\documentclass[11pt]{article}

\usepackage{amsmath, amsfonts,amsthm,amssymb,bm}
\usepackage{dsfont}

\usepackage[english]{babel}
\usepackage{float}

\usepackage[normalem]{ulem}

\usepackage{subfigure}
\usepackage{latexsym,amssymb,amsfonts,graphicx}
\usepackage{amsmath,dsfont,bbm}
\usepackage{verbatim}
\usepackage{mathrsfs}
\usepackage{bm}
\usepackage{color}
\usepackage{epsfig}
\usepackage{epstopdf}
\usepackage[title]{appendix}
\usepackage{url}
\usepackage{bm}
\usepackage{natbib}
\usepackage{tikz}
\usetikzlibrary{arrows.meta, positioning}

\usepackage[colorlinks,linkcolor=blue,citecolor=blue,anchorcolor=blue]{hyperref}

\def\esssup_#1{\underset{#1}{\mathrm{ess\,sup\, }}}
\def\essinf_#1{\underset{#1}{\mathrm{ess\,inf\, }}}

\newtheorem{theorem}{Theorem}[section]
\newtheorem{definition}{Definition}[section]
\numberwithin{equation}{section}

\newtheorem{remark}[theorem]{Remark}
\newtheorem{lemma}[theorem]{Lemma}
\newtheorem{corollary}[theorem]{Corollary}

\allowdisplaybreaks

\definecolor{Red}{rgb}{1.00, 0.00, 0.00}

\definecolor{DRed}{rgb}{0.5, 0.00, 0.00}

\definecolor{Blue}{rgb}{0.00, 0.00, 1.00}

\definecolor{Green}{rgb}{0.0, 0.4, 0.0}

\newcommand{\BU}{\mathcal{U}}
\newcommand{\db}{\operatorname{d}\! B}
\newcommand{\dd}{\operatorname{d}\! }

\newcommand{\ds}{\operatorname{d}\! s}
\newcommand{\dt}{\operatorname{d}\! t}

\newcommand{\nn}{\nonumber}

\newcommand{\argmax}{\operatorname*{arg\:max}}

\newcommand{\seta}{\mathcal{A}}
\newcommand{\setd}{\mathcal{D}}
\newcommand{\ep}{\varepsilon}

\begin{document}

\title{Optimal Reinsurance-Dividend Strategy with Fixed Transaction Costs in a Regime-Switching Brownian Risk Model: A Viscosity Solution to the Impulse Control Problem}
\author{
Wenyuan Wang\thanks{School of Mathematics and Statistics, Fujian Normal University, Fuzhou, Fujian, 350117, China. Email: \url{wwywang@xmu.edu.cn}}
\and
Zuo Quan Xu\thanks{Department of Applied Mathematics, The Hong Kong Polytechnic University, Kowloon, Hong Kong SAR, China. Email: \url{maxu@polyu.edu.hk}}
\and
Kaixin Yan\thanks{School of Mathematics and Statistics, Xi'an Jiaotong University, Xi'an, Shaanxi, 710049, China. Email: \url{kaixinyan@stu.xmu.edu.cn}}
}
\date{\today}

\maketitle
\begin{abstract}
We consider a problem of optimal proportional reinsurance-dividend distribution under a Brownian risk model, where both the drift and volatility coefficients are subject to endogenous regime-switching. Dividend payments are subject to fixed transaction costs. The problem is formulated as a two-dimensional stochastic control problem, and we prove that the value function is the unique viscosity solution of the associated Hamilton-Jacobi-Bellman equation with nonlocal operator. For almost all parameter configurations, we explicitly characterize the optimal strategy that maximizes the expected total discounted dividends net of transaction costs until ruin. The optimal dividend policy is a two-barrier impulsive strategy, while the optimal reinsurance proportion is given in feedback form. Numerical examples are provided to illustrate the optimality results.
~\\[3ex] 
\end{abstract} 

\section{Introduction}

Dividends and reinsurance constitute two fundamental forms of control in the dynamic management of an insurer's surplus. Dividend decisions allocate accumulated surplus between immediate payouts to shareholders and reserves set aside to cover potential claims and to support the underwriting of new policies. Reinsurance influences the evolution of the surplus process over time. A lower retention level reduces the insurer's claim exposure but raises the premium income ceded to the reinsurer. Dividend policy affects the current surplus level, whereas reinsurance decisions shape the future surplus trajectory through their effects on premium income and claim losses. The two controls are inherently interdependent. Insufficient dividends may fail to meet shareholder expectations, whereas excessive dividends increase the insurer's risk exposure and may, therefore, necessitate greater reinsurance protection to mitigate the elevated risk. Conversely, over-reliance on reinsurance, reflected in a low retention ratio, erodes premium income and decelerates surplus accumulation, which, in turn, constrains the capacity for future dividend distributions. Hence, the interdependence between dividend and reinsurance decisions calls for their coordinated optimization.

The classical dividend problem was introduced by de Finetti \cite{DeFinetti57}, who proposed maximizing the expected discounted dividends paid before ruin under a simple discrete time random walk model. Since then, optimal dividend policies have been studied under a wide range of surplus models and payout mechanisms. In a Brownian risk model, \cite{AsmussenTaksar97} studied the problem of maximizing expected discounted dividends until ruin and showed that the optimal policy is to always pay the maximal dividend rate under bounded dividend rates and a singular barrier strategy when dividend rates are unrestricted. For a classical Cram\'er–Lundberg model, \cite{GerberShiu06} showed the optimality of threshold dividend strategy under the bounded dividend rate constraint; while \cite{Albrecher20}, under a ratcheting constraint, characterized the value function as the unique viscosity solution of its Hamilton–Jacobi–Bellman (HJB) equation. 
For a comprehensive review, we refer to \cite{AlbrecherThonhauser09} and the references therein.
When dividend payments incur a fixed transaction cost, a sequence of infinitesimal payments is no longer economically efficient and the problem becomes one of impulse control. In this setting, for a diffusion process, \cite{Paulsen08} studied optimal dividend payments and reinvestments with fixed and proportional costs and characterized when an optimal barrier policy exists. \cite{BaiPaulsen10} considered optimal dividend payments until ruin with transaction costs for a class of diffusion processes and showed that a simple barrier strategy is not always optimal. \cite{HuntingPaulsen13} considered a diffusion with a compound Poisson jump component having negative jumps and proved the optimality of a simple lump-sum policy for a class of light-tailed jump distributions. For further studies of optimal dividend problems with fixed costs, see, among others, \cite{Cadenillas06,Paulsen07,SotomayorCadenillas13,ZhouYiu14}.

Reinsurance addresses the complementary question of how much underwriting risk the insurer should retain. Under the ruin-probability criterion, \cite{Schmidli01} derived optimal dynamic proportional reinsurance strategies in both a diffusion approximation and the classical risk model. The reinsurance decision has also been incorporated into the dividend criterion. For example, in a Brownian risk model, \cite{Hojgaard99} derived the optimal policies for risk exposure and dividend payouts under both bounded and unrestricted dividend payout rates.
In a diffusion setting with fixed dividend transaction costs, \cite{Cadenillas06} studied the joint optimization of dividend and risk policies and explicitly obtained the value function and optimal policy. In the Cram\'er--Lundberg setting, \cite{AzcueMuler05} studied optimal reinsurance-dividend payouts and proved that the optimal value function is the smallest viscosity solution of the HJB equation and that an optimal band strategy exists. In a jump-diffusion model, \cite{Belhaj10} jointly studied dividend payments and insurance against the Poisson risk, and showed that the optimal dividend policy is of barrier type, while full insurance is optimal only when the cash reserve exceeds a critical level. More recently, in a finite-horizon jump--diffusion model, \cite{Guan23} studied the optimal dividend and reinsurance problem, and proved the existence of a classical solution, the continuity, strict monotonicity, and boundedness of the dividend free boundary, and the smoothness of the reinsurance free boundary. For further studies on optimal reinsurance, we refer to \cite{AsmussenHojgaardTaksar00,ChenYuenWang21,GuanXuZhou23,WeiYangWang10,YaoYangWang16} and the references therein.

The operating environment of an insurer may evolve with business cycles, market conditions, and regulatory changes, producing shifts in underwriting profitability and claim volatility. Such external changes are commonly represented by an observable finite-state Markov chain. This modeling framework has subsequently been used to study optimal dividend and reinsurance problems under various surplus models. For a two-regime Markov-modulated Brownian risk models, \cite{SotomayorCadenillas11} studied the optimal dividend problem under bounded and unbounded dividend rates, obtained analytical solutions in both cases, and proved that the optimal policy depends on macroeconomic conditions. Focusing on unrestricted dividend rate, \cite{JiangPistorius12} further allowed finitely many regimes and regime-dependent discounting, proving the optimality of a modulated barrier strategy under positive drifts and showing the optimal strategy takes a different form when one of regimes has a small negative drift. For Markov-modulated classical compound Poisson risk models, \cite{WeiYangWang10} studied the joint reinsurance and continuous dividend-control problem and characterized the optimal value function as the unique viscosity solution of the associated HJB equation, and \cite{WeiYangWang10Impulse} considered proportional reinsurance together with impulsive dividend payouts and characterized the value function through the corresponding quasi-variational inequality.
For a Markov-modulated jump--diffusion surplus process with positive drift in every regime and exponentially distributed downward jumps, \cite{Jiang15} established the optimality of a regime-dependent barrier strategy with strictly positive barrier levels. 

When the change in operating conditions is linked directly to the insurer's financial position, surplus-dependent regime-switching provides an alternative natural description. Under this specification, crossing a prescribed surplus threshold changes the local drift and volatility without introducing an independent Markov chain. For example, \cite{KeilsonWellner78} studied oscillating Brownian motion, a driftless diffusion whose volatility takes different constant values on the two sides of an interface, and explicitly derived its transition densities, first-passage-time distributions, and occupation-time distributions. Related endogenous switching mechanisms have also been incorporated into optimal dividend problems.
\cite{WangYuZhou24} considered optimal dividend control for a surplus process with endogenous regime switching, in which the underlying risk process switches when the controlled surplus down-crosses the reorganization barrier and switches back when it up-crosses the solvency barrier.
More recently, \cite{Wang26} incorporated fixed transaction costs into De Finetti's dividend problem under an endogenous state-triggered switching mechanism, where the drift and volatility coefficients take different values on the two sides of a surplus threshold.

In this paper, motivated by these developments, we study the joint optimization of proportional reinsurance and impulsive dividend payments in a Brownian risk model with regime-switching drift and volatility coefficients. The drift and volatility alternate between $(\mu_{-},\sigma_{-})$ when the surplus is at or below a prescribed threshold $a$, and $(\mu_{+},\sigma_{+})$ when it lies above $a$. The insurer dynamically chooses a retention level $u\in[0,1]$, which scales both the drift and volatility, thereby influencing not only the current risk exposure but also the timing at which the surplus crosses the regime-switching boundary. Dividends are paid at intervention times, and each payment is subject to a fixed transaction cost $\beta>0$. The objective is to maximize the expected total discounted value of dividends, net of transaction costs, until ruin. This objective gives rise to a regular-impulsive stochastic control problem featuring two control variables. While a unified solution method for the impulsive control problems is not available, we draw inspiration from the literature (e.g., \cite{Cadenillas06}) and conjecture that the optimal dividend strategy is of a two-barrier type. Building on this conjecture, we eventually obtain explicit expressions for the optimal reinsurance-dividend strategies and the corresponding value function. 

It is worth mentioning that the drift-volatility regime‑switching makes the Hamiltonian discontinuous, and adding reinsurance control introduces nonlinearity; their delicate interplay and superposition must be handled with care. We need to locate the change points where the retention ratio switches from a feedback form to the constant 1 or vice versa, partition the domain by these change points and the regime-switching boundary (relative order unknown), and on each sub-interval construct solutions to the non-intervention part of the Hamilton-Jacobi-Bellman (HJB) equation via an elicitation-based guess-and-verify method. Meanwhile, we must check convexity/concavity requirements and inequality conditions involving both the first-order and second-order derivatives, ensuring that the candidate solution (patched from local solutions on these intervals) is a genuine solution of the non-intervention HJB equation. This constitutes a highly delicate and intricate task, demanding a deep understanding of the HJB equation, as well as logical rigor and computational skills.  Moreover, many conventional studies on dividend and reinsurance optimization (see \cite{AsmussenTaksar97,AsmussenHojgaardTaksar00,Belhaj10,Guan23,GuanXuZhou23,Hojgaard99,SotomayorCadenillas11,SotomayorCadenillas13}, among others), have shown the value function to be concave, greatly facilitating the solution of the HJB equations and the characterization of optimal strategies. In the present setting, however, the interplay between regime switching and reinsurance control destroys this concavity, making both the HJB equation harder to solve and the optimal strategy harder to identify.

However, the above construction yields a solution solely for the non‑intervention component (also referred to as the continuation part) of the HJB equation. The other component, namely the intervention part, arises due to fixed transaction costs. It is widely acknowledged that solving the HJB equation with these two coupled components is extremely challenging; indeed, no general solution method is available for equations of this type. Inspired by \cite{Cadenillas06}, we conjecture that the optimal intervention strategy (i.e., the dividend payout policy) is a two-barrier strategy. Accordingly, we decompose the solution procedure into three steps. First, we solve a family of auxiliary optimization problems, in each of which the dividend strategy is fixed as a specific two-barrier strategy while we optimize over the reinsurance strategy. This yields a family of optimal reinsurance strategies and value functions indexed by the dividend barriers. Second, we seek, among this family, the uniformly maximal value function, provided that it exists. In fact, there are a few exceptional cases where the existence of such a uniformly maximal value function is not guaranteed. For these cases, we characterize the necessary and sufficient conditions for its existence; whenever these conditions are met, the corresponding maximal value function can be determined explicitly. Finally, we verify that this maximal value function indeed solves the HJB equation, thereby proving, via the verification argument, that the corresponding two-barrier dividend strategy and reinsurance strategy are optimal for the original control problem. Fortunately, except for a few exceptional cases, this solution construction method proves successful for all other parameter configurations.

The two earlier works most closely related to the present paper are \cite{Cadenillas06} and \cite{Wang26}. The former studies the optimal reinsurance and impulsive dividend problem in a Brownian risk model with constant positive drift and volatility (no regime-switching), while the latter addresses the impulsive dividend problem in a regime-switching Brownian risk model with surplus‑dependent drift and volatility, but without reinsurance. The methodology adopted in this paper differs substantially from both. First, we solve an auxiliary problem that serves as a crucial bridge to the solution of the original control problem; by contrast, \cite{Cadenillas06} and \cite{Wang26} do not require such an auxiliary step because their frameworks do not involve the subtle interplay and superposition effects caused by the simultaneous presence of reinsurance control and regime‑switching drift and volatility. Second, strictly speaking, their solutions are only viscosity solutions to the HJB equations, since second‑order differentiability fails at intervention points and at points where the retention ratio changes. However, they do not provide a viscosity solution theory; in our more general setting, we fill this gap by rigorously developing such a theory, for instance, by adapting advanced viscosity theory and refining the dynamic programming principle, etc. The need for this is further underscored by the fact that the complexity of regime-switching leads to a few exceptional cases in which we are unable to construct a solution to the HJB equation, a difficulty also encountered in \cite{Wang26}. Finally, we note that the results of these two papers can be recovered as special cases of ours by setting, respectively, $\mu_{-}=\mu_{+}>0$, $\sigma_{-}=\sigma_{+}$ (for \cite{Cadenillas06}) and $u\equiv 1$ (for \cite{Wang26}); see Remark \ref{rem:special.cases}. In this sense, our paper offers an alternative method applicable to a broader class of problems, and the underlying idea may also provide valuable insights for other regular‑impulsive control problems.

Beyond the two closely related studies mentioned above, the present work also distinguishes itself from the broader literature. Existing dividend optimization problems with fixed transaction costs are typically formulated for spatially homogeneous uncontrolled surplus processes or for diffusion and jump‑diffusion processes with regular state‑dependent coefficients (see \cite{BaiPaulsen10,HuntingPaulsen13,Paulsen08}, etc.). In contrast, our uncontrolled process is neither spatially homogeneous nor does it admit regular coefficients. Likewise, existing studies on dividend and reinsurance optimization under regime‑switching models typically assume that regime switches are driven by an exogenous Markov chain (see \cite{JiangPistorius12,Jiang15,SotomayorCadenillas11,WeiYangWang10,WeiYangWang10Impulse}, etc.), whereas the regime‑switching in our paper is endogenous. Owing to these fundamental distinctions, none of the existing methods is directly transferable to our setting. We therefore tackle the regular-impulsive optimal control problem by developing a novel case‑by‑case explicit guess‑construct‑and‑verify methodology, which is further complemented by a newly devised viscosity‑solution framework.

The remainder of the paper is organized as follows. Section 2 formulates the original control problem and establishes the dynamic programming principle, the viscosity characterization, the comparison theorem, and the preliminary verification results. Section 3 introduces the auxiliary control problem, constructs the classical HJB solutions, and determines the candidate reinsurance-dividend strategies in the four drift configurations. Section 4 characterizes precisely when the candidate strategy obtained from the auxiliary control problem is optimal for the original control problem.

\section{Problem formulation and preliminary results}
We consider a complete filtered probability space $(\Omega, \mathcal{F}, \mathbb{F}, \mathbb{P})$, where $\mathbb{F} = (\mathcal{F}_t)_{t \ge 0}$ is the filtration generated by a standard one-dimensional Brownian motion $B = (B_t)_{t \ge 0}$ and satisfies the usual conditions.
Let $a \in (0, \infty)$, $\mu_{\pm} \in \mathbb{R}$, and $\sigma_{\pm} \in (0, \infty)$ be fixed constants. Consider the following stochastic differential equation (SDE):
\begin{equation}\label{eq:a001}
\mathrm{d}X_t=(\mu_{+}\textbf{1}_{\{X_t>a\}}+\mu_{-}\textbf{1}_{\{X_t\le a\}})\dt+(\sigma_{+}\textbf{1}_{\{X_t>a\}}+\sigma_{-}\textbf{1}_{\{X_t\le a\}})\db_t.
\end{equation} 
In \eqref{eq:a001}, the constant $a$ is referred to as the regime-switching boundary. The drift and volatility switch between $(\mu_{-},\sigma_{-})$ when the surplus is at or below $a$, and $(\mu_{+},\sigma_{+})$ when it is strictly above $a$. The existence and uniqueness of a strong solution to SDE \eqref{eq:a001} are guaranteed by Theorem 1.3 in Page 55 of \cite{LeGall84}. In this paper, $(X_t)_{t\geq 0}$ represents the surplus process of an insurance company prior to dividend payments and reinsurance purchases.
The company can manage its risk exposure through proportional reinsurance and enhance shareholder value by paying dividends. Consequently, it is natural to investigate an optimal risk control and impulsive dividend payout control problem to determine the best actions for the insurer. 

We assume that both proportional reinsurance and impulsive dividend strategies are available. Any joint risk control (reinsurance) and (right-continuous) dividend payout strategy is denoted as $(u,D)=(u_t, D_{t})_{t\geq 0}$, where $u_{t}\in[0,1]$ for all $t\geq 0$, $D_t=\sum\limits_{0\le s\le t}\Delta D_s$ with $D_{0-}=0$, and $\Delta D_{s}=D_{s}-D_{s-}\ge 0$.
Under the control $(u,D)$, the surplus process $U^{u,D}$ evolves according to the following SDE: 
\begin{equation}\label{eq:a002}
\mathrm{d}U_t^{u,D}=u_{t}\Big(\mu_{+}\textbf{1}_{\{U_t^{u,D}>a\}}+\mu_{-}\textbf{1}_{\{U_t^{u,D}\le a\}}\Big)\dt+u_{t}\Big(\sigma_{+}\textbf{1}_{\{U_t^{u,D}>a\}}+\sigma_{-}\textbf{1}_{\{U_t^{u,D}\le a\}}\Big)\db_t-\mathrm{d}D_t.
\end{equation}
The ruin time of $U^{u,D}$ is defined by
\begin{equation*}
T^{u,D}:=\inf\{t\ge0:U_t^{u,D}\leq 0\},\notag
\end{equation*}
with the convention $\inf\emptyset=\infty$.

\begin{definition}[Admissible Strategy]
\label{def.2.1}
Given an initial capital $U_{0-}=x$.
A joint reinsurance and impulsive dividend payout strategy $(u,D)=(u_{t}, D_t)_{t\ge0}$ is said to be admissible if the following conditions hold. 
\begin{itemize}
\item The SDE~\eqref{eq:a002} admits a unique strong solution $U^{u,D}$ under the strategy $(u,D)$ with the initial capital $U_{0-}=x$.
\item The process
$(u_{t})_{t\geq 0}$ is $(\mathcal{F}_{t})_{t\geq 0}$-adapted and takes values in $[0,1]$.
\item The process $(D_t)_{t \ge 0}$ is a nondecreasing, c\`adl\`ag, $(\mathcal{F}_t)_{t \geq 0}$-adapted pure jump process satisfying 
\begin{equation}
0 \le \Delta D_t := D_{t} - D_{t-} \le U^{u,D}_{t-} \vee 0, \quad t \ge 0,\nn
\end{equation}
which implies that a lump-sum dividend payment with an amount larger than the available reserve is not allowed and no ruin happens because of paying dividend. 
\end{itemize}
\end{definition}
We denote by $\Pi_x$ the collection of all admissible joint reinsurance and impulsive dividend payout strategies associated with initial capital $x$. 
It is easy to verify that $(u,D)\equiv (0, 0)\in\Pi_x\subseteq \Pi_y$ for any $x\leq y$.

Let $q>0$ be a fixed constant representing the discount rate.
Let $\beta>0$ be a fixed constant representing a transaction cost or penalty. For an admissible reinsurance and impulsive dividend payout strategy $(u,D)\in\Pi_x$, the associated reward function is defined as
\begin{equation}\label{eq:a006}
J_{u,D}(x):=\mathbb{E}_x\left[\sum_{0\le s\leq T^{u,D}}e^{-qs}\left(\Delta D_s-\beta\right)\textbf{1}_{\{\Delta D_s>0\}}\right], 
\end{equation}
where $\mathbb{E}_x$ stands for the conditional expectation given the initial value $U_{0-}^{u,D}=x$.
\begin{remark}
\label{rem.2.1.9.12}
Given any admissible strategy $(u,D)$ and initial $x>0$, we claim that $U^{u,D}_{T^{u,D}}=0$ if $T^{u,D}<\infty$. Suppose $U^{u,D}_{T^{u,D}}>0$, then by the right-continuity of the paths of $U^{u,D}$, we have $U^{u,D}_{T^{u,D}+}=U^{u,D}_{T^{u,D}}>0$, which contradicts the definition of $T^{u,D}$. Conversely, if $U^{u,D}_{T^{u,D}}<0$, then by the definition of $T^{u,D}$, we have $U^{u,D}_{T^{u,D}-}\geq 0$. This implies $$D_{T^{u,D}}-D_{T^{u,D}-}=U^{u,D}_{T^{u,D}-}-U^{u,D}_{T^{u,D}}>U^{u,D}_{T^{u,D}-}=U^{u,D}_{T^{u,D}-}\vee 0,$$
which contradicts the admissibility of $(u,D)$. 
\end{remark}
The objective is to identify an optimal reinsurance and impulsive dividend strategy $(u^{\star},D^{\star})\in\Pi_x$ such that
\begin{equation}\label{eq:a008}
J_{u^{\star},D^{\star}}(x)=\sup_{(u,D)\in\Pi_x}J_{u,D}(x),
\end{equation}
and to determine the corresponding value function
\begin{equation}
V(x):=\sup_{(u,D)\in\Pi_x}J_{u,D}(x).\nn
\end{equation}
Note that for $x\le0$, the corresponding stopping time satisfies $T^{u,D}=0$ for any admissible strategy $(u,D)\in\Pi_x$, so $V(x)=0$. Therefore, it suffices to consider the case $x>0$.

\begin{lemma}
\label{lem2.1}
The value function $V$ is non-negative, non-decreasing, and locally Lipschitz continuous on $(0,\infty)$ and satisfies
\begin{align}
\label{1-Lip}
\left\{
\begin{array}{ll}
0\le V(x)\le x+\frac{\left|\mu_{-}\right|+\left|\mu_{+}\right|}{q},\quad x\ge0,& \medskip\\
V(x)-V(y)\geq x-y-\beta, \quad x\geq y\geq 0,& 
\end{array}
\right.
\end{align}
which implies $\lim_{x\to\infty} \frac{V(x)}{x}=1$. 
In addition, 
\begin{align}
\label{shrinkage.of.set.of.add.str.}
V(x)=\sup_{(u,D)\in\Pi^{+}_x}J^+_{u,D}(x),~~~x\ge0,
\end{align}
where
\begin{equation}\label{8.8.eq:a009}
\Pi_{x}^{+}:=\big\{(u,D)\in\Pi_{x}: \text{ for each } t\ge0, \text{ either }\Delta D_t>\beta \text{ or }\Delta D_t=0\big\},
\end{equation}
and
\begin{equation} \label{eq:a006b}
J^+_{u,D}(x):=\mathbb{E}_x\left[\sum_{0\le s\leq T^{u,D}}e^{-qs}\left(\Delta D_s-\beta\right)^+ \right].
\end{equation}
Furthermore, for each fixed $x\geq 0$, $V(x)$ is non-increasing and convex in $\beta$.
\end{lemma}

\begin{proof}
The non-negativity, monotonicity and linear growth of $V$ in $x$ follow an adaptation of the proof of Proposition 3.1 in \cite{Cadenillas06}. 
Given an initial surplus $x>y\geq 0$, applying a dividend strategy that distributes a lump sum payment of $x-y$ at time $0$ yields the second inequality of \eqref{1-Lip}. 
By an analogous proof of Lemma 2.1 of \cite{Wang26}, we obtain \eqref{shrinkage.of.set.of.add.str.}. 
The monotonicity and convexity of $V$ in $\beta$ are direct consequences of the fact that $J_{u,D}(x)$ is non-increasing and linear in $\beta$.

We next prove the local Lipschitz continuity of $V$ on $(0,\infty)$. Fix any $x_{0}>0$. Suppose $x_{0}\leq x< y\leq x_{0}+1$. For any $\ep>0$, by the definition of $V$, there exists a reinsurance-dividend strategy $(u^{\ep},D^{\ep})\in\Pi^{+}_y$ such that
\begin{align}
\label{7.19.2.7}
J_{u^{\ep},D^{\ep}}(y)> V(y)-\ep.
\end{align} 
Construct a reinsurance-dividend strategy $(u,D)$ as
\begin{align}
\label{7.19.2.9}
\left\{
\begin{array}{ll}
u_{t}= 1,\,\, D_{t}= 0, &\quad t\in [0,\tau_{y}^{+}), \\
u_{t}=u_{t}^{\ep},\,\, D_{t}=D_{t}^{\ep}, &\quad t\in [\tau_{y}^{+},\infty). 
\end{array}
\right.
\end{align} 
Then it is easy to verify $(u,D)\in\Pi^{+}_x$.

Suppose $X$ is driven by \eqref{eq:a001} with $X_0=x$. 
Let $\tau_{z}:=\inf\{t\geq 0: X_{t}=z\}$. 
By Lemma 2.4 of \cite{Wang26}, we have
\begin{align}
\label{7.19.2.8}
\mathbb{E}_{x}\left[e^{-q\tau_{y}}\mathbf{1}_{\{\tau_{y}<\tau_{0}\}}\right]=\frac{\ell(x)}{\ell(y)},
\end{align}
where the function $\ell\in C^{1}((0,\infty))$ satisfies $\ell(0+)=0$ and $\ell^{\prime}>0$. 

Then, by \eqref{7.19.2.7}-\eqref{7.19.2.9}, we have
\begin{align}
\label{2.19.2.10}
V(x)\geq J^+_{u,D}(x)=&~\mathbb{E}_x\left[\sum_{0\le s< \tau_{y}\wedge T^{u,D}}e^{-qs}\left(\Delta D_s-\beta\right)^+ \right] + \mathbb{E}_x\left[\sum_{\tau_{y}\wedge T^{u,D}\le s\leq T^{u,D}}e^{-qs}\left(\Delta D_s-\beta\right)^+ \right]
\nn\\
\geq &~
\mathbb{E}_x\left[\mathbf{1}_{\{\tau_{y}<T^{u,D}\}}\sum_{\tau_{y}\le s\leq T^{u,D}}e^{-qs}\left(\Delta D_s-\beta\right)^+ \right]
\nn\\
= &~
\mathbb{E}_x\left[\mathbf{1}_{\{\tau_{y}<\tau_{0}\}}\mathbb{E}_x\left[\left.\sum_{0\le s\leq T^{u,D}-\tau_{y}}e^{-q({s+\tau_{y}})}\left(\Delta D_{s+\tau_{y}}-\beta\right)^+
\right|\mathcal{F}_{\tau_{y}}\right]\right]
\nn\\
= &~
\mathbb{E}_x\left[e^{-q\tau_{y}}\mathbf{1}_{\{\tau_{y}<\tau_{0}\}}\mathbb{E}_x\left[\left.\left[\sum_{0\le s\leq T^{u^{\ep},D^{\ep}}}e^{-qs}\left(\Delta D_{s}^{\ep}-\beta\right)^+ 
\right]\circ \theta_{\tau_{y}}\right|\mathcal{F}_{\tau_{y}}\right]\right]
\nn\\
= &~
\mathbb{E}_x\left[e^{-q\tau_{y}}\mathbf{1}_{\{\tau_{y}<\tau_{0}\}}
J^+_{u^{\ep},D^{\ep}}(y)\right]
\nn\\
\geq &~
\mathbb{E}_x\left[e^{-q\tau_{y}}\mathbf{1}_{\{\tau_{y}<\tau_{0}\}}
\right] \left(V(y)-\ep\right) 
= 
\frac{\ell(x)}{\ell(y)}\left(V(y)-\ep\right),
\end{align}
where $\theta$ is the shift operator.
By the arbitrariness of $\ep$, and the non-decreasing property of $V$, we obtain
\begin{align*}
0\geq&~ V(x)-V(y) \geq \left(\frac{\ell(x)}{\ell(y)}-1\right)V(y)=-(y-x)\ell'(z)\frac{V(y)}{\ell(y)},
\end{align*} 
for some $z\in [x,y]\subseteq [x_0, x_0+1]$. Because $\ell\in C^{1}((0,\infty))$, $\ell>0$ and \eqref{1-Lip}, we see $\ell'(z)\frac{V(y)}{\ell(y)}$ is uniformly bounded on $[x_0, x_0+1]$ for all $y,z\in [x_0, x_0+1]$. This completes the proof of the local Lipschitz continuity of $V$.
\end{proof}

For continuous function $\phi:[0,\infty)\mapsto\mathbb{R}$, we define the maximum utility operator
$\mathcal{M}$ by
\begin{equation}\label{eq:a010}
\mathcal{M}\phi(x):=\sup\limits_{0\leq\eta\leq x}\left(\phi(x-\eta)+\eta-\beta\right) -\phi(x)
=\sup\limits_{0\leq \eta\leq x}\left(\phi(\eta)-\eta\right)+x-\phi(x)-\beta. 
\nn
\end{equation}
The optimal $\eta$ in the above operator will determine the optimal dividend payout amount at surplus level $x$.

The following Lemma \ref{lem2.3} states the dynamic programming principle.

\begin{lemma}
\label{lem2.3}
Let $\Pi^{+}_x$ be defined by \eqref{8.8.eq:a009}. Given any stopping time $\tau$, we have
\begin{align}
V(x)=\sup_{(u,D)\in\Pi^{+}_x}\mathbb{E}_{x}\left[e^{-q (\tau\wedge T^{u,D})}V\left(U^{u,D}_{\tau\wedge T^{u,D}}\right)+\sum_{s\leq \tau\wedge T^{u,D}} e^{-q s} \left(\Delta D_{s}-\beta\right)^+\right],~~x>0.
\nn
\end{align}
If $\mathcal{M}V(x)<0$, then $\Pi^{+}_x$ in above can be replaced by its subset 
\begin{align}
\label{Pitilde}
\widetilde{\Pi}^{+}_x:=\{(u,D)\in\Pi^{+}_x: \Delta D_{0}=0\}.
\end{align}
\end{lemma}

\begin{proof}
We omit the proof for the first claim as it is similar to the proof of Lemma
4.1 in \cite{Albrecher20}. To prove the second claim, we only need to show that
\begin{align}
\label{8.8.shrinkage.of.set.of.add.str.}
V(x)=\sup_{(u,D)\in\widetilde{\Pi}_x^{+}}J^+_{u,D}(x).
\end{align}
Indeed, for any $(u,D)\in \Pi_x^{+}$ with $\Delta D_{0}>0$ (hence, $0< \Delta D_{0}\leq x$), we have
\begin{align}
J^+_{u,D}(x)&=\Delta D_{0}-\beta+J^+_{u,D}(x-\Delta D_{0})
\nn\\
&\leq \Delta D_{0}-\beta+V(x-\Delta D_{0})
\nn\\
&\leq \sup_{0\leq\eta\leq x}\left[V(x-\eta)+\eta-\beta\right]
\nn\\
&=\mathcal{M}V(x)+V(x)
\nn\\
&<V(x),\nn
\end{align}
which, together with \eqref{shrinkage.of.set.of.add.str.}, yields \eqref{8.8.shrinkage.of.set.of.add.str.}.
\end{proof}

We define
\begin{align*}
\seta:=\Big\{f\in C([0,\infty)): &~\textrm{nonnegative, polynomial growth, and locally} \\
&~\textrm{Lipschitz continuous on $(0,\infty)$,~$f(0)=0$}\Big\}.
\end{align*} 

\begin{lemma}
The value function $V$ is right-continues at $0$ so that $V\in\seta$. 
\end{lemma}
\begin{proof}
Thanks to Lemma \ref{lem2.1}, the right-continuity of $V$ at $0$ implies $V\in\seta$. By the first inequality of \eqref{1-Lip} and the monotonicity of $V$, we know that $V(0+)=\lim\limits_{x\rightarrow 0^{+}}V(x)\in [0,\infty)$. Our goal is to prove $V(0+)=0$. 

We first define a new problem 
\begin{align}
\widetilde{V}(x)=\sup_{(u,D)\in\Pi_x^{+}}J^+_{u,D}(x), 
\end{align}
where the stated process satisfies 
\begin{equation} \label{eq:a002a}
\mathrm{d}U_t^{u,D}=u_{t} \big(\mu_{-} \dt+ \sigma_{-} \db_t\big)-\dd D_t, ~~\mathrm{d}U_0^{u,D}=x. 
\end{equation}
This new problem is the special case of our original problem \eqref{eq:a008} where $\mu_{+}=\mu_{-}$ and $\sigma_{+}=\sigma_{-}$. So by Lemma \ref{lem2.1}, we know that $\widetilde{V}$ is non-decreasing, locally Lipschitz continuous on $(0,\infty)$ and linear growth.

The critical difference between the new and old problems lies in that the state process \eqref{eq:a002a} is a linear SDE, whereas \eqref{eq:a002} is not. 
The linearity allows us to find better properties for $\widetilde{V}$.

By the definition of $V$ and the dynamic programming principle Lemma \ref{lem2.3}, there exists $(u^{\ep},D^{\ep})\in \Pi_x^{+}$ such that 
\begin{align}
\label{2.19.9.12}
V(x)-\ep< &~ J_{u^{\ep},D^{\ep}}(x)
\nn\\
\leq &~ \mathbb{E}_x\left[\mathbf{1}_{\{ T^{u^{\ep},D^{\ep}}\leq T_{a}^{+}\}}\sum_{0\le s\leq T^{u^{\ep},D^{\ep}}}e^{-qs}\left(\Delta D_s^{\ep}-\beta\right)^{+}\right]\nn\\
&~+\mathbb{E}_x\left[\mathbf{1}_{\{T_{a}^{+}< T^{u^{\ep},D^{\ep}}\}}\left(\sum_{0\le s\leq T_{a}^{+}}e^{-qs}\left(\Delta D_s^{\ep}-\beta\right)^{+}+V(a)\right)\right]\nn\\
= &~ \mathbb{E}_x\left[\sum_{0\le s\leq T_{a}^{+}\wedge T^{u^{\ep},D^{\ep}}}e^{-qs}\left(\Delta D_s^{\ep}-\beta\right)^{+}\right] +\mathbb{E}_x\left[\mathbf{1}_{\{T_{a}^{+}< T^{u^{\ep},D^{\ep}}\}} \right]V(a)\nn\\
\leq &~ \widetilde{V}(x)+\mathbb{E}_x\left[\mathbf{1}_{\{T_{a}^{+}< T^{u^{\ep},D^{\ep}}\}} \right]V(a),
\end{align}
where
$$T_{a}^{+}:=\inf\left\{t\geq 0: x+\int_{0}^{t}\mu_{-} u_{s}^{\ep}\ds+\int_{0}^{t}\sigma_{-} u_{s}^{\ep}\db_s-D_t^{\ep}>a\right\},$$ and the last inequality is due to the fact that the state processes in \eqref{eq:a002} and \eqref{eq:a002a} are coincide before $T_{a}^{+}$. 
Let
$$\tau_{z}:=\inf\left\{t\geq 0: \int_{0}^{t}\mu_{-} u_{s}^{\ep}\ds+\int_{0}^{t}\sigma_{-} u_{s}^{\ep}\db_s=z\right\}.$$ Then since $D^{\ep}\geq 0$, we have $T_{a}^{+}\geq \tau_{a-x}\geq \tau_{a/2}$ and 
$T^{u^{\ep},D^{\ep}}\leq \tau_{-x}$ when $0<x<a/2$. So
\begin{align}
\mathbb{P}_x\left(T_{a}^{+}< T^{u^{\ep},D^{\ep}}\right)&\leq \mathbb{P}\left(\tau_{a/2}< \tau_{-x}\right) 
\rightarrow \mathbb{P}\left(\tau_{a/2}< \tau_{0}\right)=0,\quad \text{as }\,\, x\rightarrow 0+.
\end{align}
Sending $\ep\to 0$ and $x\to 0$ in \eqref{2.19.9.12} yields 
\begin{align*}
V(0+)=\lim_{x\to 0+} V(x) \leq &~ \lim_{x\to 0+} \widetilde{V}(x)=:\widetilde{V}(0+).
\end{align*}
Hence our problem is reduced to show $\widetilde{V}(0+)=0$. 

For notation simplicity, we set $u=D=0$ after the ruin time in the state process \eqref{eq:a002a}. 
Given any initial state $x_i>0$ and any admissible control $(u^i, D^i)\in \Pi_{x_i}^{+}$, we let $\BU^i$ be the corresponding state process and $T^{(i)}$ the ruin time, $i=1,2$.
Then 
\begin{align*}
\begin{cases} 
\mathrm{d}\left(\BU^{1}_t+\BU^{2}_t\right)=\mu_{-} \left( {u}_{t}^{1}+ {u}_{t}^{2}\right)\dt+\sigma_{-} \left( {u}_{t}^{1}+ {u}_{t}^{2}\right)\db_t-\mathrm{d}\left( {D}_t^{1}+ {D}_t^{2}\right),~~ t \geq 0,& \\
\BU^{1}_0+\BU^{2}_0=x_1+x_2,& 
\\
T^{(3)}:=\inf\{t\geq 0: \BU^{1}_t+\BU^{2}_t\leq 0\}\geq T^{(1)}\vee T^{(2)},&
\end{cases}
\end{align*}
where the last inequality is due to that $\BU^{i}=0$ after the ruin time $T^{(i)}$, $i=1,2$. 
Then up to $T^{(3)}$, $\BU^{1}+\BU^{2}$ is the state process under the control $({u}^{1}+ {u}^{2}, {D}^{1}+ {D}^{2})$ and initial state $x_1+x_2$, so
\begin{align*}
\widetilde{V}(x_1+x_2) 
\geq &~\mathbb{E}_{x_1+x_2}\left[\sum_{0\le s\leq T^{(3)}}e^{-qs}\left(\Delta D^{1}_s+\Delta D^{2}_s-\beta\right)\textbf{1}_{\{\Delta D^{1}_s+\Delta D^{2}_s>0\}}\right]
\nn\\
\geq&~ \mathbb{E}_{x_1+x_2}\left[\sum_{0\le s\leq T^{(3)}}e^{-qs}\Big(\left(\Delta D^{1}_s-\beta\right)\textbf{1}_{\{\Delta D^{1}_s>0\}} + \left(\Delta D^{2}_s-\beta\right)\textbf{1}_{\{\Delta D^{2}_s>0\}}\Big)\right]
\nn\\
=&~ \mathbb{E}_{x_1}\left[\sum_{0\le s\leq T^{(3)}}e^{-qs}\left(\Delta D^{1}_s-\beta\right)^+\right] +\mathbb{E}_{x_2}\left[\sum_{0\le s\leq T^{(3)}}e^{-qs}\left(\Delta D^{2}_s-\beta\right)^+\right]
\nn\\
\geq &~ \mathbb{E}_{x_1}\left[\sum_{0\le s\leq T^{(1)}}e^{-qs}\left(\Delta D^{1}_s-\beta\right)^+\right] +\mathbb{E}_{x_2}\left[\sum_{0\le s\leq T^{(2)}}e^{-qs}\left(\Delta D^{2}_s-\beta\right)^+\right]
\nn\\
= &~J_{u^1,D^1}^{+}(x_1)+J_{u^2,D^2}^{+}(x_2).
\end{align*}
By arbitrariness, we conclude 
\[\widetilde{V}(x_1+x_2)\geq \widetilde{V}(x_1)+\widetilde{V}(x_2),~~x_1,~x_2>0.\]
Mathematical induction implies, for all natural numbers $n$, that 
$$0\leq \widetilde{V}\Big(\dfrac{1}{n}\Big)\leq \frac{1}{n}\widetilde{V}(1).$$
Therefore, by monotonicity, 
\begin{align*} 
\widetilde{V}(0+)= \liminf_{x\to 0+} \widetilde{V}(x)=0.
\end{align*}
This completes the proof. 
\end{proof}

\begin{lemma}
\label{lem.8.6.2.4}
If $\phi$ is continuous $[0,\infty)$ (resp. locally Lipschitz continuous on $(0,\infty)$), then so is the map
\begin{align*}
x\mapsto \mathcal{M} \phi(x), 
\end{align*}
on the same domain. 
\end{lemma}

\begin{proof}
We write $g(x)=\sup\limits_{0\leq \eta\leq x} \phi_1(\eta)$, where $\phi_1(\eta)=\phi(\eta)-\eta$.

Now we first assume $\phi$ is continuous on $[0,\infty)$. Then so is $\phi_1$. 
To show the map $x\mapsto \mathcal{M} \phi(x)$ is continuous, it is suffices to prove 
$g$ is continuous on $[0,\infty)$. For any $0\leq x<y$, 
\begin{align*}
0\leq g(y)-g(x)&=\sup\limits_{0\leq \eta\leq y} \phi_1(\eta) 
-\sup\limits_{0\leq \eta\leq x} \phi_1(\eta) \\
&= \max\{\sup\limits_{x\leq \eta\leq y} \phi_1(\eta),~\sup\limits_{0\leq \eta\leq x}\phi_1(\eta)\}-\sup\limits_{0\leq \eta\leq x}\phi_1(\eta)\\
&= \max\{\sup\limits_{x\leq \eta\leq y} \phi_1(\eta)-\sup\limits_{0\leq \eta\leq x}\phi_1(\eta),~0\} \\
&\leq \max\{\sup\limits_{x\leq \eta\leq y} \phi_1(\eta)- \phi_1(x),~0\}\\
&=\sup\limits_{x\leq \eta\leq y} \phi_1(\eta)- \phi_1(x). 
\end{align*}
Because $\phi_1$ is continuous, the RHS converges to 0 as $y\to x+$ and as $x\to y-$, which respectively leads to $\lim_{y\to x+}g(y)=g(x)$ and $\lim_{x\to y-}g(x)=g(y)$. Hence, we conclude $g$ is continuous on $[0,\infty)$, completing the proof of the continuous case.

We now assume $\phi$ is locally Lipschitz continuous on $(0,\infty)$. Then so is $\phi_1$. 
To show the map $x\mapsto \mathcal{M} \phi(x)$ is locally Lipschitz continuous, it is suffices to prove $g$ is locally Lipschitz continuous on $[0,\infty)$. 
Fix any $x_0\in (0,\infty)$. Since $\phi_1$ is locally Lipschitz continuous on $(0,\infty)$, 
there exists a constant $L=L(x_0)>0$ such that 
\begin{align*}
\left|\phi_1(y)-\phi_1(x)\right|\leq L \left(y-x\right),~~x_0/2\leq x<y\leq 2x_0,
\end{align*}
which implies that 
\begin{align*}
0\leq g(y)-g(x)& \leq \sup\limits_{x\leq \eta\leq y} \phi_1(\eta)- \phi_1(x) \leq \phi_1(x)+L(y-x)- \phi_1(x) =L(y-x).
\end{align*}
This completes the proof. 
\end{proof}

We will use viscosity technique to study the control problem \eqref{eq:a008}. 
For any function $f\in C^2(\mathbb{R}_{+})$ and constant $u\in[0,1]$, define the operator $\mathcal{L}^u$ as
\begin{equation}\label{eq:a011}
\mathcal{L}^uf(x):=\frac{1}{2}(\sigma_{+}^2\textbf{1}_{\{x>a\}}+\sigma_{-}^2\textbf{1}_{\{x\le a\}})u^2f^{\prime\prime}(x)+(\mu_{+}\textbf{1}_{\{x>a\}}+\mu_{-}\textbf{1}_{\{x\le a\}})uf^{\prime}(x)-qf(x).\nn
\end{equation}
We notice that the map $x\mapsto \mathcal{L}^uf(x)$ is left-continuous on $(0,\infty)$,
and continuous on 
\begin{align}\label{def:setd}
\setd:=(0,a)\cup(a,\infty).
\end{align}
This fact may be used without claim in the subsequent analysis.

Furthermore, with the help of Lemma \ref{lem2.1}, we can motivate the HJB equation associated with the control problem \eqref{eq:a008} as
\begin{align}
\label{7.17.HJB}
\max\Big\{\max_{u\in[0,1]}\mathcal{L}^{u}f(x),~\mathcal{M}f(x)\Big\}=0.
\end{align}
Since the operator $\mathcal{L}^{u}$ is not continuous on $(0,\infty)$, it forces us to define the 
viscosity super-solution on $\setd$ (see \eqref{def:setd}). 

\begin{definition}
\label{def.8.6.2.3}
A locally Lipschitz continuous function $u: (0,\infty)\rightarrow \mathbb{R}$ is called a viscosity super-solution to \eqref{7.17.HJB} if it satisfies the following conditions. 
Given any point $x\in \setd$, and any function $\phi\in C^{2}((0,\infty))$ such that $u(x)=\phi(x)$ and $u\geq \phi$ on $(0,\infty)$, it holds that 
\begin{align}
\label{7.17.HJB.1}
\max\{\max_{u\in[0,1]}\mathcal{L}^{u}\phi(x),~\mathcal{M}u(x)\}\leq 0.
\end{align}
The function $\phi$ is called a super-solution test function for $u$ at $x$.

A locally Lipschitz continuous function $v: (0,\infty)\rightarrow \mathbb{R}$ is called a viscosity sub-solution to \eqref{7.17.HJB} if it satisfies the following conditions. 
Given any point $x\in (0,\infty)$, and any function $\varphi\in C^{2}((0,\infty))$ such that $v(x)=\varphi(x)$ and $v\leq \varphi$ on $(0,\infty)$, it always holds that
\begin{align}
\label{7.17.HJB.1}
\max\{\max_{u\in[0,1]}\mathcal{L}^{u}\varphi(x),~\mathcal{M}v(x)\}\geq 0.
\end{align}
The function $\varphi$ is called a sub-solution test function for $v$ at $x$.

A locally Lipschitz continuous function that is both a viscosity super-solution and sub-solution to \eqref{7.17.HJB} is called a viscosity solution to it.
\end{definition}

The following Theorem \ref{8.7.thm.2.5} states that the value function $V\in\seta$ is a viscosity solution to the HJB equation \eqref{7.17.HJB}. 

\begin{theorem}[Existence]
\label{8.7.thm.2.5}
The value function $V$ is a viscosity solution of \eqref{7.17.HJB}.
\end{theorem}

\begin{proof}
\textbf{(Super-solution.)} 
Fix any $x\in\setd$. Let $\phi\in C^{2}((0,\infty))$ be a super-solution test function for $V$ at $x$, so that $V(x)=\phi(x)$ and $V\geq \phi$ on $(0,\infty)$. 
Our below discussion will be restricted to a small interval $[x/2,2x]$. Since $V\geq 0$, we may assume $\phi=0$ on $(0,x/3]\cup[3x,\infty)$ which implies 
\begin{align}
\label{dominated1} 
\sup_{u\in[0,1]}\sup\limits_{z\in (0,\infty)}|\mathcal{L}^{u}\phi(z)|<\infty.
\end{align} 
Fix $0<\delta<\min\{x/2, |x-a|/2\}$. Then $[x-\delta,x+\delta]$ is contained in $(0,a)$ when $x\in (0,a)$ and in $(a,\infty)$ when $x\in (a,\infty)$. 
Fix any constant $u_0\in[0,1]$. 
Define $\BU_t=U_{t}^{u\equiv u_{0},D\equiv0}$ and 
$$\tau=\inf\{t\geq 0: \left|\BU_t-x\right|> \delta\}.$$
Then $\BU$ has continuous sample paths on $[0,\tau]$. 

Since $V$ is locally Lipschitz continues on $(0,\infty)$, we may 
choose $L$ to be a Lipschitz constant of $V$ on $\left[x-\delta,x+\delta\right]$. For an integer $N\geq1$, divide $\left[x-\delta,x+\delta\right]$ into $N$ equal sub-intervals $[y_i,y_{i+1})$ with $\,i=0,1,\ldots,N-1.$ Here and below $[y_{N-1},y_{N})$ means $[y_{N-1},y_{N}]$.
For every $i$, choose an $\ep$-optimal strategy $(u^{\ep,i},D^{\ep,i})\in \Pi_{y_i}^{+}$ for the initial surplus $y_i$ such that 
$$J^+_{u^{\ep,i},D^{\ep,i}}(y_i)> V(y_i)-\ep.$$
For small $h>0$, consider a special admissible reinsurance-dividend strategy $(u,D)\in \Pi_{x}^+$ associated with initial surplus $x$ such that: $u_t\equiv u_{0}$ and $D_t\equiv 0$ for all $t\in[0,h\wedge \tau)$, and, $u_t=\sum_{i=0}^{N-1}\mathbf{1}_{\{\BU_{h\wedge\tau}\in[y_i,y_{i+1})\}}u^{\ep,i}_t$ and $D_t=\sum_{i=0}^{N-1}\mathbf{1}_{\{\BU_{h\wedge\tau}\in[y_i,y_{i+1})\}}D^{\ep,i}_t$ for $t\in[h\wedge \tau,\infty)$. 
Then 
\begin{align} 
\label{8.18.2.20}
V(x)\geq~ J^+_{u,D}(x) 
=&~\mathbb{E}_x\left[\sum_{h\wedge \tau\le s\leq T^{u,D}}e^{-qs}\left(\Delta D_s-\beta\right)^+ \right]\nn\\
\geq&~ \mathbb{E}_x\left[e^{-q(h\wedge \tau)}\mathbb{E}_x\left[\left.\sum_{0\le s\leq T^{u,D}-h\wedge \tau}e^{-qs}\left(\Delta D_{s+h\wedge \tau}-\beta\right)^+ \right|{\mathcal{F}_{h\wedge \tau}}\right]\right]\nn\\
=&~ \mathbb{E}_x\left[e^{-q(h\wedge \tau)}\mathbb{E}_{\BU_{h\wedge \tau}}\left[\sum_{0\le s\leq T^{u^{\ep},D^{\ep}}}e^{-qs}\left(\Delta D_{s}^{\ep}-\beta\right)^+ \right]\right]\nn\\
=&~\mathbb{E}_x\left[e^{-q(h\wedge \tau)}J^+_{u,D}(\BU_{h\wedge \tau})\right]
\nn\\
= &~ \mathbb{E}_x\left[e^{-q(h\wedge \tau)}\sum_{i=0}^{N-1}J^+_{u^{\ep,i},D^{\ep,i}}(\BU_{h\wedge \tau})\mathbf{1}_{\{\BU_{h\wedge \tau}\in[y_i,y_{i+1})\}}\right]\nn\\
\geq &~ \mathbb{E}_x\left[e^{-q(h\wedge \tau)}\sum_{i=0}^{N-1}J^+_{u^{\ep,i},D^{\ep,i}}(y_i)\mathbf{1}_{\{\BU_{h\wedge \tau}\in[y_i,y_{i+1})\}}\right]\nn\\
\geq &~\mathbb{E}_x\left[e^{-q(h\wedge \tau)}\left(\sum_{i=0}^{N-1}(V(y_i)-\ep)\mathbf{1}_{\{\BU_{h\wedge \tau}\in[y_i,y_{i+1})\}}\right)\right]\nn\\
\geq&~ \mathbb{E}_x\left[e^{-q(h\wedge \tau)}\left(\sum_{i=0}^{N-1}(V(\BU_{h\wedge \tau})-L(\BU_{h\wedge \tau}-y_i)-\ep)\mathbf{1}_{\{\BU_{h\wedge \tau}\in[y_i,y_{i+1})\}}\right)\right]\nn\\
\geq&~ \mathbb{E}_x\left[e^{-q(h\wedge \tau)}\left(V(\BU_{h\wedge \tau})-L \frac{\delta}{N}-\ep\right)\right].
\end{align}
Letting $N\uparrow\infty$ and $\ep\downarrow 0$ in \eqref{8.18.2.20} yield
\begin{align}
\label{7.16.2.11}
V(x)\geq&~ \mathbb{E}_x\left[e^{-q(h\wedge \tau)} V(\BU_{h\wedge \tau}) \right]. 
\end{align}
Since $V\geq \phi$, It\^{o}'s formula yields 
\begin{align}
\label{7.16.2.11}
V(x) 
\geq&~\mathbb{E}_x\left[e^{-q(h\wedge \tau)}\phi(\BU_{h\wedge \tau}) \right]
\nn\\
=&~\phi(x)+\mathbb{E}_x\left[\int_{0}^{h\wedge \tau}e^{-q s}\mathcal{L}^{u_{0}}\phi(\BU_s)\ds\right]
\nn\\
&+\mathbb{E}_x\left[\int_{0}^{h\wedge \tau} e^{-q s}\phi^{\prime}(\BU_s)u_{0}\left(\sigma_{+}\mathbf{1}_{\{\BU_s>a\}}+\sigma_{-}\mathbf{1}_{\{\BU_s\leq a\}}\right)\db_{s}\right]\nn\\
=&~\phi(x)+\mathbb{E}_x\left[\int_{0}^{h\wedge \tau}e^{-q s}\mathcal{L}^{u_{0}}\phi(\BU_s)\ds\right].
\end{align}
Since $V(x)=\phi(x)$, it follows
\begin{align}
\label{7.16.2.12}
0\geq&~ \mathbb{E}_{x}\left[\frac{1}{h}\int_{0}^{h\wedge \tau}e^{-q s}\mathcal{L}^{u_{0}}\phi(\BU_s)\ds\right]. 
\end{align}
Since $\BU$ has continuous sample paths on $[0,\tau]$ and the function $y\mapsto \mathcal{L}^{u_{0}}\phi(y)$ is continuous on $\setd$, we have 
\begin{align*}
\lim_{h\to 0}\frac{1}{h}\int_{0}^{h\wedge \tau}e^{-q s}\mathcal{L}^{u_{0}}\phi(\BU_s)\ds= \mathcal{L}^{u_{0}}\phi(x), \quad \mathbb{P}_{x}-\text{a.s..} 
\end{align*}
Therefore, by \eqref{7.16.2.12} and \eqref{dominated1}, the dominated convergence theorem yields 
\begin{align}
\label{7.17.2.13}
\mathcal{L}^{u_{0}}\phi(x)\leq 0.
\end{align}
Combining \eqref{7.17.2.13} and the arbitrariness of $u_{0}\in[0,1]$ yields
\begin{align}
\label{7.17.2.15}
\max_{u\in[0,1]}\mathcal{L}^{u}\phi(x)\leq 0. 
\end{align}
Other other hand, by the second inequality of \eqref{1-Lip}, we have
\begin{align}
\label{7.17.2.16}
\mathcal{M}V(x)=&~\sup\limits_{0\leq\eta\leq x}\left(V(x-\eta)+\eta-\beta\right)-V(x)
\nn\\
=&\sup\limits_{0\leq y\leq x}\left(V(y)+x-y-\beta-V(x)\right)\leq 0.
\end{align}
Combining \eqref{7.17.2.15} and \eqref{7.17.2.16}, we conclude that $V$ is a viscosity super-solution to \eqref{7.17.HJB}.

\textbf{(Sub-solution.)}
We now prove that $V$ is a viscosity sub-solution to \eqref{7.17.HJB} by a contradiction argument (see, e.g., Albrecher et al. \cite{Albrecher20}). 

Assume that $V$ is not a viscosity sub-solution to \eqref{7.17.HJB}. 
Then, by the definition of viscosity sub-solution, there should exist $x_{0}\in(0,\infty)$, $\ep_{0}>0$ and a sub-solution test function $\psi_{0}$ such that $\psi_{0}\in C^{2}((0,\infty))$, $\psi_{0} \geq V$, $\psi_{0}(x_{0})= V(x_{0})$, and
\begin{align}
\label{7.17.2.19}
&\max\{\max_{u\in[0,1]}\mathcal{L}^{u}\psi_{0}(x_{0}),~\mathcal{M}V(x_{0})\}< -2\ep_{0}<0.
\end{align}
Because of left-continuity, the above estimate still holds if we move $x_0$ to its left a little bit, so we may assume $x_0\neq a$. 

Set 
\begin{equation}
\psi(x):=
\psi_{0}(x)+ (x-x_{0})^{4}, \quad x\in(0,\infty).
\end{equation}
Then, it is easy to see that $\psi\in C^{2}((0,\infty))$, $\psi^{\prime\prime}(x_{0})=\psi_{0}^{\prime\prime}(x_{0})$, $\psi^{\prime}(x_{0})=\psi_{0}^{\prime}(x_{0})$, $\psi(x_{0})=\psi_{0}(x_{0})=V(x_{0})$, and $\psi\geq \psi_{0}\geq V$ on $(0,\infty)$. Therefore, $\mathcal{L}^{u}\psi(x_{0})=\mathcal{L}^{u}\psi_{0}(x_{0})$ for all $u\in[0,1]$ and $\psi$ is a sub-solution test function for $V$ at $x_0$. 
By \eqref{7.17.2.19}, 
\begin{align}
\label{A.2.v1} 
\max\{\max_{u\in[0,1]}\mathcal{L}^{u}\psi(x_{0}),~\mathcal{M}V(x_{0})\}< - 2\ep_{0} <0.
\end{align}
By \eqref{A.2.v1}, $x_0\neq a$ and the twice continuous differentiability of $\psi$, 
there exists some small $h\in (0,x_{0}\wedge \frac{\beta}{2})$ with $[x_0-h,x_0+h]$ contained either in $(0,a)$ or in $(a,\infty)$, such that 
\begin{align}
\label{A.3.v1}
\max\{\max_{u\in[0,1]}\mathcal{L}^{u}\psi(x),~\mathcal{M}V(x)\}
< - \ep_{0} <0,~~x\in[x_{0}-h,x_{0}+h].
\end{align}
In addition, 
\begin{align}
\label{A.4.v1}
\psi(x)&=\psi_{0}(x)+ (x-x_{0})^{4}
\geq \psi_{0}(x)+ h^{4}
\geq V(x)+ h^{4},~~x\notin(x_{0}-h,x_{0}+h). 
\end{align}
Put
$$\ep:=\min\Big\{ \ep_0 ,~q \ep_0 ,~q h^{4}\Big\}.$$
Then, by \eqref{A.3.v1} and \eqref{A.4.v1}, we know 
\begin{align}\label{A.3.v2}
\begin{cases}
{\max_{u\in[0,1]}\mathcal{L}^{u}\psi(x)\le -\ep,
		~~ \mathcal{M}V(x)\le-\frac{\ep}{q}}, &\mathrm{if}~x\in [x_{0}-h,x_{0}+h];\bigskip\\
V(x)\le \psi(x)-\frac{\ep}{q},	&\mathrm{if}~x\notin (x_{0}-h,x_{0}+h).
\end{cases}
\end{align}
Let $(u,D)\in \widetilde{\Pi}_{x_0}^{+}$ (see \eqref{Pitilde} for the definition of $\widetilde{\Pi}_{x_0}^{+}$) be any admissible reinsurance-dividend strategy. Define
$$\tau_{h}:=\inf\{t\geq0: U_{t}^{u,D}\notin [x_{0}-h,x_{0}+h]\},\quad \tau^D:=\inf\{t\geq0:\Delta D_t >0\}=\inf\{t\geq0:\Delta D_t \geq \beta\}.$$
Since $h<\frac{\beta}{2}$, we have $x_{0}-\frac{\beta}{2}<U_{t}^{u,D}<x_{0}+\frac{\beta}{2}$ for all $t\in [0,\tau_{h})$. Together with the definitions of $\tau_{h}$ and $\widetilde{\Pi}_{x_0}^{+}$, this implies 
$0<\tau_{h}\leq \tau_{D}.$
By Theorem 4.57 (It\^{o}'s formula) in \cite{Jacod03}, we have
\begin{align*}
e^{-q (t\wedge \tau_h)}\psi\left(U^{u,D}_{t\wedge \tau_h-}\right)
=&~\psi(x_{0})+\int_{0}^{t\wedge \tau_h-}e^{-q s}\mathcal{L}^{u_{s}}\psi(U^{u,D}_{s})\ds
\nn\\
&+\int_{0}^{t\wedge \tau_h-} e^{-q s}\psi^{\prime}(U^{u,D}_{s})u_{s}\left(\sigma_{+}\mathbf{1}_{\{U^{u,D}_{s}>a\}}+\sigma_{-}\mathbf{1}_{\{U^{u,D}_{s}\leq a\}}\right)\db_{s}.
\end{align*}
Taking the expectation on both sides of above, and then using the fact that $\max\limits_{u\in[0,1]}\mathcal{L}^{u}\psi(x)\leq -\ep$ for any $x\in [x_{0}-h,x_{0}+h]$, yields
\begin{align*}
\mathbb{E}_{x_{0}}\left[e^{-q (t\wedge \tau_h)}\psi\left(U^{u,D}_{t\wedge \tau_h-}\right)
\right] 
=&~\psi(x_{0})+\mathbb{E}_{x_{0}}\left[\int_{0}^{t\wedge \tau_h-}e^{-q s}\mathcal{L}^{u_{s}}\psi(U^{u,D}_{s})\ds\right]
\nn\\
\leq&~\psi(x_{0})-\frac{\ep}{q} \mathbb{E}_{x_{0}}\left[1-e^{-q (t\wedge \tau_h)}\right].
\end{align*}
Sending $t$ to $\infty$ in above, the monotone convergence theorem gives 
\begin{align}
\label{7.17.2.26}
&\mathbb{E}_{x_{0}}\left[e^{-q \tau_h}\psi\left(U^{u,D}_{\tau_h-}\right)
\right]
\leq \psi(x_{0})-\frac{\ep}{q} \mathbb{E}_{x_{0}}\left[1-e^{-q \tau_h}\right].
\end{align}
On $\{\tau_h<\tau^D\}$, we have $U_{\tau_h}^{u,D}=U_{\tau_h-}^{u,D}$ and $\Delta D_{\tau_h}=0$. This, together with \eqref{A.4.v1} and the continuity of $V$ and $\psi$, implies
\begin{align*}
V(U_{\tau_h}^{u,D})+(\Delta D_{\tau_h}-\beta)^+
=V(U_{\tau_h-}^{u,D})\leq \psi(U_{\tau_h-}^{u,D})-\frac{\ep}{q}.
\end{align*}
On $\{\tau^D=\tau_h\}$, we have $U_{\tau_h-}^{u,D}\in [x_{0}-h,x_{0}+h]$ and $\Delta D_{\tau_h}=\Delta D_{\tau^D}\geq\beta$. Consequently, we obtain from \eqref{A.3.v2} that
\begin{align*} 
V(U_{\tau_h}^{u,D})+(\Delta D_{\tau_h}-\beta)^+
&=V(U_{\tau_h-}^{u,D}-\Delta D_{\tau_h})+\Delta D_{\tau_h}-\beta\nn\\
&\leq\mathcal MV(U_{\tau_h-}^{u,D})+V(U_{\tau_h-}^{u,D}) \leq V(U_{\tau_h-}^{u,D})-\frac{\ep}{q} \leq\psi(U_{\tau_h-}^{u,D})-\frac{\ep}{q}.
\end{align*}
Combining the above two estimates yields
$$V(U_{\tau_h}^{u,D})+(\Delta D_{\tau_h}-\beta)^+ \leq \psi(U_{\tau_h-}^{u,D})-\frac{\ep}{q}.$$
Multiplying both sides by $e^{-q\tau_h}$, taking expectations, and then using \eqref{7.17.2.26}, we obtain 
\begin{align*}
\mathbb{E}_{x_{0}}\left[e^{-q \tau_h}V\left(U^{u,D}_{\tau_h}\right)+e^{-q\tau_h}(\Delta D_{\tau_h}-\beta)^+
\right] 
\leq&~ \mathbb E_{x_0}\left[e^{-q\tau_h}\psi(U_{\tau_h-}^{u,D})\right]-\frac{\ep}{q}\mathbb E_{x_0}\left[e^{-q\tau_h}\right]
\nn\\
\leq&~\psi(x_{0})-\frac{\ep}{q} \mathbb{E}_{x_{0}}\left[1-e^{-q \tau_h}\right]-\frac{\ep}{q} \mathbb{E}_{x_{0}}\left[e^{-q \tau_h}\right]
\nn
\\
=&~ \psi(x_{0})-\frac{\ep}{q}.
\end{align*}
Together with the second conclusion of Lemma \ref{lem2.3} (on recalling that $\mathcal{M}V(x_{0})<0$), it follows
$V(x_{0})\leq \psi(x_{0})-\frac{\ep}{q}$, 
contradicting to $V(x_{0})= \psi(x_{0})$.
The proof is complete.
\end{proof}

We proceed to verify that $V$ is the unique viscosity solution in $\seta$ to \eqref{7.17.HJB}.
To this end, we need the comparison theorem below.
\begin{theorem}[Comparison theorem]
\label{comparison theorem}
Assume that (i) $v_{1}$ (resp., $v_{2}$) is a non-negative viscosity sub-solution (resp., super-solution) to \eqref{7.17.HJB}, 
(ii) both $v_{1}$ and $v_{2}$ belong to $\seta$, (iii) either 
\[\sup_{x\geq 0} (v_{1}(x)-v_{2}(x))>v_{1}(a)-v_{2}(a),\]
or $v_{1}-v_{2}$ is differentiable at $a$.
Then $v_{1}\leq v_{2}$.
\end{theorem}

\begin{proof} 
We argue by contradiction and assume that
\begin{align}\label{assmp1}
\sup_{x\geq 0} f(x)> 0,
\end{align}
where $f(x):=v_{1}(x)-v_{2}(x)$.

We first claim that 
\begin{align}\label{maxata}
\sup_{x\geq 0} e^{-\rho x}f(x)>e^{-\rho a}f(a)
\end{align}
holds for all sufficiently small $\rho>0$.
Indeed, there are two possible cases:
\begin{itemize}
\item If $f(a)<\sup_{x\geq 0} f(x)$, then there exists $b>0$ such that $f(b)>f(a)$. So $e^{-\rho b}f(b)>e^{-\rho a}f(a)$ holds for all sufficiently small $\rho>0$, and consequently, so does \eqref{maxata}. 

\item Otherwise, $f(a)=\sup_{x\geq 0} f(x)$, and consequently, by condition (iii), $f$ is differentiable at $a$. By \eqref{assmp1}, we have $f(a)>0$, so for any $\rho\neq f'(a)/f(a)$, 
\begin{align*} 
\Big(e^{-\rho x}f(x)\Big)'|_{x=a}=e^{-\rho a}(f'(a)-\rho f(a))\neq 0.
\end{align*}
This indicates that $a$ is not a maximizer of the map $x\to e^{-\rho x}f(x)$, so the inequality \eqref{maxata} holds for all positive $\rho\neq f'(a)/f(a)$.
\end{itemize}

We now fix any small $\rho>0$ such that \eqref{maxata} holds and define 
$$\widetilde{v}_{k}(x):=e^{-\rho x}v_{k}(x), \quad x\in [0,\infty), \,\, k= 1,2.$$
Then, $\widetilde{v}_{1}$ (resp., $\widetilde{v}_{2}$) is a viscosity sub-solution (resp., super-solution) to
\begin{align}
\label{8.6.HJB}
\max\{\max_{u\in[0,1]}\mathcal{L}^{u}\left(e^{\rho x}v(x)\right),~\mathcal{M}\left(e^{\rho x}v(x)\right)\}=0,
\end{align}
with boundary condition $v(0)=0$. Both $v_{1}$ and $v_{2}$ are polynomial growth and continuous, so $\widetilde{v}_{1}$ and $\widetilde{v}_{2}$ are bounded and $\lim_{x\to+\infty} \widetilde{v}_{1}(x)=\lim_{x\to+\infty}\widetilde{v}_{2}(x)=0$.

According to \eqref{assmp1}, we have
\begin{align} 
M:=\sup_{x\geq 0}\left(\widetilde{v}_{1}(x)-\widetilde{v}_{2}(x)\right)=\sup_{x\geq 0} e^{-\rho x}f(x)>0.\nn
\end{align}
Because $\widetilde{v}_{1}(0)= \widetilde{v}_{2}(0)=0$, and 
$\lim_{x\to+\infty} \widetilde{v}_{1}(x)=\lim_{x\to+\infty}\widetilde{v}_{2}(x)=0$, 
by continuity, there exist real numbers $a_2>a_1>0$ such that 
\begin{align} 
\label{2.41.9.6}
\widetilde{v}_{1}(x)\leq M/3, ~~\widetilde{v}_{2}(x)\leq M/3, ~~x\in [0,a_1]\cup [a_2,\infty).
\end{align}
Then, by continuity again, there is $x^{\star}\in(a_1,a_2)$ such that 
\begin{align*}
\widetilde{v}_{1}(x^{\star})-\widetilde{v}_{2}(x^{\star})=M. 
\end{align*} 
Since $\widetilde{v}_{1}$ and $\widetilde{v}_{2}$ are locally Lipschitz continuous on $(0,\infty)$, there exists $m>0$ such that
\begin{align*}
\left|\widetilde{v}_{k}(x)-\widetilde{v}_{k}(y)\right|\leq m \left|x-y\right|, \quad x,y \in [a_1/2, a_2+2], \,\, k=1,2.
\end{align*}

Put $\Theta:=\{(x,y):0\leq x\leq y<\infty\}$. Define two sequences of functions $(\Phi^{\lambda})_{\lambda>0}$ and $(\Sigma^{\lambda})_{\lambda>0}$ as
\begin{align}
\left\{
\begin{array}{ll}
\Phi^{\lambda}(x,y):=\frac{\lambda}{2}(x-y)^{2}+\frac{2m}{\lambda^{2}(y-x)+\lambda}, & (x,y)\in \Theta,\,\,\lambda>0, \\
\Sigma^{\lambda}(x,y):=\widetilde{v}_{1}(x)-\widetilde{v}_{2}(y)-\Phi^{\lambda}(x,y),&(x,y)\in \Theta,\,\,\lambda>0.
\end{array}
\right.\nn
\end{align}
Notice that $\Phi^{\lambda}\geq 0$, and both $\widetilde{v}_{1}$ and $\widetilde{v}_{2}$ are bounded on $\Theta$, so 
\[M^{\lambda}:=\sup_{(x,y)\in \Theta}\Sigma^{\lambda} (x,y)<\infty.\] 
By the definitions of $M$ and $M^{\lambda}$, we also have
\begin{align}
M^{\lambda}\geq \Sigma^{\lambda}(x^{\star},x^{\star})=\widetilde{v}_{1}(x^{\star})-\widetilde{v}_{2}(x^{\star})-\Phi^{\lambda}(x^{\star},x^{\star})
=M-\frac{2m}{\lambda},\nn
\end{align}
which implies
\begin{align}
\label{8.6.2.41}
\liminf_{\lambda\rightarrow\infty}M^{\lambda}\geq M>0.
\end{align}
From now on, we assume $\lambda$ is sufficiently large so that $M^{\lambda}>M/2$. 

On the other hand, by $\widetilde{v}_{2}\geq 0$, $\Phi^{\lambda}\geq 0$, and \eqref{2.41.9.6}, we get 
\[\Sigma^{\lambda} (x,y)\leq \widetilde{v}_{1}(x)\leq M/3,~~x\in [0,a_1]\cup [a_2,\infty).\]
Meanwhile, if $y-x\geq 1$, then, for $\lambda>2\widetilde c_1:=2\sup_{x\geq 0}\widetilde{v}_{1}(x)\in[0,\infty)$, 
\[\Sigma^{\lambda} (x,y)\leq \widetilde{v}_{1}(x)-\frac{\lambda}{2}(x-y)^{2}
\leq \widetilde c_1-\frac{\lambda}{2}\leq 0<M^{\lambda}.\]
Therefore, for each sufficiently large $\lambda$, 
there exists $a_1< x_{\lambda}\leq y_{\lambda}< x_{\lambda}+1< a_2+1$ 
such that
\begin{align}
\label{8.4.2.36}
\Sigma^{\lambda}(x_{\lambda},y_{\lambda})=M^{\lambda}=\max_{(x,y)\in \Theta}\Sigma^{\lambda} (x,y). 
\end{align}
We now show $x_{\lambda}\neq y_{\lambda}$. 
Indeed, for each $\lambda>0$, there exists sufficiently small $\ep>0$ such that 
\begin{align*} 
\Sigma^{\lambda}(x_{\lambda},y_{\lambda})-\Sigma^{\lambda}(x_{\lambda},x_{\lambda})
&\geq \Sigma^{\lambda}(x_{\lambda}, x_{\lambda}+\ep)-\Sigma^{\lambda}(x_{\lambda}, x_{\lambda})\nn\\
&=\widetilde{v}_{2}(x_{\lambda})-\widetilde{v}_{2}(x_{\lambda}+\ep)
-\frac{\lambda}{2}\ep^{2}-\frac{2m}{\lambda^{2}\ep+\lambda}+\frac{2m}{\lambda}\nn\\
&\geq -m\ep-\frac{\lambda}{2}\ep^{2}+\frac{2m \ep}{\lambda \ep+1}>0,
\end{align*}
where we used the definition of $m$ to get the second inequality, so $y_{\lambda}\neq x_{\lambda}$. 
Moreover, 
using the fact that
\begin{align}
\Sigma^{\lambda}(x_{\lambda},x_{\lambda})+\Sigma^{\lambda}(y_{\lambda},y_{\lambda}) \leq 2\Sigma^{\lambda}(x_{\lambda},y_{\lambda}),\nn
\end{align}
one can deduce
\[\big|y_{\lambda}-x_{\lambda}\big|\leq \frac{6m}{\lambda}.\]
Therefore, 
\begin{align}
\label{8.3.2.37}
0<y_{\lambda}-x_{\lambda}\leq \frac{6m}{\lambda}.
\end{align}
Since $a_1< x_{\lambda}< y_{\lambda}< a_2+1$, there is a subsequence $(x_{\lambda_{n}},y_{\lambda_{n}})_{n\geq 1}$ (of $(x_{\lambda},y_{\lambda})_{\lambda>0}$) and $a_1\leq \widehat{x}\leq\widehat{y}\leq a_2+1$ such that $\lim\limits_{n\rightarrow \infty}\lambda_{n}=\infty$
and $(x_{\lambda_{n}},y_{\lambda_{n}})$ converges to $(\widehat{x},\widehat{y})$ as $n\rightarrow\infty$. In particular, one has
\begin{align*}
\lim\limits_{n\rightarrow\infty}\left|x_{\lambda_{n}}-y_{\lambda_{n}}\right|=|\widehat{x}-\widehat{y}|,
\end{align*}
which, together with \eqref{8.3.2.37}, yields $\widehat{x}=\widehat{y}$.
So $a_1\leq \widehat{x}=\widehat{y}\leq a_2+1$. 

By the definitions of $\Phi^{\lambda}$, $\Sigma^{\lambda}$, \eqref{8.6.2.41}, \eqref{8.4.2.36}, and \eqref{8.3.2.37}, we have
\begin{align}
\label{8.6.2.49}
\widetilde{v}_{1}(\widehat{x})-\widetilde{v}_{2}(\widehat{x})=
\widetilde{v}_{1}(\widehat{x})-\widetilde{v}_{2}(\widehat{y})=
\liminf\limits_{n\rightarrow \infty} \Sigma^{\lambda_{n}}(x_{\lambda_{n}},y_{\lambda_{n}})=
\liminf\limits_{n\rightarrow \infty} M^{\lambda_{n}}\geq M>0.
\end{align}
By the definition of $M$, this indicates
\begin{align} \label{atwidehatX=M}
\widetilde{v}_{1}(\widehat{x})-\widetilde{v}_{2}(\widehat{x})=M=\sup_{x\geq 0}\left(\widetilde{v}_{1}(x)-\widetilde{v}_{2}(x)\right). 
\end{align}
By this and \eqref{maxata}, we have
\begin{align*} 
e^{-\rho \widehat{x}}f(\widehat{x})=\sup_{x\geq 0} e^{-\rho x}f(x)>e^{-\rho a}f(a),
\end{align*}
indicating that $\widehat{x}\neq a$. 

In the sequel, we assume $\widehat{x}< a$. The case $\widehat{x}> a$ can be treated similarly, which is left to the interest readers. Since $\lim_n y_{\lambda_{n}}=\widehat{x}<a$, from now on, we assume that $n$ is sufficiently large so that $a>y_{\lambda_n}>x_{\lambda_n}$. 

By \eqref{8.6.HJB}, one can easily check that, on $(0,a)$, $\widetilde{v}_{2}$ is a viscosity super-solution to
\begin{align}
\label{8.5.2.40}
0=&~\max_{u\in[0,1]}\mathcal{L}^{u}\left(e^{\rho x}v(x)\right)
\nn\\
=&~e^{\rho x}\max_{u\in[0,1]}\left[\frac{\sigma_{-}^{2}}{2}u^{2}v^{\prime\prime}(x)+\left[\rho\sigma_{-}^{2}u^{2}+\mu_{-}u\right] v^{\prime}(x)-\left[q-\frac{\rho^{2}\sigma_{-}^{2}}{2}u^{2}-\rho\mu_{-}u\right]v(x)\right],\,x\in(0,a).
\end{align}
Recall the definition
$$ \argmax_{x\in A}f(x):=\{z\in A: f(z)\geq f(x) \text{ for all } x\in A\}.$$
Choose any 
\[y^{\star}\in \argmax\limits_{0\leq y\leq \widehat{x}}\left(v_{1}(\widehat{x}-y)+y-\beta\right).\]
There are two possible cases.
\begin{itemize}
\item \textbf{Case $y^{\star}=0$.}
Then, by Lemma \ref{lem.8.6.2.4}, it holds that
\begin{align}
\lim\limits_{n\rightarrow\infty}\mathcal{M}\left(e^{\rho x_{\lambda_{n}}}\widetilde{v}_{1}(x_{\lambda_{n}})\right)
=
\mathcal{M}\left(e^{\rho \widehat{x}}\widetilde{v}_{1}(\widehat{x})\right)
=
e^{\rho \widehat{x}}\widetilde{v}_{1}(\widehat{x})-\beta-e^{\rho \widehat{x}}\widetilde{v}_{1}(\widehat{x})
<0,\nn
\end{align} 
Consequently, there exists a positive integer $N_{1}$ such that
\begin{align}
\label{8.6.2.42}
\mathcal{M}\left(e^{\rho x_{\lambda_{n}}}\widetilde{v}_{1}(x_{\lambda_{n}})\right)<0, \quad n\geq N_{1}. 
\end{align}
For any sub-solution test function $\phi$ for the viscosity sub-solution $\widetilde{v}_{1}$ at $x_{\lambda_{n}}$, we have
\begin{align}
&\max\{\max_{u\in[0,1]}\mathcal{L}^{u}\left(e^{\rho x_{\lambda_{n}}}\phi(x_{\lambda_{n}})\right),~\mathcal{M}\left(e^{\rho x_{\lambda_{n}}}\widetilde{v}_{1}(x_{\lambda_{n}})\right)\}\geq 0,\quad n\geq N_{1},\nn
\end{align}
which, together with \eqref{8.6.2.42}, yields
\begin{align}
\max_{u\in[0,1]}\mathcal{L}^{u}\left(e^{\rho x_{\lambda_{n}}}\phi(x_{\lambda_{n}})\right)\geq 0,\quad n\geq N_{1}.
\end{align} 
Therefore, for $n\geq N_{1}$, $\widetilde{v}_{1}$ is a viscosity sub-solution at $x_{\lambda_{n}}$ to \eqref{8.5.2.40}.

\item \textbf{Case $y^{\star}>0$.} By the definition of $\mathcal{M}$, 
\begin{align*} 
\mathcal{M}\left(e^{\rho \widehat{x}}\widetilde{v}_{2}(\widehat{x})\right)
\geq e^{\rho \left(\widehat{x}-y^{\star}\right)}\widetilde{v}_{2}(\widehat{x}-y^{\star})+y^{\star}-\beta-e^{\rho \widehat{x}}\widetilde{v}_{2}(\widehat{x}).
\end{align*}
By Lemma \ref{lem.8.6.2.4}, the definition of $M$, \eqref{atwidehatX=M}, and the fact that $M$, $\rho$, $y^{\star}>0$, we have
\begin{align*}
&\lim\limits_{n\rightarrow\infty}\mathcal{M}\left(e^{\rho x_{\lambda_{n}}}\widetilde{v}_{1}(x_{\lambda_{n}})\right)
-\lim\limits_{n\rightarrow\infty}\mathcal{M}\left(e^{\rho y_{\lambda_{n}}}\widetilde{v}_{2}(y_{\lambda_{n}})\right)
\nn\\
=&~\mathcal{M}\left(e^{\rho \widehat{x}}\widetilde{v}_{1}(\widehat{x})\right)
-\mathcal{M}\left(e^{\rho \widehat{x}}\widetilde{v}_{2}(\widehat{x})\right)
\nn\\
\leq &~
e^{\rho \left(\widehat{x}-y^{\star}\right)}\widetilde{v}_{1}(\widehat{x}-y^{\star})+y^{\star}-\beta-e^{\rho \widehat{x}}\widetilde{v}_{1}(\widehat{x})-\left[e^{\rho \left(\widehat{x}-y^{\star}\right)}\widetilde{v}_{2}(\widehat{x}-y^{\star})+y^{\star}-\beta-e^{\rho \widehat{x}}\widetilde{v}_{2}(\widehat{x})\right]
\nn\\
=&~
e^{\rho \left(\widehat{x}-y^{\star}\right)}\left[\widetilde{v}_{1}(\widehat{x}-y^{\star})-\widetilde{v}_{2}(\widehat{x}-y^{\star})\right]-e^{\rho \widehat{x}}\left[\widetilde{v}_{1}(\widehat{x})-\widetilde{v}_{2}(\widehat{x})\right]
\nn\\
\leq &~ e^{\rho \left(\widehat{x}-y^{\star}\right)} M- e^{\rho \widehat{x}}M\nn\\
<&~0.
\end{align*} 
This, together with the fact that $\widetilde{v}_{2}$ is a viscosity super-solution to \eqref{8.6.HJB} at $y_{\lambda_{n}}$, yields
\begin{align}
\label{8.6.2.52}
\mathcal{M}\left(e^{\rho x_{\lambda_{n}}}\widetilde{v}_{1}(x_{\lambda_{n}})\right)
<\mathcal{M}\left(e^{\rho y_{\lambda_{n}}}\widetilde{v}_{2}(y_{\lambda_{n}})\right)
\leq 0,\quad n\geq N_{2}, 
\end{align}
for some large integer $N_{2}$. 
Hence, for any test function $\phi$ for sub-solution $\widetilde{v}_{1}$ at $x_{\lambda_{n}}$, we have 
\begin{align}
\max_{u\in[0,1]}\mathcal{L}^{u}\left(e^{\rho x_{\lambda_{n}}}\phi(x_{\lambda_{n}})\right)\geq 0, \quad n\geq N_{2}.\nn
\end{align}
Consequently, for $n\geq N_{2}$, $\widetilde{v}_{1}$ is a viscosity sub-solution at $x_{\lambda_{n}}$ to 
\eqref{8.5.2.40}.

\end{itemize}

Summing up the above arguments, we conclude that, there exists a sufficiently large $N$ such that $\widetilde{v}_{1}$ (resp., $\widetilde{v}_{2}$) is a viscosity sub-solution (resp., super-solution) at $x_{\lambda_{n}}$ (resp., $y_{\lambda_{n}}$) to 
\eqref{8.5.2.40} for all $n\geq N$. We need to point out that 
the definition of viscosity solution to \eqref{8.5.2.40} in the sense of Definition 2.2 of \cite{Crandall92}, where we replace the Hamiltonian of \eqref{8.5.2.40} as
\begin{align}
F(x,v,v^{\prime},v^{\prime\prime})&:=-\max_{u\in[0,1]}\left[\frac{\sigma_{-}^{2}}{2}u^{2}v^{\prime\prime}(x)+\left[\rho\sigma_{-}^{2}u^{2}+\mu_{-}u\right] v^{\prime}(x)-\left[q-\frac{\rho^{2}\sigma_{-}^{2}}{2}u^{2}-\rho\mu_{-}u\right]v(x)\right]
\nn\\
&\hspace{-1cm}=\min_{u\in[0,1]}\left[-\frac{\sigma_{-}^{2}}{2}u^{2}v^{\prime\prime}(x)-\left[\rho\sigma_{-}^{2}u^{2}+\mu_{-}u\right] v^{\prime}(x)+\left[q-\frac{\rho^{2}\sigma_{-}^{2}}{2}u^{2}-\rho\mu_{-}u\right]v(x)\right],\,x\in (0,a).\nn
\end{align}

It is easy to see that the function $(0,a)\times \mathbb{R}^{3}\ni (x,v,v^{\prime},v^{\prime\prime})\mapsto F(x,v,v^{\prime},v^{\prime\prime})$ is proper and degenerate elliptic in the sense of (0.2) and (0.3) in \cite{Crandall92} if and only if $\rho>0$ is small enough such that $q-\frac{\rho^{2}\sigma_{-}^{2}}{2}u^{2}-\rho\mu_{-}u>0$ for all $u\in[0,1]$. To assure the validity of this condition, we assume in the sequel that
\begin{align}
\label{rho.small}
q>\frac{\rho^{2}\sigma_{-}^{2}}{2}+\rho\left|\mu_{-}\right|.
\end{align}
This allows us to apply the viscosity solution theory of \cite{Crandall92} in the sequel.

For $n\geq N$, we can choose $\delta_n>0$ small enough such that
\begin{align}
\left\{
\begin{array}{ll}
(x_{\lambda_{n}},y_{\lambda_{n}})\in [a_{1},x_{\lambda_{n}}+\delta_n]\times [x_{\lambda_{n}}+\delta_n,a_2+1]\subseteq \Theta, & \\
(x_{\lambda_{n}},y_{\lambda_{n}})\notin \partial\left([a_{1},x_{\lambda_{n}}+\delta_n]\times [x_{\lambda_{n}}+\delta_n,a_2+1]\right). & 
\end{array}
\right.\nn
\end{align}
Define two functions as
\begin{align}
\left\{
\begin{array}{ll}
\varphi(x,y):= \Phi^{\lambda_{n}}(x,y)+\widetilde{v}_{1}(x_{\lambda_{n}})-\widetilde{v}_{2}(y_{\lambda_{n}})-\Phi^{\lambda_{n}}(x_{\lambda_{n}},y_{\lambda_{n}}),\\
w(x,y):=\widetilde{v}_{1}(x)-\widetilde{v}_{2}(y),\qquad (x,y)\in [a_{1},x_{\lambda_{n}}+\delta_n]\times [x_{\lambda_{n}}+\delta_n,a_2+1].& 
\end{array}
\right.\nn
\end{align}
Then, by \eqref{8.4.2.36}, we have
\begin{align}
\left\{
\begin{array}{ll}
w(x,y)\leq \varphi(x,y),\,\, (x,y)\in [a_{1},x_{\lambda_{n}}+\delta_n]\times [x_{\lambda_{n}}+\delta_n,a_2+1],& \\
w(x_{\lambda_{n}},y_{\lambda_{n}})= \varphi(x_{\lambda_{n}},y_{\lambda_{n}}).&
\end{array}
\right.\nn
\end{align}
That is, $(x_{\lambda_{n}},y_{\lambda_{n}})$ is a maximizer of $w-\varphi$ relative to $[a_{1},x_{\lambda_{n}}+\delta_n]\times [x_{\lambda_{n}}+\delta_n,a_2+1]$. Consequently, by Theorem 3.2 of \cite{Crandall92} with $\mathcal{O}_{1}=[a_{1},x_{\lambda_{n}}+\delta_n]$ and $\mathcal{O}_{2}=[x_{\lambda_{n}}+\delta_n,a_2+1]$, 
for any $\ep>0$, there exist two real constants $A^n_{\ep},B^n_{\ep}$ such that
\begin{align}
\label{8.6.2.53}
\left\{
\begin{array}{ll}
F(x_{\lambda_{n}},\widetilde{v}_{1}(x_{\lambda_{n}}),\frac{\partial}{\partial x}\varphi(x_{\lambda_{n}},y_{\lambda_{n}}),A^n_{\ep})\leq 0, & \\
F(y_{\lambda_{n}},\widetilde{v}_{2}(y_{\lambda_{n}}),-\frac{\partial}{\partial y}\varphi(x_{\lambda_{n}},y_{\lambda_{n}}), -B^n_{\ep})\geq 0, & \\
\left(\begin{matrix}
A^n_{\ep} & 0 \\
0 & B^n_{\ep}
\end{matrix}\right)\leq \partial_{xy}H(\varphi)(x_{\lambda_{n}},y_{\lambda_{n}})+\ep \left[\partial_{xy}H(\varphi)(x_{\lambda_{n}},y_{\lambda_{n}})\right]^{2},
\end{array}
\right.
\end{align}
where $H(\varphi)$ denotes the Hessian matrix of $\varphi$.
By the definition of $\varphi$, we have
\begin{align}
\label{8.6.2.54}
\left\{
\begin{array}{ll}
\frac{\partial}{\partial x}\varphi(x_{\lambda},y_{\lambda})=\frac{\partial}{\partial x} \Phi^{\lambda}(x_{\lambda},y_{\lambda})=\lambda (x_{\lambda}-y_{\lambda})+\frac{2m}{\left(\lambda (y_{\lambda}-x_{\lambda})+1\right)^{2}},
&
\\\frac{\partial^{2}}{\partial x^{2}}\varphi(x_{\lambda},y_{\lambda})=\frac{\partial^{2}}{\partial x^{2}} \Phi^{\lambda}(x_{\lambda},y_{\lambda})=\lambda+\frac{4m\lambda}{\left(\lambda (y_{\lambda}-x_{\lambda})+1\right)^{3}}, 
&\\
\frac{\partial}{\partial y}\varphi(x_{\lambda},y_{\lambda})=\frac{\partial}{\partial y} \Phi^{\lambda}(x_{\lambda},y_{\lambda})=-\lambda (x_{\lambda}-y_{\lambda})-\frac{2m}{\left(\lambda (y_{\lambda}-x_{\lambda})+1\right)^{2}},
& \\
\frac{\partial^{2}}{\partial y^{2}}\varphi(x_{\lambda},y_{\lambda})=\frac{\partial^{2}}{\partial y^{2}} \Phi^{\lambda}(x_{\lambda},y_{\lambda})=\lambda+\frac{4m\lambda}{\left(\lambda (y_{\lambda}-x_{\lambda})+1\right)^{3}},
& \\
\frac{\partial^{2}}{\partial x\partial y}\varphi(x_{\lambda},y_{\lambda})=\frac{\partial^{2}}{\partial x\partial y} \Phi^{\lambda}(x_{\lambda},y_{\lambda})=-\lambda-\frac{4m\lambda}{\left(\lambda (y_{\lambda}-x_{\lambda})+1\right)^{3}}.
\end{array}
\right.
\end{align}
By the third inequality of \eqref{8.6.2.53} and \eqref{8.6.2.54}, we have
\begin{align}
\label{8.6.2.55}
A^n_{\ep}+B^n_{\ep}=&~ \left(\begin{matrix}
1 & 1
\end{matrix}\right) \left(\begin{matrix}
A^n_{\ep} & 0 \\
0 & B^n_{\ep}
\end{matrix}\right)\left(\begin{matrix}
1 \\
1 
\end{matrix}\right)
\nn\\
\leq&~ \left[\frac{\partial^{2}}{\partial x^{2}} \Phi^{\lambda_n}(x_{\lambda_n},y_{\lambda_n})\right]\left(\begin{matrix}
1 & 1
\end{matrix}\right) \left(\begin{matrix}
1 & -1 \\
-1 & 1
\end{matrix}\right)\left(\begin{matrix}
1 \\
1 
\end{matrix}\right)
\nn\\
&~+2\ep \left[\frac{\partial^{2}}{\partial x^{2}} \Phi^{\lambda_n}(x_{\lambda_n},y_{\lambda_n})\right]^2\left(\begin{matrix}
1 & 1
\end{matrix}\right) \left(\begin{matrix}
1 & -1 \\
-1 & 1
\end{matrix}\right)\left(\begin{matrix}
1 \\
1 
\end{matrix}\right)
=0.
\end{align}
In addition, the first and second inequality of \eqref{8.6.2.53} can be rewritten as
\begin{align}
\left\{
\begin{array}{ll}
\max\limits_{u\in[0,1]}\left[\frac{\sigma_{-}^{2}}{2}u^{2}A^n_{\ep}+\left[\rho\sigma_{-}^{2}u^{2}+\mu_{-}u\right] \frac{\partial}{\partial x}\varphi(x_{\lambda_{n}},y_{\lambda_{n}})-\left[q-\frac{\rho^{2}\sigma_{-}^{2}}{2}u^{2}-\rho\mu_{-}u\right]\widetilde{v}_{1}(x_{\lambda_{n}})\right]\geq 0, & \\
\min\limits_{u\in[0,1]}\left[\frac{\sigma_{-}^{2}}{2}u^{2}B^n_{\ep}+\left[\rho\sigma_{-}^{2}u^{2}+\mu_{-}u\right] \frac{\partial}{\partial y}\varphi(x_{\lambda_{n}},y_{\lambda_{n}})+\left[q-\frac{\rho^{2}\sigma_{-}^{2}}{2}u^{2}-\rho\mu_{-}u\right]\widetilde{v}_{2}(y_{\lambda_{n}})\right]\geq 0. & 
\end{array}
\right.\nn
\end{align}
The second inequality and the fact of $\frac{\partial}{\partial y}\varphi(x_{\lambda_{n}},y_{\lambda_{n}})=-\frac{\partial}{\partial x}\varphi(x_{\lambda_{n}},y_{\lambda_{n}})$ (see \eqref{8.6.2.54}) imply, for all $u\in[0,1]$, 
\begin{align*} 
&~\frac{\sigma_{-}^{2}}{2}u^{2}B^n_{\ep}-\left[\rho\sigma_{-}^{2}u^{2}+\mu_{-}u\right] \frac{\partial}{\partial x}\varphi(x_{\lambda_{n}},y_{\lambda_{n}})+\left[q-\frac{\rho^{2}\sigma_{-}^{2}}{2}u^{2}-\rho\mu_{-}u\right]\widetilde{v}_{2}(y_{\lambda_{n}})
\nn\\
=&~\frac{\sigma_{-}^{2}}{2}u^{2}B^n_{\ep}+\left[\rho\sigma_{-}^{2}u^{2}+\mu_{-}u\right] \frac{\partial}{\partial y}\varphi(x_{\lambda_{n}},y_{\lambda_{n}})+\left[q-\frac{\rho^{2}\sigma_{-}^{2}}{2}u^{2}-\rho\mu_{-}u\right]\widetilde{v}_{2}(y_{\lambda_{n}})\geq 0.
\end{align*}
Together with the first inequality, we get, 
\begin{align*}
\max\limits_{u\in[0,1]}\bigg\{&\bigg(\frac{\sigma_{-}^{2}}{2}u^{2}A^n_{\ep}+\left[\rho\sigma_{-}^{2}u^{2}+\mu_{-}u\right] \frac{\partial}{\partial x}\varphi(x_{\lambda_{n}},y_{\lambda_{n}})-\left[q-\frac{\rho^{2}\sigma_{-}^{2}}{2}u^{2}-\rho\mu_{-}u\right]\widetilde{v}_{1}(x_{\lambda_{n}})\bigg)\\
& ~~ +\bigg(
\frac{\sigma_{-}^{2}}{2}u^{2}B^n_{\ep}-\left[\rho\sigma_{-}^{2}u^{2}+\mu_{-}u\right] \frac{\partial}{\partial x}\varphi(x_{\lambda_{n}},y_{\lambda_{n}})+\left[q-\frac{\rho^{2}\sigma_{-}^{2}}{2}u^{2}-\rho\mu_{-}u\right]\widetilde{v}_{2}(y_{\lambda_{n}})
\bigg)\bigg\}\geq 0,
\end{align*}
namely, 
\begin{align}
\label{2.58.9.6}
\max\limits_{u\in[0,1]}\bigg\{&\frac{\sigma_{-}^{2}}{2}u^{2}\left(A^n_{\ep}+B^n_{\ep}\right)+\left[q-\frac{\rho^{2}\sigma_{-}^{2}}{2}u^{2}-\rho\mu_{-}u\right] \big(\widetilde{v}_{2}(y_{\lambda_{n}})-\widetilde{v}_{1}(x_{\lambda_{n}})\big)\bigg\}
\geq 0.
\end{align}
Recalling \eqref{rho.small}, \eqref{8.6.2.55}, and the fact that $(x_{\lambda_{n}},y_{\lambda_{n}})$ converges to $(\widehat{x},\widehat{x})$ as $n\rightarrow\infty$, and sending $n\to\infty$ in \eqref{2.58.9.6}, we obtain
\begin{align}
\widetilde{v}_{2}(\widehat{x})- \widetilde{v}_{1}(\widehat{x})\geq 0,\nn
\end{align}
contradicting \eqref{8.6.2.49}. 
The proof is complete.
\end{proof}

\begin{corollary}
Assume the conditions (i), (ii) in Theorem \ref{comparison theorem} hold and 
$v_{1}(a)\leq v_{2}(a)$. Then $v_{1}\leq v_{2}$.
\end{corollary}
\begin{proof}
Since $\sup_{x\geq 0} (v_{1}(x)-v_{2}(x))\geq v_{1}(0)-v_{2}(0)=0$, there two possible cases.
\begin{itemize}
\item If $\sup_{x\geq 0} (v_{1}(x)-v_{2}(x))=0$, then it obviously implies $v_{1}\leq v_{2}$.
\item Otherwise, $\sup_{x\geq 0} (v_{1}(x)-v_{2}(x))>0\geq v_{1}(a)-v_{2}(a)$, so the condition (iii) in Theorem \ref{comparison theorem} holds true and thus $v_{1}\leq v_{2}$.
\end{itemize} 
The proof is complete. 
\end{proof}

The uniqueness of viscosity solutions of \eqref{7.17.HJB}, in the weak sense of Theorem \ref{8.7.thm.2.7} below, is a direct consequence of Theorems \ref{8.7.thm.2.5} and \ref{comparison theorem}.

\begin{theorem}
\label{8.7.thm.2.7}
Let $v\in\seta$ be a viscosity solution to \eqref{7.17.HJB}. If either both $v$ and $V$ are differentiable at $a$, or both inequalities
\begin{align}
\left\{
\begin{array}{ll}
\sup_{x\geq 0} (v(x)-V(x))>v(a)-V(a),& \\
\sup_{x\geq 0} (V(x)-v(x))>V(a)-v(a),& 
\end{array}
\right.
\nonumber
\end{align}
hold, then $v\equiv V$ on $\mathbb{R}_{+}$. In particular, if $v(a)= V(a)$, then $v\equiv V$ on $\mathbb{R}_{+}$.
\end{theorem}

\begin{remark}
It should be mentioned that the viscosity theory presented in Theorems \ref{8.7.thm.2.5}-\ref{8.7.thm.2.7} has its own merit since a sufficiently smooth solution to \eqref{7.17.HJB} cannot be constructed under every parameter configuration. As shown in Sections \ref{sec.3} and \ref{sec.4}, there are a few exceptional cases in which no sufficient smooth solution to \eqref{7.17.HJB} is available; consequently, explicit characterizations of the optimal strategy and the value function remain out of reach. In these exceptional cases, we must content ourselves with the statement that the value function coincides with the unique viscosity solution to \eqref{7.17.HJB}, in the sense of Theorem \ref{8.7.thm.2.7}.
\end{remark}

Although the value function has already been characterized as the unique (in the weak sense of Theorem \ref{8.7.thm.2.7}) viscosity solution to \eqref{7.17.HJB}. We still need to characterize the optimal strategy, and, if possible, the explicit expressions for the optimal strategy and the value function $V$. Theorem \ref{thm2.2} below is a first step toward this aim, where additional smooth assumption is imposed.
\begin{theorem}
\label{thm2.2}
Let $v\in C([0,\infty))$ be a viscosity solution to \eqref{7.17.HJB} with $v(0)=0$. Furthermore, suppose that there exists $x_{0}>0$ such that $v\in C^1((0,\infty))\cap C^{2}((0,\infty)\setminus\{a,x_{0}\})$ and
\begin{align}
\label{rig.lef.sec.der.}
\left\{\begin{array}{ll}
v(x)-x=v(x_{0})-x_{0}, \quad x\in [x_{0},\infty),& \\
\max\left\{\left|v^{\prime\prime}(x_{0}^{\pm})\right|,\left|v^{\prime\prime}(a^{\pm})\right|\right\}<\infty.& 
\end{array}
\right.
\end{align}
Then
\begin{equation}\label{eq:a029}
V(x)\le v(x), \quad x\in (0,\infty).\nn
\end{equation}
Let $(u^{v},D^{v})$ be a feedback reinsurance-dividend strategy defined by
\begin{align}
\left\{
\begin{array}{ll}
u^v_{t}= \argmax\limits_{u\in[0,1]}\mathcal{L}^uv(U_t^{v}),& \text{a.e. } (t,\omega)\in [0,\infty)\times\Omega,
\\
\Delta D^{v}_{t}= \mathbf{1}_{\{\mathcal{M}v\left(U^{v}_{t-}\right)=0\}} \argmax\limits_{0\leq \eta\le U^{v}_{t-}}\left(v\left(U^{v}_{t-}-\eta\right)+\eta-\beta\right),& t\in [0,\infty),
\end{array}
\right.\nn
\end{align}
where $U^{v}:=(U_t^{v})_{t\geq 0}$ represents the surplus process under control $(u^{v},D^{v})$.
If $(u^{v},D^{v})\in\Pi_x$, then
\begin{equation}\label{eq:a030}
V(x)=v(x)=J_{u^v,D^v}(x),\quad x\in (0,\infty),\nn
\end{equation}
that is, $\left(u^v,D^v\right)$ is an optimal strategy for the control problem \eqref{eq:a008}.
\end{theorem}

\begin{proof}
By assumption, $v$ is a classical solution to \eqref{7.17.HJB} on $(0,\infty)$ except for $a$ and $x_{0}$, that is
\begin{align}
\label{eq:a018}
\left\{
\begin{array}{ll}
\mathcal{L}^uv(x)\le0,& (x,u)\in \left[(0,\infty)\setminus \{a,x_{0}\}\right]\times [0,1], \\
\mathcal{M}v(x)\leq 0,& x\in [0,\infty),
\\
\mathcal{M}v(x)\left(\max_{u\in[0,1]}\mathcal{L}^uv(x)\right)=0,& (x,u)\in \left[(0,\infty)\setminus \{a,x_{0}\}\right]\times [0,1].
\end{array}
\right.
\end{align}
Let $(u, D) \in \Pi^+_x$ be an arbitrary reinsurance-dividend strategy with initial surplus level $x>0$. Recall that $(U^{u,D}_t)_{t \geq 0}$ is governed by \eqref{eq:a002}. Denote the continuous part of $(U^{u,D}_t)_{t \geq 0}$ as $(U^{u,D,c}_t)_{t \geq 0}$. Define a sequence of stopping times $(T_{n}^{u,D})_{n \geq 1}$ as
\begin{align}
T_{n}^{u,D}:=\inf\{t\geq 0: \big|U^{u,D}_{t}\big|\notin (1/n,n)\},\quad n\geq 1.\nn
\end{align}
Recalling that $v(x)$ is continuous and continuously differentiable on $(0,\infty)$ except at $x=x_{0}$ and $x=a$ (hence, $v^{\prime}(x)$ is of bounded variation), we know that $v$ can be expressed as the difference of two convex functions. 
Then, appealing to the Meyer-It\^{o} formula (see Theorem 70 of Chapter 7 in \cite{Protter05}), we obtain
\begin{align}
\label{2.2.wr00}
e^{-q (t\wedge T_{n}^{u,D})}v\left(U^{u,D}_{t\wedge T_{n}^{u,D}}\right)
=&~v(x)-\int_{0}^{t\wedge T_{n}^{u,D}}q e^{-q s}v(U^{u,D}_{s})\ds+\int_{0}^{t\wedge T_{n}^{u,D}} e^{-q s}v^{\prime}(U^{u,D}_{s})\mathrm{d}U^{u,D}_{s}\nn\\
&+\sum_{s\leq t\wedge T_{n}^{u,D}}e^{-q s}\left(v(U^{u,D}_{s})-v(U^{u,D}_{s-})+v^{\prime}(U^{u,D}_{s-})\left( D_{s-}-D_{s}\right)\right)\nn\\
&+\frac{1}{2} \int_{0}^{t\wedge T_{n}^{u,D}} e^{-q s}\int_{0}^{\infty}\mu_{v}(\mathrm{d}z)\mathrm{d}L_{s}^{z},
\end{align}
where $\mu_{v}(\mathrm{d}z)$ is the signed measure (when restricted to compact sets) corresponding to the second derivative of $v$ in the sense of distributions, and $L_{s}^{z}$ is the local time at $z$ of the process $(U^{u,D}_{t})_{t\geq 0}$. By $v\in C^1((0,\infty))\cap C^2((0,\infty)\setminus\{x_{0},a\})$ and \eqref{rig.lef.sec.der.}, for any compact set $A\subseteq (0,\infty)$, the signed measure $\mu_{v}(\mathrm{d}z)$ admits the representation
\begin{align}
\mu_{v}(A)=\int_{A}v^{\prime\prime}(z)\mathrm{d}z.\nn
\end{align}
Applying Corollary 1 of Theorem 70 of Chapter 7 in \cite{Protter05}, we obtain
\begin{align}
\label{mifor.}
& \int_{0}^{t\wedge T_{n}^{u,D}} e^{-q s}\int_{0}^{\infty}\mu_{v}(\mathrm{d}z)\mathrm{d}L_{s}^{z}
\nn\\
=&~ \int_{0}^{t\wedge T_{n}^{u,D}} e^{-q s}v^{\prime\prime}(U^{u,D}_{s-})\mathrm{d}\left\langle U_{\cdot}^{u,D,c},U_{\cdot}^{u,D,c}\right\rangle_{s}
\nn\\
=&~\int_{0}^{t\wedge T_{n}^{u,D}} e^{-q s}(\sigma_{+}^2\textbf{1}_{\{U^{u,D}_{s-}>a\}}+\sigma_{-}^2\textbf{1}_{\{U^{u,D}_{s-}\le a\}})u_{s}^{2}v^{\prime\prime}(U^{u,D}_{s-})\ds.
\end{align}
Consequently, combining \eqref{2.2.wr00} and \eqref{mifor.}, we have
\begin{align}
\label{2.2.wr}
e^{-q (t\wedge T_{n}^{u,D})}v\left(U^{u,D}_{t\wedge T_{n}^{u,D}}\right) 
=&~v(x)+\int_{0}^{t\wedge T_{n}^{u,D}}e^{-q s}\mathcal{L}^{u_{s}}v(U^{u,D}_{s})\ds
\nn\\
&+\int_{0}^{t\wedge T_{n}^{u,D}} e^{-q s}v^{\prime}(U^{u,D}_{s})u_{s}\left(\sigma_{+}\mathbf{1}_{\{U^{u,D}_{s}>a\}}+\sigma_{-}\mathbf{1}_{\{U^{u,D}_{s}\leq a\}}\right)\db_{s}\nn\\
&+\sum_{s\leq t\wedge T_{n}^{u,D}} e^{-q s}\left(v(U^{u,D}_{s-}-\Delta D_{s})-v(U^{u,D}_{s-})\right).
\end{align}
From \eqref{eq:a018} and the definition of $(u^{v},D^{v})$, it follows that
\begin{eqnarray}
\label{2.3.wr}
\left\{\hspace{-0.2cm}
\begin{array}{ll}
\mathcal{L}^{u_{s}}v(U^{u,D}_{s})\leq 0,&s<t\wedge T_{n}^{u,D}, \\
v(U^{u,D}_{s-}-( D_{s}-D_{s-}))-v(U^{u,D}_{s-})+ D_{s}-D_{s-}-\beta \le 0,&s<t\wedge T_{n}^{u,D}, 
\\
v(U^{u^{v},D^{v}}_{s-}-( D^{v}_{s}-D^{v}_{s-}))-v(U^{u^{v},D^{v}}_{s-})+ D^{v}_{s}-D^{v}_{s-}-\beta = 0,&s<t\wedge T_{n}^{u^{v},D^{v}},D^{v}_{s}>D^{v}_{s-}. 
\end{array}
\right.
\end{eqnarray}
We proceed by asserting the following claim
\begin{align}
\label{claim.of.finite.dividend.payout}
\mathbb{P}_{x}\left(\#\{s\in[0,t\wedge T_{n}^{u^{v},D^{v}}): D^{v}_{s}-D^{v}_{s-}>0\}<\infty\right)=1,
\end{align}
where $\#A$ denotes the cardinality of the set $A$. Suppose, for contradiction, that
\begin{align}
\label{contra.ass}
\mathbb{P}_{x}\left(\#\{s\in[0,t\wedge T_{n}^{u,D}): D_{s}-D_{s-}>0\}=\infty\right)>0.
\end{align}
Since $(u,D) \in \Pi^+_x$, condition \eqref{contra.ass} implies that
\begin{align}
\label{contra.ass.1}
\mathbb{P}_{x}\left(\sum_{s<t\wedge T_{n}^{u,D}}\Delta D_{s}=\infty\right)>0.
\end{align}
By the definition of $T_n^{u,D}$, we also have
\begin{align}
\label{contra.ass.2}
\mathbb{P}_{x}\left(\inf_{s\in [0,t\wedge T_{n}^{u,D})}U^{u,D}_{s}\geq \frac{1}{n}\right)=1.
\end{align}
Combining \eqref{contra.ass.1} and \eqref{contra.ass.2}, we deduce that
\begin{align}
&\mathbb{P}\left(\sup_{r\in [0,t\wedge T_{n}^{u,D})}\left(x+\int_{0}^{r}u_{s}\left(\mu_{+}\textbf{1}_{\{U_s^{u,D}>a\}}+\mu_{-}\textbf{1}_{\{U_s^{u,D}\le a\}}\right)\ds
\right.\right.
\nn\\
&\quad\quad 
\left.\left.+\int_{0}^{r}u_{s}\left(\sigma_{+}\textbf{1}_{\{U_s^{u,D}>a\}}+\sigma_{-}\textbf{1}_{\{U_s^{u,D}\le a\}}\right)\db_s
\right)=\infty\right)>0.\nn
\end{align}
Noting that
\begin{align}
&\sup_{r\in [0,t\wedge T_{n}^{u,D})}\int_{0}^{r}u_{s}\left(\mu_{+}\textbf{1}_{\{U_s^{u,D}>a\}}+\mu_{-}\textbf{1}_{\{U_s^{u,D}\le a\}}\right)\ds
\nn\\
\leq&~ \int_{0}^{t}\left|u_{s}\left(\mu_{+}\textbf{1}_{\{U_s^{u,D}>a\}}+\mu_{-}\textbf{1}_{\{U_s^{u,D}\le a\}}\right)\right|\ds
\nn\\
\leq&~\left(\left|\mu_{-}\right|+\left|\mu_{+}\right|\right)t<\infty,\quad \mathbb{P}-\mathrm{a.s.},\nn
\end{align}
we conclude that
\begin{align}
\label{contra.01}
&\mathbb{P}\left(\sup_{r\in [0,t\wedge T_{n}^{u,D})}\left|\int_{0}^{r}u_{s}\left(\sigma_{+}\textbf{1}_{\{U_s^{u,D}>a\}}+\sigma_{-}\textbf{1}_{\{U_s^{u,D}\le a\}}\right)\db_s
\right|=\infty\right)>0.
\end{align}
However, applying the Burkholder-Davis-Gundy inequality (see Theorem 3.28 in \cite{Karatzas91}), we obtain
\begin{align}
\label{contra.02}
&\mathbb{E}\left(\sup_{r\in [0,t\wedge T_{n}^{u,D})}\left|\int_{0}^{r}u_{s}\left(\sigma_{+}\textbf{1}_{\{U_s^{u,D}>a\}}+\sigma_{-}\textbf{1}_{\{U_s^{u,D}\le a\}}\right)\db_s
\right|\right)
\nn\\
\leq&~ K
\mathbb{E}\left[\left(\int_{0}^{t\wedge T_{n}^{u,D}}u_{s}^{2}\left(\sigma_{+}^{2}\textbf{1}_{\{U_s^{u,D}>a\}}+\sigma_{-}^{2}\textbf{1}_{\{U_s^{u,D}\le a\}}\right)\ds\right)^{\frac{1}{2}}\right]
\nn\\
\leq&~ K\sqrt{\left(\sigma_{+}^{2}+\sigma_{-}^{2}\right)t}<\infty,
\end{align}
for some finite constant $K>0$. This contradicts \eqref{contra.01}. Hence, assumption \eqref{contra.ass} must fail, confirming claim \eqref{claim.of.finite.dividend.payout}.
Given $(u^{v},D^{v})\in \Pi^{+}_x$, we then have
\begin{align}
\mathcal{M}v(U^{u^{v},D^{v}}_{s})<0,\quad \mathbb{P}_{x}\times \dt-\text{a.e. } (\omega,s)\in \Omega\times [0,t\wedge T_{n}^{u^{v},D^{v}}),\nn
\end{align}
which, by \eqref{eq:a018} and the definition of $(u^{v},D^{v})$, implies
\begin{align}
\label{2.32.wr} 
\int_{0}^{t\wedge T_{n}^{u^{v},D^{v}}}e^{-q s}\mathcal{L}^{u_{s}^{v}}v(U^{u^{v},D^{v}}_{s})\ds
=&~
\int_{0}^{t\wedge T_{n}^{u^{v},D^{v}}}\mathbf{1}_{\{\mathcal{M}v(U^{u^{v},D^{v}}_{s})<0\}}e^{-q s}\mathcal{L}^{u_{s}^{v}}v(U^{u^{v},D^{v}}_{s})\ds
\nn\\
=&~0, \quad \mathbb{P}_{x}-\text{a.s.}.
\end{align}
Substituting the first two equalities from \eqref{2.3.wr} into \eqref{2.2.wr}, and taking expectations, we obtain
\begin{align}
\label{2.33.wr.01}
\mathbb{E}_{x}\left(\sum_{s\leq t\wedge T_{n}^{u,D}} e^{-q s}\left(\Delta D_{s}-\beta\right)\right)
\leq
v(x)-\mathbb{E}_{x}\left(e^{-q (t\wedge T_{n}^{u,D})}v\left(U^{u,D}_{t\wedge T_{n}^{u,D}}\right)\right).
\end{align}
Similarly, substituting the third equality from \eqref{2.3.wr} and \eqref{2.32.wr} into \eqref{2.2.wr} with $(u,D)=(u^{v},D^{v})$, and taking expectations, yields
\begin{align}
\label{2.34.wr.01}
\mathbb{E}_{x}\left(\sum_{s\leq t\wedge T_{n}^{u^{v},D^{v}}} e^{-q s}\left(D_{s}^{v}-D_{s-}^{v}-\beta\right)\right)
=
v(x)-\mathbb{E}_{x}\left(e^{-q (t\wedge T_{n}^{u^{v},D^{v}})}v\left(U^{u^{v},D^{v}}_{t\wedge T_{n}^{u^{v},D^{v}}}\right)\right).
\end{align}
Next, recall the inequality $a \leq a^2 + 1^2$ and that the function $x \mapsto e^{-qx}x$ is bounded on $[0,\infty)$. It follows that
\begin{align}
\label{2.35.wr}
0\leq&~ e^{-q (t\wedge T_{n}^{u,D})}v\left(U^{u,D}_{t\wedge T_{n}^{u,D}}\right)\mathbf{1}_{\{T^{u,D}<\infty\}}
\nn\\
\leq&~ e^{-q (t\wedge T_{n}^{u,D})}\sup_{x\leq x_{0}}v\left(x\right)\mathbf{1}_{\{U^{u,D}_{t\wedge T_{n}^{u,D}}\leq x_{0}\}}\mathbf{1}_{\{T^{u,D}<\infty\}}
\nn\\
&\,+
e^{-q (t\wedge T_{n}^{u,D})}\left(v\left(x_{0}\right)+U^{u,D}_{t\wedge T_{n}^{u,D}}-x_{0}\right)\mathbf{1}_{\{U^{u,D}_{t\wedge T_{n}^{u,D}}> x_{0}\}} \mathbf{1}_{\{T^{u,D}<\infty\}}
\nn\\
\leq&~ e^{-q (t\wedge T_{n}^{u,D})}\left(\sup_{x\in[0, x_{0}]}v\left(x\right)+x+\left(\left|\mu_{-}\right|+\left|\mu_{+}\right|\right)\left(t\wedge T_{n}^{u,D}\right)\right)\mathbf{1}_{\{T^{u,D}<\infty\}}
\nn\\
&\,+
e^{-q (t\wedge T_{n}^{u,D})}\left|\int_{0}^{t\wedge T_{n}^{u,D}}u_{s}\left(\sigma_{+}\textbf{1}_{\{U^{u,D}_s>a\}}+\sigma_{-}\textbf{1}_{\{U^{u,D}_s\le a\}}\right)\db_s\right| \mathbf{1}_{\{T^{u,D}<\infty\}}
\nn\\
\leq&~ e^{-q (t\wedge T_{n}^{u,D})}\left(1+\sup_{x\in[0, x_{0}]}v\left(x\right)+x+\left(\left|\mu_{-}\right|+\left|\mu_{+}\right|\right)\left(t\wedge T_{n}^{u,D}\right)\right)\mathbf{1}_{\{T^{u,D}<\infty\}}
\nn\\
&\,+
e^{-q (t\wedge T_{n}^{u,D})}\left[\int_{0}^{t\wedge T_{n}^{u,D}}u_{s}\left(\sigma_{+}\textbf{1}_{\{U^{u,D}_s>a\}}+\sigma_{-}\textbf{1}_{\{U^{u,D}_s\le a\}}\right)\db_s\right]^{2} \mathbf{1}_{\{T^{u,D}<\infty\}}
\nn\\
\leq&~ K_{1} +
\sup_{s\geq 0} \left[e^{-q s} \left[M^{u,D}_{s}\right]^{2}\right],
\end{align}
where 
$$K_{1}=1+\sup_{x\in[0, x_{0}]}v\left(x\right)+x+\sup_{s\geq 0} e^{-q s}s(|\mu_{-}|+|\mu_{+}|)<\infty$$ 
is a constant, and $$M^{u,D}_{t}:=\int_{0}^{t}u_{s}\left(\sigma_{+}\textbf{1}_{\{U^{u,D}_s>a\}}+\sigma_{-}\textbf{1}_{\{U^{u,D}_s\le a\}}\right)\db_s,\quad t\geq 0,$$ is a continuous martingale.
Applying It\^{o}'s formula, we have
\begin{align}
e^{-q t} \left[M^{u,D}_{t}\right]^{2}=&~-q\int_{0}^{t}e^{-q s} \left[M^{u,D}_{s}\right]^{2}\ds+\int_{0}^{t}e^{-q s}
u_{s}^{2}\left(\sigma_{+}^{2}\textbf{1}_{\{U^{u,D}_s>a\}}+\sigma_{-}^{2}\textbf{1}_{\{U^{u,D}_s\le a\}}\right)\ds
\nn\\
&+2\int_{0}^{t}e^{-q s} M^{u,D}_{s}u_{s}\left(\sigma_{+}\textbf{1}_{\{U^{u,D}_s>a\}}+\sigma_{-}\textbf{1}_{\{U^{u,D}_s\le a\}}\right)\db_s
\nn\\
\leq&~ \frac{\sigma_{+}^{2}+\sigma_{-}^{2}}{q}+2\int_{0}^{t}e^{-q s} M^{u,D}_{s}u_{s}\left(\sigma_{+}\textbf{1}_{\{U^{u,D}_s>a\}}+\sigma_{-}\textbf{1}_{\{U^{u,D}_s\le a\}}\right)\db_s.\nn
\end{align}
By the Burkholder-Davis-Gundy inequality and H\"older's inequality, it follows that
\begin{align}
\label{a.leq k+k'sqrt{a}}
&\mathbb{E}_{x}\left(\sup_{t\geq 0} \left[e^{-q t} \left[M^{u,D}_{t}\right]^{2}\right]\right)
\nn\\
\leq&~
\frac{\sigma_{+}^{2}+\sigma_{-}^{2}}{q}+2\mathbb{E}_{x}\left(\sqrt{\int_{0}^{\infty}e^{-2qt}\left[M^{u,D}_{t}\right]^{2}u_{t}^{2}
\left(\sigma_{+}^{2}\textbf{1}_{\{U_T^{u,D}>a\}}+\sigma_{-}^{2}\textbf{1}_{\{U_T^{u,D}\le a\}}\right)\dt}\right)
\nn\\
\leq&~
\frac{\sigma_{+}^{2}+\sigma_{-}^{2}}{q}+2\mathbb{E}_{x}\left(\sqrt{\int_{0}^{\infty}e^{-qt}\sup_{s\geq 0}\left[e^{-qs}\left[M^{u,D}_{s}\right]^{2}\right]
\left(\sigma_{+}^{2}+\sigma_{-}^{2}\right)\dt}\right)
\nn\\
\leq&~
\frac{\sigma_{+}^{2}+\sigma_{-}^{2}}{q}+2\sqrt{\frac{\sigma_{+}^{2}+\sigma_{-}^{2}}{q}}\sqrt{\mathbb{E}_{x}\left(\sup_{t\geq 0} \left[e^{-q t} \left[M^{u,D}_{t}\right]^{2}\right]\right)}.
\end{align}
Using \eqref{a.leq k+k'sqrt{a}} and the fact that
\begin{align}
\left(\sqrt{a}\right)^{2}-k\sqrt{a}-l\leq 0\Longrightarrow
\sqrt{a}\leq \frac{k+\sqrt{k^{2}+4l}}{2}\quad (a,k,l\geq 0),
\end{align}
we have
\begin{align}
\label{integrable.upp.bound.}
\mathbb{E}_{x}\left(\sup_{t\geq 0} \left[e^{-q t} \left[M^{u,D}_{t}\right]^{2}\right]\right)\leq \left(1+\sqrt{2}\right)^{2}\frac{\sigma_{+}^{2}+\sigma_{-}^{2}}{q}.
\end{align}
Furthermore, since $U^{u,D}_{T^{u,D}}\leq 0$ and $v(x)=0$ for $x\leq 0$, we have
\begin{align}
\label{a.s.t,n.tend.oo}
&\lim\limits_{t,n\rightarrow\infty}e^{-q (t\wedge T_{n}^{u,D})}v\left(U^{u,D}_{t\wedge T_{n}^{u,D}}\right)\mathbf{1}_{\{T^{u,D}<\infty\}}
\nn\\
=&~
e^{-q T^{u,D}}v\left(U^{u,D}_{T^{u,D}}\right)\mathbf{1}_{\{T^{u,D}<\infty\}}=0,\quad \mathbb{P}_{x}-\mathrm{a.s.}.
\end{align}
Combining \eqref{2.35.wr}, \eqref{integrable.upp.bound.}, and \eqref{a.s.t,n.tend.oo}, and applying the dominated convergence theorem, we conclude
\begin{align}
\label{limit=0}
\lim\limits_{t,n\rightarrow\infty}\mathbb{E}_{x}\left(e^{-q (t\wedge T_{n}^{u,D})}v\left(U^{u,D}_{t\wedge T_{n}^{u,D}}\right)\mathbf{1}_{\{T^{u,D}<\infty\}}\right)=0.
\end{align}
Analogous arguments to those applied in deriving \eqref{2.35.wr} yield
\begin{align}
\label{2.35.wr.1}
0\leq&~\mathbb{E}_{x}\left(e^{-q (t\wedge T_{n}^{u,D})}v\left(U^{u,D}_{t\wedge T_{n}^{u,D}}\right)\mathbf{1}_{\{T^{u,D}=\infty\}}\right)
\nn\\
\leq&~ \mathbb{E}_{x}\left(e^{-q (t\wedge T_{n}^{u,D})}\left(\sup_{x\in[0, x_{0}]}v\left(x\right)+x+\left(\left|\mu_{-}\right|+\left|\mu_{+}\right|\right)\left(t\wedge T_{n}^{u,D}\right)\right)\mathbf{1}_{\{T^{u,D}=\infty\}}\right)
\nn\\
&\,+
\mathbb{E}_{x}\left(e^{-q (t\wedge T_{n}^{u,D})}\left|\int_{0}^{t\wedge T_{n}^{u,D}}u_{s}\left(\sigma_{+}\textbf{1}_{\{U^{u,D}_s>a\}}+\sigma_{-}\textbf{1}_{\{U^{u,D}_s\le a\}}\right)\db_s\right| \mathbf{1}_{\{T^{u,D}=\infty\}}\right).
\end{align}
Applying the H\"older's inequality and using the uniform bound established in \eqref{integrable.upp.bound.}, we obtain
\begin{align}
\label{2.35.wr.2}
&\mathbb{E}_{x}\left(e^{-q (t\wedge T_{n}^{u,D})}\left|\int_{0}^{t\wedge T_{n}^{u,D}}u_{s}\left(\sigma_{+}\textbf{1}_{\{U^{u,D}_s>a\}}+\sigma_{-}\textbf{1}_{\{U^{u,D}_s\le a\}}\right)\db_s\right| \mathbf{1}_{\{T^{u,D}=\infty\}}\right)
\nn\\
\leq&~\sqrt{\mathbb{E}_{x}\left(e^{-q (t\wedge T_{n}^{u,D})}\mathbf{1}_{\{T^{u,D}=\infty\}}\right)}\sqrt{\mathbb{E}_{x}\left(e^{-q (t\wedge T_{n}^{u,D})}\left[M_{t\wedge T_{n}^{u,D}}^{u,D}\right]^{2}\right)}
\nn\\
\leq&~\sqrt{\mathbb{E}_{x}\left(e^{-q (t\wedge T_{n}^{u,D})}\mathbf{1}_{\{T^{u,D}=\infty\}}\right)}\sqrt{\mathbb{E}_{x}\left(\sup_{t\geq 0} e^{-q t} \left[M^{u,D}_{t}\right]^{2}\right)}
\nn\\
\leq&~
\left(1+\sqrt{2}\right)\sqrt{\frac{\sigma_{+}^{2}+\sigma_{-}^{2}}{q}}
\sqrt{\mathbb{E}_{x}\left(e^{-q (t\wedge T_{n}^{u,D})}\mathbf{1}_{\{T^{u,D}=\infty\}}\right)}.
\end{align}
Note that
\begin{align}
\lim\limits_{t,n\rightarrow \infty}e^{-q (t\wedge T_{n}^{u,D})}\mathbf{1}_{\{T^{u,D}=\infty\}}=
e^{-q T^{u,D}}\mathbf{1}_{\{T^{u,D}=\infty\}}=0,\quad \mathbb{P}_{x}-\mathrm{a.s.}.\nn
\end{align}
By the bounded convergence theorem, it follows that
\begin{align}
\label{2.43}
\lim\limits_{t,n\rightarrow \infty}\mathbb{E}_{x}\left(e^{-q (t\wedge T_{n}^{u,D})}\mathbf{1}_{\{T^{u,D}=\infty\}}\right)=0.
\end{align}
Combining \eqref{2.35.wr.2} and \eqref{2.43}, we conclude
\begin{align}
\label{2.35.wr.2.r}
&\lim\limits_{t,n\rightarrow\infty}\mathbb{E}_{x}\left(e^{-q (t\wedge T_{n}^{u,D})}\left|\int_{0}^{t\wedge T_{n}^{u,D}}u_{s}\left(\sigma_{+}\textbf{1}_{\{U^{u,D}_s>a\}}+\sigma_{-}\textbf{1}_{\{U^{u,D}_s\le a\}}\right)\db_s\right| \mathbf{1}_{\{T^{u,D}=\infty\}}\right)=0. 
\end{align}
Since $e^{-qx}x\rightarrow0$ as $x\rightarrow\infty$, we have
\begin{align}
&\lim\limits_{t,n\rightarrow\infty}e^{-q (t\wedge T_{n}^{u,D})}\left(\sup_{x\in[0, x_{0}]}v\left(x\right)+x+\left(\left|\mu_{-}\right|+\left|\mu_{+}\right|\right)\left(t\wedge T_{n}^{u,D}\right)\right)\mathbf{1}_{\{T^{u,D}=\infty\}}=0,\quad \mathbb{P}_{x}-\mathrm{a.s.}.\nn
\end{align}
Applying the bounded convergence theorem again yields
\begin{align}
\label{2.45}
\lim\limits_{t,n\rightarrow\infty}\mathbb{E}_{x}\left(e^{-q (t\wedge T_{n}^{u,D})}\left(\sup_{x\in[0, x_{0}]}v\left(x\right)+x+\left(\left|\mu_{-}\right|+\left|\mu_{+}\right|\right)\left(t\wedge T_{n}^{u,D}\right)\right)\mathbf{1}_{\{T^{u,D}=\infty\}}\right)=0.
\end{align}
Now, combining \eqref{2.35.wr.1}, \eqref{2.35.wr.2.r}, and \eqref{2.45}, we obtain
\begin{align}
\label{con.to.0.2}
\lim\limits_{t,n\rightarrow\infty}\mathbb{E}_{x}\left(e^{-q (t\wedge T_{n}^{u,D})}v\left(U^{u,D}_{t\wedge T_{n}^{u,D}}\right)\mathbf{1}_{\{T^{u,D}=\infty\}}\right)=0.
\end{align}
Finally, putting together \eqref{limit=0} and \eqref{con.to.0.2}, we conclude
\begin{align}
\label{con.to.0}
\lim\limits_{t,n\rightarrow\infty}\mathbb{E}_{x}\left(e^{-q (t\wedge T_{n}^{u,D})}v\left(U^{u,D}_{t\wedge T_{n}^{u,D}}\right)\right)=0.
\end{align}
This, together with \eqref{2.33.wr.01} and \eqref{2.34.wr.01}, implies
\begin{align}
\begin{cases} 
\displaystyle J_{u,D}(x)=\sum_{s\leq t\wedge T^{u,D}} e^{-q s}\left(\Delta D_{s}-\beta\right) \leq v(x),& \medskip\\
\displaystyle \sum_{s\leq t\wedge T^{u^{v},D^{v}}} e^{-q s}\left(D_{s}^{v}-D_{s-}^{v}-\beta\right) =v(x).& 
\end{cases} 
\end{align}
Recalling \eqref{shrinkage.of.set.of.add.str.} and noting that $(u,D)\in\Pi_x^{+}$ is arbitrary, we conclude
\begin{align}
\begin{cases}
\displaystyle\sup_{(u,D)\in\Pi_x}J_{u,D}(x)=\sup_{(u,D)\in\Pi_x^{+}}J_{u,D}(x) \leq v(x),& \medskip\\
\displaystyle J_{u^{v},D^{v}}(x)= v(x).& 
\end{cases} 
\end{align}
The proof is complete.
\end{proof}

\begin{remark}
Theorem \ref{thm2.2} is central to explicitly characterizing the optimal strategy and the value function. To invoke this theorem, one must first construct a viscosity (or classical) solution of the HJB equation \eqref{7.17.HJB} that is $C^2$ everywhere except at two points $a$ and $x_{0}$. Such a construction is carried out in Section \ref{sec.3}.

\end{remark}

\section{An auxiliary control problem and its solution}
\label{sec.3}

The central task of this and the next section is to construct a sufficiently smooth solution to \eqref{7.17.HJB}, thereby enabling us, via Theorem~\ref{8.7.thm.2.7}, to identify the value function and the optimal strategy for the control problem \eqref{eq:a008}. We accomplish this task in two steps.
\begin{itemize}
\item[(I)] In the first step, we modify the control problem \eqref{eq:a008} by restricting its admissible set so that only two-barrier impulse dividend strategies are allowed, thereby yielding an auxiliary control problem \eqref{eq:a008.aux.}. We then solve this auxiliary problem and derive its optimal reinsurance-dividend strategy, together with the corresponding value function in explicit form.

Given the challenges encountered and sophistication required, this first step is done in a satisfactory manner. For almost all parameter configurations, the explicit optimal strategy and value function are characterized; for the few exceptional cases, necessary and sufficient conditions for the existence of an optimal strategy are given, and whenever these conditions hold, explicit forms of the optimal strategy and value function are also available.

\item[(II)] In the second step, we need to verify that the already obtained value function for the auxiliary problem \eqref{eq:a008.aux.} satisfies \eqref{7.17.HJB}; and, the optimal strategy for the auxiliary control problem \eqref{eq:a008.aux.} is also the optimal strategy for the primal control problem \eqref{eq:a008}.
\end{itemize}

In this section, we accomplish Step (I); while the work of Step (II) is deferred to Section \ref{sec.4}. In detail, we investigate the reinsurance-dividend optimization problem under the constraint that the dividend strategy is restricted to the class of double barrier strategies. Specifically, a double barrier dividend strategy, denoted by $(D_t^{\underline{x},\overline{x}})_{t \geq 0}$ for $\beta \leq \underline{x} + \beta < \overline{x}$, is defined as follows: whenever the surplus process exceeds or attempts to exceed the upper barrier $\overline{x}$, a lump-sum dividend is paid so as to bring the surplus back down to the lower barrier $\underline{x}$; no dividends are distributed if the surplus remains below $\overline{x}$.
Define
\begin{align}
\mathcal{G}_x:= \{(u,D)\in\Pi_x: D=D^{\underline{x},\overline{x}} \text{ for some } (\underline{x},\overline{x}) \text{ with } \beta\leq \beta+\underline{x}<\overline{x}\}.
\end{align}
We consider the following auxiliary stochastic control problem: identify $(u^{\star},D^{\star})\in \mathcal{G}_x$ such that
\begin{equation}\label{eq:a008.aux.}
J^+_{u^{\star},D^{\star}}(x)=\sup_{(u,D)\in\mathcal{G}_x}J^+_{u,D}(x),\quad x>0.
\end{equation}
where $J^+_{u,D}$ is as defined in equation \eqref{eq:a006b}.

\begin{lemma}
\label{lem3.1.w}
Let $(\underline{x},\overline{x})$ be such that $\beta\leq \beta+\underline{x}<\overline{x}$. Define
\begin{align}
\mathcal{G}_x^{\underline{x},\overline{x}}:= \{u:(u,D^{\underline{x},\overline{x}})\in\mathcal{G}_x\},\quad 
\end{align}
which is the cut set of $\mathcal{G}_x$ at $D^{\underline{x},\overline{x}}$.
Then the following identity holds
\begin{align}
\sup_{(u,D)\in\mathcal{G}_x}J^+_{u,D}(x)=\sup_{\beta\leq \beta+\underline{x}<\overline{x}}\sup_{u \in \mathcal{G}_x^{\underline{x},\overline{x}}} J^+_{u,D^{\underline{x},\overline{x}}}(x), \quad x \geq 0.\nn
\end{align}
\end{lemma}

\begin{proof}
Let $(u,D)\in\mathcal{G}_x$. Then, by definition, there exists $\underline{x}$ and $\overline{x}$ such that $\beta\leq \beta+\underline{x}<\overline{x}$ and $D=D^{\underline{x},\overline{x}}$. It follows that
\begin{align}
J^+_{u,D}(x)=J^+_{u,D^{\underline{x},\overline{x}}}(x)\leq \sup_{\beta\leq \beta+\underline{x}<\overline{x}}\sup_{u \in \mathcal{G}_x^{\underline{x},\overline{x}}} J^+_{u,D^{\underline{x},\overline{x}}}(x),\quad x\geq 0.\nn
\end{align}
Since $(u,D)$ is arbitrary in $\mathcal{G}_x$, we obtain
\begin{align}
\label{ineq.leq}
\sup_{(u,D)\in\mathcal{G}_x}J^+_{u,D}(x)\leq \sup_{\beta\leq \beta+\underline{x}<\overline{x}}\sup_{u \in \mathcal{G}_x^{\underline{x},\overline{x}}} J^+_{u,D^{\underline{x},\overline{x}}}(x),\quad x\geq 0.
\end{align}
Conversely, for any $u\in\mathcal{G}_x^{\underline{x},\overline{x}}$ with $\beta\leq \beta+\underline{x}<\overline{x}$, it follows that $(u,D^{\underline{x},\overline{x}})\in \mathcal{G}_x$. Therefore,
\begin{align}
\sup_{(u,D)\in\mathcal{G}_x}J^+_{u,D}(x)\geq J^+_{u,D^{\underline{x},\overline{x}}}(x), \quad x \geq 0.\nn
\end{align}
Taking the supremum over $u\in\mathcal{G}_x^{\underline{x},\overline{x}}$ and $(\underline{x},\overline{x})$, we get
\begin{align}
\label{ineq.geq}
\sup_{(u,D)\in\mathcal{G}_x}J^+_{u,D}(x)\geq \sup_{\beta\leq \beta+\underline{x}<\overline{x}}\sup_{u \in \mathcal{G}_x^{\underline{x},\overline{x}}}J^+_{u,D^{\underline{x},\overline{x}}}(x), \quad x \geq 0.
\end{align}
Combining inequalities \eqref{ineq.leq} and \eqref{ineq.geq} yields the desired result.
\end{proof}

With the aid of Lemma \ref{lem3.1.w}, the solution to the auxiliary optimization problem \eqref{eq:a008.aux.} can be decoupled into a two-stage procedure.
\begin{itemize}
\item In the first stage, for any pair $(\underline{x},\overline{x})$ satisfying $\beta\leq \beta+\underline{x}<\overline{x}$, we consider the corresponding family of intermediate control problems indexed by these parameters. For each such index pair, we determine a control $u_{\underline{x},\overline{x}}$ such that
\begin{equation}\label{auxiliary.contr.p.}
V_{\underline{x},\overline{x}}(x):= \sup_{u \in \mathcal{G}_x^{\underline{x},\overline{x}}} J^+_{u,D^{\underline{x},\overline{x}}}(x)=J^+_{u_{\underline{x},\overline{x}},D^{\underline{x},\overline{x}}}(x), \quad x \in \mathbb{R}_{+}.
\end{equation}
Consequently, we obtain a family of optimal proportional reinsurance-dividend strategies, $\left((u_{\underline{x},\overline{x}},D^{\underline{x},\overline{x}})\right)_{\beta\leq \beta+\underline{x}<\overline{x}}$, together with the associated value functions $\left(V_{\underline{x},\overline{x}}(x)\right)_{\beta\leq \beta+\underline{x}<\overline{x}}$, both indexed by the dividend barriers $(\underline{x},\overline{x})$.
\item In the second stage, we seek the uniformly maximal value function among the family $\left(V_{\underline{x},\overline{x}}(x)\right)_{\beta\leq \beta+\underline{x}<\overline{x}}$ of value functions. That is, we identify a pair of dividend barriers $(\underline{x}^{\star},\overline{x}^{\star})$ with $\beta\leq \beta+\underline{x}^{\star}<\overline{x}^{\star}$, such that
\begin{align}
\label{point.point.optimization}
V_{\underline{x}^{\star},\overline{x}^{\star}}(x)
=\sup_{\beta\leq \beta+\underline{x}<\overline{x}} V_{\underline{x},\overline{x}}(x), \quad x\in \mathbb{R}_{+}.
\end{align}
provided that such a maximizer exists.
Lemma \ref{lem3.1.w} then ensures that the admissible strategy
$(u_{\underline{x}^{\star},\overline{x}^{\star}},D^{\underline{x}^{\star},\overline{x}^{\star}})\in \mathcal{G}_x$ obtained indeed solves the auxiliary stochastic control problem \eqref{eq:a008.aux.}.
\end{itemize}
We begin with the analysis of the first stage.
By definition, the performance function $J^+_{u,D^{\underline{x},\overline{x}}}$ satisfies
\begin{align}
\left\{
\begin{array}{ll}
J^+_{u,D^{\underline{x},\overline{x}}}(x) = J^+_{u,D^{\underline{x},\overline{x}}}(\underline{x}) + x - \underline{x} - \beta, & x \geq \overline{x}, \medskip\\
J^+_{u,D^{\underline{x},\overline{x}}}(0) = 0. &
\end{array}
\right.
\nn
\end{align}
It therefore follows that
\begin{align}
\label{boundary.condition}
\left\{
\begin{array}{ll}
V_{\underline{x},\overline{x}}(x) = V_{\underline{x},\overline{x}}(\underline{x}) + x - \underline{x} - \beta, & x \geq \overline{x}, \medskip\\
V_{\underline{x},\overline{x}}(0) = 0. &
\end{array}
\right.
\end{align}
Consequently, in order to determine the expression of $V_{\underline{x},\overline{x}}(x)$ for $x \in [0, \infty)$, it suffices to characterize $V_{\underline{x},\overline{x}}(x)$ for $x \in (0, \overline{x})$.

We now assume that $V_{\underline{x},\overline{x}} \in C^1(0,\overline{x}) \cap C^2\left((0,\overline{x}) \setminus \{a\}\right)$. Under this assumption, we proceed to derive a second-order ordinary differential equation (ODE) satisfied by $V_{\underline{x},\overline{x}}$ and subsequently obtain an explicit expression for $V_{\underline{x},\overline{x}}$ by solving this differential equation. Moreover, using the derived explicit form, we verify that $V_{\underline{x},\overline{x}}$ indeed belongs to $C^1(0,\overline{x}) \cap C^2\big((0,\overline{x}) \setminus \{a\}\big)$, thereby confirming the consistency of our solution with the assumed smoothness properties. 
The ODE satisfied by $V_{\underline{x},\overline{x}}$ is given in the following theorem. Its proof is similar to those of Theorems \ref{8.7.thm.2.5} and \ref{thm2.2}, and therefore omitted.

\begin{theorem}
\label{thm3.1.w}
Fix $(\underline{x},\overline{x})$ satisfying $\beta\leq \beta+\underline{x}<\overline{x}$.
Suppose that $V_{\underline{x},\overline{x}} \in C^1(0,\overline{x}) \cap C^2\left((0,\overline{x}) \setminus \{a\}\right)$ satisfies the boundary condition \eqref{boundary.condition}. Then, $V_{\underline{x},\overline{x}}$ solves
\begin{align}
\label{vbarbar}
0=\max_{u\in[0,1]}\mathcal{L}^{u}f(x),\quad x\in (0,\overline{x})\setminus\{a\}.
\end{align}
Conversely, if $v \in C^1(0,\overline{x}) \cap C^2\left((0,\overline{x}) \setminus \{a\}\right)$ is a solution to \eqref{vbarbar} subject to the boundary condition \eqref{boundary.condition}, then 
\begin{align}
V_{\underline{x},\overline{x}}(x)={v(x)}, \quad x\geq 0,\nn 
\end{align}
and, the optimal reinsurance strategy for the control problem \eqref{auxiliary.contr.p.} is given by
\begin{align}
u^v_{t}= \argmax\limits_{u\in[0,1]}\mathcal{L}^uv(U_t^{v}),\quad \text{a.e. } (t,\omega)\in [0,\infty)\times\Omega,
\nn
\end{align}
where $U^{v}:=(U_t^{v})_{t\geq 0}$ represents the surplus process under the control of $(u^{v},D^{\underline{x},\overline{x}})$. 
\end{theorem}

To explicitly determine the value function $V_{\underline{x}, \overline{x}}$ of \eqref{auxiliary.contr.p.}, it is necessary to solve the HJB equation \eqref{vbarbar}. However, rather than addressing \eqref{vbarbar} directly, we begin by considering the following second-order ordinary differential equation
\begin{align}
\label{eq:a046}
\max_{u\in[0,1]}\bigg\{&\frac{1}{2}(\sigma_{+}^2\textbf{1}_{\{x>a\}}+\sigma_{-}^2\textbf{1}_{\{x\le a\}})u^2v^{\prime\prime}(x)\nn\\
&\quad+(\mu_{+}\textbf{1}_{\{x>a\}}+\mu_{-}\textbf{1}_{\{x\le a\}})uv^{\prime}(x)-qv(x)\bigg\}=0, \quad x\in [0,\infty).
\end{align}
We seek a solution $v$ to \eqref{eq:a046} with the properties of being non-negative, strictly increasing, continuously differentiable, and twice continuously differentiable except at finitely many points.

In the subsequent subsections \ref{subsec.3.1}-\ref{subsec.3.4}, we construct a solution to \eqref{eq:a046} under four mutually exclusive and collectively exhaustive cases based on the signs of the drift coefficients $\mu_+$ and $\mu_-$: (1) $\mu_{\pm}\leq 0$; (2) $\mu_{-}\leq 0$, $\mu_{+}>0$; (3) $\mu_{-}>0$, $\mu_{+}\leq 0$; (4) $\mu_{\pm}>0$. 

To facilitate the analysis, we introduce the following notations: 
\begin{align}
\label{3.3.wr}
\left\{
\begin{array}{ll}
\delta_{\pm}:=\frac{-\mu_{-}\pm\sqrt{\mu_{-}^2+2q\sigma_{-}^2}}{\sigma_{-}^2}, \quad \gamma_{-}:=\frac{2\sigma_{-}^2q}{\mu_{-}^2+2\sigma_{-}^2q},\quad x_{0}^{-}:=\frac{\sigma_{-}^2 \left(1-\gamma_{-}\right)}{\mu_{-}},& \\
\theta_{\pm}:=\frac{-\mu_{+}\pm\sqrt{\mu_{+}^2+2q\sigma_{+}^2}}{\sigma_{+}^2},\quad \gamma_{+}:=\frac{2\sigma_{+}^2 q}{\mu_{+}^2+2\sigma_{+}^2q}, \quad x_{0}^{+}:=\frac{\sigma_{+}^2 \left(1-\gamma_{+}\right)}{\mu_{+}}{\color{blue}-\frac{\gamma_+(e^{\delta_+a}-e^{\delta_-a})}{\delta_+e^{\delta_+a}-\delta_-e^{\delta_-a}}+a},& \\
x^+_1:=\frac{\sigma^2_+(1-\gamma_+)}{\mu_+}+\left(1-\frac{\gamma_+}{\gamma_-}\right)a.&
\end{array}
\right.
\end{align}
These quantities will play a pivotal role in characterizing the solution $v(x)$ to \eqref{eq:a046} and the structure of the solution in each case.
By appropriately truncating the solution $v(x)$ at suitably chosen values $\overline{x}$ and $\underline{x}$, while ensuring that the boundary condition \eqref{boundary.condition} is satisfied, we complete the first stage of addressing the intermediate control problem \eqref{auxiliary.contr.p.}.

To proceed to the second stage, whose objective is to solve \eqref{point.point.optimization}, we introduce several new notations in preparation. A function $g$ is said to be \textbf{piecewise concave-convex} if there exist three inflection points $(a_i)_{1\leq i\leq 3}$ satisfying one of the following conditions
\begin{itemize}
\item $0 < a_1 < a_2 < a_3$,
\item $0 = a_1 < a_2 < a_3$,
\item $0 = a_1 = a_2 < a_3$,
\item $0 = a_1 = a_2 = a_3$,
\end{itemize}
such that $g$ is concave on $[0,a_1]$, convex on $[a_1,a_2]$, concave on $[a_2,a_3]$, and convex on $[a_3,\infty)$. 
We define the generalized inverse functions of $\left.g^{\prime}\right|_{[0,a_1]}$ and $\left.g^{\prime}\right|_{[a_{2},a_{3}]}$ as
\begin{align}
(g^{\prime})_1^{-1}(x):=&~\inf\{z\in[0,a_1]:g^{\prime}(z)\le x\},\quad \,\,\,x\in [g^{\prime}(a_1),\infty),
\nn\\
(g^{\prime})_3^{-1}(x):=&~\inf\{z\in[a_{2},a_{3}]:g^{\prime}(z)\le x\},\quad x\in [g^{\prime}(a_3),\infty).\nn
\end{align}
Furthermore, let
\begin{align}
&
[g^{\prime}(a_1),g^{\prime}(a_2)]\ni x\mapsto (g^{\prime})_2^{-1}(x)\in [a_1,a_2], 
\nn\\
&\quad\,\, [g^{\prime}(a_3),\infty)\ni x\mapsto (g^{\prime})_4^{-1}(x)\in [a_3,\infty),\nn
\end{align}
denote, respectively, the inverse functions of $\left.g^{\prime}(x)\right|_{[a_1,a_2]}$ 
and $\left.g^{\prime}(x)\right|_{[a_3,\infty)}$. 
In addition, we define two functions $\psi$ and $\phi$ as
\begin{align}
\label{def.psi.}
\psi(x,y):=&~\int_{x}^{y}\left(1-\frac{g^{\prime}(z)}{g^{\prime}(y)}\right)\mathrm{d}z, \quad 0\leq x<y<\infty,
\\
\label{phi.1}
\phi(x):=&~\psi((g^{\prime})_{3}^{-1}(g^{\prime}(x)),x),
\quad x\in[a_3,\infty).
\end{align}
Additionally, let 
\begin{align}
\left\{
\begin{array}{ll}
b_1:=\inf\{x\ge a_3:g^{\prime}(x)\ge g^{\prime}(a_1)\},& \\
b_2:=(g^{\prime})_4^{-1}(g^{\prime}(a_2)).& 
\end{array}
\right.
\nn
\end{align}
Then, $b_1<b_2$ when $0 < a_1 < a_2 < a_3$ or $0 = a_1 < a_2 < a_3$; and $b_1=b_2$ when $0 = a_1 = a_2 < a_3$ or $0 = a_1 =a_2 =a_3$.
If $b_1<b_2$, we further let
\begin{align}
\label{13.17.g.}
\left\{\begin{array}{ll}
x_1:=\inf\Big\{x\in[b_1,b_2]:\int_{(g^{\prime})_1^{-1}(g^{\prime}(x))}^{(g^{\prime})_3^{-1}(g^{\prime}(x))}\Big(1-\frac{g^{\prime}(s)}{g^{\prime}(x)}\Big)\ds\ge0\Big\},& \\
x_2:=\inf\left\{x\in[b_1,b_2]:\int_{(g^{\prime})_2^{-1}(g^{\prime}(x))}^{x}\Big(1-\frac{g^{\prime}(s)}{g^{\prime}(x)}\Big)\ds\ge0\right\},& 
\end{array}
\right.
\end{align}
and
\begin{align}
\omega_1(x):=\int_{(g^{\prime})_1^{-1}(g^{\prime}(x))}^{x}\left(1-\frac{g^{\prime}(s)}{g^{\prime}(x)}\right)\ds,~~~x\in[a_1,(g^{\prime})_2^{-1}(g^{\prime}(x_2)))\cup[x_2,\infty],
\end{align}
and
\begin{equation}
\omega_2(x):=\begin{cases}
\int_{(g^{\prime})_3^{-1}(g^{\prime}(x))}^{x}\left(1-\frac{g^{\prime}(s)}{g^{\prime}(x)}\right)\ds,~~~x\in[a_3,x_1],\\
\int_{(g^{\prime})_1^{-1}(g^{\prime}(x))}^{x}\left(1-\frac{g^{\prime}(s)}{g^{\prime}(x)}\right)\ds,~~~x\in[x_1,\infty].
\end{cases}
\end{equation}
Moreover, when $b_{1}<b_{2}$,
let
\begin{align}
\label{def.max.points}
\left\{
\begin{array}{ll}
(\hat{z}_1,\hat{z}_2):=((g^{\prime})_{1}^{-1}(g^{\prime}(\omega_1^{-1}(\beta))),\omega_1^{-1}(\beta)),& \\
(\tilde{z}_1,\tilde{z}_2):=((g^{\prime})_{3}^{-1}(g^{\prime}(\omega_2^{-1}(\beta))),\omega_2^{-1}(\beta)),& \\
A_1:=\{\beta>0:g^{\prime}(\omega_1^{-1}(\beta))<g^{\prime}(\omega_2^{-1}(\beta))\},& \\
A_2:=\{\beta>0:g^{\prime}(\omega_1^{-1}(\beta))>g^{\prime}(\omega_2^{-1}(\beta))\},& \\
A_3:=(\omega_2(x_1),\infty),& 
\end{array}
\right.
\end{align}
where $\omega_1^{-1}$ and $\omega_{2}^{-1}$ are the inverse functions of $\omega_1$ and $\omega_2$, respectively. Note that $A_2\cap A_3=\emptyset$. Denote the complement of $A$ in $(0,\infty)$ by $\overline{A}:=(0,\infty)\backslash A$.

We are now ready to solve the control problem \eqref{eq:a008.aux.} in the following subsections \ref{subsec.3.1}-\ref{subsec.3.4}.

\subsection{Case $\mu_{\pm}\le0$}
\label{subsec.3.1}

\subsubsection{Construction of a solution to \eqref{eq:a046}}
\label{subsec.3.1.1}

\begin{theorem}
\label{thm:case1.solution.eq.a046}
Suppose that $\mu_{\pm}\leq0$. Let 
\begin{align}
\label{eq:ee06}
a_{11}=\dfrac{e^{\delta_{+}a}(\theta_{-}-\delta_{+})-e^{\delta_{-}a}(\theta_{-}-\delta_{-})}{\theta_{-}-\theta_{+}},\quad a_{12}=\dfrac{e^{\delta_{+}a}(\delta_{+}-\theta_{+})-e^{\delta_{-}a}(\delta_{-}-\theta_{+})}{\theta_{-}-\theta_{+}},
\end{align}
and 
\begin{align}\label{eq:ee07}
g(x)=&~\begin{cases}
(e^{\delta_{+}x}-e^{\delta_{-}x}),&x\in[0,a],\\
\left(a_{11}e^{\theta_{+}(x-a)}+a_{12}e^{\theta_{-}(x-a)}\right),&x\in[a,\infty). 
\end{cases}
\end{align}
Then, for any constant $C>0$, the function $v=C\,g_{11}$ 
is a non-negative, strictly increasing, continuously differentiable, and piecewise twice continuously differentiable solution to \eqref{eq:a046} on $\setd$. Moreover, the maximum in \eqref{eq:a046} is attained at $u^{\star}=1$ on $(0,\infty)$.
\end{theorem}
\begin{proof}
We first consider the interval $[0,a]$. On this interval, \eqref{eq:a046} reduces to
\begin{equation}\label{eq:ee01}
\max_{u\in[0,1]}\left\{\frac{1}{2}\sigma_{-}^2u^2v^{\prime\prime}(x)+\mu_{-}uv^{\prime}(x)-qv(x)\right\}=0, \quad x\in [0,a].
\end{equation}
If $v^{\prime\prime}<0$ on some subinterval of $[0,a]$, then the quadratic term in $u$ is concave and, since $\mu_{-}\leq0$ and $v^{\prime}\geq0$, the maximum is attained at $u=0$. If $v^{\prime\prime}=0$ on some subinterval, then the expression is linear in $u$ and again attains its maximum at $u=0$. In either case, the equation would reduce to $-qv=0$ on that subinterval, which is incompatible with the desired non-negative and strictly increasing solution. Hence, we seek a solution satisfying \begin{align}
\label{3.49tobe.verif.}
\left\{
\begin{array}{ll}
v^{\prime\prime}(x)> 0,& x\in[0,a], \\
-\frac{\mu_{-}v^{\prime}(x)}{\sigma_{-}^2v^{\prime\prime}(x)}\le\frac{1}{2} \,\,\,(\Longleftrightarrow \frac{\sigma_{-}^2}{2}v^{\prime\prime}(x)+\mu_{-}v^{\prime}(x)\geq 0),& x\in[0,a].
\end{array}
\right.
\end{align}
Under condition \eqref{3.49tobe.verif.}, the quadratic trinomial in $u\in[0,1]$ on the left hand side of \eqref{eq:ee01} attains its maximum at $u=1$. Hence, \eqref{eq:ee01} is reduced to
\begin{align}
\label{3.49tobe.verif.1}
\frac{1}{2}\sigma_{-}^2v^{\prime\prime}(x)+\mu_{-}v^{\prime}(x)-qv(x)=0, \quad x\in [0,a],
\end{align}
whose solution, subject to the boundary condition $v(0)=0$, is expressed as
\begin{equation}\label{eq:ee02}
v(x)=C(e^{\delta_{+}x}-e^{\delta_{-}x}), \quad x\in [0,a],
\end{equation}
where $\delta_{\pm}=\frac{-\mu_{-}\pm\sqrt{\mu_{-}^2+2q\sigma_{-}^2}}{\sigma_{-}^2}$, and $C$ is some unknown positive constant.
Since $\delta_{+}>0$, $\delta_{-}<0$, and $\delta_{+}>|\delta_{-}|$, we have
\begin{align*}
v^{\prime\prime}(x)=C(\delta_{+}^2e^{\delta_{+}x}-\delta_{-}^2e^{\delta_{-}x})>0,\quad x\in[0,a].
\end{align*}
Moreover, using the defining equation of $\delta_{\pm}$, we obtain
\begin{align*}
\frac{1}{2}\sigma_{-}^2v^{\prime\prime}(x)+\mu_{-}v^{\prime}(x)
=Cq(e^{\delta_{+}x}-e^{\delta_{-}x})\geq0,\quad x\in[0,a].
\end{align*}
Therefore, the constructed function \eqref{eq:ee02} solves \eqref{eq:a046} on $[0,a]$ and the maximizer is $u=1$.

We next consider the interval $[a,\infty)$. 
When restricted to $[a,\infty)$, \eqref{eq:a046} becomes
\begin{equation}\label{eq:ee03}
\max_{u\in[0,1]}\left\{\frac{1}{2}\sigma_{+}^2u^2v^{\prime\prime}(x)+\mu_{+}uv^{\prime}(x)-qv(x)\right\}=0,\quad x\in[a,\infty).
\end{equation} 
Under the same reasoning, a valid increasing solution should satisfy
\begin{align}
\label{3.54.to.be.veri.}
\left\{
\begin{array}{ll}
v^{\prime\prime}(x)>0,& \,\, x\in[a,\infty), \\
-\frac{\mu_{+}v^{\prime}(x)}{\sigma_{+}^2v^{\prime\prime}(x)}\le\frac{1}{2} \,\,\,(\Leftrightarrow \frac{\sigma_{-}^2}{2}v^{\prime\prime}(x)+\mu_{-}v^{\prime}(x)\geq 0),& \,\, x\in[a,\infty),
\end{array}
\right.
\end{align}
so that the maximum is attained at $u=1$. Hence \eqref{eq:ee03} reduces to
\begin{align*}
\frac{1}{2}\sigma_{+}^2v^{\prime\prime}(x)+\mu_{+}v^{\prime}(x)-qv(x)=0,\quad x\in[a,\infty).
\end{align*}
The general solution of this ordinary differential equation is
\begin{equation}\label{eq:ee04}
v(x)=C(a_{11}e^{\theta_{+}(x-a)}+a_{12}e^{\theta_{-}(x-a)}),\quad x\in[a,\infty),
\end{equation}
The continuous differentiability of $v$ at $x=a$ gives
\begin{equation}\label{eq:ee05}
\begin{cases}
e^{\delta_{+}a}-e^{\delta_{-}a}=a_{11}+a_{12},\\
\delta_{+}e^{\delta_{+}a}-\delta_{-}e^{\delta_{-}a}=a_{11}\theta_{+}+a_{12}\theta_{-}.\\
\end{cases}
\end{equation}
and solving this system yields the stated expressions of $a_{11}$ and $a_{12}$.

It remains to verify that the assumed inequalities \eqref{3.54.to.be.veri.} indeed hold. First, we claim that $a_{11}>0$. Indeed,
\begin{itemize}
\item If $\theta_{-}\ge\delta_{-}$, then
\begin{align}
a_{11}=\frac{(\delta_{+}-\theta_{-})e^{\delta_{+}a}+
(\theta_{-}-\delta_{-})e^{\delta_{-}a}}{\theta_{+}-\theta_{-}}>0.
\nn
\end{align}
\item If $\theta_{-}<\delta_{-}$, then $a_{11}$ is increasing with respect to $a$ with $\left.a_{11}\right|_{a=0}=\frac{\delta_{+}-\delta_{-}}{\theta_{+}-\theta_{-}}>0$.
Hence $a_{11}>0$, as claimed.
\end{itemize}

In addition, it is straightforward to verify that
\begin{align}
\label{C1>|C2|ifC2<0}
a_{11}+a_{12}=\frac{(\theta_{+}-\theta_{-})e^{\delta_{+}a}-(\theta_{+}-\theta_{-})e^{\delta_{-}a}}{\theta_{+}-\theta_{-}}
= e^{\delta_{+}a}-e^{\delta_{-}a}>0,
\end{align}
since $a>0$.
Combining \eqref{C1>|C2|ifC2<0} and $a_{11}>0$, we deduce the following:
\begin{itemize}
\item If $a_{12}<0$, then
\begin{align}
v^{\prime\prime}(x)=C\left(a_{11}\theta_{+}^{2}e^{\theta_{+}(x-a)}+a_{12}\theta_{-}^{2}e^{\theta_{-}(x-a)}\right),\quad x\in[a,\infty),\nn
\end{align}
is strictly increasing on $[a,\infty)$ with $v^{\prime\prime}(a+)=C(a_{11}\theta_{+}^{2}+a_{12}\theta_{-}^{2})>0,$
since $|a_{12}|< a_{11}$ (see \eqref{C1>|C2|ifC2<0}) and $|\theta_{-}|<\theta_{+}$. Hence $v^{\prime\prime}(x)> 0$ for all $x\in[a,\infty)$.
\item If $a_{12}\geq 0$, then, by \eqref{eq:ee04}, it is evident that $v^{\prime\prime}(x)> 0$ for all $x\in[a,\infty)$.
\end{itemize}
Thus, the first inequality in \eqref{3.54.to.be.veri.} is established. We now proceed to verify the second inequality in \eqref{3.54.to.be.veri.}. Equivalently, it suffices to prove
\begin{align}\label{3.60.verified}
-\mu_{+}v^{\prime}(x)-\frac{1}{2}\sigma_{+}^{2}v^{\prime\prime}(x)\leq0,\quad x\in[a,\infty).
\end{align}
Direct calculation gives
\begin{align}\label{eq:ee08}
-\mu_{+}v^{\prime}(x)-\frac{1}{2}\sigma_{+}^{2}v^{\prime\prime}(x)
=
\frac{Cq}{\theta_{+}-\theta_{-}}
\left(\zeta_{1}(a)e^{\theta_{+}(x-a)}+\zeta_{2}(a)e^{\theta_{-}(x-a)}\right),
\end{align}
where
\begin{align*}
\begin{cases}
\zeta_{1}(a):=(\delta_{-}-\theta_{-})e^{\delta_{-}a}-(\delta_{+}-\theta_{-}) e^{\delta_{+}a},\\
\zeta_{2}(a):=(\delta_{+}-\theta_{+})e^{\delta_{+}a}+(\theta_{+}-\delta_{-})e^{\delta_{-}a}.
\end{cases}
\end{align*}
We can make the following observations:
\begin{itemize}
\item $\zeta_{1}(a)<0$.

If $\delta_{-}\geq \theta_{-}$, then $\zeta_{1}$ is a decreasing function of $a$, and since $\zeta_{1}(0)=\delta_{-}-\delta_{+}<0$, it follows that $\zeta_{1}(a)<0$ for all $a>0$. On the other hand, if $\delta_{-}< \theta_{-}$, the inequality $\zeta_{1}(a)<0$ follows directly from the explicit expression for $\zeta_{1}$.
\item $\zeta_{1}(a)+\zeta_{2}(a)<0$.

From the definitions of $\zeta_{1}$ and $\zeta_{2}$, we obtain $\zeta_{1}(a) + \zeta_{2}(a)= (\theta_{+} - \theta_{-})\left( e^{\delta_{-}a} - e^{\delta_{+}a} \right) < 0,$
where the strict inequality holds because $a > 0$.
\end{itemize}
Consequently, if $\zeta_{2}(a)\leq 0$, then the RHS of \eqref{eq:ee08} is manifestly non-positive, and hence \eqref{3.60.verified} holds;
whereas if $\zeta_{2}(a)> 0$, the RHS of \eqref{eq:ee08} is decreasing in $x$, and satisfies
$$-\mu_{+}v^{\prime}(a+)-\frac{\sigma_{+}^2}{2}v^{\prime\prime}(a+)=\frac{Cq}{\theta_{+}-\theta_{-}}
\left(\zeta_{1}(a)
+\zeta_{2}(a)\right)<0.$$
Thus, in both cases, whether $\zeta_{2}(a)\leq 0$ or $\zeta_{2}(a)> 0$, inequality \eqref{3.60.verified} holds. This verifies the second inequality in \eqref{3.54.to.be.veri.}.
We therefore conclude that the function $v$ is indeed a solution to \eqref{eq:a046} on $[a,\infty)$.
\end{proof}

\subsubsection{Identification of the uniform (in $x$) maximizer of $V_{\underline{x},\overline{x}}(x)$ over $(\underline{x},\overline{x})$}
\label{sub3.1.2}

The following theorem provides the solution to the control problem \eqref{eq:a008.aux.} with the set $\mathcal{G}_x$ of admissible reinsurance-dividend strategies restricted in such a way that only the
bouble-barrier dividend strategies are allowed.

\begin{theorem}
\label{thm3.2.ww}
Suppose that $\mu_{\pm}\leq 0$.
Let $g:=g_{11}$ in \eqref{eq:ee07} and recall that $\psi$ is defined in \eqref{def.psi.}.
The reinsurance-dividend strategy $(u^{\star},D^{\star})$ with
\begin{align}
\label{op.str.bigcase1}
\left\{
\begin{array}{ll}
u_{t}^{\star}=1,& t\geq 0, \\
D^{\star}_{t}=D^{0,\overline{x}^{\star}}_{t},& t\geq 0,
\end{array}
\right.
\end{align}
is an optimal strategy for the auxiliary control problem \eqref{eq:a008.aux.}.
The corresponding value function is given by
\begin{equation}\label{eq:ee07.www}
V_{0,\overline{x}^{\star}}(x)=
\left\{\begin{array}{ll}
\frac{g(x)}{g^{\prime}(\overline{x}^{\star})},&x\in[0,\overline{x}^{\star}), \\
x-\beta,& x\in[\overline{x}^{\star},\infty),
\end{array}
\right.\nn
\end{equation}
where $\overline{x}^{\star}$ is determined by the condition $\psi(0,\overline{x}^{\star})=\beta$.
\end{theorem}
\begin{proof}
By Lemma \ref{lem3.1.w}, we can decouple the auxiliary optimization problem \eqref{eq:a008.aux.} into problem \eqref{auxiliary.contr.p.} and problem \eqref{point.point.optimization}.
From Theorem \ref{thm3.1.w} and the construction of a solution to \eqref{eq:a046} in Section \ref{subsec.3.1.1}, it follows that the value function of problem \eqref{auxiliary.contr.p.}, $V_{\underline{x},\overline{x}}$, takes the form
\begin{equation}\label{eq:ee07.w}
V_{\underline{x},\overline{x}}(x)=
\left\{\begin{array}{ll}
Cg(x),&x\in[0,\overline{x}), \\
Cg(\underline{x})+x-\underline{x}-\beta,& x\in[\overline{x},\infty),
\end{array}
\right.
\end{equation}
and the constants $C$, $\underline{x}$, and $\overline{x}$ satisfy
\begin{align}
\label{3.18.w}
\left\{
\begin{array}{ll}
C g^{\prime}(\overline{x})=1,& \\
C\left(g(\overline{x})-g(\underline{x})\right)=\overline{x}-\underline{x}-\beta.&
\end{array}
\right.
\end{align}
The first equality in \eqref{3.18.w} ensures that $V_{\underline{x},\overline{x}}\in C^{1}(0,\infty)$. Consequently, $V_{\underline{x},\overline{x}}$ can be written equivalently as
\begin{equation}\label{eq:ee07.ww}
V_{\underline{x},\overline{x}}(x)=
\left\{\begin{array}{ll}
\frac{g(x)}{g^{\prime}(\overline{x})},&x\in[0,\overline{x}), \\
\frac{g(\underline{x})}{g^{\prime}(\overline{x})}+x-\underline{x}-\beta,& x\in[\overline{x},\infty),
\end{array}
\right.
\end{equation}
where $\underline{x}$ and $\overline{x}$ satisfy
\begin{align}
\label{3.18.ww}
\frac{g(\overline{x})-g(\underline{x})}{g^{\prime}(\overline{x})}=\overline{x}-\underline{x}-\beta, 
\end{align}
which is equivalent to
\begin{align}
\label{3.20.w}
\psi(\underline{x},\overline{x})=\beta. 
\end{align} 
From the construction of $g=g_{11}$ in \eqref{eq:ee07}, we have
\begin{align}
\label{g''>0}
g^{\prime\prime}(x)>0,\quad x\in [0,\infty).
\end{align}
That is, $g$ is \textbf{piecewise concave-convex} with $a_{1}=a_{2}=a_{3}=0$.
It follows that
\begin{align}
\label{mon.prop.}
\left\{
\begin{array}{ll}
\frac{\partial}{\partial x}\psi(x,y)=-\left(1-\frac{g^{\prime}(x)}{g^{\prime}(y)}\right)<0,&0\leq x<y<\infty, \\
\frac{\partial}{\partial y}\psi(x,y)=\int_{x}^{y}\frac{g^{\prime}(z)g^{\prime\prime}(y)}{\left[g^{\prime}(y)\right]^{2}}\mathrm{d}z>0,& 0\leq x<y<\infty.
\end{array}
\right.
\end{align}

We now claim that the uniform (in $x$) maximizer of $V_{\underline{x},\overline{x}}(x)$ over $(\underline{x},\overline{x})$ is uniquely attained at $(0,\overline{x}^{\star})$, where $\overline{x}^{\star}$ is the unique solution to $\psi(0,y)=0$. To verify this claim, suppose $(\underline{x},\overline{x})$ is any solution to \eqref{3.20.w} different from $(0,\overline{x}^{\star})$. Then necessarily
\begin{align}
\label{comp.barr.}
\underline{x}>0,\quad \overline{x}>\overline{x}^{\star}.
\end{align}
Using \eqref{g''>0}, \eqref{mon.prop.}, and \eqref{comp.barr.}, we obtain the following inequalities
\begin{itemize}
\item For $x\in (0,\overline{x}^{\star}]$,
\begin{align}\label{eq:v-v.4}
V_{0,\overline{x}^{\star}}(x)-V_{\underline{x},\overline{x}}(x)=g(x)\left(\frac{1}{g^{\prime}(\overline{x}^{\star})}-\frac{1}{g^{\prime}(\overline{x})}\right)>0.
\end{align}
\item For $x\in(\overline{x}^{\star},\overline{x}]$,
\begin{align}\label{eq:v-v.5}
V_{0,\overline{x}^{\star}}(x)-V_{\underline{x},\overline{x}}(x)=&~
x-\beta-\frac{g(x)}{g^{\prime}(\overline{x})}
\nn\\
=&~\int_{0}^{\overline{x}^{\star}}\left(1-\frac{g^{\prime}(z)}{g^{\prime}(\overline{x})}\right)\mathrm{d}z
+\int_{\overline{x}^{\star}}^{x}\left(1-\frac{g^{\prime}(z)}{g^{\prime}(\overline{x})}\right)\mathrm{d}z-\beta
\nn\\
>&~\int_{0}^{\overline{x}^{\star}}\left(1-\frac{g^{\prime}(z)}{g^{\prime}(\overline{x}^{\star})}\right)\mathrm{d}z
+\int_{\overline{x}^{\star}}^{x}\left(1-\frac{g^{\prime}(z)}{g^{\prime}(\overline{x})}\right)\mathrm{d}z-\beta
\nn\\
=&~\int_{\overline{x}^{\star}}^{x}\left(1-\frac{g^{\prime}(z)}{g^{\prime}(\overline{x})}\right)\mathrm{d}z>0.
\end{align}
\item For $x\in(\overline{x},\infty)$,
\begin{align}\label{eq:v-v.6}
V_{0,\overline{x}^{\star}}(x)-V_{\underline{x},\overline{x}}(x)=&~
x-\beta-(\frac{g(\underline{x})}{g^{\prime}(\overline{x})}+x-\underline{x}-\beta)
\nn\\
=&~\underline{x}-\frac{g(\underline{x})}{g^{\prime}(\overline{x})}
\nn\\
=&~\int_{0}^{\underline{x}}\left(1-\frac{g^{\prime}(z)}{g^{\prime}(\overline{x})}\right)\mathrm{d}z>0.
\end{align}
\end{itemize}
Therefore, for any solution $(\underline{x},\overline{x})$ of \eqref{3.20.w} satisfying $\beta\leq \beta+\underline{x}<\overline{x}$, we have $V_{0,\overline{x}^{\star}}(x)\geq V_{\underline{x},\overline{x}}(x)$ for all $x\geq 0$, confirming the claim. That is to say, $(0,\overline{x}^{\star})$ is the solution to problem \eqref{point.point.optimization}. Furthermore, based on Theorem \ref{thm3.1.w} and the construction of $V_{\underline{x},\overline{x}}$ in Section \ref{subsec.3.1.1}, the optimal reinsurance strategy for the control problem \eqref{auxiliary.contr.p.} with $(\underline{x},\overline{x})=(0,\overline{x}^{\star})$ is given by $u^{\star}_t=1$ for $t\geq0$.
Then, the optimality of \eqref{op.str.bigcase1} to the problem \eqref{eq:a008.aux.} is confirmed by Lemma \ref{lem3.1.w}.
The proof is complete.
\end{proof}

\subsection{Case $\mu_{-}\le0, \,\mu_{+}>0$}
\label{subsec.3.2}

\subsubsection{Construction of a solution to \eqref{eq:a046}}
\label{subsubsec.3.2.1}
\begin{theorem}
\label{thm:case2.solution.eq.a046}
Suppose that $\mu_{-}\leq0$ and $\mu_{+}>0$. Then, for any constant $C>0$, a non-negative, strictly increasing, continuously differentiable, and piecewise twice continuously differentiable solution to \eqref{eq:a046} is given as follows. If $x_{0}^{+}>a$, then
\begin{align}
\label{21.1.w}
v(x)=&~\left\{
\begin{array}{ll}
C(e^{\delta_{+}x}-e^{\delta_{-}x}),& x\in[0,a], \\
C {\left(a_{21}\frac{x}{a}+a_{22}\right)^{\gamma_+}},& x\in [a,x_{0}^{+}),\\
C \left(a_{23}e^{\theta_{+}(x-x_{0}^{+})}+a_{24}e^{\theta_{-}(x-x_{0}^{+})}\right),& x\in [x_{0}^{+},\infty),
\end{array}
\right.
\\
:=&~C\,g_{21}(x)\nn,
\end{align}
where
\begin{align}
\begin{cases}\label{3.26.w}
a_{21}=\dfrac{a\left(\delta_{+}e^{\delta_{+}a}-\delta_{-}e^{\delta_{-}a}\right)}{\gamma_{+}\left(e^{\delta_{+}a}-e^{\delta_{-}a}\right)^{(\gamma_{+}-1)/\gamma_{+}}},\\[1.2ex]
a_{22}=\left(e^{\delta_{+}a}-e^{\delta_{-}a}\right)^{1/\gamma_{+}}
\left[1-\dfrac{a\left(\delta_{+}e^{\delta_{+}a}-\delta_{-}e^{\delta_{-}a}\right)}
{\gamma_{+}\left(e^{\delta_{+}a}-e^{\delta_{-}a}\right)}\right],\\[1.2ex]
a_{23}=\dfrac{\gamma_{+}\dfrac{a_{21}}{a}
\left(a_{21}\dfrac{x_{0}^{+}}{a}+a_{22}\right)^{\gamma_{+}-1}
-\theta_{-}\left(a_{21}\dfrac{x_{0}^{+}}{a}+a_{22}\right)^{\gamma_{+}}}{\theta_{+}-\theta_{-}},\\[1.2ex]
a_{24}=\dfrac{\theta_{+}\left(a_{21}\dfrac{x_{0}^{+}}{a}+a_{22}\right)^{\gamma_{+}}
-\gamma_{+}\dfrac{a_{21}}{a}
\left(a_{21}\dfrac{x_{0}^{+}}{a}+a_{22}\right)^{\gamma_{+}-1}}
{\theta_{+}-\theta_{-}},
\end{cases}
\end{align}
and the maximizing control is
\begin{align*}
u^{\star}(x)=\left\{
\begin{array}{ll}
1,&x\in[0,a],\\
-\dfrac{\mu_{+}v^{\prime}(x)}{\sigma_{+}^{2}v^{\prime\prime}(x)}
=\dfrac{\mu_{+}(a_{21}x+a_{22}a)}{\sigma_{+}^{2}a_{21}(1-\gamma_{+})},&x\in(a,x_{0}^{+}),\\
1,&x\in[x_{0}^{+},\infty).
\end{array}
\right.
\end{align*}
If $x_{0}^{+}\leq a$, then
\begin{align}
\label{g22.g22}
v(x)=&~\left\{
\begin{array}{ll}
C(e^{\delta_{+}x}-e^{\delta_{-}x}),& x\in[0,a], \\
C \left(a_{11}e^{\theta_{+}(x-a)}+a_{12}e^{\theta_{-}(x-a)}\right),& x\in [a,\infty),
\end{array}
\right.
\\
:=&~C\,g_{22}(x),\nn
\end{align}
where $a_{11}$ and $a_{12}$ are given by \eqref{eq:ee06}, and the maximizing control $u^{\star}(x)=1$ on $(0,\infty)$.
\end{theorem}
\begin{proof}
By arguments analogous to those employed in the case $\mu_{\pm}\leq 0$, it follows that the solution $v$ of \eqref{eq:a046} on the interval $[0,a]$ is given by \eqref{eq:ee02}, and, the maximizer is $u^{\star}(x)=1$. 

It remains to determine a solution to
\begin{equation}\label{eq:ee03.{2}}
\max_{u\in[0,1]}\left\{\frac{1}{2}\sigma_{+}^2u^2v^{\prime\prime}(x)+\mu_{+}uv^{\prime}(x)-qv(x)\right\}=0,\quad x\in[a,\infty),
\end{equation}
with the required properties, that is, $v$ must be non-negative, strictly increasing, continuously differentiable, and twice continuously differentiable except at finitely many points.
To infer the qualitative behavior of the solution to \eqref{eq:ee03.{2}}, we make the following observations.
\begin{itemize}
\item Suppose there is a sub-interval $(b,c)\subseteq [a,\infty)$ ($b<c$) such that $v^{\prime\prime}(x)<0$ for all $x\in (b,c)$. Then, for any fixed $x\in (b,c)$, the quadratic trinomial in $u\in[0,1]$ on the left hand side of \eqref{eq:ee03.{2}} attains its maximum at $u=\min\{-\frac{\mu_{+}v^{\prime}(x)}{\sigma_{+}^2v^{\prime\prime}(x)},1\}$. Consequently, on the interval $(b,c)$, \eqref{eq:ee03.{2}} is reduced to
\begin{align}
\left\{
\begin{array}{ll}
-\frac{\mu_{+}^{2}\left[v^{\prime}(x)\right]^{2}}{2\sigma_{+}^{2}v^{\prime\prime}(x)}-qv(x)=0,& x\in (b,c)\cap \{x\geq a: -\frac{\mu_{+}v^{\prime}(x)}{\sigma_{+}^2v^{\prime\prime}(x)}< 1\}, \\
\frac{1}{2}\sigma_{+}^2v^{\prime\prime}(x)+\mu_{+}v^{\prime}(x)-qv(x)=0,& x\in (b,c)\cap \{x\geq a: -\frac{\mu_{+}v^{\prime}(x)}{\sigma_{+}^2v^{\prime\prime}(x)}\geq 1\}.
\end{array}
\right.
\nn
\end{align}

\item Suppose there exists a sub-interval $(b,c)\subseteq [a,\infty)$ ($b<c$) such that $v^{\prime\prime}(x)=0$ for all $x\in (b,c)$. Then $v^{\prime}(x)\equiv v_{0}$ on $(b,c)$ for some constant $v_{0}\geq 0$. For any fixed $x\in (b,c)$, the quadratic trinomial in $u\in[0,1]$ on the left hand side of \eqref{eq:ee03.{2}} becomes a linear function. Since $\mu_{+}> 0$ and $v^{\prime}(x)\geq 0$, this linear function attains its maximum at $u=1$. Therefore, on the interval $(b,c)$, \eqref{eq:ee03.{2}} reduces to
\begin{align}
\left\{
\begin{array}{ll}
\mu_{+} v_{0}-qv(x)=0, & x\in (b,c),\\
v^{\prime}(x)=v_{0}, & x\in (b,c),\nn
\end{array}
\right.
\end{align}
to which the unique solution is $v\equiv 0$, which contradicts the requirement that $v$ be strictly increasing.
\item Now suppose there exists a sub-interval $(b,c)\subseteq [a,\infty)$ ($b<c$) such that $v^{\prime\prime}(x)>0$ for all $x\in (b,c)$. Then, for any fixed $x\in (b,c)$, the quadratic trinomial in $u\in[0,1]$ on the left hand side of \eqref{eq:ee03.{2}} attains its maximum at $u=1$. Thus, on $(b,c)$, \eqref{eq:ee03.{2}} simplifies to
\begin{align}
\frac{1}{2}\sigma_{+}^2v^{\prime\prime}(x)+\mu_{+}v^{\prime}(x)-qv(x)=0,\quad x\in (b,c).
\nn
\end{align}
\end{itemize}

We begin by assuming that $x_{0}^{+}>a$ and the solution $v$ to \eqref{eq:ee03.{2}} satisfies
\begin{align}
\label{3.24.tobeferi.}
\left\{
\begin{array}{ll}
v^{\prime\prime}(x)<0,& x\in [a,x_{0}^{+}), \\
-\frac{\mu_{+}v^{\prime}(x)}{\sigma_{+}^2v^{\prime\prime}(x)}< 1,& x\in [a,x_{0}^{+}).
\end{array}
\right.
\end{align}
Under these assumptions, on the interval $[a,x_{0}^{+})$, \eqref{eq:ee03.{2}} reduces to
\begin{align}
-\frac{\mu_{+}^{2}\left[v^{\prime}(x)\right]^{2}}{2\sigma_{+}^{2}v^{\prime\prime}(x)}-qv(x)=0,\quad x\in [a,x_{0}^{+}),
\nn
\end{align}
which admits the general solution
\begin{align}
\label{3.24.w}
v(x)&= 
C\left(a_{21}\frac{x}{a}+a_{22}\right)^{\gamma_+},\quad x\in [a,x_{0}^{+}),
\end{align}
where $\gamma_{+}=\frac{2\sigma_{+}^2 q}{\mu_{+}^2+2\sigma_{+}^2q}<1$. 
The continuous differentiability of $v$ at $x=a$ gives
\begin{align}
\label{3.25.w}
\left\{
\begin{array}{ll}
e^{\delta_{+}a}-e^{\delta_{-}a}={(a_{21}+a_{22})^{\gamma_+}},& \\
\delta_{+}e^{\delta_{+}a}-\delta_{-}e^{\delta_{-}a}={a_{21}\frac{\gamma_{+}}{a}(a_{21}+a_{22})^{\gamma_+-1}}.&
\end{array}
\right.
\end{align}
Solving this system yields the stated expressions of $a_{21}$ and $a_{22}$. In particular, $a_{21}>0$. Therefore, we verify that
\begin{align}
\label{3.27.w}
\left\{\begin{array}{ll}
v^{\prime\prime}(x)={C\gamma_+(\gamma_+-1)\frac{a_{21}^2}{a^2}\left(a_{21}\frac{x}{a}+a_{22}\right)^{\gamma_+-2}<0,} 
& x\in [a,x_{0}^{+}), \\
-\frac{\mu_{+}v^{\prime}(x)}{\sigma_{+}^2v^{\prime\prime}(x)}={\frac{ \mu_+\left(a_{21}x+a_{22}a\right)}{\sigma_+^2a_{21}(1-\gamma_+)}<1,}
& x\in [a,x_{0}^{+}),
\end{array}
\right.
\end{align}
confirming that condition \eqref{3.24.tobeferi.} is indeed satisfied. Hence, the constructed function \eqref{21.1.w} solves \eqref{eq:a046} on $[a,x^+_0]$ and the maximizer is $u^{\star}(x)=-\frac{\mu_{+}v^{\prime}(x)}{\sigma_{+}^{2}v^{\prime\prime}(x)}$ with $u^{\star}(x^+_0)=1$.

On the other hand, on the interval $[x_{0}^{+},\infty)$, the simultaneous satisfaction of $v^{\prime\prime}(x)<0$ and $-\frac{\mu_{+}v^{\prime}(x)}{\sigma_{+}^2v^{\prime\prime}(x)}<1$ is impossible. Combined with the observations stated immediately after \eqref{eq:ee03.{2}}, this implies that on $[x_{0}^{+},\infty)$, \eqref{eq:ee03.{2}} simplifies to
\begin{align}
\frac{1}{2}\sigma_{+}^2v^{\prime\prime}(x)+\mu_{+}v^{\prime}(x)-qv(x)=0,\quad x\in [x_{0}^{+},\infty),
\nn
\end{align}
which has the general solution
\begin{align}
\label{C3,C4}
v(x)=C(a_{23}e^{\theta_{+}(x-x_{0}^{+})}+a_{24}e^{\theta_{-}(x-x_{0}^{+})}),\quad x\in [x_{0}^{+},\infty).
\end{align}
Imposing the smooth-pasting conditions at $x=x_{0}^{+}$ leads to
\begin{align}
\label{3.31.w}
\left\{
\begin{array}{ll}
{\left(a_{21}\frac{x_0^+}{a}+a_{22}\right)^{\gamma_+}}
=a_{23}+a_{24},& \\
{\gamma_+\frac{a_{21}}{a}\left(a_{21}\frac{x_0^+}{a}+a_{22}\right)^{\gamma_+-1}}=a_{23}\theta_{+}+a_{24}\theta_{-}.&
\end{array}
\right.
\end{align}
Here, we note that the second equality of \eqref{3.31.w} implies that
\begin{align}
\label{str.inc.v}
v^{\prime}(x)>0,\quad x\in [x_{0}^{+},\infty).
\end{align}
Solving \eqref{3.31.w} yields the stated expressions of $a_{23}$ and $a_{24}$. 
It follows from $\theta_-<0$, $\theta_+-\theta_->0$ and $a_{21}>0$ that $a_{23}>0$ and $a_{23}+a_{24}>0,$
which implies that
\begin{itemize}
\item If $a_{24}\geq 0$, then from \eqref{C3,C4}, it is evident that $v^{\prime\prime}(x)> 0$ for all $x\in[x_0^{+},\infty)$.

\item If $a_{24}<0$, then
\begin{align}
v^{\prime\prime}(x)=C(a_{23}\theta_{+}^{2}e^{\theta_{+}(x-x_0^{+})}+a_{24}\theta_{-}^{2}e^{\theta_{-}(x-x_0^{+})}),\quad x\in [x_0^{+},\infty),\nn
\end{align}
is strictly increasing on $[x_0^{+},\infty)$. At $x=x_0^{+}$, we have
$v^{\prime\prime}(x_0^{+})
=C(a_{23}\theta_{+}^{2}+a_{24}\theta_{-}^{2})$.
{We claim that $a_{23}\theta_{+}^{2}+a_{24}\theta_{-}^{2}<0$. Indeed, by \eqref{3.26.w}, we have
\begin{eqnarray}\label{eq:a23theta+a24theta}
a_{23}\theta_{+}^{2}+a_{24}\theta_{-}^{2}
\hspace{-0.3cm}&=&\hspace{-0.3cm}
\frac{\gamma_+\frac{a_{21}}{a}\left(a_{21}\frac{x_0^+}{a}+a_{22}\right)^{\gamma_+-1}-\theta_-\left(a_{21}\frac{x_0^+}{a}+a_{22}\right)^{\gamma_+}}{\theta_+-\theta_-}\theta_+^2\nn\\
\hspace{-0.3cm}&&\hspace{-0.3cm}
+\frac{\theta_+\left(a_{21}\frac{x_0^+}{a}+a_{22}\right)^{\gamma_+}-\gamma_+\frac{a_{21}}{a}\left(a_{21}\frac{x_0^+}{a}+a_{22}\right)^{\gamma_+-1}}{\theta_+-\theta_-}\theta_-^2
\nn\\
\hspace{-0.3cm}&=&\hspace{-0.3cm}
\left(a_{21}\frac{x_0^+}{a}+a_{22}\right)^{\gamma_+-1}\frac{a_{21}}{a}\gamma_+(\theta_++\theta_-)
-\left(a_{21}\frac{x_0^+}{a}+a_{22}\right)^{\gamma_+}\theta_+\theta_-
\nn\\
\hspace{-0.3cm}&=&\hspace{-0.3cm}
\frac{2}{\sigma_+^2}\left(a_{21}\frac{x_0^+}{a}+a_{22}\right)^{\gamma_+-1}\left(q(a_{21}\frac{x^+_0}{a}+a_{22})-\frac{a_{21}}{a}\mu_+\gamma_+\right).
\end{eqnarray}
For brevity, denote $S_1=e^{\delta_+a}-e^{\delta_-a}$ and $S_2=\delta_+e^{\delta_+a}-\delta_-e^{\delta_-a}$. It follows from \eqref{3.26.w} and the definition of $x^+_0$ that
\begin{eqnarray}
\hspace{-0.3cm}&&\hspace{-0.3cm}
q(a_{21}\frac{x^+_0}{a}+a_{22})-\frac{a_{21}}{a}\mu_+\gamma_+
\nn\\
\hspace{-0.3cm}&=&\hspace{-0.3cm}
q\left(\frac{a}{\gamma_+}\frac{S_2}{S_1}S_1^{\frac{1}{\gamma_+}}\frac{x^+_0}{a}+S_1^{\frac{1}{\gamma_+}}\left(1-\frac{a}{\gamma_+}\frac{S_2}{S_1}\right)\right)-\frac{\mu_+\gamma_+}{a}\frac{a}{\gamma_+}\frac{S_2}{S_1}S_1^{\frac{1}{\gamma_+}}
\nn\\
\hspace{-0.3cm}&=&\hspace{-0.3cm}
S_1^{\frac{1}{\gamma_+}}\left(q\left(\frac{x^+_0-a}{\gamma_+}\frac{S_2}{S_1}+1\right)-\mu_+\frac{S_2}{S_1}\right)
\nn\\
\hspace{-0.3cm}&=&\hspace{-0.3cm}
-S_1^{\frac{1}{\gamma_+}}\frac{\mu_+}{2}\frac{S_2}{S_1}<0,
\end{eqnarray}
which combined with \eqref{eq:a23theta+a24theta} implies $a_{23}\theta^2_++a_{24}\theta^2_-<0.$
} 
Noting that $a_{23}\theta_{+}^{2}+a_{24}\theta_{-}^{2}<0$ occurs if and only if $\ln{\frac{-a_{24}\theta_{-}^2}{a_{23}\theta_{+}^2}}> 0$,
we have
\begin{align}
\label{3.36.ww}
\left\{
\begin{array}{ll}
v^{\prime\prime}(x)>0,& x\in (x_0^{+}+\frac{\ln{\frac{-a_{24}\theta_{-}^2}{a_{23}\theta_{+}^2}}}{\theta_{+}-\theta_{-}},\infty), \\
v^{\prime\prime}(x)<0,& x\in [x_0^{+}, x_0^{+}+\frac{\ln{\frac{-a_{24}\theta_{-}^2}{a_{23}\theta_{+}^2}}}{\theta_{+}-\theta_{-}}).
\end{array}
\right.
\end{align}
From \eqref{C3,C4}, we derive
\begin{align}
\label{3.37.w}
&\sigma_{+}^2v^{\prime\prime}(x)+\mu_{+}v^{\prime}(x)
\nn\\
=&~ C\left(a_{23} \left(\sigma_{+}^2\theta_{+}^{2}+\mu_{+}\theta_{+}\right)
e^{\theta_{+}(x-x_0^{+})}
+ a_{24}\left(\sigma_{+}^2\theta_{-}^{2}+\mu_{+}\theta_{-}\right)
e^{\theta_{-}(x-x_0^{+})}\right)
\nn\\
=&~
C\left(a_{23} \left(2q-\mu_{+}\theta_{+}\right)
e^{\theta_{+}(x-x_0^{+})}
+ a_{24}\left(2q-\mu_{+}\theta_{-}\right)
e^{\theta_{-}(x-x_0^{+})}\right), \quad x\in [x_0^{+},\infty),
\end{align}
where we have used the identity $\sigma_{+}^2\theta_{\pm}^{2}+\mu_{+}\theta_{\pm}=2q-\mu_{+}\theta_{\pm}$, which follows from the definition of $\theta_{\pm}$.
Using the explicit formula for $\theta_{+}$, we compute
\begin{align}
\label{3.40.w}
2q-\mu_{+}\theta_{+}=&~2q-\mu_{+}\frac{-\mu_{+}+\sqrt{\mu_{+}^2+2q\sigma_{+}^2}}{\sigma_{+}^2}
\nn\\
=&~2q-\frac{\mu_{+}\left(-\mu_{+}+\sqrt{\mu_{+}^2+2q\sigma_{+}^2}\right)\left(\mu_{+}+\sqrt{\mu_{+}^2+2q\sigma_{+}^2}\right)}{\sigma_{+}^2\left(\mu_{+}+\sqrt{\mu_{+}^2+2q\sigma_{+}^2}\right)}
\nn\\
=&~2q-\frac{2q\mu_{+}\sigma_{+}^2}{\sigma_{+}^2\left(\mu_{+}+\sqrt{\mu_{+}^2+2q\sigma_{+}^2}\right)}
\nn\\
=&~\frac{2q\sqrt{\mu_{+}^2+2q\sigma_{+}^2}}{\mu_{+}+\sqrt{\mu_{+}^2+2q\sigma_{+}^2}}>0.
\end{align}
Combining \eqref{C3,C4}, \eqref{3.40.w} and the facts of $2q-\mu_{+}\theta_{-}>0$ and $a_{23}>0$, we conclude that the RHS of \eqref{3.37.w} is strictly increasing in $x$, since the dominant exponential term (with $\theta_{+}>\theta_{-}$) grows faster. In this case, the minimum occurs at $x=x_{0}^{+}$:
\begin{align}
\label{3.39.w}
&\sigma_{+}^2v^{\prime\prime}(x_{0}^{+}+)+\mu_{+}v^{\prime}(x_{0}^{+}+)
\nn\\
=&~
2\left(qv(x_{0}^{+}+)-\mu_{+}v^{\prime}(x_{0}^{+}+)-\frac{1}{2}\sigma_{+}^2v^{\prime\prime}(x_{0}^{+}+)\right)
+\sigma_{+}^2v^{\prime\prime}(x_{0}^{+}+)+\mu_{+}v^{\prime}(x_{0}^{+}+)
\nn\\
=&~2qv(x_{0}^{+}+)-\mu_{+}v^{\prime}(x_{0}^{+}+)
\nn\\
=&~2qv(x_{0}^{+}-)-\mu_{+}v^{\prime}(x_{0}^{+}-)
\nn\\
=&~{C\left(a_{21}\frac{x^+_0}{a}+a_{22}\right)^{\gamma_+-1}\left(2qa_{21}\frac{x^+_0}{a}+2qa_{22}-\mu_+\gamma_+\frac{a_{21}}{a}\right)}
\nn\\
=&~0,
\end{align}
where we have used \eqref{3.26.w} and the definition of $x^+_0$.
We further verify the second-order smooth fit at $x_0^{+}$. By \eqref{3.26.w} and the definition of $x_0^{+}$ yields
\begin{align*}
x_0^{+}+\frac{a_{22}}{a_{21}}a
=\frac{\sigma_{+}^{2}(1-\gamma_{+})}{\mu_{+}}.
\end{align*}
Therefore, using the expression of the maximizing control on $[a,x_0^{+})$, we obtain
\begin{align*}
-\frac{\mu_{+}v^{\prime}(x_0^{+}-)}{\sigma_{+}^{2}v^{\prime\prime}(x_0^{+}-)}
=\lim_{x\uparrow x_0^{+}}u^{\star}(x)=\frac{\mu_{+}}{\sigma_{+}^{2}(1-\gamma_{+})}\left(x_0^{+}+\frac{a_{22}}{a_{21}}a\right)=1.
\end{align*}
Equivalently,
\begin{align}
\sigma_{+}^{2}v^{\prime\prime}(x_0^{+}-)+\mu_{+}v^{\prime}(x_0^{+}-)=0.
\label{eq:C2.x0.plus.left}
\end{align}
On the other hand, \eqref{3.39.w} gives
\begin{align}
\sigma_{+}^{2}v^{\prime\prime}(x_0^{+}+)+\mu_{+}v^{\prime}(x_0^{+}+)=0.
\label{eq:C2.x0.plus.right}
\end{align}
By the second equality in \eqref{3.31.w}, $v^{\prime}(x_0^{+}-)=v^{\prime}(x_0^{+}+)$; hence,
\begin{align}
v^{\prime\prime}(x_0^{+}-)
=v^{\prime\prime}(x_0^{+}+).
\label{eq:C2.x0.plus}
\end{align}
Thus, $v$ is twice continuously differentiable at $x_0^{+}$.

By \eqref{3.39.w} and the strict increasing property of $\sigma_{+}^2v^{\prime\prime}+\mu_{+}v^{\prime}$ on $[x_{0}^{+},\infty)$, we have
\begin{align}
\label{3.41.w}
\sigma_{+}^2v^{\prime\prime}(x)+\mu_{+}v^{\prime}(x)\geq 0. 
\end{align}
\end{itemize}
Therefore, we conclude that
\begin{itemize}
\item[$\bullet$] {Suppose $a_{24}\geq 0$.} Then $v^{\prime\prime}(x)>0, \quad x\in (x_{0}^{+},\infty),$
confirming that $v$ given by \eqref{C3,C4} and \eqref{3.26.w} is indeed a solution to \eqref{eq:ee03.{2}} on $[x_{0}^{+},\infty)$.

\item[$\bullet$] {Suppose $a_{24}<0$.} 
By \eqref{3.36.ww} and \eqref{3.41.w}, we have
\begin{align}
\left\{
\begin{array}{ll}
v^{\prime\prime}(x)<0,& x\in [x_0^{+}, x_0^{+}+\frac{\ln{\frac{-a_{24}\theta_{-}^2}{a_{23}\theta_{+}^2}}}{\theta_{+}-\theta_{-}}),
\\
\sigma_{+}^2 v^{\prime\prime}(x) +\mu_{+}v^{\prime}(x)\geq 0 \,\,\,(\Leftrightarrow \frac{-\mu_{+}v^{\prime}(x)}{\sigma_{+}^2 v^{\prime\prime}(x)}\geq 1),& x\in [x_0^{+}, x_0^{+}+\frac{\ln{\frac{-a_{24}\theta_{-}^2}{a_{23}\theta_{+}^2}}}{\theta_{+}-\theta_{-}}), \\
v^{\prime\prime}(x)>0,& x\in (x_0^{+}+\frac{\ln{\frac{-a_{24}\theta_{-}^2}{a_{23}\theta_{+}^2}}}{\theta_{+}-\theta_{-}}, \infty),
\end{array}
\right.\nn
\end{align}
confirming that $v$ given by \eqref{C3,C4} and \eqref{3.26.w} is indeed a solution to \eqref{eq:ee03.{2}} on $[x_{0}^{+},\infty)$.

\end{itemize}
In summary, when $\mu_{-}\le0, \,\mu_{+}>0$ and $x_{0}^{+}>a$, 
the solution to \eqref{eq:a046} takes the form \eqref{21.1.w}.

We now proceed under the assumption that $x_{0}^{+}\leq a$.
Under this assumption, on the interval $[a,\infty)$, \eqref{eq:ee03.{2}} reduces to
\begin{align}
\frac{1}{2}\sigma_{+}^2v^{\prime\prime}(x)+\mu_{+}v^{\prime}(x)-qv(x)=0,\quad x\in [a,\infty),
\nn
\end{align}
whose general solution is given by \eqref{eq:ee04} and \eqref{eq:ee06}. As discussed below \eqref{eq:ee04}, we recall that $a_{11}+a_{12}>0$, $a_{11}>0$, and
\begin{itemize}
\item If $a_{12}\geq 0$, then $v^{\prime\prime}(x)> 0$ for any $x\in[a,\infty)$.
\item If $a_{12}<0$, then
\begin{align}
v^{\prime\prime}(x)=C(a_{11}\theta_{+}^{2}e^{\theta_{+}(x-a)}+a_{12}\theta_{-}^{2}e^{\theta_{-}(x-a)}),\quad x\in[a,\infty),\nn
\end{align}
is strictly increasing with $v^{\prime\prime}(a+)=C(a_{11}\theta_{+}^{2}+a_{12}\theta_{-}^{2}).$
\begin{itemize}
\item[$\bullet$] If $a_{11}\theta_{+}^{2}+a_{12}\theta_{-}^{2}\geq 0$, then $v^{\prime\prime}(x)> 0$ for any $x\in(a,\infty)$.
\item[$\bullet$] If $a_{11}\theta_{+}^{2}+a_{12}\theta_{-}^{2}< 0$ ($\Leftrightarrow \ln{\frac{-a_{12}\theta_{-}^2}{a_{11}\theta_{+}^2}}>0)$, then
\begin{equation}\label{eq:gg01}
\begin{cases}
v^{\prime\prime}(x)<0,& x\in[a,a+\frac{\ln{\frac{-a_{12}\theta_{-}^2}{a_{11}\theta_{+}^2}}}{\theta_{+}-\theta_{-}}),\\
v^{\prime\prime}(x)>0,& x\in(a+\frac{\ln{\frac{-a_{12}\theta_{-}^2}{a_{11}\theta_{+}^2}}}{\theta_{+}-\theta_{-}},\infty).
\end{cases}
\end{equation}
\end{itemize}
By \eqref{eq:ee04}, we obtain
\begin{align}
\label{ff01}
&\mu_{+}v^{\prime}(x)+\sigma_{+}^2v^{\prime\prime}(x)
\nn\\
=&~C\left(a_{11}\left(\mu_{+}\theta_{+}+\sigma_{+}^2\theta_{+}^2\right)e^{\theta_{+}(x-a)}+a_{12}\left(\mu_{+}\theta_{-}+\sigma_{+}^2\theta_{-}^2\right)e^{\theta_{-}(x-a)}\right)
\nn\\
=&~C\left(a_{11}\left(2q-\mu_{+}\theta_{+}\right)e^{\theta_{+}(x-a)}+a_{12}\left(2q-\mu_{+}\theta_{-}\right)e^{\theta_{-}(x-a)}\right),\quad x\in [a,\infty).
\end{align}
From \eqref{3.40.w}, along with the facts that $2q-\mu_{+}\theta_{-}>0$, and $a_{11}>0$, we deduce that the function on the right-hand side of \eqref{ff01} is strictly increasing, and attains its minimum at $x=a$:
\begin{align}
\label{ff055}
&\sigma_{+}^2v^{\prime\prime}(a+)+\mu_{+}v^{\prime}(a+)
\nn\\
=&~
2\left(qv(a+)-\mu_{+}v^{\prime}(a+)-\frac{1}{2}\sigma_{+}^2v^{\prime\prime}(a+)\right)
+\sigma_{+}^2v^{\prime\prime}(a+)+\mu_{+}v^{\prime}(a+)
\nn\\
=&~2qv(a+)-\mu_{+}v^{\prime}(a+)
\nn\\
=&~2qv(a-)-\mu_{+}v^{\prime}(a-)
\nn\\
=&~C\left[2q\left(e^{\delta_{+}a}-e^{\delta_{-}a}\right)-\mu_{+}\left(\delta_{+}e^{\delta_{+}a}-\delta_{-}e^{\delta_{-}a}\right)\right] 
\nn\\
{=}&{
-C\left[\mu_{+}-2q\frac{e^{\delta_{+}a}-e^{\delta_{-}a}}{\delta_{+}e^{\delta_{+}a}-\delta_{-}e^{\delta_{-}a}}\right]}
\nn\\
{=}&{
-C\frac{\mu_{+}^{2}+2q\sigma_{+}^{2}}{\sigma_{+}^{2}}\left(x_{0}^{+}-a\right)}\nn\\
{\geq} &{0.}
\end{align}
\end{itemize}

Therefore, when $\mu_{-}\le0,\, \mu_{+}>0$ and $x_{0}^{+}\leq a$,
the solution to \eqref{eq:a046} takes the form \eqref{g22.g22}.
\end{proof}

\subsubsection{Identification of the uniform (in $x$) maximizer of $V_{\underline{x},\overline{x}}(x)$ over $(\underline{x},\overline{x})$}
\label{subs.3.2.2}

By analyzing various parameter configurations under case $\mu_{-}\leq 0,\, \mu_{+}>0$, we identify four mutually exclusive sub-cases denoted by $(\text{II}_{1})$-$(\text{II}_{3})$. For each sub-case, we specify the function $g$ and its inflection points $(a_i)_{1\leq i\leq 3}$, and derive conclusions based on the constructions of $g_{21}$ and $g_{22}$ presented in Subsection \ref{subsubsec.3.2.1}.
\begin{itemize}

\item[($\text{II}_{1}$)] If $x_0^{+}>a$ and either one of the following conditions holds
\begin{itemize}
\item[$\bullet$] $a_{24}\geq 0$,
\item[$\bullet$] $a_{24}<0$ (recall that $a_{23}\theta_{+}^{2}+a_{24}\theta_{-}^{2}<0$ in this case),
\end{itemize}
letting
\begin{align}
g:=g_{21}, \, a_1:=0, \, a_2:=a, \, a_3:=x_0^{+}+\frac{\ln{\left(1\vee\frac{-a_{24}\theta_{-}^2}{a_{23}\theta_{+}^2}\right)}}{\theta_{+}-\theta_{-}},\nn
\end{align}
then $g$ is \textbf{piecewise concave-convex} with inflection points $(a_i)_{1\leq i\leq 3}$ satisfying $0=a_{1}<a_{2}<a_{3}$.

\item[($\text{II}_{2}$)] If $x_0^{+}\le a$, $a_{12}<0$ and $a_{11}\theta_{+}^{2}+a_{12}\theta_{-}^{2}<0$, letting
\begin{align}
g:=g_{22},\, a_1:=0,\, a_2:=a, \, a_3:=a+\frac{\ln{\frac{-a_{12}\theta_{-}^2}{a_{11}\theta_{+}^2}}}{\theta_{+}-\theta_{-}},\nn
\end{align}
then $g$ is \textbf{piecewise concave-convex} with inflection points $(a_i)_{1\leq i\leq 3}$ satisfying $0=a_{1}<a_{2}<a_{3}$.

\item[($\text{II}_{3}$)] If $x_0^{+}\le a$ and either one of the following conditions holds
\begin{itemize}
\item[$\bullet$] $a_{12}\geq 0$,
\item[$\bullet$] $a_{12}<0$ and $
a_{11}\theta_{+}^{2}+a_{12}\theta_{-}^{2}\geq 0$,
\end{itemize}
letting
\begin{align}
g:=g_{22}, \, a_1:=0,\, a_{2}:=0,\, a_{3}:=0,\nn
\end{align}
then $g$ is \textbf{piecewise concave-convex} with all inflection points coinciding at $0$.
\end{itemize}
We note that the three sub-cases mentioned above are collectively exhaustive. Under Cases $(\text{II}_{1})$-$(\text{II}_{2})$, let us define 
\begin{align}
\label{parametr.maxmizers}
(\underline{x}^{\star},\overline{x}^{\star}):=\left\{
\begin{array}{ll}
(\hat{z}_1,\hat{z}_2),&\text{if } \beta\in (A_1\cup A_3)\backslash\{\omega_1(x_2)\}, \\
(0,x_2),&\text{if } \beta\in (A_1\cup A_3)\cap\{\omega_1(x_2)\}, \\
(\tilde{z}_1,\tilde{z}_2),& \text{if } \beta\in A_2\cap \overline{A}_3, \\
(\tilde{z}_1,\tilde{z}_2),&\text{if } \beta\in \overline{A_1\cup A_2\cup A_3}\backslash\{\omega_1(x_2)\},\\
\text{ any one of } 
(0,x_2) \text{ and } (\tilde{z}_1,\tilde{z}_2),&\text{if }\beta\in \overline{A_1\cup A_2\cup A_3}\cap\{\omega_1(x_2)\},
\end{array}
\right.
\end{align}
where $(\hat{z}_1,\hat{z}_2)$, $(\tilde{z}_1,\tilde{z}_2)$, $A_{1}$, $A_{2}$, and $A_{3}$ are given by \eqref{def.max.points}

\begin{lemma}
\label{lem3.4}
Let $(\underline{x}^{\star},\overline{x}^{\star})$ be defined by \eqref{parametr.maxmizers}. Then
\begin{align}
\label{3.95.77}
g^{\prime}(\overline{x}^{\star})=\inf \{g^{\prime}(\overline{x}): \text{ \emph{there exists} } \underline{x}\geq0 \text{ \emph{such that} } \psi(\underline{x},\overline{x})=\beta\}.
\end{align}
\end{lemma}

\begin{proof}
We prove \eqref{3.95.77} only for the case where $\beta\in (A_1\cup A_3)\backslash\{\omega_1(x_2)\}$ and $\overline{x}^{\star}\in (0,(g^{\prime})^{-1}_2(g^{\prime}(x_2)))$, as this scenario sufficiently captures the core proof strategy.
Then $\underline{x}^{\star}=0$, $\overline{x}^{\star}=\omega_{1}^{-1}(\beta)$, $\psi(\underline{x}^{\star},\overline{x}^{\star})=\beta$, and $g^{\prime}(0)<g^{\prime}(\overline{x}^{\star})<g^{\prime}(x_2)<g^{\prime}(a)$. Let $(\underline{x},\overline{x})\neq (\underline{x}^{\star},\overline{x}^{\star})$ be such that $\psi(\underline{x},\overline{x})=\beta$. We examine the following cases.
\begin{itemize}
\item Suppose $\underline{x}=\underline{x}^{\star}=0$ and $\overline{x}\in [0,a]$. By the continuity and strict increasing property of the function 
$x\mapsto \psi(0,x)$ on $[0,a]$, we must have $\overline{x}=\overline{x}^{\star}$. This contradicts the assumption $(\underline{x},\overline{x})\neq (\underline{x}^{\star},\overline{x}^{\star})$.
\item Suppose $\underline{x}\in (0, \overline{x}^{\star})$ and $\overline{x}\in [0,a]$. By the continuity and strict decreasing property of the function 
$x\mapsto \psi(x,\overline{x}^{\star})$ on $[0,\overline{x}^{\star}]$, it follows that
\begin{align} \psi(\underline{x},\overline{x}^{\star})<\psi(\underline{x}^{\star},\overline{x}^{\star})=\beta.\nn
\end{align}
Combined with the strict increasing property of the function $x\mapsto \psi(\underline{x},x)$ on $[\underline{x},a]$ and $\psi(\underline{x},\overline{x})=\beta$, this yields
$\overline{x}\in (\overline{x}^{\star},a]$. Since $g^{\prime}$ is increasing on $[0,a]$, it follows that $g^{\prime}(\overline{x})>g^{\prime}(\overline{x}^{\star})$.
\item Suppose $\underline{x}\in [\overline{x}^{\star},a]$ and $\overline{x}\in [0,a]$. Since $g^{\prime}$ is increasing on $[0,a]$, we have $g^{\prime}(\overline{x})>g^{\prime}(\underline{x})\geq g^{\prime}(\overline{x}^{\star})$.
\item Suppose $\underline{x} \in [0,a]$ and $\overline{x}\in (a,a_{3}]$. We prove $g^{\prime}(\overline{x}^{\star})\leq g^{\prime}(\overline{x})$ by contradiction. Indeed, if $g^{\prime}(\overline{x}^{\star})> g^{\prime}(\overline{x})$, we examine the following cases.
\begin{itemize}
\item When $g^{\prime}(\overline{x})\leq g^{\prime}(0)$, then ${g^{\prime}(x)}> {g^{\prime}(\overline{x})}$ for all $x\in (\underline{x},\overline{x})$, implying
\begin{align*}
\psi(\underline{x},\overline{x})=\int_{\underline{x}}^{\overline{x}}\left(1-\frac{g^{\prime}(z)}{g^{\prime}(\overline{x})}\right)\mathrm{d}z<0,
\end{align*}
which contradicts $\psi(\underline{x},\overline{x})=\beta$.

\item When $g^{\prime}(\overline{x})> g^{\prime}(0)$, then since $g^{\prime}(\overline{x}^{\star})> g^{\prime}(\overline{x})$ we have $(g^{\prime})_{2}^{-1}(g^{\prime}(\overline{x}))\in (0,\overline{x}^{\star})$. Therefore, by the strict decreasing property of the function $ x\mapsto \int_{0}^{(g^{\prime})_{2}^{-1}(g^{\prime}(x))}\left(1-\frac{g^{\prime}(z)}{g^{\prime}(x)}\right)\mathrm{d}z$ on $[(g^{\prime})_{3}^{-1}(g^{\prime}(\overline{x}^{\star})),\overline{x}]$, we obtain
\begin{align*}
\psi(\underline{x},\overline{x})&=\int_{\underline{x}}^{\underline{x}\vee (g^{\prime})_{2}^{-1}(g^{\prime}(\overline{x}))}\left(1-\frac{g^{\prime}(z)}{g^{\prime}(\overline{x})}\right)\mathrm{d}z
+
\int_{\underline{x}\vee (g^{\prime})_{2}^{-1}(g^{\prime}(\overline{x}))}^{\overline{x}}\left(1-\frac{g^{\prime}(z)}{g^{\prime}(\overline{x})}\right)\mathrm{d}z
\nn\\
&< 
\int_{\underline{x}}^{\underline{x}\vee (g^{\prime})_{2}^{-1}(g^{\prime}(\overline{x}))}\left(1-\frac{g^{\prime}(z)}{g^{\prime}(\overline{x})}\right)\mathrm{d}z
\nn\\
&= 
\mathbf{1}_{\{\underline{x}< (g^{\prime})_{2}^{-1}(g^{\prime}(\overline{x}))\}}\int_{\underline{x}}^{(g^{\prime})_{2}^{-1}(g^{\prime}(\overline{x}))}\left(1-\frac{g^{\prime}(z)}{g^{\prime}(\overline{x})}\right)\mathrm{d}z
\nn\\
&\leq 
\mathbf{1}_{\{\underline{x}< (g^{\prime})_{2}^{-1}(g^{\prime}(\overline{x}))\}}\int_{0}^{(g^{\prime})_{2}^{-1}(g^{\prime}(\overline{x}))}\left(1-\frac{g^{\prime}(z)}{g^{\prime}(\overline{x})}\right)\mathrm{d}z
\nn\\
&\leq 
\mathbf{1}_{\{\underline{x}< (g^{\prime})_{2}^{-1}(g^{\prime}(\overline{x}))\}}\int_{0}^{(g^{\prime})_{2}^{-1}(g^{\prime}((g^{\prime})_{3}^{-1}(g^{\prime}(\overline{x}^{\star}))))}\left(1-\frac{g^{\prime}(z)}{g^{\prime}((g^{\prime})_{3}^{-1}(g^{\prime}(\overline{x}^{\star})))}\right)\mathrm{d}z
\nn\\
&=
\mathbf{1}_{\{\underline{x}< (g^{\prime})_{2}^{-1}(g^{\prime}(\overline{x}))\}}\int_{0}^{\overline{x}^{\star}}\left(1-\frac{g^{\prime}(z)}{g^{\prime}(\overline{x}^{\star})}\right)\mathrm{d}z
\nn\\
&=\beta \mathbf{1}_{\{\underline{x}< (g^{\prime})_{2}^{-1}(g^{\prime}(\overline{x}))\}},
\end{align*}
which contradicts the fact that $\psi(\underline{x},\overline{x})=\beta$.
\end{itemize}

\item Suppose $\underline{x} \in [0,a]$ and $\overline{x}\in (a_{3},\infty)$. We prove $g^{\prime}(\overline{x}^{\star})\leq g^{\prime}(\overline{x})$ by contradiction. Indeed, if $g^{\prime}(\overline{x}^{\star})> g^{\prime}(\overline{x})$, we examine the following cases.
\begin{itemize}
\item When $\int_{\underline{x}}^{(g^{\prime})_{3}^{-1}(g^{\prime}(\overline{x}))}\left(1-\frac{g^{\prime}(z)}{g^{\prime}(\overline{x})}\right)\mathrm{d}z<0$, then
\begin{align}
\beta=\psi(\underline{x},\overline{x})&=\int_{\underline{x}}^{(g^{\prime})_{3}^{-1}(g^{\prime}(\overline{x}))}\left(1-\frac{g^{\prime}(z)}{g^{\prime}(\overline{x})}\right)\mathrm{d}z
+
\int_{(g^{\prime})_{3}^{-1}(g^{\prime}(\overline{x}))}^{\overline{x}}\left(1-\frac{g^{\prime}(z)}{g^{\prime}(\overline{x})}\right)\mathrm{d}z
\nn\\
&< \int_{(g^{\prime})_{3}^{-1}(g^{\prime}(\overline{x}))}^{\overline{x}}\left(1-\frac{g^{\prime}(z)}{g^{\prime}(\overline{x})}\right)\mathrm{d}z
\nn\\
&\leq \omega_{2}(\overline{x}),\nn
\end{align}
which implies $\omega_{2}^{-1}(\beta)\leq \overline{x}$. Combined with $g^{\prime}(\overline{x}^{\star})> g^{\prime}(\overline{x})$ and the monotonicity of $g^{\prime}$ on $[a_{3},\infty)$, this yields 
\begin{align*}
g^{\prime}(\omega_{2}^{-1}(\beta))\leq g^{\prime}(\overline{x})< g^{\prime}(\overline{x}^{\star})=g^{\prime}(\omega_{1}^{-1}(\beta)),
\end{align*}
that is $\beta\in A_{2}$, contradicting the assumption of $\beta\in (A_1\cup A_3)\backslash\{\omega_1(x_2)\}$ (recall that $A_{2}\cap (A_{1}\cup A_{3})=\emptyset$).
\item When $\int_{\underline{x}}^{(g^{\prime})_{3}^{-1}(g^{\prime}(\overline{x}))}\left(1-\frac{g^{\prime}(z)}{g^{\prime}(\overline{x})}\right)\mathrm{d}z\geq 0$, then $g^{\prime}(\underline{x})<g^{\prime}(\overline{x})<g^{\prime}(\overline{x}^{\star})<g^{\prime}(a)$, $\overline{x}<x_{2}$, and
\begin{align}
\beta=\psi(\underline{x},\overline{x})&=\int_{\underline{x}}^{(g^{\prime})_{2}^{-1}(g^{\prime}(\overline{x}))}\left(1-\frac{g^{\prime}(z)}{g^{\prime}(\overline{x})}\right)\mathrm{d}z
+
\int_{(g^{\prime})_{2}^{-1}(g^{\prime}(\overline{x}))}^{\overline{x}}\left(1-\frac{g^{\prime}(z)}{g^{\prime}(\overline{x})}\right)\mathrm{d}z
\nn\\
&< \int_{\underline{x}}^{(g^{\prime})_{2}^{-1}(g^{\prime}(\overline{x}))}\left(1-\frac{g^{\prime}(z)}{g^{\prime}(\overline{x})}\right)\mathrm{d}z
\nn\\
&\leq \int_{0}^{(g^{\prime})_{2}^{-1}(g^{\prime}(\overline{x}))}\left(1-\frac{g^{\prime}(z)}{g^{\prime}(\overline{x})}\right)\mathrm{d}z
\nn\\
&\leq \int_{0}^{(g^{\prime})_{2}^{-1}(g^{\prime}(\overline{x}))}\left(1-\frac{g^{\prime}(z)}{g^{\prime}(\overline{x}^{\star})}\right)\mathrm{d}z
\nn\\
&\leq \int_{0}^{\overline{x}^{\star}}\left(1-\frac{g^{\prime}(z)}{g^{\prime}(\overline{x}^{\star})}\right)\mathrm{d}z
=\beta,\nn
\end{align}
a contradiction.
\end{itemize}

\item Suppose that $\underline{x}\in(a,\infty)$ and $\overline{x}\in(a_{3},\infty)$. One can check that
\begin{align*}
\beta&=\psi(\underline{x},\overline{x})
\nn\\
&=\left[\psi(\underline{x},(g^{\prime})_{3}^{-1}(g^{\prime}(\overline{x})))
+\psi((g^{\prime})_{3}^{-1}(g^{\prime}(\overline{x})),\overline{x})\right]\mathbf{1}_{\{g^{\prime}(\overline{x})\leq g^{\prime}(a),\, \underline{x}\leq (g^{\prime})_{3}^{-1}(g^{\prime}(\overline{x}))\}}
\nn\\
&\quad \,\,+\left[-\int_{(g^{\prime})_{3}^{-1}(g^{\prime}(\overline{x}))}^{\underline{x}}\left(1-\frac{g^{\prime}(z)}{g^{\prime}(\overline{x})}\right)\mathrm{d}z
+\psi((g^{\prime})_{3}^{-1}(g^{\prime}(\overline{x})),\overline{x})\right]\mathbf{1}_{\{g^{\prime}(\overline{x})\leq g^{\prime}(a),\, \underline{x}> (g^{\prime})_{3}^{-1}(g^{\prime}(\overline{x}))\}}
\nn\\
&\quad \,\,+\psi(\underline{x},\overline{x})\mathbf{1}_{\{g^{\prime}(\overline{x})> g^{\prime}(a)\}}
\nn\\
&\leq \psi((g^{\prime})_{3}^{-1}(g^{\prime}(\overline{x})),\overline{x}) \mathbf{1}_{\{g^{\prime}(\overline{x})\leq g^{\prime}(a)\}} 
+\psi(0,\overline{x})\mathbf{1}_{\{g^{\prime}(\overline{x})> g^{\prime}(a)\}}
\nn\\
&\leq \omega_{2}(\overline{x}).
\end{align*}
\begin{itemize}
\item If $\beta\in A_{1}$, then
\begin{align}
g^{\prime}(\overline{x}^{\star})=g^{\prime}(\omega_1^{-1}(\beta))<g^{\prime}(\omega_2^{-1}(\beta))\leq g^{\prime}(\overline{x}).\nn
\end{align}
\item If $\beta\in A_{3}$, then $\omega_{2}(\overline{x})\geq \beta > \omega_{2}(x_{1})$, that is $\overline{x}>x_{1}$. In addition, when $\overline{x}< x_{2}$, we have
\begin{align}
\psi(0,\overline{x}^{\star})=\beta\leq \omega_{2}(\overline{x})&=
\psi(0,(g^{\prime})_{2}^{-1}(g^{\prime}(\overline{x})))+\psi((g^{\prime})_{2}^{-1}(g^{\prime}(\overline{x})),\overline{x})
\nn\\
&<\psi(0,(g^{\prime})_{2}^{-1}(g^{\prime}(\overline{x}))),
\nn
\end{align}
which implies $\overline{x}^{\star}< (g^{\prime})_{2}^{-1}(g^{\prime}(\overline{x}))$, which combined with the increasing property of $g^{\prime}$ on $[0,a]$ gives $g^{\prime}(\overline{x}^{\star})< g^{\prime}(\overline{x})$. While, when $\overline{x} \geq x_{2}$, by the assumption of $\overline{x}^{\star}<(g^{\prime})^{-1}_2(g^{\prime}(x_2))$ we have
\begin{align}
g^{\prime}(\overline{x}^{\star})<g^{\prime}((g^{\prime})^{-1}_2(g^{\prime}(x_2)))=g^{\prime}(x_2)\leq g^{\prime}(\overline{x}).\nn
\end{align}
\end{itemize}
\end{itemize}
The proof is complete.
\end{proof}

Under the case of $\mu_{-}\leq 0, \,\mu_{+}>0$, the optimal strategy and value function of the control problem \eqref{eq:a008.aux.} are presented in the following Theorem \ref{Theorem3.4}.
To state the results of Theorem \ref{Theorem3.4}, we may need to impose the following assumption.
\begin{itemize}
\item[] \textbf{Assumption A}: For any pair $(\underline{x},\overline{x})$ such that $\psi(\underline{x},\overline{x})=\beta$, $\underline{x}\geq 0$, $\overline{x}\geq a_{3}$ and $g^{\prime}(0)< g^{\prime}(\overline{x})< g^{\prime}(a)$, we have
\begin{align}
\psi(0,(g^{\prime})_{3}^{-1}(g^{\prime}(\overline{x})))\geq \beta. 
\nn
\end{align}
\end{itemize}
\begin{remark}\label{remark3.7}
For any pair $(\underline{x},\overline{x})$ satisfying $\psi(\underline{x},\overline{x})=\beta$, $\underline{x}\geq 0$, $\overline{x}\geq a_{3}$, and\, $g^{\prime}(0)< g^{\prime}(\overline{x})< g^{\prime}(a)$, we have
\begin{align}
\psi\bigl(0,(g^{\prime})_{3}^{-1}(g^{\prime}(\overline{x}))\bigr)
&=\int_{0}^{(g^{\prime})_{3}^{-1}(g^{\prime}(\overline{x}))}\left(1-\frac{g^{\prime}(z)}{g^{\prime}(\overline{x})}\right)\mathrm{d}z
\nn\\
&\geq \int_{0}^{a_{3}}\left(1-\frac{g^{\prime}(z)}{g^{\prime}(a_{3})}\right)\mathrm{d}z
=\psi(0,a_{3}),\nn
\end{align}
since the function $\displaystyle x\mapsto\int_{0}^{(g^{\prime})_{3}^{-1}(g^{\prime}(x))}\left(1-\frac{g^{\prime}(z)}{g^{\prime}(x)}\right)\mathrm{d}z$ is increasing on $[a_{3},(g^{\prime})_{4}^{-1}(g^{\prime}(a))]$.
Consequently, a sufficient condition for \textbf{Assumption A} to hold is
\begin{align*}
\psi(0,a_{3})\geq \beta.
\end{align*}
\end{remark}

\begin{theorem}
\label{Theorem3.4}
Suppose $\mu_{-}\leq 0,\,\mu_{+}>0$. Recall that $\psi$ and $(\underline{x}^{\star},\overline{x}^{\star})$ are defined by \eqref{def.psi.} and \eqref{parametr.maxmizers}. Let $g$ and $(a_{i})_{1\leq i\leq 3}$ be defined separately in three cases $(\text{\emph{II}}_{1})$-$(\text{\emph{II}}_{3})$.

Under cases $(\text{\emph{II}}_{1})$-$(\text{\emph{II}}_{2})$ with $\overline{x}^{*}\geq a$, the reinsurance-dividend strategy $(u^{\star},D^{\star})$ given by
\begin{equation}
\label{optimal.stra.3.4}
\begin{cases} 
u_{t}^{\star}=\frac{\mu_{+}\left(U^{u^{\star},D^{\underline{x}^{\star},\overline{x}^{\star}}}_{t}+\frac{a_{22}}{a_{21}}a\right)}{\sigma_{+}^2(1-\gamma_{+})}\emph{\textbf{1}}_{\{U^{u^{\star},D^{\underline{x}^{\star},\overline{x}^{\star}}}_{t}\in[a,\,a\vee x_0^{+})\}}+\emph{\textbf{1}}_{\{U^{u^{\star},D^{\underline{x}^{\star},\overline{x}^{\star}}}_{t}\notin[a,\,a\vee x_0^{+})\}},& t\geq 0, \\
D^{\star}_{t}=D^{\underline{x}^{\star},\overline{x}^{\star}}_{t},& t\geq 0.
\end{cases}
\end{equation}
is optimal for the auxiliary control problem \eqref{eq:a008.aux.}, and, the associated value function is
\begin{equation}\label{eq:vfunc.3.4}
V_{\underline{x}^{\star},\overline{x}^{\star}}(x)=
\left\{\begin{array}{ll}
\frac{g(x)}{g^{\prime}(\overline{x}^{\star})},&x\in[0,\overline{x}^{\star}), \\
x-\overline{x}^{\star}+\frac{g(\overline{x}^{\star})}{g^{\prime}(\overline{x}^{\star})}=x-\underline{x}^{\star}-\beta+\frac{g(\underline{x}^{\star})}{g^{\prime}(\overline{x}^{\star})},& x\in[\overline{x}^{\star},\infty).
\end{array}
\right.
\end{equation}

Under cases $(\text{\emph{II}}_{1})$-$(\text{\emph{II}}_{2})$ with $\overline{x}^{*}\in [0,a)$ and \textbf{Assumption A} holding, 
the reinsurance-dividend strategy $(u^{\star},D^{\star})$ given by \eqref{optimal.stra.3.4} is optimal for the auxiliary control problem \eqref{eq:a008.aux.}, and the associated value function is given by \eqref{eq:vfunc.3.4}.

Under Case $(\text{\emph{II}}_{3})$, the reinsurance-dividend strategy $(u^{\star},D^{\star})$ defined by
\begin{align}
\begin{cases}
u_{t}^{\star}=1,&t\geq 0, \\
D^{\star}_{t}=D^{\underline{x}^{\star},\overline{x}^{\star}}_{t},&t\geq 0,
\end{cases}\nn
\nn
\end{align}
is optimal for \eqref{eq:a008.aux.}. The corresponding value function is
\begin{equation}
V_{0,\overline{x}^{\star}}(x)=
\left\{\begin{array}{ll}
\frac{g(x)}{g^{\prime}(\overline{x}^{\star})},&x\in[0,\overline{x}^{\star}), \\
x-\beta,& x\in[\overline{x}^{\star},\infty),
\end{array}
\right.\nn
\end{equation}
where $\overline{x}^{\star}$ is uniquely determined by the condition $\psi(0,\overline{x}^{\star})=\beta$.
\end{theorem}

\begin{proof}
It suffices to prove the theorem for the case $0=a_{1}<a_{2}<a_{3}$ (see Case $\text{{II}}_{1}$-$\text{II}_{2}$), as the scenario of $0=a_{1}=a_{2}=a_{3}$ (see Case $\text{{II}}_{3}$) follows analogously to Theorem \ref{thm3.2.ww}. By Lemma \ref{lem3.1.w}, we decouple the auxiliary optimization problem \eqref{eq:a008.aux.} into problems \eqref{auxiliary.contr.p.} and problem \eqref{point.point.optimization}. Applying theorem \ref{thm3.1.w} and arguments similar to those in the proof of Theorem \ref{thm3.2.ww}, the value function of \eqref{auxiliary.contr.p.} is given by 
\begin{equation}
V_{\underline{x},\overline{x}}(x)=
\left\{\begin{array}{ll}
\frac{g(x)}{g^{\prime}(\overline{x})},&x\in[0,\overline{x}), \\
\frac{g(\underline{x})}{g^{\prime}(\overline{x})}+x-\underline{x}-\beta,& x\in[\overline{x},\infty),
\end{array}
\right.
\end{equation}
where $g$ is defined at the beginning of Subsection \ref{subs.3.2.2}, and the pair $(\underline{x}, \overline{x})$ satisfies
\begin{align}
\label{3.20.w2}
\psi(\underline{x},\overline{x})=\beta.
\end{align}
We proceed to verify that $(\underline{x}^{\star},\overline{x}^{\star})$ given by \eqref{parametr.maxmizers} solves problem \eqref{point.point.optimization}. 

Suppose that $(\underline{x}^{\star},\overline{x}^{\star})$ is such that $\underline{x}^{\star}=0$ and $\overline{x}^{\star}\geq a$. Then $\overline{x}^{\star}\geq x_{2}$.
We examine the following cases.
\begin{itemize}
\item Consider a solution $(\underline{x},\overline{x})$ of \eqref{3.20.w2} such that $0\leq \underline{x}\leq \overline{x}\leq a$. By Lemma \ref{lem3.4}, we have $g^{\prime}(\overline{x})\geq g^{\prime}(\overline{x}^{\star})$. 
\begin{itemize}
\item For $x\in [0,\overline{x}]$,
\begin{align}
\label{al.01}
V_{\underline{x}^{\star},\overline{x}^{\star}}(x)-V_{\underline{x},\overline{x}}(x)=g(x)\left(\frac{1}{g^{\prime}(\overline{x}^{\star})}-\frac{1}{g^{\prime}(\overline{x})}\right)\geq 0.
\end{align}
\item For $x\in(\overline{x},\overline{x}^{\star}]$,
\begin{align}
\label{al.02}
V_{\underline{x}^{\star},\overline{x}^{\star}}(x)-V_{\underline{x},\overline{x}}(x)=&~
\frac{g(x)}{g^{\prime}(\overline{x}^{\star})}-\left[x-\underline{x}-\beta+\frac{g(\underline{x})}{g^{\prime}(\overline{x})}\right]
\nn\\
\geq&~\beta -\int_{\underline{x}}^{x}\left(1-\frac{g^{\prime}(z)}{g^{\prime}(\overline{x}^{\star})}\right)\mathrm{d}z.
\end{align}

In case where $g^{\prime}(\overline{x}^{\star})\geq g^{\prime}(a)$, we obtain
\begin{align}
\label{7.10.3.100}
\int_{\underline{x}}^{x}\left(1-\frac{g^{\prime}(z)}{g^{\prime}(\overline{x}^{\star})}\right)\mathrm{d}z
=&~\int_{0}^{\overline{x}^{\star}}\left(1-\frac{g^{\prime}(z)}{g^{\prime}(\overline{x}^{\star})}\right)\mathrm{d}z-\int_{0}^{\underline{x}}\left(1-\frac{g^{\prime}(z)}{g^{\prime}(\overline{x}^{\star})}\right)\mathrm{d}z
\nn\\
&-\int_{x}^{\overline{x}^{\star}}\left(1-\frac{g^{\prime}(z)}{g^{\prime}(\overline{x}^{\star})}\right)\mathrm{d}z
\nn\\
\leq&~\int_{0}^{\overline{x}^{\star}}\left(1-\frac{g^{\prime}(z)}{g^{\prime}(\overline{x}^{\star})}\right)\mathrm{d}z=\beta,
\end{align}
since ${g^{\prime}(x)}\leq {g^{\prime}(\overline{x}^{\star})}$ for any $x\in(\overline{x},\overline{x}^{\star}]$.

In case where $g^{\prime}(\overline{x}^{\star})< g^{\prime}(a)$, 
consider the function
\begin{align*}
x\mapsto V_{\underline{x}^{\star},\overline{x}^{\star}}(x)-V_{\underline{x},\overline{x}}(x)=\frac{g(x)}{g^{\prime}(\overline{x}^{\star})}-\left[x-\underline{x}-\beta+\frac{g(\underline{x})}{g^{\prime}(\overline{x})}\right],
\qquad x\in(\overline{x},\overline{x}^{\star}].
\end{align*}
Its derivative is
$\frac{g^{\prime}(x)}{g^{\prime}(\overline{x}^{\star})}-1.$
It follows from $g^{\prime}(\overline{x})\geq g^{\prime}(\overline{x}^{\star})$ and the piecewise concave-convex shape of $g$ that the function $x\mapsto V_{\underline{x}^{\star},\overline{x}^{\star}}(x)-V_{\underline{x},\overline{x}}(x)$ is increasing on $[\overline{x},(g^{\prime})_{3}^{-1}(g^{\prime}(\overline{x}^{\star}))]$ and decreasing on $[(g^{\prime})_{3}^{-1}(g^{\prime}(\overline{x}^{\star})),\overline{x}^{\star}].$
At the two endpoints, the matching condition $\psi(\underline{x},\overline{x})=\beta$ gives
\begin{align*}
V_{\underline{x}^{\star},\overline{x}^{\star}}(\overline{x})-V_{\underline{x},\overline{x}}(\overline{x})
&=g(\overline{x})\left(\frac{1}{g^{\prime}(\overline{x}^{\star})}-\frac{1}{g^{\prime}(\overline{x})}\right)\geq0,
\end{align*}
and, since $\underline{x}^{\star}=0$,
\begin{align*}
V_{\underline{x}^{\star},\overline{x}^{\star}}(\overline{x}^{\star})-V_{\underline{x},\overline{x}}(\overline{x}^{\star})
&=\underline{x}-\frac{g(\underline{x})}{g^{\prime}(\overline{x})}
=\int_{0}^{\underline{x}}\left(1-\frac{g^{\prime}(z)}{g^{\prime}(\overline{x})}\right)\mathrm{d}z\geq0.
\end{align*}
Therefore, we have 
\begin{align}
V_{\underline{x}^{\star},\overline{x}^{\star}}(x)-V_{\underline{x},\overline{x}}(x)
\geq 0, \quad x\in(\overline{x},\overline{x}^{\star}].\nn
\end{align}
\item For $x\in(\overline{x}^{\star},\infty)$,
\begin{align} 
\label{al.03}
V_{0,\overline{x}^{\star}}(x)-V_{\underline{x},\overline{x}}(x)&=
V_{0,\overline{x}^{\star}}(\overline{x}^{\star})-V_{\underline{x},\overline{x}}(\overline{x}^{\star})\geq 0.
\end{align}
\end{itemize}
\item Employing arguments analogous to \eqref{al.01}-\eqref{al.03}, one can verify that for any solution $(\underline{x},\overline{x})$ of \eqref{3.20.w2} satisfying $0\leq \underline{x}\leq a$ and $a<\overline{x}\leq a_{3}$, the inequality $V_{\underline{x}^{\star},\overline{x}^{\star}}(x)-V_{\underline{x},\overline{x}}(x)\geq 0$ holds for all $x\geq 0$.
\item Consider a solution $(\underline{x},\overline{x})$ of \eqref{3.20.w2} satisfying $0\leq \underline{x}\leq a$ and $\overline{x}>a_{3}$. By Lemma \ref{lem3.4}, $g^{\prime}(\overline{x})\geq g^{\prime}(\overline{x}^{\star})$. Since $g^{\prime}$ is strictly increasing on $[a_{3},\infty)$, it follows that $\overline{x}>\overline{x}^{\star}$.
\begin{itemize}
\item For $x\in [0,\overline{x}^{\star}]$,
\begin{align}
\label{al.04}
V_{\underline{x}^{\star},\overline{x}^{\star}}(x)-V_{\underline{x},\overline{x}}(x)=g(x)\left(\frac{1}{g^{\prime}(\overline{x}^{\star})}-\frac{1}{g^{\prime}(\overline{x})}\right)\geq 0.
\end{align}
\item For $x\in(\overline{x}^{\star},\overline{x}]$,
\begin{align}
\label{al.05}
V_{\underline{x}^{\star},\overline{x}^{\star}}(x)-V_{\underline{x},\overline{x}}(x)=&~
x-\beta-\frac{g(x)}{g^{\prime}(\overline{x})}
\nn\\
\geq&~\int_{0}^{x}\left(1-\frac{g^{\prime}(z)}{g^{\prime}(\overline{x})}\right)\mathrm{d}z-\beta
\nn\\
\geq&~\int_{0}^{\overline{x}^{\star}}\left(1-\frac{g^{\prime}(z)}{g^{\prime}(\overline{x})}\right)\mathrm{d}z-\beta
\nn\\
\geq&~\int_{0}^{\overline{x}^{\star}}\left(1-\frac{g^{\prime}(z)}{g^{\prime}(\overline{x}^{\star})}\right)\mathrm{d}z-\beta
\nn\\
=&~0.
\end{align}
\item For $x\in(\overline{x},\infty)$,
\begin{align} 
\label{al.06}
V_{0,\overline{x}^{\star}}(x)-V_{\underline{x},\overline{x}}(x)=&~
V_{0,\overline{x}^{\star}}(\overline{x})-V_{\underline{x},\overline{x}}(\overline{x}) \geq 0.
\end{align}
\end{itemize}

\item Consider a solution $(\underline{x},\overline{x})$ of \eqref{3.20.w2} such that either $a<\underline{x}<a_{3}<\overline{x}$ or $\min\{\underline{x}, \overline{x}\}\geq a_{3}$. The inequalities \eqref{al.04}-\eqref{al.06} remain valid.
\end{itemize}

Suppose $(\underline{x}^{\star},\overline{x}^{\star})$ is such that $\underline{x}^{\star}\in (a,a_{3}]$ and $\overline{x}^{\star}> a$. Then $\overline{x}^{\star}> a_{3}$, $\underline{x}^{\star}=(g^{\prime})_{3}^{-1}(g^{\prime}(\overline{x}^{\star}))$, $\overline{x}^{\star}<x_{1}$, and $g^{\prime}(0)<g^{\prime}(x_{1})<g^{\prime}(a)$.
We examine the following cases.
\begin{itemize}
\item Consider a solution $(\underline{x},\overline{x})$ of \eqref{3.20.w2} such that $0\leq \underline{x}\leq \overline{x}\leq a$. By Lemma \ref{lem3.4}, we have $g^{\prime}(\overline{x})\geq g^{\prime}(\overline{x}^{\star})$. 
\begin{itemize}
\item For $x\in [0,\overline{x}]$,
\begin{align}
V_{\underline{x}^{\star},\overline{x}^{\star}}(x)-V_{\underline{x},\overline{x}}(x)=g(x)\left(\frac{1}{g^{\prime}(\overline{x}^{\star})}-\frac{1}{g^{\prime}(\overline{x})}\right)\geq 0.\nn
\end{align}
\item For $x\in(\overline{x},\underline{x}^{\star}]$,
\begin{align}
V_{\underline{x}^{\star},\overline{x}^{\star}}(x)-V_{\underline{x},\overline{x}}(x)&=
\frac{g(x)}{g^{\prime}(\overline{x}^{\star})}-\left[x-\underline{x}-\beta+\frac{g(\underline{x})}{g^{\prime}(\overline{x})}\right]
\nn\\
&\geq \beta -\int_{\underline{x}}^{x}\left(1-\frac{g^{\prime}(z)}{g^{\prime}(\overline{x}^{\star})}\right)\mathrm{d}z
\nn\\
&\geq \beta -\int_{\underline{x}}^{\overline{x}}\left(1-\frac{g^{\prime}(z)}{g^{\prime}(\overline{x}^{\star})}\right)\mathrm{d}z-\int_{\overline{x}}^{x}\left(1-\frac{g^{\prime}(z)}{g^{\prime}(\overline{x}^{\star})}\right)\mathrm{d}z
\nn\\
&\geq \beta -\int_{\underline{x}}^{\overline{x}}\left(1-\frac{g^{\prime}(z)}{g^{\prime}(\overline{x}^{\star})}\right)\mathrm{d}z
\nn\\
&\geq \beta -\int_{\underline{x}}^{\overline{x}}\left(1-\frac{g^{\prime}(z)}{g^{\prime}(\overline{x})}\right)\mathrm{d}z
\nn\\
&=0,
\nn
\end{align}
where the last inequality uses $g^{\prime}(\overline{x})\geq g^{\prime}(\overline{x}^{\star})$, and the third inequality utilizes the fact that
\begin{align}
g^{\prime}(x)\geq g^{\prime}(\overline{x}^{\star}), \quad x\in(\overline{x},\underline{x}^{\star}].\nn
\end{align}
\item For $x\in (\underline{x}^{\star},\overline{x}^{\star}]$,
\begin{align}
\label{7.10.3.1060}
V_{\underline{x}^{\star},\overline{x}^{\star}}(x)-V_{\underline{x},\overline{x}}(x)&=\frac{g(x)}{g^{\prime}(\overline{x}^{\star})}-\left[x-\underline{x}-\beta+\frac{g(\underline{x})}{g^{\prime}(\overline{x})}\right]
\nn\\
&\geq \beta -\int_{\underline{x}}^{x}\left(1-\frac{g^{\prime}(z)}{g^{\prime}(\overline{x}^{\star})}\right)\mathrm{d}z.
\end{align}
Note that the function $x\mapsto \int_{x}^{(g^{\prime})_{3}^{-1}(g^{\prime}(x_{1}))}\left(1-\frac{g^{\prime}(z)}{g^{\prime}(x_{1})}\right)\mathrm{d}z$ is decreasing on $[0,(g^{\prime})_{2}^{-1}(g^{\prime}(x_{1}))]$ and increasing on $[(g^{\prime})_{2}^{-1}(g^{\prime}(x_{1})),(g^{\prime})_{3}^{-1}(g^{\prime}(x_{1}))]$. Additionally, by the definition of $x_{1}$, we have
\begin{align*}
&\int_{0}^{(g^{\prime})_{3}^{-1}(g^{\prime}(x_{1}))}\left(1-\frac{g^{\prime}(z)}{g^{\prime}(x_{1})}\right)\mathrm{d}z=0,
\nn\\
&
\int_{(g^{\prime})_{3}^{-1}(g^{\prime}(x_{1}))}^{(g^{\prime})_{3}^{-1}(g^{\prime}(x_{1}))}\left(1-\frac{g^{\prime}(z)}{g^{\prime}(x_{1})}\right)\mathrm{d}z=0.
\end{align*}
Consequently, since the function $x\mapsto \int_{\underline{x}}^{(g^{\prime})_{3}^{-1}(g^{\prime}(x))}\left(1-\frac{g^{\prime}(z)}{g^{\prime}(x)}\right)\mathrm{d}z$ is increasing on $[a_{3},\infty)$, we obtain
\begin{align}
\label{7.10.3.106}
0&\geq \int_{\underline{x}}^{(g^{\prime})_{3}^{-1}(g^{\prime}(x_{1}))}\left(1-\frac{g^{\prime}(z)}{g^{\prime}(x_{1})}\right)\mathrm{d}z
\nn\\
&\geq \int_{\underline{x}}^{(g^{\prime})_{3}^{-1}(g^{\prime}(\overline{x}^{\star}))}\left(1-\frac{g^{\prime}(z)}{g^{\prime}(\overline{x}^{\star})}\right)\mathrm{d}z.
\end{align}
By \eqref{7.10.3.106}, it follows that
\begin{align}
\int_{\underline{x}}^{x}\left(1-\frac{g^{\prime}(z)}{g^{\prime}(\overline{x}^{\star})}\right)\mathrm{d}z
&=\int_{\underline{x}}^{(g^{\prime})_{3}^{-1}(g^{\prime}(\overline{x}^{\star}))}\left(1-\frac{g^{\prime}(z)}{g^{\prime}(\overline{x}^{\star})}\right)\mathrm{d}z
+\int_{(g^{\prime})_{3}^{-1}(g^{\prime}(\overline{x}^{\star}))}^{x}\left(1-\frac{g^{\prime}(z)}{g^{\prime}(\overline{x}^{\star})}\right)\mathrm{d}z
\nn\\
&\leq \int_{(g^{\prime})_{3}^{-1}(g^{\prime}(\overline{x}^{\star}))}^{x}\left(1-\frac{g^{\prime}(z)}{g^{\prime}(\overline{x}^{\star})}\right)\mathrm{d}z
\nn\\
&\leq \int_{(g^{\prime})_{3}^{-1}(g^{\prime}(\overline{x}^{\star}))}^{\overline{x}^{\star}}\left(1-\frac{g^{\prime}(z)}{g^{\prime}(\overline{x}^{\star})}\right)\mathrm{d}z
\nn\\
&=\beta,\quad x\in (\underline{x}^{\star},\overline{x}^{\star}].\nn
\end{align}
Combining this with \eqref{7.10.3.1060} yields
\begin{align*}
V_{\underline{x}^{\star},\overline{x}^{\star}}(x)-V_{\underline{x},\overline{x}}(x)\geq 0,\quad x\in (\underline{x}^{\star},\overline{x}^{\star}].
\end{align*}
\item For $x\in(\overline{x}^{\star},\infty)$,
\begin{align} 
V_{0,\overline{x}^{\star}}(x)-V_{\underline{x},\overline{x}}(x)=&~
x-\beta-(\frac{g(\underline{x})}{g^{\prime}(\overline{x})}+x-\underline{x}-\beta)
\nn\\
=&~\underline{x}-\frac{g(\underline{x})}{g^{\prime}(\overline{x})}
\nn\\
=&~\int_{0}^{\underline{x}}\left(1-\frac{g^{\prime}(z)}{g^{\prime}(\overline{x})}\right)\mathrm{d}z\geq 0,\nn
\end{align}
since $g^{\prime}$ is increasing on $[0,a]$.
\end{itemize}
\item Consider a solution $(\underline{x},\overline{x})$ of \eqref{3.20.w2} satisfying $0\leq \underline{x}\leq a$ and $a<\overline{x}\leq a_{3}$. Lemma \ref{lem3.4} implies $g^{\prime}(\overline{x})\geq g^{\prime}(\overline{x}^{\star})$, so $a<\overline{x}< \underline{x}^{\star} \leq a_{3}$.
\begin{itemize}
\item For $x\in [0,\overline{x}]$,
\begin{align}
V_{\underline{x}^{\star},\overline{x}^{\star}}(x)-V_{\underline{x},\overline{x}}(x)=g(x)\left(\frac{1}{g^{\prime}(\overline{x}^{\star})}-\frac{1}{g^{\prime}(\overline{x})}\right)\geq 0.\nn
\end{align}
\item For $x\in(\overline{x},\overline{x}^{\star}]$,
\begin{align}
\label{7.10.3.108}
V_{\underline{x}^{\star},\overline{x}^{\star}}(x)-V_{\underline{x},\overline{x}}(x)&=
\frac{g(x)}{g^{\prime}(\overline{x}^{\star})}-\left[x-\underline{x}-\beta+\frac{g(\underline{x})}{g^{\prime}(\overline{x})}\right]
\nn\\
&\geq \beta -\int_{\underline{x}}^{x}\left(1-\frac{g^{\prime}(z)}{g^{\prime}(\overline{x}^{\star})}\right)\mathrm{d}z.
\end{align}
The function $x\mapsto \int_{\underline{x}}^{x}\left(1-\frac{g^{\prime}(z)}{g^{\prime}(\overline{x}^{\star})}\right)\mathrm{d}z$decreases on $[\overline{x},\underline{x}^{\star}]$, and increases on $[\underline{x}^{\star},\overline{x}^{\star}]$. Moreover, we have
\begin{align}
\int_{\underline{x}}^{\overline{x}}\left(1-\frac{g^{\prime}(z)}{g^{\prime}(\overline{x}^{\star})}\right)\mathrm{d}z\leq \int_{\underline{x}}^{\overline{x}}\left(1-\frac{g^{\prime}(z)}{g^{\prime}(\overline{x})}\right)\mathrm{d}z=\beta,\nn
\end{align}
and, by \eqref{7.10.3.106},
\begin{align}
\int_{\underline{x}}^{\overline{x}^{\star}}\left(1-\frac{g^{\prime}(z)}{g^{\prime}(\overline{x}^{\star})}\right)\mathrm{d}z
&=
\int_{\underline{x}}^{\underline{x}^{\star}}\left(1-\frac{g^{\prime}(z)}{g^{\prime}(\overline{x}^{\star})}\right)\mathrm{d}z
+
\int_{\underline{x}^{\star}}^{\overline{x}^{\star}}\left(1-\frac{g^{\prime}(z)}{g^{\prime}(\overline{x}^{\star})}\right)\mathrm{d}z
\nn\\
&\leq \int_{\underline{x}^{\star}}^{\overline{x}^{\star}}\left(1-\frac{g^{\prime}(z)}{g^{\prime}(\overline{x}^{\star})}\right)\mathrm{d}z
\nn\\
&=\beta.\nn
\end{align}
Therefore,
\begin{align}
\label{7.10.3.109}
\int_{\underline{x}}^{x}\left(1-\frac{g^{\prime}(z)}{g^{\prime}(\overline{x}^{\star})}\right)\mathrm{d}z\leq \beta, \quad x\in(\overline{x},\overline{x}^{\star}].
\end{align}

Combining \eqref{7.10.3.108} and \eqref{7.10.3.109} yields
\begin{align*}
V_{\underline{x}^{\star},\overline{x}^{\star}}(x)-V_{\underline{x},\overline{x}}(x)\geq 0, \quad x\in(\overline{x},\overline{x}^{\star}].
\end{align*}

\item For $x\in(\overline{x}^{\star},\infty)$,
\begin{align} 
V_{0,\overline{x}^{\star}}(x)-V_{\underline{x},\overline{x}}(x)=&~
V_{0,\overline{x}^{\star}}(\overline{x}^{\star})-V_{\underline{x},\overline{x}}(\overline{x}^{\star})\geq 0.\nn
\end{align}
\end{itemize}

\item Consider a solution $(\underline{x},\overline{x})$ of \eqref{3.20.w2} satisfying $0\leq \underline{x}\leq a$ and $\overline{x}>a_{3}$. Lemma \ref{lem3.4} gives $g^{\prime}(\overline{x})\geq g^{\prime}(\overline{x}^{\star})$. Since $g^{\prime}$ is strict increasing on $[a_{3},\infty)$, we have $\overline{x}>\overline{x}^{\star}$.
\begin{itemize}
\item For $x\in [0,\overline{x}^{\star}]$,
\begin{align}
\label{al.1}
V_{\underline{x}^{\star},\overline{x}^{\star}}(x)-V_{\underline{x},\overline{x}}(x)=g(x)\left(\frac{1}{g^{\prime}(\overline{x}^{\star})}-\frac{1}{g^{\prime}(\overline{x})}\right)\geq 0.
\end{align}
\item For $x\in(\overline{x}^{\star},\overline{x}]$,
\begin{align}
\label{al.2}
V_{\underline{x}^{\star},\overline{x}^{\star}}(x)-V_{\underline{x},\overline{x}}(x)&=
V_{\underline{x}^{\star},\overline{x}^{\star}}(\overline{x}^{\star})-V_{\underline{x},\overline{x}}(\overline{x}^{\star})+
x-\overline{x}^{\star}-\frac{g(x)-g(\overline{x}^{\star})}{g^{\prime}(\overline{x})}
\nn\\
&\geq x-\overline{x}^{\star}-\frac{g(x)-g(\overline{x}^{\star})}{g^{\prime}(\overline{x})}
\nn\\
&=\int_{\overline{x}^{\star}}^{x}\left(1-\frac{g^{\prime}(z)}{g^{\prime}(\overline{x})}\right)\mathrm{d}z
\nn\\
&\geq 0,
\end{align}
where the last inequality follows from
\begin{align}
g^{\prime}(x)\leq g^{\prime}(\overline{x}), \quad x\in(\overline{x}^{\star},\overline{x}].\nn
\end{align}
\item For $x\in(\overline{x},\infty)$,
\begin{align} 
\label{al.3}
V_{\underline{x}^{\star},\overline{x}^{\star}}(x)-V_{\underline{x},\overline{x}}(x)&=
V_{\underline{x}^{\star},\overline{x}^{\star}}(\overline{x})-V_{\underline{x},\overline{x}}(\overline{x})
\nn\\
&\geq 0.
\end{align}
\end{itemize}

\item Consider a solution $(\underline{x},\overline{x})$ of \eqref{3.20.w2} such that 
either $a<\underline{x}<a_{3}<\overline{x}$ or $\min\{\underline{x}, \overline{x}\}\geq a_{3}$. The inequalities \eqref{al.1}-\eqref{al.3} remain valid.
\end{itemize}

Suppose $(\underline{x}^{\star},\overline{x}^{\star})$ satisfies $0\leq \underline{x}^{\star}<\overline{x}^{\star}\leq a$. Then we must have $g^{\prime}(0)<g^{\prime}(x_{2})<g^{\prime}(a)$, $\underline{x}^{\star}=0$, and $\overline{x}^{\star}\in (0,
(g^{\prime})_{2}^{-1}(g^{\prime}(x_{2})))$.
We examine the following cases.
\begin{itemize}
\item Consider a solution $(\underline{x},\overline{x})$ of \eqref{3.20.w2} satisfying $0\leq \underline{x}\leq \overline{x}\leq a$. By Lemma \ref{lem3.4}, we have $g^{\prime}(\overline{x})\geq g^{\prime}(\overline{x}^{\star})$. If $\underline{x}=0$, then we must have $\overline{x}=\overline{x}^{\star}$, which implies $V_{\underline{x}^{\star},\overline{x}^{\star}}(x)=V_{\underline{x},\overline{x}}(x)$ for all $x\geq 0$. Thus, we focus on the nontrivial case $\underline{x}>0$, which implies $\overline{x}>\overline{x}^{\star}$.
\begin{itemize}
\item For $x\in [0,\overline{x}^{\star}]$,
\begin{align}
V_{\underline{x}^{\star},\overline{x}^{\star}}(x)-V_{\underline{x},\overline{x}}(x)=g(x)\left(\frac{1}{g^{\prime}(\overline{x}^{\star})}-\frac{1}{g^{\prime}(\overline{x})}\right)\geq 0.\nn
\end{align}
\item For $x\in(\overline{x}^{\star},\overline{x}]$,
\begin{align}
V_{\underline{x}^{\star},\overline{x}^{\star}}(x)-V_{\underline{x},\overline{x}}(x)&=
x-\beta-\frac{g(x)}{g^{\prime}(\overline{x})}
\nn\\
&= \int_{0}^{x}\left(1-\frac{g^{\prime}(z)}{g^{\prime}(\overline{x})}\right)\mathrm{d}z-\beta
\nn\\
&\geq
\int_{0}^{\overline{x}^{\star}}\left(1-\frac{g^{\prime}(z)}{g^{\prime}(\overline{x})}\right)\mathrm{d}z-\beta
\nn\\
&\geq \int_{0}^{\overline{x}^{\star}}\left(1-\frac{g^{\prime}(z)}{g^{\prime}(\overline{x}^{\star})}\right)\mathrm{d}z-\beta
\nn\\
&= 0,
\nn
\end{align}
where the second inequality uses the fact that $g^{\prime}$ is increasing on $[0,a]$.
\item For $x\in(\overline{x},\infty)$,
\begin{align} 
V_{0,\overline{x}^{\star}}(x)-V_{\underline{x},\overline{x}}(x)=&~
V_{0,\overline{x}^{\star}}(\overline{x})-V_{\underline{x},\overline{x}}(\overline{x})\geq 0.\nn
\end{align}
\end{itemize}
\item Consider a solution $(\underline{x},\overline{x})$ of \eqref{3.20.w2} satisfying $\underline{x}\in [0,a]$ and $\overline{x}\in (a, a_{3}]$. By Lemma \ref{lem3.4}, $g^{\prime}(\overline{x})\geq g^{\prime}(\overline{x}^{\star})$, and it holds that $g^{\prime}(\underline{x})<g^{\prime}(\overline{x})$.
\begin{itemize}
\item For $x\in [0,\overline{x}^{\star}]$,
\begin{align}
V_{\underline{x}^{\star},\overline{x}^{\star}}(x)-V_{\underline{x},\overline{x}}(x)=g(x)\left(\frac{1}{g^{\prime}(\overline{x}^{\star})}-\frac{1}{g^{\prime}(\overline{x})}\right)\geq 0.\nn
\end{align}
\item For $x\in(\overline{x}^{\star},(g^{\prime})_{2}^{-1}(g^{\prime}(\overline{x}))]$,
\begin{align}
V_{\underline{x}^{\star},\overline{x}^{\star}}(x)-V_{\underline{x},\overline{x}}(x)&=
x-\beta-\frac{g(x)}{g^{\prime}(\overline{x})}
\nn\\
&= \int_{0}^{x}\left(1-\frac{g^{\prime}(z)}{g^{\prime}(\overline{x})}\right)\mathrm{d}z-\beta
\nn\\
&\geq \int_{0}^{\overline{x}^{\star}}\left(1-\frac{g^{\prime}(z)}{g^{\prime}(\overline{x})}\right)\mathrm{d}z-\beta
\nn\\
&\geq \int_{0}^{\overline{x}^{\star}}\left(1-\frac{g^{\prime}(z)}{g^{\prime}(\overline{x}^{\star})}\right)\mathrm{d}z-\beta
\nn\\
&= 0.
\nn
\end{align}
\item For $x\in((g^{\prime})_{2}^{-1}(g^{\prime}(\overline{x})), \overline{x}]$,
\begin{align}
V_{\underline{x}^{\star},\overline{x}^{\star}}(x)-V_{\underline{x},\overline{x}}(x)&=
x-\beta-\frac{g(x)}{g^{\prime}(\overline{x})}
\nn\\
&= \int_{0}^{x}\left(1-\frac{g^{\prime}(z)}{g^{\prime}(\overline{x})}\right)\mathrm{d}z-\beta
\nn\\
&\geq \int_{0}^{\overline{x}}\left(1-\frac{g^{\prime}(z)}{g^{\prime}(\overline{x})}\right)\mathrm{d}z-\beta
\nn\\
&\geq \int_{\underline{x}}^{\overline{x}}\left(1-\frac{g^{\prime}(z)}{g^{\prime}(\overline{x})}\right)\mathrm{d}z-\beta
\nn\\
&= 0,
\nn
\end{align}
where we used $g^{\prime}(x)\geq g^{\prime}(\overline{x})$ for $x\in ((g^{\prime})_{2}^{-1}(g^{\prime}(\overline{x})), \overline{x}]$, and $g^{\prime}(x)\leq g^{\prime}(\underline{x})< g^{\prime}(\overline{x})$ for any $x\in [0, \underline{x}]$.
\item For $x\in(\overline{x},\infty)$,
\begin{align} 
V_{0,\overline{x}^{\star}}(x)-V_{\underline{x},\overline{x}}(x)=&~
V_{0,\overline{x}^{\star}}(\overline{x})-V_{\underline{x},\overline{x}}(\overline{x})\geq 0.\nn
\end{align}
\end{itemize}
\item Consider a solution $(\underline{x},\overline{x})$ of \eqref{3.20.w2} satisfying $\underline{x}\geq 0$ and $\overline{x}\in [a_{3},\infty)$. By Lemma \ref{lem3.4}, $g^{\prime}(\overline{x})\geq g^{\prime}(\overline{x}^{\star})>g^{\prime}(0)$. 
\begin{itemize}
\item For $x\in [0,\overline{x}^{\star}]$,
\begin{align}
V_{\underline{x}^{\star},\overline{x}^{\star}}(x)-V_{\underline{x},\overline{x}}(x)=g(x)\left(\frac{1}{g^{\prime}(\overline{x}^{\star})}-\frac{1}{g^{\prime}(\overline{x})}\right)\geq 0.\nn
\end{align}
\item For $x\in(\overline{x}^{\star},\overline{x}]$,
\begin{align}
V_{\underline{x}^{\star},\overline{x}^{\star}}(x)-V_{\underline{x},\overline{x}}(x)&=
x-\beta-\frac{g(x)}{g^{\prime}(\overline{x})}
\nn\\
&= \int_{0}^{x}\left(1-\frac{g^{\prime}(z)}{g^{\prime}(\overline{x})}\right)\mathrm{d}z-\beta.
\nn
\end{align}
If $g^{\prime}(0)< g^{\prime}(\overline{x})< g^{\prime}(a)$, the difference $V_{\underline{x}^{\star},\overline{x}^{\star}}(x)-V_{\underline{x},\overline{x}}(x)$ increases on $[\overline{x}^{\star},(g^{\prime})_{2}^{-1}(g^{\prime}(\overline{x}))]$, decreases on $[(g^{\prime})_{2}^{-1}(g^{\prime}(\overline{x})),(g^{\prime})_{3}^{-1}(g^{\prime}(\overline{x}))]$, and increases on $[(g^{\prime})_{3}^{-1}(g^{\prime}(\overline{x})),\overline{x}]$. Consequently, the minimum occurs at $\overline{x}^{\star}$ or $(g^{\prime})_{3}^{-1}(g^{\prime}(\overline{x}))$, leading to
\begin{align}
V_{\underline{x}^{\star},\overline{x}^{\star}}(x)-V_{\underline{x},\overline{x}}(x)\geq 0 \,\text{ on }\, (\overline{x}^{\star},\overline{x}] &\Longleftrightarrow \int_{0}^{(g^{\prime})_{3}^{-1}(g^{\prime}(\overline{x}))}\left(1-\frac{g^{\prime}(z)}{g^{\prime}(\overline{x})}\right)\mathrm{d}z\geq \beta
\nn\\
&\Longleftrightarrow \psi(0,(g^{\prime})_{3}^{-1}(g^{\prime}(\overline{x})))\geq \beta.
\end{align}

If $g^{\prime}(\overline{x}) \geq g^{\prime}(a)$, then $V_{\underline{x}^{\star},\overline{x}^{\star}}(x)-V_{\underline{x},\overline{x}}(x)$ is increasing on $[\overline{x}^{\star},\overline{x}]$. Since $V_{\underline{x}^{\star},\overline{x}^{\star}}(\overline{x}^{\star})-V_{\underline{x},\overline{x}}(\overline{x}^{\star})\geq 0$, we conclude that $V_{\underline{x}^{\star},\overline{x}^{\star}}(x)-V_{\underline{x},\overline{x}}(x)\geq 0$ for all $x\in [\overline{x}^{\star},\overline{x}]$.
\item For $x\in(\overline{x},\infty)$,
\begin{align} 
V_{\underline{x}^{\star},\overline{x}^{\star}}(x)-V_{\underline{x},\overline{x}}(x)=
V_{0,\overline{x}^{\star}}(\overline{x})-V_{\underline{x},\overline{x}}(\overline{x})\geq 0, \quad \text{ if }\,\, g^{\prime}(\overline{x}) \geq g^{\prime}(a), \nn
\end{align}
and
\begin{align} 
V_{\underline{x}^{\star},\overline{x}^{\star}}(x)-V_{\underline{x},\overline{x}}(x)&=
V_{0,\overline{x}^{\star}}(\overline{x})-V_{\underline{x},\overline{x}}(\overline{x})
\nn\\
&\geq 0,\quad \text{ if }\,\, g^{\prime}(0)< g^{\prime}(\overline{x})< g^{\prime}(a)\,\text{ and }\,\psi(0,(g^{\prime})_{3}^{-1}(g^{\prime}(\overline{x})))\geq \beta.\nn
\end{align}

Therefore, under \textbf{Assumption A}, we conclude that $V_{\underline{x}^{\star},\overline{x}^{\star}}(x)-V_{\underline{x},\overline{x}}(x)\geq 0$ for all $x\geq 0$, and all solutions $(\underline{x},\overline{x})$ of \eqref{3.20.w2} such that $\underline{x}\geq 0$ and $\overline{x}\in [a_{3},\infty)$.
\end{itemize}

\end{itemize}

The proof is complete.
\end{proof}

\begin{remark}
From the proof of Theorem \ref{Theorem3.4}, it is clear that, in cases $(\text{\emph{II}}_{1})$-$(\text{\emph{II}}_{2})$ with $\overline{x}^{*}\in [0,a)$, \textbf{Assumption A} is necessary and sufficient for the existence of an optimal strategy for the auxiliary stochastic control problem \eqref{eq:a008.aux.}
\end{remark}

\subsection{Case $\mu_{-}>0, \mu_{+}\le0$}
\label{subsec.3.3}

\subsubsection{Construction of a solution to \eqref{eq:a046}}
\label{subsubsec.3.3.1}
\begin{theorem}
Suppose that $\mu_{-}>0$ and $\mu_{+}\leq0$. Then, for every constant $C>0$, a non-negative, strictly increasing, continuously differentiable, and piecewise twice continuously differentiable solution to \eqref{eq:a046} is given as follows. If $x_{0}^{-}\geq a$, then
\begin{align}
\label{3.44.g.}
v(x)=&~\left\{
\begin{array}{ll}
Cx^{\gamma_{-}},& x\in[0,a], \\
C(a_{31}e^{\theta_{+}(x-a)}+a_{32}e^{\theta_{-}(x-a)}),& x\in [a,\infty),
\end{array}
\right.
\\
:=&~C\,g_{31}(x),\nn
\end{align}
where
\begin{align}\label{3.40.g.}
a_{31}=\dfrac{\gamma_{-}a^{\gamma_{-}-1}-\theta_{-}a^{\gamma_{-}}}{\theta_{+}-\theta_{-}},\quad 
a_{32}=\dfrac{\theta_{+}a^{\gamma_{-}}-\gamma_{-}a^{\gamma_{-}-1}}{\theta_{+}-\theta_{-}}.
\end{align}
If $x_{0}^{-}<a$, then
\begin{align}
\label{3.63.g.}
v(x)=&~\left\{
\begin{array}{ll}
Cx^{\gamma_{-}},& x\in[0,x_0^{-}], \\
C(a_{33}e^{\delta_{+}(x-x_0^{-})}+a_{34}e^{\delta_{-}(x-x_0^{-})}),& x\in [x_0^{-},a], \\
C(a_{35}e^{\theta_{+}(x-a)}+a_{36}e^{\theta_{-}(x-a)}),& x\in [a,\infty),
\end{array}
\right.
\\
:=&~C\,g_{32}(x),\nn
\end{align}
where
\begin{align}\label{3.48.g.}
\begin{cases}
a_{33}=\dfrac{\gamma_{-}(x_{0}^{-})^{\gamma_{-}-1}-\delta_{-}(x_{0}^{-})^{\gamma_{-}}}{\delta_{+}-\delta_{-}},\\[1.2ex]
a_{34}=\dfrac{\delta_{+}(x_{0}^{-})^{\gamma_{-}}-\gamma_{-}(x_{0}^{-})^{\gamma_{-}-1}}{\delta_{+}-\delta_{-}},\\[1.2ex]
a_{35}=\dfrac{a_{33}\left(\delta_{+}-\theta_{-}\right)e^{\delta_{+}(a-x_{0}^{-})}+a_{34}\left(\delta_{-}-\theta_{-}\right)e^{\delta_{-}(a-x_{0}^{-})}}{\theta_{+}-\theta_{-}},\\[1.2ex]
a_{36}=\dfrac{a_{33}\left(\theta_{+}-\delta_{+}\right)e^{\delta_{+}(a-x_{0}^{-})}+a_{34}\left(\theta_{+}-\delta_{-}\right)e^{\delta_{-}(a-x_{0}^{-})}}{\theta_{+}-\theta_{-}}.
\end{cases}
\end{align}
In both cases, a maximizing control in \eqref{eq:a046} is
\begin{align*}
u^{\star}(x)=
\begin{cases}
-\frac{\mu_{-}v^{\prime}(x)}{\sigma_{-}^{2}v^{\prime\prime}(x)}=\frac{\mu_{-}x}{\sigma_{-}^{2}(1-\gamma_{-})},&x\in[0,a\wedge x_{0}^{-}],\\[1.2ex]
1,&x>a\wedge x_{0}^{-}.
\end{cases}
\end{align*}
\end{theorem}
\begin{proof}
When confined in the interval $[0,a]$, \eqref{eq:a046} reduces to \eqref{eq:ee01}.
To guess the geometric nature of the solution to \eqref{eq:ee01}, we make the following observations.
\begin{itemize}
\item If there is a sub-interval $(b,c)\subseteq [0,a]$ ($b<c$) such that $v^{\prime\prime}(x)<0$ for all $x\in (b,c)$, then, for any fixed $x\in (b,c)$, the quadratic trinomial in $u\in[0,1]$ on the left hand side of \eqref{eq:ee01} attains its maximum at $u=\min\{-\frac{\mu_{-}v^{\prime}(x)}{\sigma_{-}^2v^{\prime\prime}(x)},1\}$. Hence, on $(b,c)$, \eqref{eq:ee01} is reduced to
\begin{align}
\left\{
\begin{array}{ll}
-\frac{\mu_{-}^{2}\left[v^{\prime}(x)\right]^{2}}{2\sigma_{-}^{2}v^{\prime\prime}(x)}-qv(x)=0,& x\in (b,c)\cap \{x\le a: -\frac{\mu_{-}v^{\prime}(x)}{\sigma_{-}^2v^{\prime\prime}(x)}< 1\}, \\
\frac{1}{2}\sigma_{-}^2v^{\prime\prime}(x)+\mu_{-}v^{\prime}(x)-qv(x)=0,& x\in (b,c)\cap \{x\le a: -\frac{\mu_{-}v^{\prime}(x)}{\sigma_{-}^2v^{\prime\prime}(x)}\geq 1\}.
\end{array}
\right.
\nn
\end{align}

\item If there exists a sub-interval $(b,c)\subseteq [0,a]$ ($b<c$) such that $v^{\prime\prime}(x)=0$ for all $x\in (b,c)$, then, $v^{\prime}(x)\equiv v$ on $(b,c)$ for some constant $v\geq 0$; and, for any fixed $x\in (b,c)$, the quadratic trinomial in $u\in[0,1]$ on the left hand side of \eqref{eq:ee01} reduces to a linear function that attains its maximum at $u=1$ (since $\mu_{-}> 0$ and $v^{\prime}(x)\geq 0$). Hence, on the interval $(b,c)$, \eqref{eq:ee01} is again reduced to
\begin{align}
\left\{
\begin{array}{ll}
\mu_{-} v-qv(x)=0, & x\in (b,c),\\
v^{\prime}(x)=v, & x\in (b,c),\nn
\end{array}
\right.
\end{align}
to which the unique solution $v$ is a constant, which is not strictly increasing.
\item If there is a sub-interval $(b,c)\subseteq [0,a]$ ($b<c$) such that $v^{\prime\prime}(x)>0$ for all $x\in (b,c)$, then, for any fixed $x\in (b,c)$, the quadratic trinomial in $u\in[0,1]$ on the left hand side of \eqref{eq:ee01} attains its maximum at $u=1$. Hence, on $(b,c)$, \eqref{eq:ee01} is reduced to
\begin{align}
\frac{1}{2}\sigma_{-}^2v^{\prime\prime}(x)+\mu_{-}v^{\prime}(x)-qv(x)=0,\quad x\in (b,c).
\nn
\end{align}
\end{itemize}

We begin by assuming that $x_{0}^{-}\geq a$ and that the solution $v$ to \eqref{eq:ee01} satisfies
\begin{align}
\label{3.31.g.}
\left\{
\begin{array}{ll}
v^{\prime\prime}(x)<0,& x\in [0,a], \\
-\frac{\mu_{-}v^{\prime}(x)}{\sigma_{-}^2v^{\prime\prime}(x)}< 1,& x\in [0,a].
\end{array}
\right.
\end{align}
Within the interval $[0,a]$, under the assumptions stated in \eqref{3.31.g.}, \eqref{eq:ee01} simplifies to
\begin{align}
-\frac{\mu_{-}^{2}\left[v^{\prime}(x)\right]^{2}}{2\sigma_{-}^{2}v^{\prime\prime}(x)}-qv(x)=0,\quad x\in [0,a].
\nn
\end{align}
The general solution to this ordinary differential equation is
\begin{align}
\label{3.32.g.}
v(x)=Cx^{\gamma_{-}},\quad x\in [0,a],
\end{align}
where $\gamma_{-}=\frac{2\sigma_{-}^2q}{\mu_{-}^2+2\sigma_{-}^2q}$, and $C>0$ is some undetermined constant.
By \eqref{3.32.g.} and the assumption of $x_{0}^{-}=\frac{\sigma_{-}^2 \left(1-\gamma_{-}\right)}{\mu_{-}}>a$, we compute
\begin{align}
\label{3.33.w}
\left\{\begin{array}{ll}
v^{\prime\prime}(x)=C\gamma_{-}(\gamma_{-}-1)x^{\gamma_{-}-2}<0, & x\in [0,a], \\
-\frac{\mu_{-}v^{\prime}(x)}{\sigma_{-}^2v^{\prime\prime}(x)}=\frac{\mu_{-}x}{\sigma_{-}^2 \left(1-\gamma_{-}\right)}<1, & x\in [0,a],
\end{array}
\right.
\end{align}
which confirms that the assumptions in \eqref{3.31.g.} hold.

Next, we proceed to seek a non-negative and strictly increasing solution to \eqref{eq:a046} over the interval $[a,\infty)$. On this domain, \eqref{eq:a046} reduces to \eqref{eq:ee03}.
We now impose the following structural conditions
\begin{align}
\label{3.37.g.}
\left\{
\begin{array}{ll}
v^{\prime\prime}(x)>0,& x\in [a,\infty), \\
-\frac{\mu_{+}v^{\prime}(x)}{\sigma_{+}^2v^{\prime\prime}(x)}\le\frac{1}{2},& x\in [a,\infty).
\end{array}
\right.
\end{align}
Under the assumptions of \eqref{3.37.g.}, the maximum of the left hand side of \eqref{eq:ee03} occurs at $u=1$, then \eqref{eq:ee03} is simplified to
\begin{align}
\frac{1}{2}\sigma_{+}^2v^{\prime\prime}(x)+\mu_{+}v^{\prime}(x)-qv(x)=0,\quad x\in [a,\infty).
\nn
\end{align}
Its general solution takes the form
\begin{align}
\label{3.38.g.}
v(x)=C(a_{31}e^{\theta_{+}(x-a)}+a_{32}e^{\theta_{-}(x-a)}),\quad x\in [a,\infty).
\end{align}
Matching the value and the first derivative at $a$ yields the stated $a_{31}$ and $a_{32}$. 
Note also that
\begin{equation}\label{3.41.g.}
a_{31}+a_{32}=\frac{\theta_{+}a^{\gamma_{-}}-\theta_{-}a^{\gamma_{-}}}{\theta_{+}-\theta_{-}}=a^{\gamma_{-}}>0.
\end{equation}
Combining $a_{31}>0$ and \eqref{3.41.g.}, one deduces that
\begin{itemize}
\item If $a_{32}\geq 0$, then, by \eqref{3.38.g.}, it is obvious that $v^{\prime\prime}(x)> 0$ for any $x\in[a,\infty)$.

\item If $a_{32}<0$, observe that
\begin{align}
v^{\prime\prime}(x)=C(a_{31}\theta_{+}^{2}e^{\theta_{+}(x-a)}+a_{32}\theta_{-}^{2}e^{\theta_{-}(x-a)}),\quad x\in [a,\infty),\nn
\end{align}
is strictly increasing on $[a,\infty)$ with
$
v^{\prime\prime}(a+)
=C(a_{31}\theta_{+}^{2}+a_{32}\theta_{-}^{2})>0,
$
due to $a_{31}>|a_{32}|$ and $\theta_{+}>|\theta_{-}|$. It follows that $v^{\prime\prime}(x)>0$ for any $x\in[a,\infty)$.
\end{itemize}
Then the first inequality of \eqref{3.37.g.} is verified. To establish the second inequality, we compute
\begin{align}
\label{3.42.g.}
&-\mu_{+}v^{\prime}(x)-\frac{1}{2}\sigma_{+}^2v^{\prime\prime}(x)
\nn\\
=&~ Ca_{31}e^{\theta_{+}(x-a)}\left(-\mu_{+}\theta_{+}-\frac{1}{2}\sigma_{+}^2\theta_{+}^2\right)
+Ca_{32}e^{\theta_{-}(x-a)}\left(-\mu_{+}\theta_{-}-\frac{1}{2}\sigma_{+}^2\theta_{-}^2\right)
\nn\\
=&~
-\frac{Cqa^{\gamma_{-}-1}}{\theta_{+}-\theta_{-}}\left((\gamma_{-}-\theta_{-}a)e^{\theta_{+}(x-a)}
+(\theta_{+}a-\gamma_{-})e^{\theta_{-}(x-a)}\right),
\quad x\in [a,\infty).
\end{align}
Two cases arise
\begin{itemize}
\item If $\theta_{+}a-\gamma_{-}\ge0$, then the function on the RHS of \eqref{3.42.g.} is non-positive valued.
\item If $\theta_{+}a-\gamma_{-}<0$, then the function on the RHS of \eqref{3.42.g.} is strictly decreasing in $x$, achieving its maximum at $x=a$
\begin{align}
\label{3.43.g.}
-\mu_{+}v^{\prime}(a+)-\frac{1}{2}\sigma_{+}^2v^{\prime\prime}(a+)
=-\frac{Cqa^{\gamma_{-}}}{\theta_{+}-\theta_{-}}(\theta_{+}-\theta_{-})<0.
\end{align}
\end{itemize}
In both cases, the inequality $-\frac{\mu_{+}v^{\prime}(x)}{\sigma_{+}^2v^{\prime\prime}(x)}\le\frac{1}{2}$ holds for all $x\in[a,\infty)$, thereby verifying the second condition in \eqref{3.37.g.}.
In summary, the constructed function \eqref{3.38.g.} solves \eqref{eq:a046} and the maximizer is $u^{\star}(x)=1$ on $[a,\infty)$.

We now suppose that $x_0^{-}<a$. From our earlier analysis, the solution $v$ to \eqref{eq:a046} on $[0,x_0^{-}]$ must take the power law form, that is,
\begin{align}
\label{3.45.g.}
v(x)=Cx^{\gamma_{-}},\quad x\in [0,x_0^{-}],
\end{align}
where $\gamma_{-}=\frac{2\sigma_{-}^2q}{\mu_{-}^2+2\sigma_{-}^2q}$ and $C>0$ is a constant.
On the subsequent interval $[x_0^{-},a]$, the simultaneous satisfaction of $v^{\prime\prime}(x)<0$ and $-\frac{\mu_{-}v^{\prime}(x)}{\sigma_{-}^2v^{\prime\prime}(x)}<1$ is no longer possible. Consequently, we claim that \eqref{eq:ee01} reduces to
\begin{align}
\frac{1}{2}\sigma_{-}^2v^{\prime\prime}(x)+\mu_{-}v^{\prime}(x)-qv(x)=0,\quad x\in [x_0^{-},a].
\nn
\end{align}
Its general solution is
\begin{align}
\label{3.46.g.}
v(x)=C(a_{33}e^{\delta_{+}(x-x_0^{-})}+a_{34}e^{\delta_{-}(x-x_0^{-})}),\quad x\in [x_0^{-},a].
\end{align}
We expect that $v$ is continuously differentiable, then
\begin{equation}\label{3.47.g.}
\begin{cases}
(x_0^{-})^{\gamma_{-}}=a_{33}+a_{34},\\
\gamma_{-}(x_0^{-})^{(\gamma_{-}-1)}=\delta_{+}a_{33}+\delta_{-}a_{34}.\\
\end{cases}
\end{equation}
Solving \eqref{3.47.g.} yields $a_{33},a_{34}$ are determined in \eqref{3.48.g.}. 
It follows immediately that $a_{33}>0$ and
\begin{equation}\label{3.49.g.}
a_{33}+a_{34}=\frac{\delta_{+}(x_0^{-})^{\gamma_{-}}-\delta_{-}(x_0^{-})^{\gamma_{-}}}{\delta_{+}-\delta_{-}}=(x_0^{-})^{\gamma_{-}}>0.
\end{equation}
Combining $a_{33}>0$ and \eqref{3.49.g.}, we now examine the sign of $v^{\prime\prime}$ on $[x_0^{-},a]$ under various cases. 
\begin{itemize}
\item If $a_{34}\geq 0$, then, by \eqref{3.46.g.}, it is obvious that $v^{\prime\prime}(x)> 0$ for any $x\in[x_0^{-},a]$.

\item If $a_{34}<0$, then
\begin{align}
v^{\prime\prime}(x)=C(a_{33}\delta_{+}^2 e^{\delta_{+}(x-x_0^{-})}+a_{34}\delta_{-}^2e^{\delta_{-}(x-x_0^{-})}),\quad x\in [x_0^{-},a],\nn
\end{align}
is strictly increasing on $[x_0^{-},a]$. At $x=x_0^{-}$, we have $v^{\prime\prime}(x_0^{-}+)=C(a_{33}\delta_{+}^{2}+a_{34}\delta_{-}^{2})$.
\begin{itemize}
\item[$\bullet$] If $a_{33}\delta_{+}^{2}+a_{34}\delta_{-}^{2}\ge0$, we have $v^{\prime\prime}(x)>0$ for all $x\in[x_0^{-},a]$. 
\item[$\bullet$] If $a_{33}\delta_{+}^{2}+a_{34}\delta_{-}^{2}<0$, which occurs if and only if $\ln{\frac{-a_{34}\delta_{-}^2}{a_{33}\delta_{+}^2}}>0$.
\begin{itemize}
\item[$\bullet$] If $x_0^{-}+\frac{\ln{\frac{-a_{34}\delta_{-}^2}{a_{33}\delta_{+}^2}}}{\delta_{+}-\delta_{-}}\in(x_0^{-},a)$, we have
\begin{align}
\label{3.50.g.}
\left\{
\begin{array}{ll}
v^{\prime\prime}(x)>0,& x\in (x_0^{-}+\frac{\ln{\frac{-a_{34}\delta_{-}^2}{a_{33}\delta_{+}^2}}}{\delta_{+}-\delta_{-}},a], \\
v^{\prime\prime}(x)<0,& x\in [x_0^{-}, x_0^{-}+\frac{\ln{\frac{-a_{34}\delta_{-}^2}{a_{33}\delta_{+}^2}}}{\delta_{+}-\delta_{-}}).
\end{array}
\right.
\end{align}
\item[$\bullet$] If $x_0^{-}+\frac{\ln{\frac{-a_{34}\delta_{-}^2}{a_{33}\delta_{+}^2}}}{\delta_{+}-\delta_{-}}\in(a,\infty)$, we have
\begin{align}
\label{3.005.g.}
v^{\prime\prime}(x)<0,~~~ x\in [x_0^{-},a].
\end{align}
\end{itemize}
\end{itemize} 
From \eqref{3.46.g.}, we compute
\begin{align}
\label{3.52.g.}
&\mu_{-}v^{\prime}(x)+\sigma_{-}^2v^{\prime\prime}(x)
\nn\\
=&~ C\left(a_{33} \left(\mu_{-}\delta_{+}+\sigma_{-}^2\delta_{+}^{2}\right)
e^{\delta_{+}(x-x_0^{-})}
+ a_{34}\left(\mu_{-}\delta_{-}+\sigma_{-}^2\delta_{-}^{2}\right)
e^{\delta_{-}(x-x_0^{-})}\right)
\nn\\
=&~
C\left(a_{33} \left(2q-\mu_{-}\delta_{+}\right)
e^{\delta_{+}(x-x_0^{-})}
+ a_{34}\left(2q-\mu_{-}\delta_{-}\right)
e^{\delta_{-}(x-x_0^{-})}\right), \quad x\in [x_0^{-},a].
\end{align}

Using the definition of $\delta_{+}$, one gets
\begin{align}
2q-\mu_{-}\delta_{+}=&~2q-\mu_{-}\frac{-\mu_{-}+\sqrt{\mu_{-}^2+2q\sigma_{-}^2}}{\sigma_{-}^2}
\nn\\
=&~2q-\frac{\mu_{-}\left(-\mu_{-}+\sqrt{\mu_{-}^2+2q\sigma_{-}^2}\right)\left(\mu_{-}+\sqrt{\mu_{-}^2+2q\sigma_{-}^2}\right)}{\sigma_{-}^2\left(\mu_{-}+\sqrt{\mu_{-}^2+2q\sigma_{-}^2}\right)}
\nn\\
=&~2q-\frac{2q\mu_{-}\sigma_{-}^2}{\sigma_{-}^2\left(\mu_{-}+\sqrt{\mu_{-}^2+2q\sigma_{-}^2}\right)}
\nn\\
=&~\frac{2q\sqrt{\mu_{-}^2+2q\sigma_{-}^2}}{\mu_{-}+\sqrt{\mu_{-}^2+2q\sigma_{-}^2}}>0.
\end{align}
\end{itemize}
This, together with $a_{33}>0$ and the fact of $2q-\mu_{-}\delta_{-}>0$, implies that the function on the right-hand side of \eqref{3.52.g.} is strictly increasing in $x$. In this case, the minimum occurs at $x=x_{0}^{-}$
\begin{align}
\label{3.53.g.}
&\sigma_{-}^2v^{\prime\prime}(x_{0}^{-}+)+\mu_{-}v^{\prime}(x_{0}^{-}+)
\nn\\
=&~
2\left(qv(x_{0}^{-}+)-\mu_{-}v^{\prime}(x_{0}^{-}+)-\frac{1}{2}\sigma_{-}^2v^{\prime\prime}(x_{0}^{-}+)\right)
+\sigma_{-}^2v^{\prime\prime}(x_{0}^{-}+)+\mu_{-}v^{\prime}(x_{0}^{-}+)
\nn\\
=&~2qv(x_{0}^{-}+)-\mu_{-}v^{\prime}(x_{0}^{-}+)
\nn\\
=&~2qv(x_{0}^{-}-)-\mu_{-}v^{\prime}(x_{0}^{-}-)
\nn\\
=&~C(x_0^{-})^{\gamma_{-}-1}(2qx_0^{-}-\mu_{-}\gamma_{-})=0,
\end{align}
where we have used the fact that
\begin{align}
\label{3.54.g.}
2qx_0^{-}-\mu_{-}\gamma_{-}=&~
2q\frac{\sigma_{-}^2 \left(1-\gamma_{-}\right)}{\mu_{-}}-\mu_{-}\frac{2\sigma_{-}^2 q}{\mu_{-}^2+2\sigma_{-}^2q}
\nn\\
=&~
2q\frac{\sigma_{-}^2 }{\mu_{-}}\frac{\mu_{-}^2}{\mu_{-}^2+2\sigma_{-}^2q}-\mu_{-}\frac{2\sigma_{-}^2 q}{\mu_{-}^2+2\sigma_{-}^2q}
\nn\\
=&~0.
\end{align}
We now verify the second-order smooth fit at $x_0^{-}$. From \eqref{3.45.g.} and the definition $x_0^{-}=\frac{\sigma_{-}^{2}(1-\gamma_{-})}{\mu_{-}}$, we obtain
\begin{align}
-\frac{\mu_{-}v^{\prime}(x_0^{-}-)}{\sigma_{-}^{2}v^{\prime\prime}(x_0^{-}-)}=\frac{\mu_{-}x_0^{-}}{\sigma_{-}^{2}(1-\gamma_{-})}=1,
\end{align}
and therefore
\begin{align}
\sigma_{-}^{2}v^{\prime\prime}(x_0^{-}-)+\mu_{-}v^{\prime}(x_0^{-}-)=0.
\label{eq:C2.x0.minus.left}
\end{align}
Equation \eqref{3.53.g.} gives
\begin{align}
\sigma_{-}^{2}v^{\prime\prime}(x_0^{-}+)+\mu_{-}v^{\prime}(x_0^{-}+)=0.
\label{eq:C2.x0.minus.right}
\end{align}
Since \eqref{3.47.g.} implies $v^{\prime}(x_0^{-}-)=v^{\prime}(x_0^{-}+)$, it follows that
\begin{align}
v^{\prime\prime}(x_0^{-}-)
=v^{\prime\prime}(x_0^{-}+).
\label{eq:C2.x0.minus}
\end{align}
Thus, $v$ is twice continuously differentiable at $x_0^{-}$.
Then, the function on the right-hand side of
\eqref{3.52.g.} is non-negative valued.
Consequently, we conclude that
\begin{itemize}
\item[$\bullet$] Suppose either $a_{34}\geq 0$ or $
a_{33}\delta_{+}^{2}+a_{34}\delta_{-}^{2}\geq 0$. Then
\begin{align}
v^{\prime\prime}(x)>0, \quad x\in [x_0^{-},a],\nn
\end{align}
confirming that $v$ given by \eqref{3.46.g.} and \eqref{3.48.g.} is indeed a solution to \eqref{eq:ee01} on $[x_0^{-},a]$.

\item[$\bullet$] Suppose $a_{34}<0$ and $
a_{33}\delta_{+}^{2}+a_{34}\delta_{-}^{2}<0$. By \eqref{3.50.g.} and \eqref{3.005.g.}, if $x_0^{-}+\frac{\ln{\frac{-a_{34}\delta_{-}^2}{a_{33}\delta_{+}^2}}}{\delta_{+}-\delta_{-}}\in(x_0^{-},a)$, then we have
\begin{align*}
\begin{cases}
v^{\prime\prime}(x)>0, & x\in (x_0^{-}+\frac{\ln{\frac{-a_{34}\delta_{-}^2}{a_{33}\delta_{+}^2}}}{\delta_{+}-\delta_{-}},a], \\
v^{\prime\prime}(x)<0, & x\in [x_0^{-}, x_0^{-}+\frac{\ln{\frac{-a_{34}\delta_{-}^2}{a_{33}\delta_{+}^2}}}{\delta_{+}-\delta_{-}}),
\\
\sigma_{-}^2 v^{\prime\prime}(x) +\mu_{-}v^{\prime}(x)\geq 0, & x\in [x_0^{-}, x_0^{-}+\frac{\ln{\frac{-a_{34}\delta_{-}^2}{a_{33}\delta_{+}^2}}}{\delta_{+}-\delta_{-}});
\end{cases}
\end{align*}
while, if
$x_0^{-}+\frac{\ln{\frac{-a_{34}\delta_{-}^2}{a_{33}\delta_{+}^2}}}{\delta_{+}-\delta_{-}}\in(a,\infty)$, then we have
\begin{align*}
\begin{cases}
v^{\prime\prime}(x)<0, & x\in [x_0^{-}, a),
\\
\sigma_{-}^2 v^{\prime\prime}(x) +\mu_{-}v^{\prime}(x)\geq 0, & x\in [x_0^{-}, a).
\end{cases}
\end{align*}
This confirms that $v$ given by \eqref{3.46.g.} and \eqref{3.48.g.} is indeed a solution to \eqref{eq:ee01} on $[x_0^{-},a]$.
\end{itemize}
We continue to seek a non-negative and strictly increasing solution to \eqref{eq:a046} on the interval $[a,\infty)$. When confined in the interval $[a,\infty)$, \eqref{eq:a046} reduces to \eqref{eq:ee03}.
We assume that
\begin{align}
\label{3.56.g.}
\left\{
\begin{array}{ll}
v^{\prime\prime}(x)>0,& x\in [a,\infty), \\
-\frac{\mu_{+}v^{\prime}(x)}{\sigma_{+}^2v^{\prime\prime}(x)}\le\frac{1}{2},& x\in [a,\infty).
\end{array}
\right.
\end{align}
Under \eqref{3.56.g.}, the maximum value on the left hand side of \eqref{eq:ee03} is obtained at $u=1$, then \eqref{eq:ee03} is simplified to
\begin{align}
\frac{1}{2}\sigma_{+}^2v^{\prime\prime}(x)+\mu_{+}v^{\prime}(x)-qv(x)=0,\quad x\in [a,\infty).
\nn
\end{align}
Its general solution is
\begin{align}
\label{3.57.g.}
v(x)=C(a_{35}e^{\theta_{+}(x-a)}+a_{36}e^{\theta_{-}(x-a)}),\quad x\in [a,\infty).
\end{align}
To ensure that $v$ is continuously differentiable, we must have
\begin{equation}\label{3.58.g.}
\begin{cases}
a_{33}e^{\delta_{+}(a-x_{0}^{-})}+a_{34}e^{\delta_{-}(a-x_{0}^{-})}=a_{35}+a_{36},\\
\delta_{+}a_{33}e^{\delta_{+}(a-x_{0}^{-})}+\delta_{-}a_{34}e^{\delta_{-}(a-x_{0}^{-})}=\theta_{+}a_{35}+\theta_{-}a_{36}.\\
\end{cases}
\end{equation}
Solving \eqref{3.58.g.} yields that $a_{35},a_{36}$ are determined in \eqref{3.48.g.}.
Note that 
\begin{align}
a_{35} &= \frac{\delta_{+}a_{33}e^{\delta_{+}(a-x_{0}^{-})}+\delta_{-}a_{34}e^{\delta_{-}(a-x_{0}^{-})}}{\theta_{+}-\theta_{-}}-\frac{\theta_{-}\left(a_{33}e^{\delta_{+}(a-x_{0}^{-})}+a_{34}e^{\delta_{-}(a-x_{0}^{-})}\right)}{\theta_{+}-\theta_{-}}>0,\nn
\end{align}
since 
\begin{align*}
a_{33}e^{\delta_{+}(a-x_{0}^{-})}+a_{34}e^{\delta_{-}(a-x_{0}^{-})}&=\frac{v(a-)}{C}>0,
\\
\delta_{+}a_{33}e^{\delta_{+}(a-x_{0}^{-})}+\delta_{-}a_{34}e^{\delta_{-}(a-x_{0}^{-})}&=\frac{v^{\prime}(a-)}{C}>0.
\end{align*}
Moreover, from the first equality of \eqref{3.58.g.} follows
\begin{equation}\label{3.60.g.}
a_{35}+a_{36}=\frac{v(a-)}{C}>0.
\end{equation}
Combining $a_{35}>0$ and \eqref{3.60.g.}, one deduces that
\begin{itemize}
\item If $a_{36}\geq 0$, then, by \eqref{3.57.g.}, it is obvious that $v^{\prime\prime}(x)> 0$ for any $x\in[a,\infty)$.

\item If $a_{36}<0$, then
\begin{align}
v^{\prime\prime}(x)=C(a_{35}\theta_{+}^2 e^{\theta_{+}(x-a)}+a_{36}\theta_{-}^2e^{\theta_{-}(x-a)}),\quad x\in [a,\infty),\nn
\end{align}
is strictly increasing on $[a,\infty)$ with
\begin{align}
v^{\prime\prime}(a+)
=&~C(a_{35}\theta_{+}^{2}+a_{36}\theta_{-}^{2})>0,
\end{align}
since $a_{35}>|a_{36}|$ and $\theta_{+}>|\theta_{-}|$. Then $v^{\prime\prime}(x)>0$ for any $x\in[a,\infty)$.

\end{itemize}
Then the first inequality of \eqref{3.56.g.} is verified. 
Since $v(a+)=v(a-)>0$, $v^{\prime}(a+)=v^{\prime}(a-)>0$, and
$v^{\prime\prime}(x)>0$ on $[a,\infty)$, we have
$v^{\prime}(x)>0$ and $v(x)>0$ for all $x\in[a,\infty)$. 
From \eqref{3.57.g.} and the fact that
$
\frac{1}{2}\sigma_+^2\theta_\pm^2+\mu_+\theta_\pm-q=0,
$
it follows that 
\begin{align}\label{3.61.g.}
&-\mu_+v^{\prime}(x)-\frac{1}{2}\sigma_+^2V^{\prime\prime}(x) \nn\\
={}&C a_{35}e^{\theta_+(x-a)}
\left(-\mu_+\theta_+-\frac{1}{2}\sigma_+^2\theta_+^2\right)
+C a_{36}e^{\theta_-(x-a)}
\left(-\mu_+\theta_--\frac{1}{2}\sigma_+^2\theta_-^2\right)
\nn\\
={}&-qC\left(a_{35}e^{\theta_+(x-a)}
+a_{36}e^{\theta_-(x-a)}\right)
=-qv(x)<0,\qquad x\in[a,\infty).
\end{align}
That is, we have $-\frac{\mu_{+}v^{\prime}(x)}{\sigma_{+}^2v^{\prime\prime}(x)}\le\frac{1}{2}$ for all $x\in[a,\infty)$. The second inequality of \eqref{3.56.g.} is verified.
Therefore, the function $v$ given by \eqref{3.57.g.} and \eqref{3.48.g.} is indeed a solution to \eqref{eq:a046} on $[a,\infty)$.
\end{proof}

\subsubsection{Identification of the uniform (in $x$) maximizer of $V_{\underline{x},\overline{x}}(x)$ over $(\underline{x},\overline{x})$}

By analyzing various parameter configurations under case $\mu_{-}>0,\, \mu_{+}\le0$, we identify two mutually exclusive and collectively exhaustive sub-cases denoted by $(\text{III}_{1})$-$(\text{III}_{2})$. For each sub-case, we specify the function $g$ and its inflection points $(a_i)_{1\leq i\leq 3}$, and derive conclusions based on the constructions of $g_{31}$ and $g_{32}$ presented in Subsection \ref{subsubsec.3.3.1}.
\begin{itemize}

\item[($\text{III}_{1}$)] If $x_0^{-}\geq a$, letting
\begin{align}
g:=g_{31}, \, a_1:=0, \, a_2:=0, \, a_3:=a,\nn
\end{align}
then $g$ is \textbf{piecewise concave-convex} with inflection points $(a_i)_{1\leq i\leq 3}$ satisfying $0=a_{1}=a_{2}<a_{3}$.

\item[($\text{III}_{2}$)] If $x_0^{-}< a$ and either one of the following conditions holds
\begin{itemize}
\item[$\bullet$] $a_{34}\geq 0$ or $
a_{33}\delta_{+}^{2}+a_{34}\delta_{-}^{2}\geq 0$,
\item[$\bullet$] $a_{34}<0$, $
a_{33}\delta_{+}^{2}+a_{34}\delta_{-}^{2}<0$,
\end{itemize}
letting
\begin{align}
g:=g_{32}, \, a_1=0, \, a_2=0, \, a_3=\left(x_0^{-}+\frac{\ln{\left(1\vee\frac{-a_{34}\delta_{-}^2}{a_{33}\delta_{+}^2}\right)}}{\delta_{+}-\delta_{-}}\right)\wedge a,\nn
\end{align}
then $g$ is \textbf{piecewise concave-convex} with inflection points $(a_i)_{1\leq i\leq 3}$ satisfying $0=a_{1}=a_{2}<a_{3}$.
\end{itemize}

\begin{theorem}
\label{Theorem3.5}
Suppose $\mu_{-}>0,\,\mu_{+}\le0$. Let $\psi$ and $\phi$ be defined by \eqref{def.psi.} and \eqref{phi.1}. Let $g$ and $(a_{i})_{1\leq i\leq 3}$ be specified separately in the four cases $(\text{\emph{III}}_{1})$-$(\text{\emph{III}}_{2})$.

Under Cases $(\text{\emph{III}}_{1})$-$(\text{\emph{III}}_{2})$, the reinsurance-dividend strategy $(u^{\star},D^{\star})$ given by
\begin{align}\label{optimal.stra.thm3.5}
\begin{cases} u_{t}^{\star}=\frac{\mu_{-}U^{u^{\star},D^{\underline{x}^{\star},\overline{x}^{\star}}}_{t}}{\sigma_{-}^2(1-\gamma_{-})}\emph{\textbf{1}}_{\{U^{u^{\star},D^{\underline{x}^{\star},\overline{x}^{\star}}}_{t}\in[0,\,a\wedge x_0^{-})\}}+\emph{\textbf{1}}_{\{U^{u^{\star},D^{\underline{x}^{\star},\overline{x}^{\star}}}_{t}\not\in[0,\,a\wedge x_0^{-})\}},& t\geq 0, \\ D^{\star}_{t}=D^{\underline{x}^{\star},\overline{x}^{\star}}_{t},& t\geq 0,
\end{cases}
\end{align}
is optimal for the auxiliary control problem \eqref{eq:a008.aux.}. The associated value function is
\begin{equation}
V_{\underline{x}^{\star},\overline{x}^{\star}}(x)=
\left\{\begin{array}{ll}
\frac{g(x)}{g^{\prime}(\overline{x}^{\star})},&x\in[0,\overline{x}^{\star}), \\
x-\overline{x}^{\star}+\frac{g(\overline{x}^{\star})}{g^{\prime}(\overline{x}^{\star})}=x-\underline{x}^{\star}-\beta+\frac{g(\underline{x}^{\star})}{g^{\prime}(\overline{x}^{\star})},& x\in[\overline{x}^{\star},\infty),
\end{array}
\right.\nn
\end{equation}
with
\begin{align}\label{optimal.stra.thm3.5}
(\underline{x}^{\star},\overline{x}^{\star})
=((g^{\prime})_{3}^{-1}(g^{\prime}(\phi^{-1}(\beta))),\phi^{-1}(\beta)).
\end{align}
\end{theorem}
\begin{proof}
We only need to give a proof for the scenario of $0=a_{1}=a_{2}<a_{3}$ (see case $\text{{III}}_{2}$), the proof for the scenario of $0=a_{1}<a_{2}<a_{3}$ (see case $\text{{III}}_{1}$) is similar to that of Theorem \ref{thm3.2.ww}. By Lemma \ref{lem3.1.w}, we can decouple the auxiliary optimization problem \eqref{eq:a008.aux.} into problem \eqref{auxiliary.contr.p.} and problem \eqref{point.point.optimization}. By theorem \ref{thm3.1.w} and similar arguments adopted at the proof of Theorem \ref{thm3.2.ww}, the value function of the auxiliary optimization problem \eqref{auxiliary.contr.p.} is given by 
\begin{equation}
V_{\underline{x},\overline{x}}(x)=
\left\{\begin{array}{ll}
\frac{g(x)}{g^{\prime}(\overline{x})},&x\in[0,\overline{x}), \\
\frac{g(\underline{x})}{g^{\prime}(\overline{x})}+x-\underline{x}-\beta,& x\in[\overline{x},\infty),
\end{array}
\right.
\end{equation}
where $g=g_{31}$ is defined in \eqref{3.44.g.}, and $\underline{x}$ and $\overline{x}$ satisfy
\begin{align}
\psi(\underline{x},\overline{x})=\beta.
\end{align}
We next verify that $(\underline{x}^{\star},\overline{x}^{\star})$ given by \eqref{optimal.stra.thm3.5} is the solution to problem \eqref{point.point.optimization}. 

Suppose $(\underline{x},\overline{x})$ with $\underline{x}\in(0,a_3)$ is any solution to $\psi(\underline{x},\overline{x})=\beta$ different from $(\underline{x}^{\star},\overline{x}^{\star})$. Then, we have the following observations.
\begin{itemize}
\item Suppose $\overline{x}\in(0,a_3]$. It follows form the concavity of $g=g_{31}$ on $[0,a_3]$ that 
\begin{eqnarray}
\psi(\underline{x},\overline{x})=\int_{\underline{x}}^{\overline{x}}\left(1-\frac{g^{\prime}(z)}{g^{\prime}(\overline{x})}\right)\mathrm{d}z<0,
\end{eqnarray}
which contradicts the assumption of $\psi(\underline{x},\overline{x})=\beta$. Hence, this case is inadmissible.
\item Suppose $\overline{x}\in(a_3,\infty)$. It follows from the construction of $g=g_{31}$ defined in \eqref{3.44.g.} that 
\begin{align}\label{psi'xy2}
\left\{
\begin{array}{ll}
\frac{\partial}{\partial x}\psi(x,y)=-\left(1-\frac{g^{\prime}(x)}{g^{\prime}(y)}\right)>0,&0\leq x<(g^{\prime})^{-1}_3(g^{\prime}(y))<y<\infty, \\
\frac{\partial}{\partial x}\psi(x,y)=-\left(1-\frac{g^{\prime}(x)}{g^{\prime}(y)}\right)<0,&(g^{\prime})^{-1}_3(g^{\prime}(y))\leq x<y<\infty,\\
\frac{\partial}{\partial y}\psi(x,y)=\int_{x}^{y}\frac{g^{\prime}(z)g^{\prime\prime}(y)}{\left[g^{\prime}(y)\right]^{2}}\mathrm{d}z>0,& 0<x<y<\infty,
\end{array}
\right.
\end{align}
which, combined with $\psi(\underline{x}^{\star},\overline{x}^{\star})=\beta$, implies $\overline{x}>\overline{x}^{\star}$. Using similar arguments adopted at the proof of Theorem \ref{thm3.2.ww}, we obtain $V_{\underline{x}^{\star},\overline{x}^{\star}}(x)\geq V_{\underline{x},\overline{x}}(x)$ for all $x\in(0,\infty)$.
\end{itemize}
In summary, for $(\underline{x},\overline{x})$ with $\underline{x}\in(0,a_3)$ being any solution to $\psi(\underline{x},\overline{x})=\beta$ different from $(\underline{x}^{\star},\overline{x}^{\star})$, we have necessarily $V_{\underline{x}^{\star},\overline{x}^{\star}}(x)\geq V_{\underline{x},\overline{x}}(x)$ for all $x\in(0,\infty)$.

Suppose $(\underline{x},\overline{x})$ with $\underline{x}\in[a_3,\infty)$ is any solution to $\psi(\underline{x},\overline{x})=\beta$ different from $(\underline{x}^{\star},\overline{x}^{\star})$. Then, it should be $\overline{x}\in[a_3,\infty)$.
It follows from \eqref{psi'xy2} that $\psi(\cdot,y)$ obtains its maximizer at $(g^{\prime})^{-1}_3(g^{\prime}(y))$, hence, we have 
\begin{eqnarray}\label{eq:beta<psi}
\beta=\psi(\underline{x},\overline{x})\leq \psi((g^{\prime})^{-1}_3(g^{\prime}(\overline{x})),\overline{x}).
\end{eqnarray}
It is easy to check that the function
$$\phi: x\mapsto \psi((g^{\prime})^{-1}_3(g^{\prime}(x)),x),$$
is increasing on $[a_3,\infty)$, which, combined with \eqref{eq:beta<psi} and $\psi(\underline{x}^{\star},\overline{x}^{\star})=\beta$, implies that $\overline{x}>\overline{x}^{\star}$. Using similar arguments adopted at the proof of Theorem \ref{thm3.2.ww}, we obtain $V_{\underline{x}^{\star},\overline{x}^{\star}}(x)\geq V_{\underline{x},\overline{x}}(x)$ for all $x\in(0,\infty)$. 

In conclusion, $(\underline{x}^{\star},\overline{x}^{\star})$ given by \eqref{optimal.stra.thm3.5} is the solution to problem \eqref{point.point.optimization}. Furthermore, based on Theorem \ref{thm3.1.w} and the construction of $V_{\underline{x},\overline{x}}$ in Section \ref{subsubsec.3.3.1}, the optimal reinsurance strategy for the control problem \eqref{auxiliary.contr.p.} with $(\underline{x},\overline{x})=(\underline{x}^{\star},\overline{x}^{\star})$ is given by \eqref{optimal.stra.thm3.5}. The proof is completed.
\end{proof}

\subsection{Case $\mu_{\pm}>0$}
\label{subsec.3.4}

\subsubsection{Construction of a solution to \eqref{eq:a046}}
\label{subsubsec.3.4.1}
\begin{theorem}
Suppose that $\mu_{-}>0$ and $\mu_{+}>0$. Then, for every constant $C>0$, a non-negative, strictly increasing, continuously differentiable, and piecewise twice continuously differentiable solution to \eqref{eq:a046} is given as follows. If $a\leq x_{1}^{+}\wedge x_{0}^{-}$, then
\begin{align}\label{3.94.g.}
v(x)=&~\left\{
\begin{array}{ll}
Cx^{\gamma_{-}},&x\in[0,a],\\
{C(a_{41}\frac{x}{a}+a_{42})^{\gamma_+}},&x\in [a,x_{1}^{+}),\\
C(a_{43}e^{\theta_{+}(x-x_{1}^{+})}+a_{44}e^{\theta_{-}(x-x_{1}^{+})}),&x\in[x_1^{+},\infty),
\end{array}
\right.
\\
:=&~C\,g_{41}(x),\nn
\end{align}
where
\begin{align}\label{3.67.g.}
\begin{cases}
a_{41}=\dfrac{\gamma_{-}}{\gamma_{+}}a^{\frac{\gamma_{-}}{\gamma_{+}}},\\[1.2ex]
a_{42}=a^{\frac{\gamma_{-}}{\gamma_{+}}}\left(1-\dfrac{\gamma_{-}}{\gamma_{+}}\right),\\[1.2ex]
a_{43}=\dfrac{\left(a_{41}\dfrac{x_{1}^{+}}{a}+a_{42}\right)^{\gamma_{+}-1}\left[\gamma_{+}\dfrac{a_{41}}{a}-\theta_{-}\left(a_{41}\dfrac{x_{1}^{+}}{a}+a_{42}\right)\right]}{\theta_{+}-\theta_{-}},\\[1.2ex]
a_{44}=\dfrac{\left(a_{41}\dfrac{x_{1}^{+}}{a}+a_{42}\right)^{\gamma_{+}-1}\left[\theta_{+}\left(a_{41}\dfrac{x_{1}^{+}}{a}+a_{42}\right)-\gamma_{+}\dfrac{a_{41}}{a}\right]}{\theta_{+}-\theta_{-}}.
\end{cases}
\end{align}
If $x_{1}^{+}<a\leq x_{0}^{-}$, then
\begin{align}
\label{3.95.g.}
v(x)=&~\left\{
\begin{array}{ll}
Cx^{\gamma_{-}},& x\in[0,a], \\
C(a_{31}e^{\theta_{+}(x-a)}+a_{32}e^{\theta_{-}(x-a)}),& x\in [a,\infty),
\end{array}
\right.
\\
:=&~C\,g_{42}(x).\nn
\end{align}
If $x_{0}^{-}<a\leq x_{2}^{+}$, then
\begin{align}\label{3.114.g.}
v(x)=&~\left\{
\begin{array}{ll}
Cx^{\gamma_{-}},&x\in[0,x_0^{-}],\\
C(a_{33}e^{\delta_{+}(x-x_0^{-})}+a_{34}e^{\delta_{-}(x-x_0^{-})}),&x\in [x_0^{-},a],\\
{C\left(a_{45}\frac{x}{a}+a_{46}\right)^{\gamma_{+}}},&x\in [a,x_{2}^{+}),\\
C(a_{47}e^{\theta_{+}(x-x_{2}^{+})}+a_{48}e^{\theta_{-}(x-x_{2}^{+})}),&x\in [x_{2}^{+},\infty),
\end{array}
\right.
\\
:=&~C\,g_{43}(x),\nn
\end{align}
where
\begin{align}\label{3.101.g.}
\begin{cases}
a_{45}=\dfrac{a}{\gamma_{+}}\left(\delta_{+}a_{33}e^{\delta_{+}(a-x_{0}^{-})}+\delta_{-}a_{34}e^{\delta_{-}(a-x_{0}^{-})}\right)
\left(a_{33}e^{\delta_{+}(a-x_{0}^{-})}+a_{34}e^{\delta_{-}(a-x_{0}^{-})}\right)^{\frac{1}{\gamma_{+}}-1},\\[1.2ex]
a_{46}=\left(a_{33}e^{\delta_{+}(a-x_{0}^{-})}+a_{34}e^{\delta_{-}(a-x_{0}^{-})}\right)^{\frac{1}{\gamma_{+}}}-a_{45},\\[1.2ex]
a_{47}=\dfrac{\left(a_{45}\dfrac{x_{2}^{+}}{a}+a_{46}\right)^{\gamma_{+}-1}\left[\gamma_{+}\dfrac{a_{45}}{a}-\theta_{-}\left(a_{45}\dfrac{x_{2}^{+}}{a}+a_{46}\right)\right]}{\theta_{+}-\theta_{-}},\\[1.2ex]
a_{48}=\dfrac{\left(a_{45}\dfrac{x_{2}^{+}}{a}+a_{46}\right)^{\gamma_{+}-1}\left[\theta_{+}\left(a_{45}\dfrac{x_{2}^{+}}{a}+a_{46}\right)-\gamma_{+}\dfrac{a_{45}}{a}\right]}{\theta_{+}-\theta_{-}}.
\end{cases}
\end{align}
If $a>x_{0}^{-}\vee x_{2}^{+}$, then
\begin{align}
\label{3.tt1.g.}
v(x)=&~\left\{
\begin{array}{ll}
Cx^{\gamma_{-}},& x\in[0,x_0^{-}], \\
C(a_{33}e^{\delta_{+}(x-x_0^{-})}+a_{34}e^{\delta_{-}(x-x_0^{-})}),& x\in [x_0^{-},a], \\
C(a_{35}e^{\theta_{+}(x-a)}+a_{36}e^{\theta_{-}(x-a)}),& x\in [a,\infty),
\end{array}
\right.
\\
:=&~C\,g_{44}(x).\nn
\end{align}
Here, $a_{31},a_{32}$ are given by \eqref{3.40.g.}, and $a_{33},a_{34},a_{35},a_{36}$ are given by \eqref{3.48.g.}.
Moreover, the maximizing control can be written in the following piecewise form:
\begin{align*}
u^{\star}(x)=\begin{cases}
\dfrac{\mu_{-}x}{\sigma_{-}^{2}(1-\gamma_{-})},&x\in[0,a\wedge x_{0}^{-}],\\[1.2ex]
\dfrac{\mu_{+}\left(x+\dfrac{a_{42}}{a_{41}}a\right)}{\sigma_{+}^{2}(1-\gamma_{+})},&x\in(a,x_{1}^{+}),\quad a\leq x_{1}^{+}\wedge x_{0}^{-},\\[1.2ex]
\dfrac{\mu_{+}\left(x+\dfrac{a_{46}}{a_{45}}a\right)}{\sigma_{+}^{2}(1-\gamma_{+})},&x\in(a,x_{2}^{+}),\quad x_{0}^{-}<a\leq x_{2}^{+},\\[1.2ex]
1,&\text{otherwise.}
\end{cases}
\end{align*}
\end{theorem}
\begin{proof}
Recall that $x^+_1:=\frac{\sigma^2_+(1-\gamma_+)}{\mu_+}+\left(1-\frac{\gamma_+}{\gamma_-}\right)a$. We begin by assuming that $a\leq x_1^{+}\wedge x_0^{-}$.
Based on the previous analysis, one knows that, on the interval $[0,a]$, the solution $v$ to \eqref{eq:a046} (or, equivalently, \eqref{eq:ee01}) is given by \eqref{3.32.g.}.
On the interval $[a,x_1^{+})$, we conjecture that the solution $v$ to \eqref{eq:a046} satisfies
\begin{align}
\label{3.64.g.}
\left\{
\begin{array}{ll}
v^{\prime\prime}(x)<0,& x\in [a,x_{1}^{+}), \\
-\frac{\mu_{+}v^{\prime}(x)}{\sigma_{+}^2v^{\prime\prime}(x)}< 1,& x\in [a,x_{1}^{+}).
\end{array}
\right.
\end{align}
Under \eqref{3.64.g.}, on the interval $[a,x_{1}^{+})$, \eqref{eq:a046} (or, equivalently, \eqref{eq:ee03}) is reduced to
\begin{align}
-\frac{\mu_{+}^{2}\left[v^{\prime}(x)\right]^{2}}{2\sigma_{+}^{2}v^{\prime\prime}(x)}-qv(x)=0,\quad x\in [a,x_{1}^{+}),
\nn
\end{align}
whose solution is given by
\begin{equation}\label{3.65.g.}
v(x)={C\left(a_{41}\frac{x}{a}+a_{42}\right)^{\gamma_+}},\quad x\in [a,x_{1}^{+}),
\end{equation}
where $a_{41},\,a_{42}$ are constants to be determined. Using the smooth fit condition yields
\begin{equation}\label{3.66.g.}
\begin{cases}
a^{\gamma_{-}}={(a_{41}+a_{42})^{\gamma_+}},\\
\gamma_{-}a^{(\gamma_{-}-1)}={\gamma_+\frac{a_{41}}{a}(a_{41}+a_{42})^{\gamma_+-1}}.
\end{cases}
\end{equation}
Solving \eqref{3.66.g.} gives $a_{41}$ and $a_{42}$ are given by \eqref{3.67.g.}.
By \eqref{3.65.g.} and \eqref{3.67.g.}, we verify that
\begin{align}\label{3.68.g.}
\begin{cases}
v^{\prime\prime}(x)=C{\gamma_+(\gamma_+-1)\left(a_{41}\frac{x}{a}+a_{42}\right)^{\gamma_+-2}\frac{a_{41}^2}{a^2}<0},\quad &x\in [a,x_{1}^{+}),\\
-\frac{\mu_{+}v^{\prime}(x)}{\sigma_{+}^2v^{\prime\prime}(x)}=
{\frac{\mu_+\left(x+\frac{a_{42}}{a_{41}}a\right)}{\sigma^2_+(1-\gamma_+)}<1},\quad &x\in [a,x_{1}^{+}),
\end{cases}
\end{align}
confirming that condition \eqref{3.64.g.} is indeed satisfied.
On the interval $[x_1^{+},\infty)$, \eqref{eq:a046} shall be simplified to
\begin{align}
\frac{1}{2}\sigma_{+}^2v^{\prime\prime}(x)+\mu_{+}v^{\prime}(x)-qv(x)=0,\quad x\in [x_{1}^{+},\infty),
\nn
\end{align}
whose solution is
\begin{align}
\label{3.71.g.}
v(x)=C(a_{43}e^{\theta_{+}(x-x_{1}^{+})}+a_{44}e^{\theta_{-}(x-x_{1}^{+})}),\quad x\in [x_{1}^{+},\infty),
\end{align}
where $a_{43}, \,a_{44}$ are constants to be determined. Imposing the smooth-pasting conditions at $x=x_1^{+}$ leads to
\begin{align}
\label{3.72.g.}
\left\{
\begin{array}{ll}
{\left(a_{41}\frac{x^+_1}{a}+a_{42}\right)^{\gamma_+}}=a_{43}+a_{44},& \\
{\gamma_+\frac{a_{41}}{a}\left(a_{41}\frac{x^+_1}{a}+a_{42}\right)^{\gamma_+-1}}=a_{43}\theta_{+}+a_{44}\theta_{-}.&
\end{array}
\right.
\end{align}
Here, we also note that the second equality of \eqref{3.72.g.} implies that
\begin{align}
v^{\prime}(x)>0,\quad x\in [x_{1}^{+},\infty).\nn
\end{align}
Solving \eqref{3.72.g.} yields that $a_{43}$ and $a_{44}$ are given by \eqref{3.67.g.}.
{By $a_{41}>0$ and $\theta_-<0$, we have $a_{43}>0$ and $a_{43}+a_{44}>0$.}
Therefore, one deduces that
\begin{itemize}
\item If $a_{44}\geq 0$, then, by \eqref{3.71.g.}, it is obvious that $v^{\prime\prime}(x)> 0$ for any $x\in[x_1^{+},\infty)$.

\item If $a_{44}<0$, then
\begin{align}
v^{\prime\prime}(x)=C(a_{43}\theta_{+}^{2}e^{\theta_{+}(x-x_1^{+})}+a_{44}\theta_{-}^{2}e^{\theta_{-}(x-x_1^{+})}),\quad x\in [x_1^{+},\infty),\nn
\end{align}
is strictly increasing on $[x_1^{+},\infty)$. At $x=x_1^{+}$, we have $v^{\prime\prime}(x_1^{+})
=C(a_{43}\theta_{+}^{2}+a_{44}\theta_{-}^{2})$. {We claim that $a_{43}\theta_{+}^{2}+a_{44}\theta_{-}^{2}<0$. Indeed, by \eqref{3.67.g.}, we have
\begin{eqnarray}
a_{43}\theta_{+}^{2}+a_{44}\theta_{-}^{2}
\hspace{-0.3cm}&=&\hspace{-0.3cm}
\frac{\left(a_{41}\frac{x^+_1}{a}+a_{42}\right)^{\gamma_+-1}\left[\gamma_+\frac{a_{41}}{a}-\theta_-\left(a_{41}\frac{x^+_1}{a}+a_{42}\right)\right]}{\theta_+-\theta_-}\theta^2_+
\nn\\
\hspace{-0.3cm}&&\hspace{-0.3cm}
+\frac{\left(a_{41}\frac{x^+_1}{a}+a_{42}\right)^{\gamma_+-1}\left[\theta_+\left(a_{41}\frac{x^+_1}{a}+a_{42}\right)-\gamma_+\frac{a_{41}}{a}\right]}{\theta_+-\theta_-}\theta^2_-
\nn\\
\hspace{-0.3cm}&=&\hspace{-0.3cm}
\left(a_{41}\frac{x^+_1}{a}+a_{42}\right)^{\gamma_+-1}\left[\gamma_+\frac{a_{41}}{a}(\theta_++\theta_-)-\left(a_{41}\frac{x^+_1}{a}+a_{42}\right)\theta_-\theta_+\right]
\nn\\
\hspace{-0.3cm}&=&\hspace{-0.3cm}
-\left(a_{41}\frac{x^+_1}{a}+a_{42}\right)^{\gamma_+-1}\frac{\mu_+\gamma_-}{\sigma^2_+}a^{\frac{\gamma_-}{\gamma_+}-1}<0.
\end{eqnarray}
} 
Noting that $a_{43}\theta_{+}^{2}+a_{44}\theta_{-}^{2}<0$ occurs if and only if $\ln{\frac{-a_{44}\theta_{-}^2}{a_{43}\theta_{+}^2}}>0$,
we have
\begin{align}
\label{3.76.g.}
\left\{
\begin{array}{ll}
v^{\prime\prime}(x)>0,& x\in (x_1^{+}+\frac{\ln{\frac{-a_{44}\theta_{-}^2}{a_{43}\theta_{+}^2}}}{\theta_{+}-\theta_{-}},\infty), \\
v^{\prime\prime}(x)<0,& x\in [x_1^{+}, x_1^{+}+\frac{\ln{\frac{-a_{44}\theta_{-}^2}{a_{43}\theta_{+}^2}}}{\theta_{+}-\theta_{-}}).
\end{array}
\right.
\end{align}

By \eqref{3.71.g.}, one gets
\begin{align}
\label{3.77.g.}
&\sigma_{+}^2v^{\prime\prime}(x)+\mu_{+}v^{\prime}(x)
\nn\\
=&~ C\left(a_{43} \left(\sigma_{+}^2\theta_{+}^{2}+\mu_{+}\theta_{+}\right)
e^{\theta_{+}(x-x_1^{+})}
+ a_{44}\left(\sigma_{+}^2\theta_{-}^{2}+\mu_{+}\theta_{-}\right)
e^{\theta_{-}(x-x_1^{+})}\right)
\nn\\
=&~
C\left(a_{43} \left(2q-\mu_{+}\theta_{+}\right)
e^{\theta_{+}(x-x_1^{+})}
+ a_{44}\left(2q-\mu_{+}\theta_{-}\right)
e^{\theta_{-}(x-x_1^{+})}\right), \quad x\in [x_1^{+},\infty).
\end{align}
Based on the previously proven $2q-\mu_{+}\theta_{+}>0$ and $a_{43}>0$, we know that the function on the RHS of \eqref{3.77.g.} is strictly increasing in $x$, since the dominant exponential term (with $\theta_{+}>\theta_{-}$) grows faster. In this case, the minimum occurs at $x=x_{1}^{+}$:
\begin{align}
\label{3.78.g.}
&\sigma_{+}^2v^{\prime\prime}(x_{1}^{+}+)+\mu_{+}v^{\prime}(x_{1}^{+}+)
\nn\\
=&~
2\left(qv(x_{1}^{+}+)-\mu_{+}v^{\prime}(x_{1}^{+}+)-\frac{1}{2}\sigma_{+}^2v^{\prime\prime}(x_{1}^{+}+)\right)
+\sigma_{+}^2v^{\prime\prime}(x_{1}^{+}+)+\mu_{+}v^{\prime}(x_{1}^{+}+)
\nn\\
=&~2qv(x_{1}^{+}+)-\mu_{+}v^{\prime}(x_{1}^{+}+)
\nn\\
=&~2qv(x_{1}^{+}-)-\mu_{+}v^{\prime}(x_{1}^{+}-)
\nn\\
=&~{C\left(a_{41}\frac{x^+_1}{a}+a_{42}\right)^{\gamma_+-1}\left(2qa_{41}\frac{x^+_1}{a}+2qa_{42}-\mu_+\gamma_+\frac{a_{41}}{a}\right)} 
\nn\\
=&~0,
\end{align}
where we have used \eqref{3.67.g.} and the definition of $x^+_1$.
{
We further verify the second-order smooth fit at $x_1^{+}$. By \eqref{3.68.g.} and the definition of $x_1^{+}$,
\begin{align}
-\frac{\mu_{+}v^{\prime}(x_1^{+}-)}{\sigma_{+}^{2}v^{\prime\prime}(x_1^{+}-)}=1,
\end{align}
so that
\begin{align}
\sigma_{+}^{2}v^{\prime\prime}(x_1^{+}-)+\mu_{+}v^{\prime}(x_1^{+}-)=0.
\label{eq:C2.x1.plus.left}
\end{align}
Equation \eqref{3.78.g.} yields
\begin{align}
\sigma_{+}^{2}v^{\prime\prime}(x_1^{+}+)+\mu_{+}v^{\prime}(x_1^{+}+)=0.
\label{eq:C2.x1.plus.right}
\end{align}
Since \eqref{3.72.g.} gives $v^{\prime}(x_1^{+}-)=v^{\prime}(x_1^{+}+)$, we conclude that
\begin{align}
v^{\prime\prime}(x_1^{+}-)
=v^{\prime\prime}(x_1^{+}+).
\label{eq:C2.x1.plus}
\end{align}
Thus, $v$ is twice continuously differentiable at $x_1^{+}$.}
By \eqref{3.78.g.} and the strict increasing property of $\sigma_{+}^2v^{\prime\prime}+\mu_{+}v^{\prime}$ on $[x_{1}^{+},\infty)$, we have
\begin{align}\label{3.7801.g.}
\sigma_{+}^2v^{\prime\prime}(x)+\mu_{+}v^{\prime}(x)\geq 0, \quad x\in [x_{1}^{+},\infty).
\end{align}
\end{itemize}
Therefore, we conclude that
\begin{itemize}
\item[$\bullet$] {Suppose $a_{44}\geq 0$.}
Then
\begin{align}
v^{\prime\prime}(x)>0, \quad x\in (x_{1}^{+},\infty),\nn
\end{align}
confirming that $v$ given by \eqref{3.71.g.} and \eqref{3.67.g.} is indeed a solution to \eqref{eq:ee03.{2}} on $[x_{1}^{+},\infty)$.
\item[$\bullet$] {
Suppose $a_{44}<0$.}
By \eqref{3.76.g.} and \eqref{3.7801.g.}, we have
\begin{align}
\label{3.7611.g.}
\left\{
\begin{array}{ll}
v^{\prime\prime}(x)>0,& x\in (x_1^{+}+\frac{\ln{\frac{-a_{44}\theta_{-}^2}{a_{43}\theta_{+}^2}}}{\theta_{+}-\theta_{-}},\infty), \\
v^{\prime\prime}(x)<0,& x\in [x_1^{+}, x_1^{+}+\frac{\ln{\frac{-a_{44}\theta_{-}^2}{a_{43}\theta_{+}^2}}}{\theta_{+}-\theta_{-}}),
\\
\sigma_{+}^2 v^{\prime\prime}(x) +\mu_{+}v^{\prime}(x)\geq 0,& x\in [x_1^{+}, x_1^{+}+\frac{\ln{\frac{-a_{44}\theta_{-}^2}{a_{43}\theta_{+}^2}}}{\theta_{+}-\theta_{-}}),
\end{array}
\right.
\end{align}
confirming that $v$ given by \eqref{3.71.g.} and \eqref{3.67.g.} is indeed a solution to \eqref{eq:ee03.{2}} on $[x_{1}^{+},\infty)$.
\end{itemize}
In summary, when $\mu_{\pm}>0$ and {$a\leq x^+_1\wedge x_0^{-}$},
the solution to \eqref{eq:a046} takes the form \eqref{3.94.g.}.

We now proceed under the assumption that {$x^+_1<a \leq x_0^{-}$.}
Based on the previous analysis, the solution $v$ of \eqref{eq:a046} on $[0,a]$ should be given by \eqref{3.32.g.}.
On the interval $[a,\infty)$, based on the previous analysis, the solution $v$ of \eqref{eq:a046} on $[a,\infty)$ should be given by \eqref{3.38.g.}. Since $a_{31}>0$ and $a_{31}+a_{32}>0$ (see \eqref{3.40.g.} and \eqref{3.41.g.}), one deduces that

\begin{itemize}
\item If $a_{32}\geq 0$, then, by \eqref{3.38.g.}, it is obvious that $v^{\prime\prime}(x)> 0$ for any $x\in[a,\infty)$.

\item If $a_{32}<0$, then
\begin{align}
v^{\prime\prime}(x)=C(a_{31}\theta_{+}^{2}e^{\theta_{+}(x-a)}+a_{32}\theta_{-}^{2}e^{\theta_{-}(x-a)}),\quad x\in [a,\infty),\nn
\end{align}
is strictly increasing on $[a,\infty)$. At $x=a$, we have $v^{\prime\prime}(a)=C(a_{31}\theta_{+}^{2}+a_{32}\theta_{-}^{2})$. If $a_{31}\theta_{+}^{2}+a_{32}\theta_{-}^{2}\ge0$, we have $v^{\prime\prime}(x)>0$ for all $x\in[a,\infty)$. If $a_{31}\theta_{+}^{2}+a_{32}\theta_{-}^{2}<0$, which occurs if and only if $\ln{\frac{-a_{32}\theta_{-}^2}{a_{31}\theta_{+}^2}}>0$,
then
\begin{align}
\label{3.1.wwc}
\left\{
\begin{array}{ll}
v^{\prime\prime}(x)>0,& x\in (a+\frac{\ln{\frac{-a_{32}\theta_{-}^2}{a_{31}\theta_{+}^2}}}{\theta_{+}-\theta_{-}},\infty), \\
v^{\prime\prime}(x)<0,& x\in [a, a+\frac{\ln{\frac{-a_{32}\theta_{-}^2}{a_{31}\theta_{+}^2}}}{\theta_{+}-\theta_{-}}).
\end{array}
\right.
\end{align}
By \eqref{3.38.g.}, one gets
\begin{align}
\label{3.2.wwc}
&\sigma_{+}^2v^{\prime\prime}(x)+\mu_{+}v^{\prime}(x)
\nn\\
=&~ C\left(a_{31} \left(\sigma_{+}^2\theta_{+}^{2}+\mu_{+}\theta_{+}\right)
e^{\theta_{+}(x-a)}
+ a_{32}\left(\sigma_{+}^2\theta_{-}^{2}+\mu_{+}\theta_{-}\right)
e^{\theta_{-}(x-a)}\right)
\nn\\
=&~
C\left(a_{31} \left(2q-\mu_{+}\theta_{+}\right)
e^{\theta_{+}(x-a)}
+ a_{32}\left(2q-\mu_{+}\theta_{-}\right)
e^{\theta_{-}(x-a)}\right), \quad x\in [a,\infty).
\end{align}
Based on the previously proven $2q-\mu_{+}\theta_{+}>0$ and $a_{31}>0$ and the fact of $2q-\mu_{+}\theta_{-}>0$, implies that the function on the RHS of \eqref{3.2.wwc} is strictly increasing in $x$, since the dominant exponential term (with $\theta_{+}>\theta_{-}$) grows faster. In this case, the minimum occurs at $x=a$:
\begin{align}
\label{3.3.wwc}
&\sigma_{+}^2v^{\prime\prime}(a+)+\mu_{+}v^{\prime}(a+)
\nn\\
=&~
2\left(qv(a+)-\mu_{+}v^{\prime}(a+)-\frac{1}{2}\sigma_{+}^2v^{\prime\prime}(a+)\right)
+\sigma_{+}^2v^{\prime\prime}(a+)+\mu_{+}v^{\prime}(a+)
\nn\\
=&~2qv(a+)-\mu_{+}v^{\prime}(a+)
\nn\\
=&~2qv(a-)-\mu_{+}v^{\prime}(a-)
\nn\\
=&~Ca^{\gamma_{-}-1}(2qa-\mu_{+}\gamma_{-})\nn\\
=&~Ca^{\gamma_--1}(a-x^+_1)\frac{2q\gamma_-}{\gamma_+}
\geq 0.
\end{align}
\end{itemize}
Therefore, we conclude that
\begin{itemize}
\item[$\bullet$] Suppose either $a_{32}\geq 0$ or $
a_{31}\theta_{+}^{2}+a_{32}\theta_{-}^{2}\geq 0$. Then,
\begin{align}
v^{\prime\prime}(x)>0, \quad x\in (a,\infty),\nn
\end{align}
confirming that $v$ given by \eqref{3.38.g.} and \eqref{3.40.g.} is indeed a solution to \eqref{eq:a046} on $[a,\infty)$.

\item[$\bullet$] Suppose $a_{32}<0$, $
a_{31}\theta_{+}^{2}+a_{32}\theta_{-}^{2}<0$. 
By \eqref{3.1.wwc} and \eqref{3.3.wwc}, we have
\begin{align}
\label{3.1.wwc112}
\left\{
\begin{array}{ll}
v^{\prime\prime}(x)>0,& x\in (a+\frac{\ln{\frac{-a_{32}\theta_{-}^2}{a_{31}\theta_{+}^2}}}{\theta_{+}-\theta_{-}},\infty), \\
v^{\prime\prime}(x)<0,& x\in [a, a+\frac{\ln{\frac{-a_{32}\theta_{-}^2}{a_{31}\theta_{+}^2}}}{\theta_{+}-\theta_{-}}),
\\
\sigma_{+}^2 v^{\prime\prime}(x) +\mu_{+}v^{\prime}(x)\geq 0,& x\in [a, a+\frac{\ln{\frac{-a_{32}\theta_{-}^2}{a_{31}\theta_{+}^2}}}{\theta_{+}-\theta_{-}}),
\end{array}
\right.
\end{align}
confirming that $v$ given by \eqref{3.38.g.} and \eqref{3.40.g.} is indeed a solution to \eqref{eq:a046} on $[a,\infty)$.
\end{itemize}
In summary, when $\mu_{\pm}>0$ and $x_1^{+}<a\leq x_0^{-}$, 
the solution to \eqref{eq:a046} is given by \eqref{3.95.g.}.

Define $x^+_2=\frac{\sigma^2_+(1-\gamma_+)}{\mu_+}-\frac{a_{46}}{a_{45}}a$.
We now proceed under the assumption that {$x_0^{-}< a \leq x_2^{+}$}. Based on the previous analysis, one knows that the solution $v$ to \eqref{eq:a046} on $[0,x_0^{-}]$ should be given by \eqref{3.45.g.} and the solution $v$ of \eqref{eq:a046} on $[x_0^{-},a]$ should be given by \eqref{3.46.g.}. By \eqref{3.47.g.} and \eqref{eq:C2.x0.minus}, these two pieces are twice continuously differentiable across $x_0^{-}$.
On the interval $[a,x_2^{+})$, we assume that the solution $v$ to \eqref{eq:a046} satisfies
\begin{align}
\label{3.98.g.}
\left\{
\begin{array}{ll}
v^{\prime\prime}(x)<0,& x\in [a,x_{2}^{+}), \\
-\frac{\mu_{+}v^{\prime}(x)}{\sigma_{+}^2v^{\prime\prime}(x)}< 1,& x\in [a,x_{2}^{+}).
\end{array}
\right.
\end{align}
Under \eqref{3.98.g.}, on the interval $[a,x_{2}^{+})$, \eqref{eq:a046} is reduced to
\begin{align}
-\frac{\mu_{+}^{2}\left[v^{\prime}(x)\right]^{2}}{2\sigma_{+}^{2}v^{\prime\prime}(x)}-qv(x)=0,\quad x\in [a,x_{2}^{+}),
\nn
\end{align}
whose solution is given by
\begin{equation}\label{3.99.g.}
v(x)=C\left(a_{45}\frac{x}{a}+a_{46}\right)^{\gamma_{+}},\quad x\in [a,x_{2}^{+}),
\end{equation}
where $a_{45},\,a_{46}$ are unknown constants to be determined. Using the smooth fit condition yields
\begin{equation}\label{3.100.g.}
\begin{cases}
a_{33}e^{\delta_+(a-x^-_0)}+a_{34}e^{\delta_-(a-x^-_0)}={(a_{45}+a_{46})^{\gamma_+}},\\
a_{33}\delta_+e^{\delta_+(a-x^-_0)}+a_{34}\delta_-e^{\delta_-(a-x^-_0)}={\gamma_+(a_{45}+a_{46})^{\gamma_+-1}\frac{a_{45}}{a}},\\
\end{cases}
\end{equation}
solving \eqref{3.100.g.} gives $a_{45}$ and $a_{46}$ are given by \eqref{3.101.g.}. Note that
\begin{align}
a_{45}&=\frac{a}{\gamma_{+}}\left(\delta_{+}a_{33}e^{\delta_{+}(a-x_{0}^{-})}+\delta_{-}a_{34}e^{\delta_{-}(a-x_{0}^{-})}\right)
\left(a_{33}e^{\delta_{+}(a-x_{0}^{-})}+a_{34}e^{\delta_{-}(a-x_{0}^{-})}\right)^{\frac{1}{\gamma_{+}}-1}>0,\nn
\end{align}
since
\begin{align}
&\delta_{+}a_{33}e^{\delta_{+}(a-x_{0}^{-})}+\delta_{-}a_{34}e^{\delta_{-}(a-x_{0}^{-})}=\frac{v^{\prime}(a-)}{C}>0,\nn\\
&a_{33}e^{\delta_{+}(a-x_{0}^{-})}+a_{34}e^{\delta_{-}(a-x_{0}^{-})}=\frac{v(a-)}{C}>0.
\end{align}
By \eqref{3.99.g.} and \eqref{3.101.g.}, we verify that
\begin{align}\label{3.102.g.}
\begin{cases}
v^{\prime\prime}(x)={C\gamma_+(\gamma_+-1)\left(a_{45}\frac{x}{a}+a_{46}\right)^{\gamma_+-2}\frac{a_{45}^2}{a^2}<0,}
\quad &x\in [a,x_{2}^{+}),\\
-\frac{\mu_{+}v^{\prime}(x)}{\sigma_{+}^2v^{\prime\prime}(x)}={\frac{\mu_+\left(x+\frac{a_{46}}{a_{45}}a\right)}{\sigma^2_+(1-\gamma_+)}<1,}
\quad &x\in [a,x_{2}^{+}),\\
\end{cases}
\end{align}
confirming that condition \eqref{3.98.g.} is indeed satisfied.
On the interval $[x_2^{+},\infty)$, \eqref{eq:ee03.{2}} shall be simplified to
\begin{align}
\frac{1}{2}\sigma_{+}^2v^{\prime\prime}(x)+\mu_{+}v^{\prime}(x)-qv(x)=0,\quad x\in [x_{2}^{+},\infty),
\nn
\end{align}
whose solution is
\begin{align}
\label{3.105.g.}
v(x)=C(a_{47}e^{\theta_{+}(x-x_{2}^{+})}+a_{48}e^{\theta_{-}(x-x_{2}^{+})}),\quad x\in [x_{2}^{+},\infty),
\end{align}
where $a_{47}, \,a_{48}$ are unknown constants to be determined. Using the smooth fit condition, one gets
\begin{align}
\label{3.106.g.}
\left\{
\begin{array}{ll}
{\left(a_{45} \frac{x_{2}^{+}}{a}+a_{46}\right)^{\gamma_{+}}}=a_{47}+a_{48},& \\
{\gamma_+\frac{a_{45}}{a}\left(a_{45}\frac{x_{2}^{+}}{a}+a_{46}\right)^{\gamma_{+}-1}}=a_{47}\theta_{+}+a_{48}\theta_{-}.&
\end{array}
\right.
\end{align}
Here, we note that the second equality of \eqref{3.106.g.} implies that
\begin{align}
v^{\prime}(x)>0,\quad x\in [x_{2}^{+},\infty).\nn
\end{align}
Solving \eqref{3.106.g.} yields that $a_{47}$ and $a_{48}$ are given by \eqref{3.101.g.}.
By $a_{45}>0$ and $\theta_-<0$, we have $a_{47}>0$ and $a_{47}+a_{48}>0.$ 
Therefore, one deduces that
\begin{itemize}
\item If $a_{48}\geq 0$, then, by \eqref{3.105.g.}, it is obvious that $v^{\prime\prime}(x)> 0$ for any $x\in[x_2^{+},\infty)$.

\item If $a_{48}<0$, then
\begin{align}
v^{\prime\prime}(x)=C(a_{47}\theta_{+}^{2}e^{\theta_{+}(x-x_2^{+})}+a_{48}\theta_{-}^{2}e^{\theta_{-}(x-x_2^{+})}),\quad x\in [x_2^{+},\infty),\nn
\end{align}
is strictly increasing on $[x_2^{+},\infty)$. At $x=x_2^{+}$, we have $v^{\prime\prime}(x_2^{+})
=C(a_{47}\theta_{+}^{2}+a_{48}\theta_{-}^{2})$. 
We claim that $a_{47}\theta_{+}^{2}+a_{48}\theta_{-}^{2}<0.$ Indeed, by \eqref{3.101.g.}, we have
\begin{eqnarray}
a_{47}\theta_{+}^{2}+a_{48}\theta_{-}^{2}
\hspace{-0.3cm}&=&\hspace{-0.3cm}
\frac{\left(a_{45}\frac{x^+_2}{a}+a_{46}\right)^{\gamma_+-1}\left[\gamma_+\frac{a_{45}}{a}-\theta_-\left(a_{45}\frac{x^+_2}{a}+a_{46}\right)\right]}{\theta_+-\theta_-}\theta^2_+
\nn\\
\hspace{-0.3cm}&&\hspace{-0.3cm}
+\frac{\left(a_{45}\frac{x^+_2}{a}+a_{46}\right)^{\gamma_+-1}\left[\theta_+\left(a_{45}\frac{x^+_2}{a}+a_{46}\right)-\gamma_+\frac{a_{45}}{a}\right]}{\theta_+-\theta_-}\theta^2_-
\nn\\
\hspace{-0.3cm}&=&\hspace{-0.3cm}
\left(a_{45}\frac{x^+_2}{a}+a_{46}\right)^{\gamma_+-1}\left[\gamma_+\frac{a_{45}}{a}(\theta_++\theta_-)-\theta_+\theta_-\left(a_{45}\frac{x^+_2}{a}+a_{46}\right)\right]
\nn\\
\hspace{-0.3cm}&=&\hspace{-0.3cm}
-\frac{2q}{\sigma^2_+}\left(a_{45}\frac{x^+_2}{a}+a_{46}\right)^{\gamma_+}<0. 
\end{eqnarray}
Noting that $a_{47}\theta_{+}^{2}+a_{48}\theta_{-}^{2}<0$ occurs if and only if $\ln{\frac{-a_{48}\theta_{-}^2}{a_{47}\theta_{+}^2}}> 0$,
we have
\begin{align}
\label{3.110.g.}
\left\{
\begin{array}{ll}
v^{\prime\prime}(x)>0,& x\in (x_2^{+}+\frac{\ln{\frac{-a_{48}\theta_{-}^2}{a_{47}\theta_{+}^2}}}{\theta_{+}-\theta_{-}},\infty), \\
v^{\prime\prime}(x)<0,& x\in [x_2^{+}, x_2^{+}+\frac{\ln{\frac{-a_{48}\theta_{-}^2}{a_{47}\theta_{+}^2}}}{\theta_{+}-\theta_{-}}).
\end{array}
\right.
\end{align}

By \eqref{3.105.g.}, one gets
\begin{align}
\label{3.111.g.}
&\sigma_{+}^2v^{\prime\prime}(x)+\mu_{+}v^{\prime}(x)
\nn\\
=&~ C\left(a_{47} \left(\sigma_{+}^2\theta_{+}^{2}+\mu_{+}\theta_{+}\right)
e^{\theta_{+}(x-x_2^{+})}
+ a_{48}\left(\sigma_{+}^2\theta_{-}^{2}+\mu_{+}\theta_{-}\right)
e^{\theta_{-}(x-x_2^{+})}\right)
\nn\\
=&~
C\left(a_{47} \left(2q-\mu_{+}\theta_{+}\right)
e^{\theta_{+}(x-x_2^{+})}
+ a_{48}\left(2q-\mu_{+}\theta_{-}\right)
e^{\theta_{-}(x-x_2^{+})}\right), \quad x\in [x_2^{+},\infty).
\end{align}
Based on the previously proven $2q-\mu_{+}\theta_{+}>0$ and $a_{47}>0$, the function on the RHS of \eqref{3.111.g.} is strictly increasing in $x$, since the dominant exponential term (with $\theta_{+}>\theta_{-}$) grows faster. In this case, the minimum occurs at $x=x_{2}^{+}$:
\begin{align}
\label{3.112.g.}
&\sigma_{+}^2v^{\prime\prime}(x_{2}^{+}+)+\mu_{+}v^{\prime}(x_{2}^{+}+)
\nn\\
=&~
2\left(qv(x_{2}^{+}+)-\mu_{+}v^{\prime}(x_{2}^{+}+)-\frac{1}{2}\sigma_{+}^2v^{\prime\prime}(x_{2}^{+}+)\right)
+\sigma_{+}^2v^{\prime\prime}(x_{2}^{+}+)+\mu_{+}v^{\prime}(x_{2}^{+}+)
\nn\\
=&~2qv(x_{2}^{+}+)-\mu_{+}v^{\prime}(x_{2}^{+}+)
\nn\\
=&~2qv(x_{2}^{+}-)-\mu_{+}v^{\prime}(x_{2}^{+}-)
\nn\\
=&~{C\left(a_{45}\frac{x^+_2}{a}+a_{46}\right)^{\gamma_+-1}\left[2q\left(a_{45}\frac{x^+_2}{a}+a_{46}\right)-\mu_+\gamma_+\frac{a_{45}}{a}\right]}
\nn\\
=&~0.
\end{align}
We further verify the second-order smooth fit at $x_2^{+}$. By \eqref{3.102.g.} and the definition of $x_2^{+}$,
\begin{align}
-\frac{\mu_{+}v^{\prime}(x_2^{+}-)}{\sigma_{+}^{2}v^{\prime\prime}(x_2^{+}-)}=1,
\end{align}
and hence
\begin{align}
\sigma_{+}^{2}v^{\prime\prime}(x_2^{+}-)+\mu_{+}v^{\prime}(x_2^{+}-)=0.
\label{eq:C2.x2.plus.left}
\end{align}
Equation \eqref{3.112.g.} gives
\begin{align}
\sigma_{+}^{2}v^{\prime\prime}(x_2^{+}+)+\mu_{+}v^{\prime}(x_2^{+}+)=0.
\label{eq:C2.x2.plus.right}
\end{align}
Since \eqref{3.106.g.} implies $v^{\prime}(x_2^{+}-)=v^{\prime}(x_2^{+}+)$, it follows that
\begin{align}
v^{\prime\prime}(x_2^{+}-)
=v^{\prime\prime}(x_2^{+}+).
\label{eq:C2.x2.plus}
\end{align}
Thus, $v$ is twice continuously differentiable at $x_2^{+}$.
By \eqref{3.112.g.} and the strictly increasing property of $\sigma_{+}^2v^{\prime\prime}+\mu_{+}v^{\prime}$ on $[x_2^{+},\infty)$, we have
\begin{align}
\label{3.112.g.1}
\sigma_{+}^2v^{\prime\prime}(x)+\mu_{+}v^{\prime}(x)\geq 0.
\end{align}
\end{itemize}
Therefore, we conclude that
\begin{itemize}
\item[$\bullet$] {Suppose either $a_{48}\geq 0$. }
Then
\begin{align}
v^{\prime\prime}(x)>0, \quad x\in (x_{2}^{+},\infty),\nn
\end{align}
confirming that $v$ given by \eqref{3.105.g.} and \eqref{3.101.g.} is indeed a solution to \eqref{eq:a046} on $[x_{2}^{+},\infty)$.

\item[$\bullet$] {Suppose $a_{48}<0$.}
By \eqref{3.110.g.} and \eqref{3.112.g.1}, we have
\begin{align}
\left\{
\begin{array}{ll}
v^{\prime\prime}(x)>0,& x\in (x_2^{+}+\frac{\ln{\frac{-a_{48}\theta_{-}^2}{a_{47}\theta_{+}^2}}}{\theta_{+}-\theta_{-}},\infty), \\
v^{\prime\prime}(x)<0,& x\in [x_2^{+}, x_2^{+}+\frac{\ln{\frac{-a_{48}\theta_{-}^2}{a_{47}\theta_{+}^2}}}{\theta_{+}-\theta_{-}}),
\\
\sigma_{+}^2 v^{\prime\prime}(x) +\mu_{+}v^{\prime}(x)\geq 0,& x\in [x_2^{+}, x_2^{+}+\frac{\ln{\frac{-a_{48}\theta_{-}^2}{a_{47}\theta_{+}^2}}}{\theta_{+}-\theta_{-}}),
\end{array}
\right.
\end{align}
confirming that $v$ given by \eqref{3.105.g.} and \eqref{3.101.g.} is indeed a solution to \eqref{eq:a046} on $[x_{2}^{+},\infty)$.

\end{itemize}
In summary, when $\mu_{\pm}>0$ and $x_0^{-}< a\leq x_2^{+}$, 
the solution to \eqref{eq:a046} takes the form \eqref{3.114.g.}.

We proceed under the assumption that {\color{blue}$a>x_0^{-}\vee x_2^{+}$.} Based on the previous analysis, one knows that the solution $v$ of \eqref{eq:a046} on $[0,x_0^{-}]$ should be given by \eqref{3.45.g.} and the solution $v$ of \eqref{eq:a046} on $[x_0^{-},a]$ should be given by \eqref{3.46.g.}. By \eqref{3.47.g.} and \eqref{eq:C2.x0.minus}, these two pieces are twice continuously differentiable across $x_0^{-}$.
On the interval $[a,\infty)$, based on the previous analysis, the solution $v$ of \eqref{eq:a046} on $[a,\infty)$ should be given by \eqref{3.57.g.}. 
Since it is proved that $a_{35}>0$ and $a_{35}+a_{36}>0$, one deduces that
\begin{itemize}
\item If $a_{36}\geq 0$, then, by \eqref{3.57.g.}, it is obvious that $v^{\prime\prime}(x)> 0$ for any $x\in[a,\infty)$.

\item If $a_{36}<0$, then
\begin{align}
v^{\prime\prime}(x)=C(a_{35}\theta_{+}^{2}e^{\theta_{+}(x-a)}+a_{36}\theta_{-}^{2}e^{\theta_{-}(x-a)}),\quad x\in [a,\infty),\nn
\end{align}
is strictly increasing on $[a,\infty)$. At $x=a$, we have $v^{\prime\prime}(a)=C(a_{35}\theta_{+}^{2}+a_{36}\theta_{-}^{2})$. If $a_{35}\theta_{+}^{2}+a_{36}\theta_{-}^{2}\ge0$, we have $v^{\prime\prime}(x)>0$ for all $x\in[a,\infty)$. If $a_{35}\theta_{+}^{2}+a_{36}\theta_{-}^{2}<0$, which occurs if and only if $\ln{\frac{-a_{36}\theta_{-}^2}{a_{35}\theta_{+}^2}}> 0$,
then
\begin{align}
\label{3.1.wwc1}
\left\{
\begin{array}{ll}
v^{\prime\prime}(x)>0,& x\in (a+\frac{\ln{\frac{-a_{36}\theta_{-}^2}{a_{35}\theta_{+}^2}}}{\theta_{+}-\theta_{-}},\infty), \\
v^{\prime\prime}(x)<0,& x\in [a, a+\frac{\ln{\frac{-a_{36}\theta_{-}^2}{a_{35}\theta_{+}^2}}}{\theta_{+}-\theta_{-}}).
\end{array}
\right.
\end{align}

By \eqref{3.57.g.}, one gets
\begin{align}
\label{3.2.wwc1}
&\sigma_{+}^2v^{\prime\prime}(x)+\mu_{+}v^{\prime}(x)
\nn\\
=&~ C\left(a_{35} \left(\sigma_{+}^2\theta_{+}^{2}+\mu_{+}\theta_{+}\right)
e^{\theta_{+}(x-a)}
+ a_{36}\left(\sigma_{+}^2\theta_{-}^{2}+\mu_{+}\theta_{-}\right)
e^{\theta_{-}(x-a)}\right)
\nn\\
=&~
C\left(a_{35} \left(2q-\mu_{+}\theta_{+}\right)
e^{\theta_{+}(x-a)}
+ a_{36}\left(2q-\mu_{+}\theta_{-}\right)
e^{\theta_{-}(x-a)}\right), \quad x\in [a,\infty).
\end{align}
Based on the previously proven $2q-\mu_{+}\theta_{+}>0$ and $a_{35}>0$ and the fact that $2q-\mu_{+}\theta_{-}>0$, the function on the RHS of \eqref{3.2.wwc1} is strictly increasing in $x$, since the dominant exponential term (with $\theta_{+}>\theta_{-}$) grows faster. In this case, the minimum occurs at $x=a$:
\begin{align}
\label{3.3.wwc1}
&\sigma_{+}^2v^{\prime\prime}(a+)+\mu_{+}v^{\prime}(a+)
\nn\\
=&~
2\left(qv(a+)-\mu_{+}v^{\prime}(a+)-\frac{1}{2}\sigma_{+}^2v^{\prime\prime}(a+)\right)
+\sigma_{+}^2v^{\prime\prime}(a+)+\mu_{+}v^{\prime}(a+)
\nn\\
=&~2qv(a+)-\mu_{+}v^{\prime}(a+)
\nn\\
=&~2qv(a-)-\mu_{+}v^{\prime}(a-)
\nn\\
=&~Ca_{33}e^{\delta_{+}(a-x_0^{-})}(2q-\mu_{+}\delta_{+})+Ca_{34}e^{\delta_{-}(a-x_0^{-})}(2q-\mu_{+}\delta_{-})
\nn\\
=&~Ca_{35}(2q-\mu_{+}\theta_{+})+Ca_{36}(2q-\mu_{+}\theta_{-})
\nn\\
=&~C(a_{45}+a_{46})^{\gamma_{+}-1}\left[2q(a_{45}+a_{46})-\mu_{+}\gamma_{+}\frac{a_{45}}{a}\right]\geq0.
\end{align}
Indeed, $a>x_2^+$ is equivalent to
\[
\frac{\mu_{+}a(a_{45}+a_{46})}{\sigma_{+}^{2}(1-\gamma_{+})a_{45}}>1,
\]
which, by the definition of $\gamma_{+}$, 
\end{itemize}
Therefore, we conclude that
\begin{itemize}
\item[$\bullet$] Suppose either $a_{36}\geq 0$ or $
a_{35}\theta_{+}^{2}+a_{36}\theta_{-}^{2}\geq 0$. Then,
\begin{align}
v^{\prime\prime}(x)>0, \quad x\in (a,\infty),\nn
\end{align}
confirming that $v$ given by \eqref{3.57.g.} and \eqref{3.48.g.} is indeed a solution to \eqref{eq:a046} on $[a,\infty)$.

\item[$\bullet$] Suppose $a_{36}<0$,
$a_{35}\theta_{+}^{2}+a_{36}\theta_{-}^{2}<0$. 
By \eqref{3.1.wwc1} and \eqref{3.3.wwc1}, we have
\begin{align}
\left\{
\begin{array}{ll}
v^{\prime\prime}(x)>0,& x\in (a+\frac{\ln{\frac{-a_{36}\theta_{-}^2}{a_{35}\theta_{+}^2}}}{\theta_{+}-\theta_{-}},\infty), \\
v^{\prime\prime}(x)<0,& x\in [a, a+\frac{\ln{\frac{-a_{36}\theta_{-}^2}{a_{35}\theta_{+}^2}}}{\theta_{+}-\theta_{-}}),
\\
\sigma_{+}^2v^{\prime\prime}(x)+\mu_{+}v^{\prime}(x)\geq 0,& x\in [a, a+\frac{\ln{\frac{-a_{36}\theta_{-}^2}{a_{35}\theta_{+}^2}}}{\theta_{+}-\theta_{-}}),
\end{array}
\right.
\end{align}
confirming that $v$ given by \eqref{3.57.g.} and \eqref{3.48.g.} is indeed a solution to \eqref{eq:a046} on $[a,\infty)$.
\end{itemize}
In summary, if $\mu_{\pm}>0$ and $a> x_0^{-}\vee x_2^{+}$, 
the solution to \eqref{eq:a046} is \eqref{3.tt1.g.}.
\end{proof}

\subsubsection{Identification of the uniform (in $x$) maximizer of $V_{\underline{x},\overline{x}}(x)$ over $(\underline{x},\overline{x})$}
\label{subs.3.4.2}

By analyzing various parameter configurations under case $\mu_{\pm}>0$, we identify seven mutually exclusive sub-cases denoted by $(\text{IV}_{1})$-$(\text{IV}_{7})$. For each sub-case, we specify the function $g$ and its inflection points $(a_i)_{1\leq i\leq 3}$, and derive conclusions based on the constructions of $g_{41}$, $g_{42}$, $g_{43}$ and $g_{44}$ presented in Subsection \ref{subsubsec.3.4.1}.

\begin{itemize}

\item[($\text{IV}_{1}$)] If $a\leq x_0^{-}\wedge x_{1}^{+}$, and either of the following conditions holds
\begin{itemize}
\item[$\bullet$] $a_{44}\geq 0$,
\item[$\bullet$] $a_{44}< 0$ (recall that $a_{43}\theta_{+}^{2}+a_{44}\theta_{-}^{2}< 0$ in this case),
\end{itemize}
letting
\begin{align}
g:=g_{41}, \, a_1:=0, \, a_2:=0, \, a_3:=x_1^{+}+\frac{\ln{\left(1\vee\frac{-a_{44}\theta_{-}^2}{a_{43}\theta_{+}^2}\right)}}{\theta_{+}-\theta_{-}},\nn
\end{align}
then $g$ is \textbf{piecewise concave-convex} with inflection points $(a_i)_{1\leq i\leq 3}$ satisfying $0=a_{1}=a_{2}<a_{3}$.

\item[($\text{IV}_{2}$)] If $x_{1}^{+}<a\leq x_0^{-}$ and either of the following conditions holds
\begin{itemize}
\item[$\bullet$] $a_{32}\geq 0$ or $a_{31}\theta_{+}^{2}+a_{32}\theta_{-}^{2}\geq 0$,
\item[$\bullet$] $a_{32}< 0$ and $a_{31}\theta_{+}^{2}+a_{32}\theta_{-}^{2}< 0$, 
\end{itemize}
letting
\begin{align}
g:=g_{42},\, a_1:=0,\, a_2:=0, \, a_3:=a+\frac{\ln{\left(1\vee\frac{-a_{32}\theta_{-}^2}{a_{31}\theta_{+}^2}\right)}}{\theta_{+}-\theta_{-}},
\nn
\end{align}
then $g$ is \textbf{piecewise concave-convex} with inflection points $(a_i)_{1\leq i\leq 3}$ satisfying $0=a_{1}=a_{2}<a_{3}$.

\item[($\text{IV}_{3}$)] If $x_0^{-}<a\le x_{2}^{+}$, $a_{34}<0$, $a_{33}\delta_{+}^2+a_{34}\delta_{-}^2<0$, $x_0^{-}+\frac{\ln{\frac{-a_{34}\delta_{-}^2}{a_{33}\delta_{+}^2}}}{\delta_{+}-\delta_{-}}\geq a$,
and either of the following conditions holds
\begin{itemize}
\item[$\bullet$] $a_{48}\ge0$,
\item[$\bullet$] $a_{48}<0$ \,(recall that $a_{47}\theta_{+}^2+a_{48}\theta_{-}^2<0$ in this case),
\end{itemize}
letting
\begin{align}
g:=g_{43}, \, a_1:=0,\, a_{2}:=0,\, a_{3}:=x_{2}^{+}+\frac{\ln{\left(1\vee\frac{-a_{48}\theta_{-}^2}{a_{47}\theta_{+}^2}\right)}}{\theta_{+}-\theta_{-}},\nn
\end{align}
then $g$ is \textbf{piecewise concave-convex} with inflection points $(a_i)_{1\leq i\leq 3}$ satisfying $0=a_{1}=a_{2}<a_{3}$. 

\item[($\text{IV}_{4}$)] If $a>x_0^{-}\vee x_{2}^{+}$, either $a_{36}\geq 0$ or $a_{35}\theta_{+}^{2}+a_{36}\theta_{-}^{2}\geq 0$ holds, and either of the following conditions holds
\begin{itemize}
\item[$\bullet$] $a_{34}<0$, $a_{33}\delta_{+}^2+a_{34}\delta_{-}^2<0$, and $x_0^{-}+\frac{\ln{\frac{-a_{34}\delta_{-}^2}{a_{33}\delta_{+}^2}}}{\delta_{+}-\delta_{-}}<a$,
\item[$\bullet$] $a_{34}\geq 0$ or $a_{33}\delta_{+}^2+a_{34}\delta_{-}^2\geq 0$,
\end{itemize}
letting
\begin{align}
g:=g_{44}, \, a_1:=0, \, a_2:=0, \, a_3:=x_0^{-}+\frac{\ln{\left(1\vee\frac{-a_{34}\delta_{-}^2}{a_{33}\delta_{+}^2}\right)}}{\delta_{+}-\delta_{-}},\nn
\end{align}
then $g$ is \textbf{piecewise concave-convex} with inflection points $(a_i)_{1\leq i\leq 3}$ satisfying $0=a_{1}=a_{2}< a_{3}$.

\item[($\text{IV}_{5}$)] If $a>x_0^{-}\vee x_{2}^{+}$, $a_{34}<0$, $a_{33}\delta_{+}^2+a_{34}\delta_{-}^2<0$, $x_0^{-}+\frac{\ln{\frac{-a_{34}\delta_{-}^2}{a_{33}\delta_{+}^2}}}{\delta_{+}-\delta_{-}}\geq a$, and either of the following conditions holds
\begin{itemize}
\item[$\bullet$] $a_{36}\geq 0$ or $a_{35}\theta_{+}^{2}+a_{36}\theta_{-}^{2}\geq 0$,
\item[$\bullet$] $a_{36}< 0$ and $a_{35}\theta_{+}^{2}+a_{36}\theta_{-}^{2}< 0$,
\end{itemize}
letting
\begin{align}
g:=g_{44}, \, a_1:=0, \, a_2:=0, \, a_3:=a+\frac{\ln{\left(1\vee\frac{-a_{36}\theta_{-}^2}{a_{35}\theta_{+}^2}\right)}}{\theta_{+}-\theta_{-}},\nn
\end{align}
then $g$ is \textbf{piecewise concave-convex} with inflection points $(a_i)_{1\leq i\leq 3}$ satisfying $0=a_{1}=a_{2}<a_{3}$.

\item[($\text{IV}_{6}$)] If $a>x_0^{-}\vee x_{2}^{+}$, $a_{36}< 0$, $a_{35}\theta_{+}^{2}+a_{36}\theta_{-}^{2}< 0$, and either of the following conditions holds
\begin{itemize}
\item[$\bullet$] $a_{34}<0$, $a_{33}\delta_{+}^2+a_{34}\delta_{-}^2<0$, $x_0^{-}+\frac{\ln{\frac{-a_{34}\delta_{-}^2}{a_{33}\delta_{+}^2}}}{\delta_{+}-\delta_{-}}<a$,
\item[$\bullet$] $a_{34}\geq 0$ or $a_{33}\delta_{+}^2+a_{34}\delta_{-}^2\geq 0$,
\end{itemize}
letting
\begin{align}
g:=g_{44}, \, a_1:=x_0^{-}+\frac{\ln{\left(1\vee\frac{-a_{34}\delta_{-}^2}{a_{33}\delta_{+}^2}\right)}}{\delta_{+}-\delta_{-}}, \, a_2:=a, \, a_3:=a+\frac{\ln{\frac{-a_{36}\theta_{-}^2}{a_{35}\theta_{+}^2}}}{\theta_{+}-\theta_{-}},\nn
\end{align}
then $g$ is \textbf{piecewise concave-convex} with inflection points $(a_i)_{1\leq i\leq 3}$ satisfying $0<a_{1}<a_{2}< a_{3}$.

\item[($\text{IV}_{7}$)] If $x_0^{-}<a\le x_{2}^{+}$, either one of the following conditions holds
\begin{itemize}
\item[$\bullet$] $a_{34}\geq 0$ \text{ or } $a_{33}\delta_{+}^2+a_{34}\delta_{-}^2\geq 0$,
\item[$\bullet$] $a_{34}<0$
and
$a_{33}\delta_{+}^2+a_{34}\delta_{-}^2<0$, and $x_0^{-}+\frac{\ln{\frac{-a_{34}\delta_{-}^2}{a_{33}\delta_{+}^2}}}{\delta_{+}-\delta_{-}}<a$,
\end{itemize}
and, either one of the following conditions holds
\begin{itemize}
\item[$\bullet$] $a_{48}\ge0$,
\item[$\bullet$] $a_{48}<0$ \,(recall that $a_{47}\theta_{+}^2+a_{48}\theta_{-}^2<0$ in this case),
\end{itemize}
letting
\begin{align}
g:=g_{43}, \, a_1:=x_0^{-}+\frac{\ln{\left(1\vee\frac{-a_{34}\delta_{-}^2}{a_{33}\delta_{+}^2}\right)}}{\delta_{+}-\delta_{-}},\, a_{2}:=a,\, a_{3}:=x_{2}^{+}+\frac{\ln{\left(1\vee\frac{-a_{48}\theta_{-}^2}{a_{47}\theta_{+}^2}\right)}}{\theta_{+}-\theta_{-}},\nn
\end{align}
then $g$ is \textbf{piecewise concave-convex} with inflection points $(a_i)_{1\leq i\leq 3}$ satisfying $0<a_{1}<a_{2}<a_{3}$.

\end{itemize}
We note that the seven sub-cases mentioned above are collectively exhaustive. Under Cases $(\text{\emph{IV}}_{6})$-$(\text{\emph{IV}}_{7})$, define
\begin{align}\label{opti.stra.3.4}
(\underline{x}^{\star},\overline{x}^{\star})= \left\{
\begin{array}{ll}
(\hat{z}_1,\hat{z}_2),&\text{\emph{if} } \beta\in (A_1\cup A_3)\backslash\{\omega_1(x_2)\}, \\
((g^{\prime})^{-1}_1(g^{\prime}(x_2),x_2),&\text{\emph{if} } \beta\in (A_1\cup A_3)\cap\{\omega_1(x_2)\}, \\
(\tilde{z}_1,\tilde{z}_2),& \text{\emph{if} } \beta\in A_2\cap \overline{A}_3, \\
(\tilde{z}_1,\tilde{z}_2),&\text{\emph{if} } \beta\in \overline{A_1\cup A_2\cup A_3}\backslash\{\omega_1(x_2)\}
\\
\text{ \emph{any one of } } ((g^{\prime})^{-1}_1(g^{\prime}(x_2)),x_2) \text{ \emph{and} } (\tilde{z}_1,\tilde{z}_2),&\text{\emph{if} } \beta\in \overline{A_1\cup A_2\cup A_3}\cap\{\omega_1(x_2)\}.
\end{array}
\right.
\end{align}

\begin{lemma}
\label{lem:slope.comparison.case.IV}
Suppose that either Case $(\text{\emph{IV}}_{6})$ or Case $(\text{\emph{IV}}_{7})$ holds, and let
$(\underline{x}^{\star},\overline{x}^{\star})$ be defined by \eqref{opti.stra.3.4}. Then
\begin{align}
g^{\prime}(\overline{x}^{\star})=\inf\left\{g^{\prime}(\overline{x}):
\text{ there exists }\underline{x}\geq0\text{ such that }
\psi(\underline{x},\overline{x})=\beta\right\}.
\label{eq:slope.minimizer.case.IV}
\end{align}
	\end{lemma}
\begin{proof}
We first record the monotonicity properties used below. Whenever the relevant inverse function exists and $i\in\{1,3\}$, we have
\begin{align}
\frac{\mathrm{d}}{\mathrm{d}y}
\psi\bigl((g^{\prime})_i^{-1}(g^{\prime}(y)),y\bigr)=
\frac{g^{\prime\prime}(y)}{[g^{\prime}(y)]^2}
\left[g(y)-g\bigl((g^{\prime})_i^{-1}(g^{\prime}(y))\bigr)\right]>0,
\label{eq:omega.branch.increasing.case.IV}
\end{align}
for $y\in(a_1,a_2)\cup(b_1,\infty)$ when $i=1$ and $y\in(a_3,b_2)$ when $i=3.$ Moreover,
\begin{align}
\frac{\mathrm{d}}{\mathrm{d}y}
\int_{(g^{\prime})_1^{-1}(g^{\prime}(y))}^{(g^{\prime})_3^{-1}(g^{\prime}(y))}
\left(1-\frac{g^{\prime}(z)}{g^{\prime}(y)}\right)\mathrm{d}z
=\frac{g^{\prime\prime}(y)}{[g^{\prime}(y)]^2}
\left[g\bigl((g^{\prime})_3^{-1}(g^{\prime}(y))\bigr)-g\bigl((g^{\prime})_1^{-1}(g^{\prime}(y))\bigr)
\right]>0,\quad y\in(b_1,b_2),
\label{eq:x1.crossing.case.IV}
\end{align}
and
\begin{align}
\frac{\mathrm{d}}{\mathrm{d}y}
\int_{(g^{\prime})_2^{-1}(g^{\prime}(y))}^{y}
\left(1-\frac{g^{\prime}(z)}{g^{\prime}(y)}\right)\mathrm{d}z=
\frac{g^{\prime\prime}(y)}{[g^{\prime}(y)]^2}
\left[g(y)-g\bigl((g^{\prime})_2^{-1}(g^{\prime}(y))\bigr)\right]>0,\quad y\in(b_1,b_2).
\label{eq:x2.crossing.case.IV}
\end{align}
By \eqref{eq:x1.crossing.case.IV} and the definition of $x_1$, we have
\begin{align}
			&\psi\bigl((g^{\prime})_1^{-1}(g^{\prime}(y)),y\bigr)
			-\psi\bigl((g^{\prime})_3^{-1}(g^{\prime}(y)),y\bigr)
			\begin{cases}
				<0,&y\in[b_1,x_1),\\
				>0,&y\in(x_1,b_2],
			\end{cases}
			\label{eq:x1.sign.case.IV}\\
			&\psi\bigl((g^{\prime})_1^{-1}(g^{\prime}(x_1)),x_1\bigr)
			-\psi\bigl((g^{\prime})_3^{-1}(g^{\prime}(x_1)),x_1\bigr)
			\begin{cases}
				=0,&x_1\in(b_1,b_2],\\
				\geq0,&x_1=b_1.
			\end{cases}
			\label{eq:x1.endpoint.case.IV}
		\end{align}
Consequently, by the definition of $\omega_2$, we have
\begin{align}
\omega_2(y)=\max\Big\{
&\psi\bigl((g^{\prime})_1^{-1}(g^{\prime}(y)),y\bigr),
\psi\bigl((g^{\prime})_3^{-1}(g^{\prime}(y)),y\bigr)\Big\},
\label{eq:omega2.maximum.case.IV}
\end{align}
whenever both inverse functions exist.

We prove \eqref{eq:slope.minimizer.case.IV} only for the case where $\beta\in A_1\backslash\{\omega_1(x_2)\}$ and \begin{align*}
\overline{x}^{\star}=\omega_1^{-1}(\beta)
\in\left(a_1,(g^{\prime})_2^{-1}(g^{\prime}(x_2))\right).
\end{align*}
The proofs for the other possible locations of $(\underline{x}^{\star},\overline{x}^{\star})$ specified in \eqref{opti.stra.3.4} are similar. In the present case, \begin{align*}
\underline{x}^{\star}=(g^{\prime})_1^{-1}(g^{\prime}(\overline{x}^{\star}))<a_1<\overline{x}^{\star}<a_2,\qquad
\psi(\underline{x}^{\star},\overline{x}^{\star})=\beta.
\end{align*}
Let $(\underline{x},\overline{x})\neq(\underline{x}^{\star},\overline{x}^{\star})$ be any pair such that
$0\leq\underline{x}<\overline{x}$ and
$\psi(\underline{x},\overline{x})=\beta$. We prove
$g^{\prime}(\overline{x})\geq g^{\prime}(\overline{x}^{\star})$ by contradiction.
Suppose that
\begin{align}
g^{\prime}(\overline{x})<g^{\prime}(\overline{x}^{\star}).
\label{eq:slope.case.IV.contradiction}
\end{align}
We examine the following cases.
\begin{itemize}
\item Suppose $0\leq\underline{x}<\overline{x}\leq a_1$.
Since $g$ is concave on $[0,a_1]$, we have
\begin{align*}
\psi(\underline{x},\overline{x})
=\int_{\underline{x}}^{\overline{x}}
\left(1-\frac{g^{\prime}(z)}{g^{\prime}(\overline{x})}\right)\mathrm{d}z\leq0,
\end{align*}
which contradicts
$\psi(\underline{x},\overline{x})=\beta>0$.
\item Suppose $0\leq\underline{x}<\overline{x}$ and
$\overline{x}\in(a_1,a_2)$. By
\eqref{eq:slope.case.IV.contradiction} and the strictly increasing property of $g^{\prime}$ on $[a_1,a_2]$, we have
$\overline{x}<\overline{x}^{\star}$. For fixed $\overline{x}\in(a_1,a_2)$, it follows from the construction of $g$ that the function $x\mapsto\psi(x,\overline{x})$ is increasing on $[0,(g^{\prime})_1^{-1}(g^{\prime}(\overline{x}))]$ and decreasing on $[(g^{\prime})_1^{-1}(g^{\prime}(\overline{x})),\overline{x}]$. Therefore,
\begin{align*}
\beta=\psi(\underline{x},\overline{x})\leq
\psi((g^{\prime})_1^{-1}(g^{\prime}(\overline{x})),\overline{x})=\omega_1(\overline{x})<\omega_1(\overline{x}^{\star})=\beta,
\end{align*}
where the strict inequality follows from the strictly increasing property of $\omega_1$ on its first domain interval. This is a contradiction.
\item Suppose $0\leq\underline{x}<\overline{x}$ and $\overline{x}\in[a_2,a_3)$. If $g^{\prime}(\overline{x})\leq g^{\prime}(a_1)$, then $g^{\prime}(z)\geq g^{\prime}(\overline{x})$ for every $z\in[0,\overline{x}]$, and hence $\psi(\underline{x},\overline{x})\leq0$, a contradiction. 	Thus, it remains to consider $g^{\prime}(\overline{x})>g^{\prime}(a_1)$. The sign of $1-g^{\prime}(z)/g^{\prime}(\overline{x})$ shows that $x\mapsto\psi(x,\overline{x})$ attains its positive maximum at $(g^{\prime})_1^{-1}(g^{\prime}(\overline{x}))$ on the interval $[0,\overline{x}]$. Consequently, 
\begin{align*}
\beta
&=\psi(\underline{x},\overline{x})\nn\\
&\leq
\psi((g^{\prime})_1^{-1}(g^{\prime}(\overline{x})),\overline{x})\nn\\
&=\int_{(g^{\prime})_1^{-1}(g^{\prime}(\overline{x}))}^{(g^{\prime})_2^{-1}(g^{\prime}(\overline{x}))}
\left(1-\frac{g^{\prime}(z)}{g^{\prime}(\overline{x})}\right)\mathrm{d}z
+\int_{(g^{\prime})_2^{-1}(g^{\prime}(\overline{x}))}^{\overline{x}}
\left(1-\frac{g^{\prime}(z)}{g^{\prime}(\overline{x})}\right)\mathrm{d}z\nn\\
&<\int_{(g^{\prime})_1^{-1}(g^{\prime}(\overline{x}))}^{(g^{\prime})_2^{-1}(g^{\prime}(\overline{x}))}
\left(1-\frac{g^{\prime}(z)}{g^{\prime}(\overline{x})}\right)\mathrm{d}z\nn\\
&=\omega_1((g^{\prime})_2^{-1}(g^{\prime}(\overline{x})))
<\omega_1(\overline{x}^{\star})=\beta.
\end{align*}
Here, the first strict inequality follows from
$g^{\prime}(z)>g^{\prime}(\overline{x})$ on
$((g^{\prime})_2^{-1}(g^{\prime}(\overline{x})),\overline{x})$, while the second follows from \eqref{eq:slope.case.IV.contradiction}, which gives
$(g^{\prime})_2^{-1}(g^{\prime}(\overline{x}))<\overline{x}^{\star}$, and the increasing property of $\omega_1$. This is again a contradiction.
\item Suppose $\underline{x}\in[0,a_1]$ and $\overline{x}\in[a_3,\infty)$. We have
\begin{align}
\frac{\partial}{\partial x}\psi(x,\overline{x})
=-\left(1-\frac{g^{\prime}(x)}{g^{\prime}(\overline{x})}\right).
\label{eq:psi.derivative.case.IV}
\end{align}
Under \eqref{eq:slope.case.IV.contradiction}, \[ g^{\prime}(\overline{x}) <g^{\prime}(\overline{x}^{\star}) <g^{\prime}(x_2)\leq g^{\prime}(a_2). \]
Thus, only the following two cases are possible.
\begin{itemize}
\item If $g^{\prime}(\overline{x})<g^{\prime}(a_1)$, then
$(g^{\prime})_1^{-1}(g^{\prime}(\overline{x}))$ does not exist, while $(g^{\prime})_3^{-1}(g^{\prime}(\overline{x}))$ exists, and $	g^{\prime}(z)>g^{\prime}(\overline{x}),$ for $z\in[\underline{x},(g^{\prime})_3^{-1}(g^{\prime}(\overline{x}))),$ and $g^{\prime}(z)<g^{\prime}(\overline{x}),$ for $z\in((g^{\prime})_3^{-1}(g^{\prime}(\overline{x})),\overline{x}).$
Therefore, by \eqref{eq:psi.derivative.case.IV},
\begin{align*}
\psi(\underline{x},\overline{x})
&\leq
\psi((g^{\prime})_3^{-1}(g^{\prime}(\overline{x})),\overline{x})=\omega_2(\overline{x}).
\end{align*}
\item If $g^{\prime}(a_1)\leq g^{\prime}(\overline{x})\leq g^{\prime}(a_2)$, then both $(g^{\prime})_1^{-1}(g^{\prime}(\overline{x}))$ and $(g^{\prime})_3^{-1}(g^{\prime}(\overline{x}))$ exist. Since $\underline{x}\in[0,a_1]$, \eqref{eq:psi.derivative.case.IV} gives
\begin{align*}
\psi(\underline{x},\overline{x})\leq\psi((g^{\prime})_1^{-1}(g^{\prime}(\overline{x})),\overline{x}).
\end{align*}
Moreover,
\begin{align*}
\psi((g^{\prime})_1^{-1}(g^{\prime}(\overline{x})),\overline{x})
-\psi((g^{\prime})_3^{-1}(g^{\prime}(\overline{x})),\overline{x})=\int_{(g^{\prime})_1^{-1}(g^{\prime}(\overline{x}))}^{(g^{\prime})_3^{-1}(g^{\prime}(\overline{x}))}
\left(1-\frac{g^{\prime}(z)}{g^{\prime}(\overline{x})}\right)\mathrm{d}z.
\end{align*}
Hence, by \eqref{eq:omega2.maximum.case.IV}, 

\item If $g^{\prime}(\overline{x})>g^{\prime}(a_2)$, then $(g^{\prime})_3^{-1}(g^{\prime}(\overline{x}))$ does not exist, while $(g^{\prime})_1^{-1}(g^{\prime}(\overline{x}))$ exists. By \eqref{eq:psi.derivative.case.IV},
$
\psi(\underline{x},\overline{x}) \leq
\psi((g^{\prime})_1^{-1}(g^{\prime}(\overline{x})),\overline{x}) =\omega_2(\overline{x}).
$
\end{itemize}
Thus, in all three cases,
\begin{align}
\beta=\psi(\underline{x},\overline{x}) \leq\omega_2(\overline{x}).
\label{eq:psi.less.omega2.case.IV}
\end{align}
Since $\omega_2$ and $g^{\prime}$ are increasing on their corresponding domains, \eqref{eq:psi.less.omega2.case.IV} and $\beta\in A_1$ imply
\begin{align}
g^{\prime}(\overline{x})\geq g^{\prime}(\omega_2^{-1}(\beta))
>g^{\prime}(\omega_1^{-1}(\beta))
=g^{\prime}(\overline{x}^{\star}).
\label{eq:slope.from.omega2.case.IV}
\end{align}
contradicting \eqref{eq:slope.case.IV.contradiction}.
\item Suppose $\underline{x}\in(a_1,a_2)$ and $\overline{x}\in[a_3,\infty)$. By the same three-case analysis based on \eqref{eq:psi.derivative.case.IV} as in the preceding item,
\begin{align*}
\psi(\underline{x},\overline{x})\leq
\begin{cases}
\psi\bigl((g^{\prime})_3^{-1}(g^{\prime}(\overline{x})),\overline{x}\bigr),
&g^{\prime}(\overline{x})<g^{\prime}(a_1),\\[1mm]
\max\Big\{
\psi\bigl((g^{\prime})_1^{-1}(g^{\prime}(\overline{x})),\overline{x}\bigr),
\psi\bigl((g^{\prime})_3^{-1}(g^{\prime}(\overline{x})),\overline{x}\bigr)\Big\},
&g^{\prime}(a_1)\leq g^{\prime}(\overline{x})\leq g^{\prime}(a_2),\\[1mm]
\psi\bigl((g^{\prime})_1^{-1}(g^{\prime}(\overline{x})),\overline{x}\bigr),
&g^{\prime}(\overline{x})>g^{\prime}(a_2).
\end{cases}
\end{align*}
Indeed, in the middle case, the minimum of $x\mapsto\psi(x,\overline{x})$ on $[a_1,a_2]$ is attained at $(g^{\prime})_2^{-1}(g^{\prime}(\overline{x}))$, so its maximum over $[a_1,a_2]$ is attained at $a_1$ or $a_2$; moreover,
\begin{align*}
\psi(a_1,\overline{x})
&\leq\psi\bigl((g^{\prime})_1^{-1}(g^{\prime}(\overline{x})),\overline{x}\bigr),\\
\psi(a_2,\overline{x})
&\leq\psi\bigl((g^{\prime})_3^{-1}(g^{\prime}(\overline{x})),\overline{x}\bigr).
\end{align*} 
Hence, by \eqref{eq:omega2.maximum.case.IV},
\begin{align*}
\beta=\psi(\underline{x},\overline{x})\leq\omega_2(\overline{x}).
\end{align*}
	Thus, \eqref{eq:slope.from.omega2.case.IV} also holds for this position, contradicting \eqref{eq:slope.case.IV.contradiction}.

\item Suppose $\underline{x}\in[a_2,a_3)$ and $\overline{x}\in[a_3,\infty)$. If $g^{\prime}(\overline{x})\leq g^{\prime}(a_2)$, then the monotonicity of $x\mapsto\psi(x,\overline{x})$ on $[a_2,a_3]$ gives
$
\psi(\underline{x},\overline{x})
\leq
\psi((g^{\prime})_3^{-1}(g^{\prime}(\overline{x})),\overline{x})
\leq\omega_2(\overline{x}).
$
If $g^{\prime}(\overline{x})>g^{\prime}(a_2)$, then $g^{\prime}(z)\leq g^{\prime}(\overline{x}),$ for $z\in[(g^{\prime})_1^{-1}(g^{\prime}(\overline{x})),\underline{x}],$ and hence
\begin{align*}
\psi((g^{\prime})_1^{-1}(g^{\prime}(\overline{x})),\overline{x})
-\psi(\underline{x},\overline{x})=\int_{(g^{\prime})_1^{-1}(g^{\prime}(\overline{x}))}^{\underline{x}}
\left(1-\frac{g^{\prime}(z)}{g^{\prime}(\overline{x})}\right)\mathrm{d}z\geq0.
\end{align*}
Therefore, $\psi(\underline{x},\overline{x})\leq\psi((g^{\prime})_1^{-1}(g^{\prime}(\overline{x})),\overline{x})=\omega_2(\overline{x}).
$
Thus, in both cases, \eqref{eq:psi.less.omega2.case.IV} and \eqref{eq:slope.from.omega2.case.IV} hold, contradicting \eqref{eq:slope.case.IV.contradiction}. 
\item Suppose $a_3\leq\underline{x}<\overline{x}<\infty$. If $g^{\prime}(\overline{x})\leq g^{\prime}(a_2)$, then $	g^{\prime}(z)\leq g^{\prime}(\overline{x}),$ for $z\in[(g^{\prime})_3^{-1}(g^{\prime}(\overline{x})),\underline{x}]$ and 
$$\psi((g^{\prime})_3^{-1}(g^{\prime}(\overline{x})),\overline{x})-\psi(\underline{x},\overline{x})=\int_{(g^{\prime})_3^{-1}(g^{\prime}(\overline{x}))}^{\underline{x}}\left(1-\frac{g^{\prime}(z)}{g^{\prime}(\overline{x})}\right)\mathrm{d}z\geq0,$$ and hence
$
\psi(\underline{x},\overline{x})\leq\psi((g^{\prime})_3^{-1}(g^{\prime}(\overline{x})),\overline{x})\leq\omega_2(\overline{x}).$ If $g^{\prime}(\overline{x})>g^{\prime}(a_2)$, then $	g^{\prime}(z)\leq g^{\prime}(\overline{x}),$ for $	z\in[(g^{\prime})_1^{-1}(g^{\prime}(\overline{x})),\underline{x}],$ and $$\psi((g^{\prime})_1^{-1}(g^{\prime}(\overline{x})),\overline{x})-\psi(\underline{x},\overline{x})=\int_{(g^{\prime})_1^{-1}(g^{\prime}(\overline{x}))}^{\underline{x}}\left(1-\frac{g^{\prime}(z)}{g^{\prime}(\overline{x})}\right)\mathrm{d}z\geq0,$$
and hence $\psi(\underline{x},\overline{x})\leq\psi((g^{\prime})_1^{-1}(g^{\prime}(\overline{x})),\overline{x})=\omega_2(\overline{x}).$ Therefore, \eqref{eq:psi.less.omega2.case.IV} and \eqref{eq:slope.from.omega2.case.IV} hold in either case, contradicting \eqref{eq:slope.case.IV.contradiction}. 
\end{itemize}
	Thus every pair $(\underline{x},\overline{x})$ satisfying
			$\psi(\underline{x},\overline{x})=\beta$ obeys
			$g^{\prime}(\overline{x})\geq g^{\prime}(\overline{x}^{\star})$. Since
			$(\underline{x}^{\star},\overline{x}^{\star})$ itself satisfies the
			constraint, \eqref{eq:slope.minimizer.case.IV} follows. The proof for
			the remaining ranges of $\beta$ in \eqref{opti.stra.3.4} is obtained by
			the same comparison.
\end{proof}

	\begin{itemize}
		\item[] \textbf{Assumption B}: If $\overline{x}^{\star}\in[a_1,a_2)$ and $\psi(0,a_3)\leq \psi(0,\overline{x}^{\star})$, then, for any pair $(\underline{x},\overline{x})$ such that
		$\psi(\underline{x},\overline{x})=\beta$, $\underline{x}\geq0$,
		$\overline{x}\geq a_3$, and
		$g^{\prime}(\overline{x}^{\star})<g^{\prime}(\overline{x})<g^{\prime}(a_2)$, we have
		\begin{align*}
			\psi\left(0,(g^{\prime})_3^{-1}(g^{\prime}(\overline{x}))\right)
			\geq\psi(0,\overline{x}^{\star}).
		\end{align*}
	\end{itemize}

\begin{theorem}
\label{Theorem3.6}
Suppose $\mu_{\pm}>0$. Let $\psi$ be defined by \eqref{def.psi.}. Let $g$ and $(a_{i})_{1\leq i\leq 3}$ be defined separately in seven cases $(\text{\emph{IV}}_{1})$-$(\text{\emph{IV}}_{7})$.

Under Cases $(\text{\emph{IV}}_{1})$-$(\text{\emph{IV}}_{5})$, define $(\underline{x}^{\star},\overline{x}^{\star})
=((g^{\prime})_{3}^{-1}(g^{\prime}(\phi^{-1}(\beta))),\phi^{-1}(\beta))$. The reinsurance-dividend strategy $(u^{\star},D^{\star})$ given by 
\begin{align}\label{optimal.stra.3.13}
\begin{cases}
u_{t}^{\star}=
\frac{\mu_{-}U^{u^{\star},D^{\underline{x}^{\star},\overline{x}^{\star}}}_{t}}{\sigma_{-}^2(1-\gamma_{-})}\emph{\textbf{1}}_{\{U^{u^{\star},D^{\underline{x}^{\star},\overline{x}^{\star}}}_{t}\in[0,x_0^{-}\wedge a)\}}+\frac{\mu_{+}\left(U^{u^{\star},D^{\underline{x}^{\star},\overline{x}^{\star}}}_{t}+\frac{a_{42}}{a_{41}}a\right)}{\sigma_{+}^2(1-\gamma_{+})}\emph{\textbf{1}}_{\{U^{u^{\star},D^{\underline{x}^{\star},\overline{x}^{\star}}}_{t}\in{\color{blue}[a,a \vee x^+_1)},\,x^+_1\geq x^+_2
\}}\\~~~~~\,\,\frac{\mu_{+}\left(U^{u^{\star},D^{\underline{x}^{\star},\overline{x}^{\star}}}_{t}+\frac{a_{46}}{a_{45}}a\right)}{\sigma_{+}^2(1-\gamma_{+})}\emph{\textbf{1}}_{\{U^{u^{\star},D^{\underline{x}^{\star},\overline{x}^{\star}}}_{t}\in{\color{blue}[a,a \vee x^+_2)},\,x^{+}_1<x^+_2\}}+\emph{\textbf{1}}_{\{U^{u^{\star},D^{\underline{x}^{\star},\overline{x}^{\star}}}_{t}\in[x_0^{-}\wedge a,\infty)\backslash{\color{blue}[a,a \vee x^+_1\vee x^+_2)}\}}
,& t\geq 0, \\
D^{\star}_{t}=D^{\underline{x}^{\star},\overline{x}^{\star}}_{t},& t\geq 0,
\end{cases}
\end{align}
is optimal for the auxiliary control problem \eqref{eq:a008.aux.}, and, the associated value function is
\begin{equation}\label{eq:V3.3}
V_{\underline{x}^{\star},\overline{x}^{\star}}(x)=
\left\{\begin{array}{ll}
\frac{g(x)}{g^{\prime}(\overline{x}^{\star})},&x\in[0,\overline{x}^{\star}), \\
x-\overline{x}^{\star}+\frac{g(\overline{x}^{\star})}{g^{\prime}(\overline{x}^{\star})}=x-\underline{x}^{\star}-\beta+\frac{g(\underline{x}^{\star})}{g^{\prime}(\overline{x}^{\star})},& x\in[\overline{x}^{\star},\infty).
\end{array}
\right.
\end{equation}

Under Cases $(\text{\emph{IV}}_{6})$-$(\text{\emph{IV}}_{7})$ with $\overline{x}^{\star}<a$, or with $\overline{x}^{\star}\geq a$ and
$\psi(0,a_3)\leq\psi(0,\overline{x}^{\star})$, let $(\underline{x}^{\star},\overline{x}^{\star})$ be defined by \eqref{opti.stra.3.4}. The reinsurance-dividend strategy $(u^{\star},D^{\star})$ given by \eqref{optimal.stra.3.13} is optimal for the auxiliary control problem	\eqref{eq:a008.aux.}, and,
the associated value function is given by \eqref{eq:V3.3}.

Under Cases $(\text{\emph{IV}}_{6})$-$(\text{\emph{IV}}_{7})$ with $\overline{x}^{\star}< a$ and $\psi(0,a_3)\leq\psi(0,\overline{x}^{\star})$, let $(\underline{x}^{\star},\overline{x}^{\star})$ be defined by \eqref{opti.stra.3.4}. Suppose further that \textbf{Assumption B} holds. The reinsurance-dividend strategy $(u^{\star},D^{\star})$ given by \eqref{optimal.stra.3.13} is optimal for the auxiliary control problem	\eqref{eq:a008.aux.}, and,
the associated value function is given by \eqref{eq:V3.3}. 
\end{theorem}

\begin{proof}
We only need to give a proof for the scenario of $0<a_{1}<a_{2}<a_{3}$ (see case $\text{{IV}}_{6}-\text{{IV}}_{7}$), the proof for the scenario of $0=a_{1}=a_{2}<a_{3}$ (see case $\text{{IV}}_{1}-\text{{IV}}_{5}$) is similar to that of Theorem \ref{Theorem3.5}. By Lemma \ref{lem3.1.w}, we can decouple the auxiliary optimization problem \eqref{eq:a008.aux.} into problem \eqref{auxiliary.contr.p.} and problem \eqref{point.point.optimization}. By theorem \ref{thm3.1.w} and similar arguments adopted at the proof of Theorem \ref{thm3.2.ww}, the value function of the auxiliary optimization problem \eqref{auxiliary.contr.p.} is given by \begin{equation}
V_{\underline{x},\overline{x}}(x)=
\left\{\begin{array}{ll}
\frac{g(x)}{g^{\prime}(\overline{x})},&x\in[0,\overline{x}), \\
\frac{g(\underline{x})}{g^{\prime}(\overline{x})}+x-\underline{x}-\beta,& x\in[\overline{x},\infty),
\end{array}
\right.
\end{equation}
where $g=g_{44}$ ($g=g_{43}$, resp.) is defined in \eqref{3.tt1.g.} under Case $\text{{IV}}_{6}$ (\eqref{3.114.g.} under Case $\text{{IV}}_{7}$, resp.), and $\underline{x}$ and $\overline{x}$ satisfy
\begin{align}
\label{3.20.w2.4}
\psi(\underline{x},\overline{x})=\beta.
\end{align}

We next verify that $(\underline{x}^{\star},\overline{x}^{\star})$ given by \eqref{opti.stra.3.4} is the solution to problem \eqref{point.point.optimization}. Suppose $(\underline{x}^{\star},\overline{x}^{\star})$ is such that $\underline{x}^{\star}\in[0,a_1]$ and $\overline{x}^{\star}\in[a_1,(g^{\prime})^{-1}_2(g^{\prime}(x_2)))\cup[x_2,\infty)$.
Then $\underline{x}^{\star}=(g^{\prime})^{-1}_1(g^{\prime}(\overline{x}^{\star}))$. We examine the following cases. 

Suppose 
$(\underline{x},\overline{x})$ with $\underline{x}\in[0,a_1)$ is
any solution to \eqref{3.20.w2.4} different from $(\underline{x}^{\star},\overline{x}^{\star})$. Note that we necessarily have $\overline{x}>a_1$ to ensure $\psi(\underline{x},\overline{x})=\beta$. Then, we have the following observations.
\begin{itemize}
\item 
Suppose 
$\overline{x}\in (a_1,a_2)$ such that $\psi(\underline{x},\overline{x})=\beta$. 
Then, if $\overline{x}^{\star}\in[0,(g^{\prime})^{-1}_2(g^{\prime}(x_2)))$, we obtain $\overline{x}>\overline{x}^{\star}$ and $g^{\prime}(\overline{x})\geq g^{\prime}(\overline{x}^{\star})$, since Lemma \ref{lem:slope.comparison.case.IV}, and the following inequalities
\begin{itemize}
\item For $x\in (0,\overline{x}^{\star}]$,
\begin{align}\label{eq:v-v.4.3.6}
V_{\underline{x}^{\star},\overline{x}^{\star}}(x)-V_{\underline{x},\overline{x}}(x)=g(x)\left(\frac{1}{g^{\prime}(\overline{x}^{\star})}-\frac{1}{g^{\prime}(\overline{x})}\right)\geq0.
\end{align}
\item For $x\in(\overline{x}^{\star},\overline{x}]$,
\begin{align}\label{eq:v-v.5.3.6}
V_{\underline{x}^{\star},\overline{x}^{\star}}(x)-V_{\underline{x},\overline{x}}(x)=&~
\frac{g(\underline{x}^{\star})}{g^{\prime}(\overline{x}^{\star})}+x-\underline{x}^{\star}-\beta-\frac{g(x)}{g^{\prime}(\overline{x})}
\nn\\
=&~-\int_{0}^{\overline{x}^{\star}}\left(1-\frac{g^{\prime}(z)}{g^{\prime}(\overline{x}^{\star})}\right)\mathrm{d}z
+\int_{0}^{x}\left(1-\frac{g^{\prime}(z)}{g^{\prime}(\overline{x})}\right)\mathrm{d}z
\nn\\
\geq&~-\int_{0}^{\overline{x}^{\star}}\left(1-\frac{g^{\prime}(z)}{g^{\prime}(\overline{x})}\right)\mathrm{d}z
+\int_{0}^{x}\left(1-\frac{g^{\prime}(z)}{g^{\prime}(\overline{x})}\right)\mathrm{d}z
\nn\\
=&~\int_{\overline{x}^{\star}}^{x}\left(1-\frac{g^{\prime}(z)}{g^{\prime}(\overline{x})}\right)\mathrm{d}z\geq0.
\end{align}
\item For $x\in(\overline{x},\infty)$,
\begin{align}\label{eq:v-v.6.3.6}
V_{\underline{x}^{\star},\overline{x}^{\star}}(x)-V_{\underline{x},\overline{x}}(x)=&~\frac{g(\overline{x}^{\star})}{g^{\prime}(\overline{x}^{\star})}+
x-\overline{x}^{\star}-(\frac{g(\overline{x})}{g^{\prime}(\overline{x})}+x-\overline{x})
\nn\\
=&~\int_0^{\overline{x}}\left(1-\frac{g^{\prime}(z)}{g^{\prime}(\overline{x})}\right)dz-\int_0^{\overline{x}^{\star}}\left(1-\frac{g^{\prime}(z)}{g^{\prime}(\overline{x}^{\star})}\right)dz
\nn\\
\geq&~\int_{\overline{x}^{\star}}^{\overline{x}}\left(1-\frac{g^{\prime}(z)}{g^{\prime}(\overline{x})}\right)\mathrm{d}z\geq0.
\end{align}
\end{itemize}
Hence, $V_{\underline{x}^{\star},\overline{x}^{\star}}(x)\geq V_{\underline{x},\overline{x}}(x)$ for all $x\in(0,\infty)$. If $\overline{x}^{\star}\in[x_2,\infty)$, then we have $\overline{x}<\overline{x}^{\star}$ and $g^{\prime}(\overline{x})\geq g^{\prime}(\overline{x}^{\star})$ by Lemma \ref{lem:slope.comparison.case.IV}. Thus,
\begin{itemize}
\item for $x\in(0,\overline{x})$, 
\begin{eqnarray}\label{eq:v-v.1.3.6}
V_{\underline{x}^{\star},\overline{x}^{\star}}(x)-V_{\underline{x},\overline{x}}(x)=\frac{g(x)}{g^{\prime}(\overline{x}^{\star})}-\frac{g(x)}{g^{\prime}(\overline{x})}\geq0;
\end{eqnarray}

\item for $x\in[\overline{x},\overline{x}^{\star})$, \begin{eqnarray}\label{eq:v-v.2.3.6}
V_{\underline{x}^{\star},\overline{x}^{\star}}(x)-V_{\underline{x},\overline{x}}(x)\hspace{-0.3cm}&=&\hspace{-0.3cm}\frac{g(x)}{g^{\prime}(\overline{x}^{\star})}-\frac{g(\underline{x})}{g^{\prime}(\overline{x})}-x+\underline{x}+\beta
\nn\\
\hspace{-0.3cm}&=&\hspace{-0.3cm}\int_{0}^{\overline{x}}\left(1-\frac{g^{\prime}(z)}{g^{\prime}(\overline{x})}\right)\mathrm{d}z-\int_{0}^{x}\left(1-\frac{g^{\prime}(z)}{g^{\prime}(\overline{x}^{\star})}\right)\mathrm{d}z
\nn\\
\hspace{-0.3cm}&\geq&\hspace{-0.3cm}\int_{0}^{\overline{x}}\left(1-\frac{g^{\prime}(z)}{g^{\prime}(\overline{x})}\right)\mathrm{d}z-\int_{0}^{\overline{x}^{\star}}\left(1-\frac{g^{\prime}(z)}{g^{\prime}(\overline{x}^{\star})}\right)\mathrm{d}z
\nn\\
\hspace{-0.3cm}&=&\hspace{-0.3cm}
\int_{0}^{\underline{x}}\left(1-\frac{g^{\prime}(z)}{g^{\prime}(\overline{x})}\right)\mathrm{d}z-\int_{0}^{\underline{x}^{\star}}\left(1-\frac{g^{\prime}(z)}{g^{\prime}(\overline{x}^{\star})}\right)\mathrm{d}z
\nn\\
\hspace{-0.3cm}&\geq&\hspace{-0.3cm}
\int_{\underline{x}^{\star}}^{\underline{x}}\left(1-\frac{g^{\prime}(z)}{g^{\prime}(\overline{x}^{\star})}\right)\mathrm{d}z\geq0;
\end{eqnarray}
\item for $x\in[\overline{x}^{\star},\infty)$,
\begin{eqnarray}\label{eq:v-v.3.3.6}
V_{\underline{x}^{\star},\overline{x}^{\star}}(x)-V_{\underline{x},\overline{x}}(x)\hspace{-0.3cm}&=&\hspace{-0.3cm}
\frac{g(\underline{x}^{\star})}{g^{\prime}(\overline{x}^{\star})}+
x-\underline{x}^{\star}-\beta-(\frac{g(\underline{x})}{g^{\prime}(\overline{x})}+x-\underline{x}-\beta)
\nn\\
\hspace{-0.3cm}&=&\hspace{-0.3cm}
\int_0^{\underline{x}}\left(1-\frac{g^{\prime}(z)}{g^{\prime}(\overline{x})}\right)dz-\int_0^{\underline{x}^{\star}}\left(1-\frac{g^{\prime}(z)}{g^{\prime}(\overline{x}^{\star})}\right)dz
\nn\\
\hspace{-0.3cm}&\geq&\hspace{-0.3cm}
\int_{\underline{x}^{\star}}^{\underline{x}}\left(1-\frac{g^{\prime}(z)}{g^{\prime}(\overline{x}^{\star})}\right)dz\geq0.
\end{eqnarray}

\end{itemize}
Hence, it holds that $V_{\underline{x}^{\star},\overline{x}^{\star}}(x)\geq V_{\underline{x},\overline{x}}(x)$ for all $x\in(0,\infty)$.
\item Suppose 
$\overline{x}\in [a_2,a_3)$ such that $\psi(\underline{x},\overline{x})=\beta$. By Lemma \ref{lem:slope.comparison.case.IV}, we have $g^{\prime}(\overline{x})\geq g^{\prime}(\overline{x}^{\star})$. 
Then, if $\overline{x}^{\star}\in[0,(g^{\prime})^{-1}_2(g^{\prime}(x_2)))$, we have $\overline{x}^{\star}<\overline{x}$ and the following inequalities
\begin{itemize}
\item For $x\in (0,\overline{x}^{\star}]$, we have \eqref{eq:v-v.4.3.6}.
\item For $x\in(\overline{x}^{\star},\overline{x}]$,
\begin{align}
V_{\underline{x}^{\star},\overline{x}^{\star}}(x)-V_{\underline{x},\overline{x}}(x)=&~
\frac{g(\underline{x}^{\star})}{g^{\prime}(\overline{x}^{\star})}+x-\underline{x}^{\star}-\beta-\frac{g(x)}{g^{\prime}(\overline{x})}
\nn\\
=&~-\int_{0}^{\overline{x}^{\star}}\left(1-\frac{g^{\prime}(z)}{g^{\prime}(\overline{x}^{\star})}\right)\mathrm{d}z
+\int_{0}^{x}\left(1-\frac{g^{\prime}(z)}{g^{\prime}(\overline{x})}\right)\mathrm{d}z.
\end{align}
Noting that the function $$x\mapsto -\int_{0}^{\overline{x}^{\star}}\left(1-\frac{g^{\prime}(z)}{g^{\prime}(\overline{x}^{\star})}\right)\mathrm{d}z
+\int_{0}^{x}\left(1-\frac{g^{\prime}(z)}{g^{\prime}(\overline{x})}\right)\mathrm{d}z$$
is increasing on $[\overline{x}^{\star},(g^{\prime})^{-1}_2(g^{\prime}(\overline{x})))$ and decreasing on $((g^{\prime})^{-1}_2(g^{\prime}(\overline{x}))),\overline{x}]$, $V_{\underline{x}^{\star},\overline{x}^{\star}}(x)-V_{\underline{x},\overline{x}}(x)\geq0$ on $[\overline{x}^{\star},\overline{x})$ holds if and only if 
\begin{eqnarray}\label{eq:condi1.1.3.6}
-\int_{0}^{\overline{x}^{\star}}\left(1-\frac{g^{\prime}(z)}{g^{\prime}(\overline{x}^{\star})}\right)\mathrm{d}z
+\int_{0}^{\overline{x}^{\star}}\left(1-\frac{g^{\prime}(z)}{g^{\prime}(\overline{x})}\right)\mathrm{d}z\geq0,
\end{eqnarray}
and 
\begin{eqnarray}\label{eq:condi1.2.3.6}
-\int_{0}^{\overline{x}^{\star}}\left(1-\frac{g^{\prime}(z)}{g^{\prime}(\overline{x}^{\star})}\right)\mathrm{d}z
+\int_{0}^{\overline{x}}\left(1-\frac{g^{\prime}(z)}{g^{\prime}(\overline{x})}\right)\mathrm{d}z\geq0.
\end{eqnarray}
The first inequality holds, since
\begin{eqnarray}
\hspace{-0.3cm}&&\hspace{-0.3cm}
-\int_{0}^{\overline{x}^{\star}}\left(1-\frac{g^{\prime}(z)}{g^{\prime}(\overline{x}^{\star})}\right)\mathrm{d}z
+\int_{0}^{\overline{x}^{\star}}\left(1-\frac{g^{\prime}(z)}{g^{\prime}(\overline{x})}\right)\mathrm{d}z
\nn\\
\hspace{-0.3cm}&\geq&\hspace{-0.3cm}
-\int_{0}^{\overline{x}^{\star}}\left(1-\frac{g^{\prime}(z)}{g^{\prime}(\overline{x}^{\star})}\right)\mathrm{d}z
+\int_{0}^{\overline{x}^{\star}}\left(1-\frac{g^{\prime}(z)}{g^{\prime}(\overline{x}^{\star})}\right)\mathrm{d}z=0.
\end{eqnarray}
The second inequality holds, since
\begin{align}
&-\int_{0}^{\overline{x}^{\star}}\left(1-\frac{g^{\prime}(z)}{g^{\prime}(\overline{x}^{\star})}\right)\mathrm{d}z
+\int_{0}^{\overline{x}}\left(1-\frac{g^{\prime}(z)}{g^{\prime}(\overline{x})}\right)\mathrm{d}z\nn\\
=&~\int_{0}^{\underline{x}}\left(1-\frac{g^{\prime}(z)}{g^{\prime}(\overline{x})}\right)\mathrm{d}z
-\int_{0}^{\underline{x}^{\star}}\left(1-\frac{g^{\prime}(z)}{g^{\prime}(\overline{x}^{\star})}\right)\mathrm{d}z\nn\\
\geq&~\int_{0}^{\underline{x}}\left(1-\frac{g^{\prime}(z)}{g^{\prime}(\overline{x}^{\star})}\right)\mathrm{d}z
-\int_{0}^{\underline{x}^{\star}}\left(1-\frac{g^{\prime}(z)}{g^{\prime}(\overline{x}^{\star})}\right)\mathrm{d}z\nn\\
=&~\int_{\underline{x}^{\star}}^{\underline{x}}\left(1-\frac{g^{\prime}(z)}{g^{\prime}(\overline{x}^{\star})}\right)\mathrm{d}z\geq0.
\end{align}
\item For $x\in(\overline{x},\infty)$, \eqref{eq:condi1.2.3.6} gives
\begin{align*}
V_{\underline{x}^{\star},\overline{x}^{\star}}(x)-V_{\underline{x},\overline{x}}(x)
&=V_{\underline{x}^{\star},\overline{x}^{\star}}(\overline{x})-V_{\underline{x},\overline{x}}(\overline{x})\\
&=-\int_{0}^{\overline{x}^{\star}}\left(1-\frac{g^{\prime}(z)}{g^{\prime}(\overline{x}^{\star})}\right)\mathrm{d}z+\int_{0}^{\overline{x}}\left(1-\frac{g^{\prime}(z)}{g^{\prime}(\overline{x})}\right)\mathrm{d}z\geq0.
\end{align*}

\end{itemize}
Hence, $V_{\underline{x}^{\star},\overline{x}^{\star}}(x)\geq V_{\underline{x},\overline{x}}(x)$ for all $x\in(0,\infty)$. If $\overline{x}^{\star}\in[x_2,\infty)$, then we have \eqref{eq:v-v.1.3.6}-\eqref{eq:v-v.3.3.6}. 
\item Suppose 
$\overline{x}\in[a_3,\infty)$ such that $\psi(\underline{x},\overline{x})=\beta$. By Lemma \ref{lem:slope.comparison.case.IV}, we have $g^{\prime}(\overline{x})\geq g^{\prime}(\overline{x}^{\star})$.
If $\overline{x}^{\star}\in[x_2,\infty),$ then we have $\overline{x}^{\star}<\overline{x}$ and \eqref{eq:v-v.4.3.6}-\eqref{eq:v-v.6.3.6}. If $\overline{x}^{\star}\in[0,(g^{\prime})^{-1}_2(g^{\prime}(x_2)))$, then $\overline{x}^{\star}<\overline{x}$. 
It follows from the construction of $g$ that \begin{align}\label{psi'y.3.6}
\left\{
\begin{array}{ll}
\frac{\partial}{\partial y}\psi(x,y)=\int_{x}^{y}\frac{g^{\prime}(z)g^{\prime\prime}(y)}{\left[g^{\prime}(y)\right]^{2}}\mathrm{d}z<0,& x<a_2<y<a_3,\\
\frac{\partial}{\partial y}\psi(x,y)=\int_{x}^{y}\frac{g^{\prime}(z)g^{\prime\prime}(y)}{\left[g^{\prime}(y)\right]^{2}}\mathrm{d}z>0,& 0\leq x<a_3<y.
\end{array}
\right.
\end{align}

For $x\in (0,\overline{x}^{\star}]$, we have \eqref{eq:v-v.4.3.6}. For $x\in(\overline{x}^{\star},\overline{x}]$,
\begin{align} \label{eq:comparison.high.upper.case.IV}V_{\underline{x}^{\star},\overline{x}^{\star}}(x)-V_{\underline{x},\overline{x}}(x)=&~
\frac{g(\underline{x}^{\star})}{g^{\prime}(\overline{x}^{\star})}+x-\underline{x}^{\star}-\beta-\frac{g(x)}{g^{\prime}(\overline{x})}
\nn\\
=&~-\int_{0}^{\overline{x}^{\star}}\left(1-\frac{g^{\prime}(z)}{g^{\prime}(\overline{x}^{\star})}\right)\mathrm{d}z
+\int_{0}^{x}\left(1-\frac{g^{\prime}(z)}{g^{\prime}(\overline{x})}\right)\mathrm{d}z.
\end{align}
If $g^{\prime}(a_2)\leq g^{\prime}(\overline{x}),$ then $1-g^{\prime}(x)/g^{\prime}(\overline{x})\geq0$ on $[\overline{x}^{\star},\overline{x}]$, so \eqref{eq:comparison.high.upper.case.IV} is increasing. If $g^{\prime}(a_2)>g^{\prime}(\overline{x})$, the function \eqref{eq:comparison.high.upper.case.IV}
is increasing on $[\overline{x}^{\star},(g^{\prime})^{-1}_2(g^{\prime}(\overline{x})))$, decreasing on $((g^{\prime})^{-1}_2(g^{\prime}(\overline{x})),(g^{\prime})^{-1}_3(g^{\prime}(\overline{x})))$, and increasing on $((g^{\prime})^{-1}_3(g^{\prime}(\overline{x}))),\overline{x})$. At $x=\overline{x}^{\star}$, we have \begin{eqnarray}
\hspace{-0.3cm}&&\hspace{-0.3cm}
-\int_{0}^{\overline{x}^{\star}}\left(1-\frac{g^{\prime}(z)}{g^{\prime}(\overline{x}^{\star})}\right)\mathrm{d}z
+\int_{0}^{\overline{x}^{\star}}\left(1-\frac{g^{\prime}(z)}{g^{\prime}(\overline{x})}\right)\mathrm{d}z\nn\\
\hspace{-0.3cm}&\geq&\hspace{-0.3cm}
-\int_{0}^{\overline{x}^{\star}}\left(1-\frac{g^{\prime}(z)}{g^{\prime}(\overline{x}^{\star})}\right)\mathrm{d}z
+\int_{0}^{\overline{x}^{\star}}\left(1-\frac{g^{\prime}(z)}{g^{\prime}(\overline{x}^{\star})}\right)\mathrm{d}z=0.
\end{eqnarray}
\begin{itemize}
\item Suppose $\psi(0,a_3)>\psi(0,\overline{x}^{\star})$. 
It follows from \eqref{psi'y.3.6} that
\begin{eqnarray}\label{eq:3.6.1}
\psi(0,\overline{x}^{\star})<\psi(0,a_3)<\psi(0,\overline{x}),
\end{eqnarray}
and 
\begin{eqnarray}\label{eq:condi1.6.3.6}
\psi(0,a_3)=\int_0^{a_3}\left(1-\frac{g^{\prime}(z)}{g^{\prime}(a_3)}\right)\mathrm{d}z<\int_0^{(g^{\prime})^{-1}_3(g^{\prime}(\overline{x}))}\left(1-\frac{g^{\prime}(z)}{g^{\prime}(\overline{x})}\right)\mathrm{d}z.
\end{eqnarray}
By \eqref{eq:3.6.1} and \eqref{eq:condi1.6.3.6}, we have
\begin{align}
-&~\int_{0}^{\overline{x}^{\star}}\left(1-\frac{g^{\prime}(z)}{g^{\prime}(\overline{x}^{\star})}\right)\mathrm{d}z
+\int_{0}^{(g^{\prime})^{-1}_3(g^{\prime}(\overline{x}))}\left(1-\frac{g^{\prime}(z)}{g^{\prime}(\overline{x})}\right)\mathrm{d}z\nn\\
&>-\psi(0,\overline{x}^{\star})+\psi(0,a_3)>0,
\end{align}
yields that $V_{\underline{x}^{\star},\overline{x}^{\star}}(x)-V_{\underline{x},\overline{x}}(x)>0$ for $x\in (\overline{x}^{\star},\overline{x})$. 

\item 
Suppose $\psi(0,a_3)\leq \psi(0,\overline{x}^{\star})$. 
By \textbf{Assumption B}, at $x=(g^{\prime})_3^{-1}(g^{\prime}(\overline{x}))$,
\begin{align*}
&-\int_{0}^{\overline{x}^{\star}}\left(1-\frac{g^{\prime}(z)}{g^{\prime}(\overline{x}^{\star})}\right)\mathrm{d}z
+\int_{0}^{(g^{\prime})_3^{-1}(g^{\prime}(\overline{x}))}\left(1-\frac{g^{\prime}(z)}{g^{\prime}(\overline{x})}\right)\mathrm{d}z\\
&=\psi\left(0,(g^{\prime})_3^{-1}(g^{\prime}(\overline{x}))\right)-\psi(0,\overline{x}^{\star})\geq0.
\end{align*}
Thus, $V_{\underline{x}^{\star},\overline{x}^{\star}}(x)-V_{\underline{x},\overline{x}}(x)\geq0$ on $(\overline{x}^{\star},\overline{x}]$. 
\end{itemize}
For $x\in[\overline{x},\infty)$,
\begin{align*}
V_{\underline{x}^{\star},\overline{x}^{\star}}(x)-V_{\underline{x},\overline{x}}(x)
&=V_{\underline{x}^{\star},\overline{x}^{\star}}(\overline{x})-V_{\underline{x},\overline{x}}(\overline{x})\\
&=\int_0^{\overline{x}}\left(1-\frac{g^{\prime}(z)}{g^{\prime}(\overline{x})}\right)\mathrm{d}z
-\int_0^{\overline{x}^{\star}}\left(1-\frac{g^{\prime}(z)}{g^{\prime}(\overline{x}^{\star})}\right)\mathrm{d}z\geq0.
\end{align*}
\end{itemize}
Suppose $(\underline{x},\overline{x})$ with $\underline{x}\in[a_1,a_2)$ is
any solution to \eqref{3.20.w2.4} different from $(\underline{x}^{\star},\overline{x}^{\star})$. Then, we have the following observations.
\begin{itemize}
\item Suppose 
$\overline{x}\in (a_1,a_2]$ such that $\psi(\underline{x},\overline{x})=\beta$. By Lemma \ref{lem:slope.comparison.case.IV}, we have $g^{\prime}(\overline{x})\geq g^{\prime}(\overline{x}^{\star})$. Then, if $\overline{x}^{\star}\in[a_1,(g^{\prime})^{-1}_2(g^{\prime}(x_2)))$, we obtain $\overline{x}>\overline{x}^{\star}$ and the inequalities \eqref{eq:v-v.4.3.6}-\eqref{eq:v-v.6.3.6}. Hence, $V_{\underline{x}^{\star},\overline{x}^{\star}}(x)\geq V_{\underline{x},\overline{x}}(x)$ for all $x\in(0,\infty)$. If $\overline{x}^{\star}\in[x_2,\infty)$, then we have $\overline{x}<\overline{x}^{\star}$ and 
\begin{itemize}
\item for $x\in(0,\overline{x}],$
\begin{align}
\label{eq:v-v3.215}
V_{\underline{x}^{\star},\overline{x}^{\star}}(x)-V_{\underline{x},\overline{x}}(x)=\frac{g(x)}{g^{\prime}(\overline{x}^{\star})}-\frac{g(x)}{g^{\prime}(\overline{x})}\geq0;
\end{align}
\item for $x\in[\overline{x},\overline{x}^{\star}]$,
\begin{align}
&V_{\underline{x}^{\star},\overline{x}^{\star}}(x)
-V_{\underline{x},\overline{x}}(x)
\nn\\
&=\frac{g(x)}{g^{\prime}(\overline{x}^{\star})}
-\left(
\frac{g(\overline{x})}{g^{\prime}(\overline{x})}
+x-\overline{x}
\right)
\nn\\
&=\int_0^{\overline{x}}
\left(1-\frac{g^{\prime}(z)}{g^{\prime}(\overline{x})}\right)\mathrm{d}z
-\int_0^x
\left(1-\frac{g^{\prime}(z)}{g^{\prime}(\overline{x}^{\star})}\right)\mathrm{d}z.
\label{eq:comparison.lower.a1.high.candidate}
\end{align} 
At $x=\overline{x}^{\star}$, we have
\begin{align}\label{eq:v-v3.216}
&V_{\underline{x}^{\star},\overline{x}^{\star}}
(\overline{x}^{\star})
-V_{\underline{x},\overline{x}}
(\overline{x}^{\star})\nn\\
&=
\int_0^{\overline{x}}
\left(1-\frac{g^{\prime}(z)}{g^{\prime}(\overline{x})}\right)\mathrm{d}z
-\int_0^{\overline{x}^{\star}}
\left(1-\frac{g^{\prime}(z)}{g^{\prime}(\overline{x}^{\star})}\right)\mathrm{d}z\nn\\
&=
\int_0^{\underline{x}}
\left(1-\frac{g^{\prime}(z)}{g^{\prime}(\overline{x})}\right)\mathrm{d}z
-\int_0^{\underline{x}^{\star}}
\left(1-\frac{g^{\prime}(z)}{g^{\prime}(\overline{x}^{\star})}\right)\mathrm{d}z\nn\\
&\geq\int_0^{(g^{\prime})_1^{-1}(g^{\prime}(\overline{x}))}\left(1-\frac{g^{\prime}(z)}{g^{\prime}(\overline{x})}\right)\mathrm{d}z-\int_0^{(g^{\prime})_1^{-1}(g^{\prime}(\overline{x}^{\star}))}
\left(1-\frac{g^{\prime}(z)}{g^{\prime}(\overline{x}^{\star})}\right)\mathrm{d}z
\geq0.
\end{align}
Here, the first inequality follows from $\underline{x}<\overline{x}=(g^{\prime})_2^{-1}(g^{\prime}(\overline{x}))$, while the second follows from the increasing property of
\begin{align*}
y\mapsto
\int_0^{(g^{\prime})_1^{-1}(y)}
\left(1-\frac{g^{\prime}(z)}{y}\right)\mathrm{d}z.
\end{align*}

It is easy to check the function $x\mapsto V_{\underline{x}^{\star},\overline{x}^{\star}}
(x)
-V_{\underline{x},\overline{x}}
(x)$ increases in $(\overline{x},(g^{\prime})_3^{-1}(g^{\prime}(\overline{x}^{\star})))$ and decreases in $((g^{\prime})_3^{-1}(g^{\prime}(\overline{x}^{\star})),\overline{x}^{\star})$. Combined with \eqref{eq:v-v3.215} and \eqref{eq:v-v3.216}, we have $V_{\underline{x}^{\star},\overline{x}^{\star}}
(x)
-V_{\underline{x},\overline{x}}
(x)\geq0$ on $[\overline{x},\overline{x}^{\star}]$.
\item for $x\in[\overline{x}^{\star},\infty)$, 
\begin{align*}
V_{\underline{x}^{\star},\overline{x}^{\star}}(x)
-V_{\underline{x},\overline{x}}(x)=
V_{\underline{x}^{\star},\overline{x}^{\star}}
(\overline{x}^{\star})
-V_{\underline{x},\overline{x}}
(\overline{x}^{\star})
\geq0.
\end{align*}
\end{itemize}
\item Suppose 
$\overline{x}\in (a_2,a_3]$ such that $\psi(\underline{x},\overline{x})=\beta$. Then, if $\overline{x}^{\star}\in[a_1,(g^{\prime})^{-1}_2(g^{\prime}(x_2)))$, then
$\overline{x}^{\star}<a_2<\overline{x}.$
Moreover,
$\underline{x}^{\star}
=(g^{\prime})_1^{-1}(g^{\prime}(\overline{x}^{\star})).$
Since $\psi(\underline{x},\overline{x})=\beta>0$, we have
$\underline{x}<(g^{\prime})_2^{-1}(g^{\prime}(\overline{x})).$
The monotonicity of $g^{\prime}$ and
\(g^{\prime}(\overline{x})\geq g^{\prime}(\overline{x}^{\star}),\) since Lemma \ref{lem:slope.comparison.case.IV} yield
\begin{align}
\int_0^{\underline{x}}
\left(1-\frac{g^{\prime}(z)}{g^{\prime}(\overline{x})}\right)\mathrm{d}z
&\geq
\int_0^{(g^{\prime})_1^{-1}(g^{\prime}(\overline{x}))}
\left(1-\frac{g^{\prime}(z)}{g^{\prime}(\overline{x})}\right)\mathrm{d}z
\nn\\
&\geq
\int_0^{(g^{\prime})_1^{-1}(g^{\prime}(\overline{x}^{\star}))}
\left(1-\frac{g^{\prime}(z)}
{g^{\prime}(\overline{x}^{\star})}\right)\mathrm{d}z
\nn\\
&=
\int_0^{\underline{x}^{\star}}
\left(1-\frac{g^{\prime}(z)}
{g^{\prime}(\overline{x}^{\star})}\right)\mathrm{d}z.
\label{eq:middle.upper.lower.comparison}
\end{align}
Here, the second inequality follows from the increasing property of
\[
y\mapsto
\int_0^{(g^{\prime})_1^{-1}(y)}
\left(1-\frac{g^{\prime}(z)}{y}\right)\mathrm{d}z.
\]
We have the following inequalities.
\begin{itemize}
\item For $x\in(0,\overline{x}^{\star}]$,
\begin{align*}
V_{\underline{x}^{\star},\overline{x}^{\star}}(x)
-V_{\underline{x},\overline{x}}(x)=g(x)\left(
\frac{1}{g^{\prime}(\overline{x}^{\star})}
-\frac{1}{g^{\prime}(\overline{x})}
\right)\geq0.
\end{align*}
\item For $x\in
(\overline{x}^{\star},\overline{x}]$,
\begin{align}
V_{\underline{x}^{\star},\overline{x}^{\star}}(x)
-V_{\underline{x},\overline{x}}(x)
=
-\int_0^{\overline{x}^{\star}}
\left(1-\frac{g^{\prime}(z)}
{g^{\prime}(\overline{x}^{\star})}\right)\mathrm{d}z
+\int_0^x
\left(1-\frac{g^{\prime}(z)}
{g^{\prime}(\overline{x})}\right)\mathrm{d}z.
\label{eq:middle.upper.value.difference}
\end{align} 
At $x=\overline{x}^{\star}$, the difference is non-negative.
At $x=\overline{x}^{\star}$, the difference is strictly positive.
At $x=\overline{x}$, 
\eqref{eq:middle.upper.lower.comparison} gives
\begin{align*}
&-\int_0^{\overline{x}^{\star}}
\left(1-\frac{g^{\prime}(z)}
{g^{\prime}(\overline{x}^{\star})}\right)\mathrm{d}z
+\int_0^{\overline{x}}
\left(1-\frac{g^{\prime}(z)}
{g^{\prime}(\overline{x})}\right)\mathrm{d}z\\
&=
\int_0^{\underline{x}}
\left(1-\frac{g^{\prime}(z)}
{g^{\prime}(\overline{x})}\right)\mathrm{d}z
-\int_0^{\underline{x}^{\star}}
\left(1-\frac{g^{\prime}(z)}
{g^{\prime}(\overline{x}^{\star})}\right)\mathrm{d}z
\geq0.
\end{align*}
The right-hand side of
\eqref{eq:middle.upper.value.difference}
is increasing on
$[\overline{x}^{\star},
(g^{\prime})_2^{-1}(g^{\prime}(\overline{x}))]$
and decreasing on
$
[(g^{\prime})_2^{-1}(g^{\prime}(\overline{x})),
\overline{x}].
$
Hence, its minimum is attained at one of the two endpoints.
Therefore, \[
V_{\underline{x}^{\star},\overline{x}^{\star}}(x)
-V_{\underline{x},\overline{x}}(x)\geq0,
\qquad
x\in(\overline{x}^{\star},\overline{x}].
\]
\item For $x\in[\overline{x},\infty)$,
\begin{align*}
V_{\underline{x}^{\star},\overline{x}^{\star}}(x)
-V_{\underline{x},\overline{x}}(x)=
V_{\underline{x}^{\star},\overline{x}^{\star}}
(\overline{x})
-V_{\underline{x},\overline{x}}(\overline{x})
\geq0.
\end{align*}
\end{itemize}
If $\overline{x}^{\star}\in[x_2,\infty)$, then $\overline{x}<a_3<\overline{x}^{\star}$. Moreover, $\psi(\underline{x},\overline{x})=\beta>0$ implies $\underline{x}<(g^{\prime})_2^{-1}(g^{\prime}(\overline{x})).$ Consequently,
\begin{align*}
\int_0^{\underline{x}}
\left(1-\frac{g^{\prime}(z)}{g^{\prime}(\overline{x})}\right)\mathrm{d}z&\geq
\int_0^{(g^{\prime})_1^{-1}(g^{\prime}(\overline{x}))}
\left(1-\frac{g^{\prime}(z)}{g^{\prime}(\overline{x})}\right)\mathrm{d}z\\
&\geq
\int_0^{(g^{\prime})_1^{-1}(g^{\prime}(\overline{x}^{\star}))}
\left(1-\frac{g^{\prime}(z)}{g^{\prime}(\overline{x}^{\star})}\right)\mathrm{d}z.
\end{align*}
Here, the first inequality follows from $\underline{x}<(g^{\prime})_2^{-1}(g^{\prime}(\overline{x}))$, while the second follows from the increasing property of
\begin{align*}
y\mapsto
\int_0^{(g^{\prime})_1^{-1}(y)}
\left(1-\frac{g^{\prime}(z)}{y}\right)\mathrm{d}z.
\end{align*}
We now compare the two value functions on each interval.
\begin{itemize}
\item For $x\in(0,\overline{x}]$,
\begin{align*}
V_{\underline{x}^{\star},\overline{x}^{\star}}(x)-V_{\underline{x},\overline{x}}(x)=g(x)\left(\frac{1}{g^{\prime}(\overline{x}^{\star})}-\frac{1}{g^{\prime}(\overline{x})}\right)\geq0.
\end{align*}
\item For $x\in(\overline{x},\overline{x}^{\star}]$,
\begin{align*}
V_{\underline{x}^{\star},\overline{x}^{\star}}(x)-V_{\underline{x},\overline{x}}(x)=V_{\underline{x}^{\star},\overline{x}^{\star}}(\overline{x})-V_{\underline{x},\overline{x}}(\overline{x})+\int_{\overline{x}}^x\left(\frac{g^{\prime}(z)}{g^{\prime}(\overline{x}^{\star})}-1\right)\mathrm{d}z.
\end{align*}
At $x=\overline{x}^{\star}$, the difference is non-negative. At $x=\overline{x}^{\star}$, we have
\begin{align*}
V_{\underline{x}^{\star},\overline{x}^{\star}}(\overline{x}^{\star})-V_{\underline{x},\overline{x}}(\overline{x}^{\star})
&=\int_0^{\underline{x}}\left(1-\frac{g^{\prime}(z)}{g^{\prime}(\overline{x})}\right)\mathrm{d}z-\int_0^{\underline{x}^{\star}}\left(1-\frac{g^{\prime}(z)}{g^{\prime}(\overline{x}^{\star})}\right)\mathrm{d}z\geq0.
\end{align*}
Since $g^{\prime}(\overline{x})\geq g^{\prime}(\overline{x}^{\star})$ and $g^{\prime}$ is strictly decreasing on $[a_2,a_3]$, we have $\overline{x}\leq(g^{\prime})_3^{-1}(g^{\prime}(\overline{x}^{\star})).$
Therefore, the function $x\mapsto V_{\underline{x}^{\star},\overline{x}^{\star}}(x)-V_{\underline{x},\overline{x}}(x)$ is increasing on $[\overline{x},(g^{\prime})_3^{-1}(g^{\prime}(\overline{x}^{\star}))],$ and decreasing on $[(g^{\prime})_3^{-1}(g^{\prime}(\overline{x}^{\star})),\overline{x}^{\star}].$ Thus $V_{\underline{x}^{\star},\overline{x}^{\star}}(x)-V_{\underline{x},\overline{x}}(x)\geq0$ on $(\overline{x},\overline{x}^{\star}].$
\item For $x\in[\overline{x}^{\star},\infty)$,
\begin{align*}
V_{\underline{x}^{\star},\overline{x}^{\star}}(x)-V_{\underline{x},\overline{x}}(x)=V_{\underline{x}^{\star},\overline{x}^{\star}}(\overline{x}^{\star})-V_{\underline{x},\overline{x}}(\overline{x}^{\star})\geq0.
\end{align*}
\end{itemize}

\item Suppose 
$\overline{x}\in (a_3,\infty)$ such that $\psi(\underline{x},\overline{x})=\beta$. By Lemma \ref{lem:slope.comparison.case.IV}, we have $g^{\prime}(\overline{x})\geq g^{\prime}(\overline{x}^{\star})$. Then, if $\overline{x}^{\star}\in[a_1,(g^{\prime})^{-1}_2(g^{\prime}(x_2)))$, then \(\overline{x}^{\star}<a_2<a_3<\overline{x}\). We have the following inequalities.
\begin{itemize}
\item 
For $x\in(0,\overline{x}^{\star}]$,
\begin{align}
\label{eq:comparison.a1.high.upper.low.1}
V_{\underline{x}^{\star},\overline{x}^{\star}}(x)
-
V_{\underline{x},\overline{x}}(x)
=
g(x)\left(
\frac{1}{g^{\prime}(\overline{x}^{\star})}
-
\frac{1}{g^{\prime}(\overline{x})}
\right)\geq0.
\end{align}

\item For \(x\in(\overline{x}^{\star},\overline{x}]\),
\begin{align}\label{eq:comparison.a1.high.upper}
V_{\underline{x}^{\star},\overline{x}^{\star}}(x)
-
V_{\underline{x},\overline{x}}(x)=
-\int_0^{\overline{x}^{\star}}
\left(
1-\frac{g^{\prime}(z)}{g^{\prime}(\overline{x}^{\star})}
\right)\mathrm dz
+
\int_0^x
\left(
1-\frac{g^{\prime}(z)}{g^{\prime}(\overline{x})}
\right)\mathrm dz .
\end{align} 
If
$g^{\prime}(\overline{x})\geq g^{\prime}(a_2),$
then $g^{\prime}(z)\leq g^{\prime}(\overline{x})$ for $z\in[\overline{x}^{\star},\overline{x}]$.
Consequently,
$1-\frac{g^{\prime}(z)}{g^{\prime}(\overline{x})}\geq0$ for $z\in[\overline{x}^{\star},\overline{x}]$, and the right-hand side of \eqref{eq:comparison.a1.high.upper} is increasing on $[\overline{x}^{\star},\overline{x}]$. Therefore,
\begin{align*}
V_{\underline{x}^{\star},\overline{x}^{\star}}(x)
-V_{\underline{x},\overline{x}}(x)\geq0,
\qquad
x\in(\overline{x}^{\star},\overline{x}].
\end{align*}
If $g^{\prime}(\overline{x})<g^{\prime}(a_2),$
then, since $\overline{x}>a_3$, $g^{\prime}(\overline{x})\in[g^{\prime}(a_3),g^{\prime}(a_2)).$
Hence, we have
\begin{align*}
\overline{x}^{\star}<(g^{\prime})_2^{-1}(g^{\prime}(\overline{x}))<(g^{\prime})_3^{-1}(g^{\prime}(\overline{x}))<\overline{x}.
\end{align*}
The right-hand side of \eqref{eq:comparison.a1.high.upper} is increasing on
$[\overline{x}^{\star},(g^{\prime})_2^{-1}(g^{\prime}(\overline{x}))],$ and
decreasing on $[(g^{\prime})_2^{-1}(g^{\prime}(\overline{x})),(g^{\prime})_3^{-1}(g^{\prime}(\overline{x}))],$
and then increasing on
$[(g^{\prime})_3^{-1}(g^{\prime}(\overline{x})),\overline{x}].$
\begin{itemize}
\item Suppose $\psi(0,a_3)>\psi(0,\overline{x}^{\star})$. It follows from \eqref{psi'y.3.6} that $\psi(0,\overline{x}^{\star})<\psi(0,a_3)<\psi(0,\overline{x}),$
and
\begin{align*}
\psi(0,a_3)
=\int_0^{a_3}\left(1-\frac{g^{\prime}(z)}{g^{\prime}(a_3)}\right)\mathrm{d}z<\int_0^{(g^{\prime})_3^{-1}(g^{\prime}(\overline{x}))}
\left(1-\frac{g^{\prime}(z)}{g^{\prime}(\overline{x})}\right)\mathrm{d}z.
\end{align*}
Consequently, 
\begin{align*}
&-\int_0^{\overline{x}^{\star}}\left(1-\frac{g^{\prime}(z)}{g^{\prime}(\overline{x}^{\star})}\right)\mathrm{d}z+\int_0^{(g^{\prime})_3^{-1}(g^{\prime}(\overline{x}))}\left(1-\frac{g^{\prime}(z)}{g^{\prime}(\overline{x})}\right)\mathrm{d}z\\
&>-\psi(0,\overline{x}^{\star})+\psi(0,a_3)>0.
\end{align*}
\item Suppose $\psi(0,a_3)\leq\psi(0,\overline{x}^{\star})$.
By \textbf{Assumption B}, at $x=(g^{\prime})_3^{-1}(g^{\prime}(\overline{x}))$,
\begin{align*}
&-\int_0^{\overline{x}^{\star}}
\left(1-\frac{g^{\prime}(z)}{g^{\prime}(\overline{x}^{\star})}\right)\mathrm{d}z+\int_0^{(g^{\prime})_3^{-1}(g^{\prime}(\overline{x}))}\left(1-\frac{g^{\prime}(z)}{g^{\prime}(\overline{x})}\right)\mathrm{d}z\\
&=\psi\left(0,(g^{\prime})_3^{-1}(g^{\prime}(\overline{x}))\right)-\psi(0,\overline{x}^{\star})\geq0.
\end{align*}
\end{itemize}

Consequently,
$
V_{\underline{x}^{\star},\overline{x}^{\star}}(x)-V_{\underline{x},\overline{x}}(x)\geq0,$ for $x\in(\overline{x}^{\star},\overline{x}].$
\item For $x\in[\overline{x},\infty)$, we have
\begin{align}\label{eq:comparison.a1.high.upper.low.3}
V_{\underline{x}^{\star},\overline{x}^{\star}}(x)-V_{\underline{x},\overline{x}}(x)=V_{\underline{x}^{\star},\overline{x}^{\star}}(\overline{x})-V_{\underline{x},\overline{x}}(\overline{x})\geq0.
\end{align}
\end{itemize}
If $\overline{x}^{\star}\in[x_2,\infty)$.
	Since both $\overline{x}^{\star}$ and $\overline{x}$ belong to
	$(a_3,\infty)$ and $g^{\prime}$ is strictly increasing on
	$[a_3,\infty)$, it follows from
	$g^{\prime}(\overline{x})>g^{\prime}(\overline{x}^{\star})$ that $\overline{x}>\overline{x}^{\star}.$ We have the following inequalities.
\begin{itemize}
\item For $x\in(0,\overline{x}^{\star}]$,
\begin{align}\label{eq:comparison.a1.high.upper.high.1}
V_{\underline{x}^{\star},\overline{x}^{\star}}(x)-V_{\underline{x},\overline{x}}(x)=g(x)\left(\frac{1}{g^{\prime}(\overline{x}^{\star})}-\frac{1}{g^{\prime}(\overline{x})}\right)>0.
\end{align}
\item For $x\in(\overline{x}^{\star},\overline{x}]$,
\begin{align*}
V_{\underline{x}^{\star},\overline{x}^{\star}}(x)-V_{\underline{x},\overline{x}}(x)=-\int_0^{\overline{x}^{\star}}
\left(1-\frac{g^{\prime}(z)}{g^{\prime}(\overline{x}^{\star})}\right)\mathrm{d}z+\int_0^x
\left(1-\frac{g^{\prime}(z)}{g^{\prime}(\overline{x})}\right)\mathrm{d}z.
\end{align*}
At $x=\overline{x}^{\star}$, the difference is strictly positive.
Moreover, because $g^{\prime}$ is strictly increasing on $[a_3,\infty)$,
$1-\frac{g^{\prime}(x)}{g^{\prime}(\overline{x})}>0,$ for 
$x\in[\overline{x}^{\star},\overline{x}).$
		Thus the difference is strictly increasing on
		$[\overline{x}^{\star},\overline{x}]$, and hence
		\begin{align*}
			V_{\underline{x}^{\star},\overline{x}^{\star}}(x)
			-V_{\underline{x},\overline{x}}(x)>0,
			\qquad
			x\in(\overline{x}^{\star},\overline{x}].
		\end{align*}
\item For $x\in[\overline{x},\infty)$,
		\begin{align}\label{eq:comparison.a1.high.upper.high.3}
			V_{\underline{x}^{\star},\overline{x}^{\star}}(x)
			-V_{\underline{x},\overline{x}}(x)=
			V_{\underline{x}^{\star},\overline{x}^{\star}}
			(\overline{x})
			-V_{\underline{x},\overline{x}}(\overline{x})
			>0.
		\end{align}
\end{itemize}
\end{itemize}
Suppose $(\underline{x},\overline{x})$ with $\underline{x}\in[a_2,\infty)$ is
any solution to \eqref{3.20.w2.4} different from $(\underline{x}^{\star},\overline{x}^{\star})$. Note that we necessarily have $\overline{x}>a_3$. Then, we have the following observations.
			\begin{itemize}
				\item If $\overline{x}^{\star}\in[a_1,(g^{\prime})^{-1}_2(g^{\prime}(x_2)))$, then
				$\overline{x}^{\star}<a_2<a_3<\overline{x}$. By Lemma
				\ref{lem:slope.comparison.case.IV}, we have
				$g^{\prime}(\overline{x})\geq g^{\prime}(\overline{x}^{\star})$. By arguments analogous
				to those used for $\underline{x}\in[a_1,a_2)$ and
				$\overline{x}\in(a_3,\infty)$, where \textbf{Assumption B} is used when
				$g^{\prime}(\overline{x})<g^{\prime}(a_2)$ and $\psi(0,a_3)\leq\psi(0,\overline{x}^{\star})$ to verify the non-negativity at
				$(g^{\prime})_3^{-1}(g^{\prime}(\overline{x}))$, we obtain from
				\eqref{eq:comparison.a1.high.upper.low.1}--\eqref{eq:comparison.a1.high.upper.low.3} that
				\[
				V_{\underline{x}^{\star},\overline{x}^{\star}}(x)
				-
				V_{\underline{x},\overline{x}}(x)\geq0,
				\qquad x\in[0,\infty).
				\]
				\item If $\overline{x}^{\star}\in[x_2,\infty)$, then
				$\overline{x}^{\star}>a_3$. Since $\overline{x}>a_3$ and $g^{\prime}$ is
				strictly increasing on $[a_3,\infty)$, Lemma
				\ref{lem:slope.comparison.case.IV} implies
				$\overline{x}>\overline{x}^{\star}$. By arguments analogous to
				those used for $\underline{x}\in[a_1,a_2)$ and
				$\overline{x}\in(a_3,\infty)$, we obtain from
				\eqref{eq:comparison.a1.high.upper.high.1}--\eqref{eq:comparison.a1.high.upper.high.3} that
				\[
				V_{\underline{x}^{\star},\overline{x}^{\star}}(x)
				-
				V_{\underline{x},\overline{x}}(x)\geq0,
				\qquad x\in[0,\infty).
				\]
			\end{itemize}
Suppose $(\underline{x}^{\star},\overline{x}^{\star})$ is such that $\underline{x}^{\star}
			=(g^{\prime})_3^{-1}(g^{\prime}(\overline{x}^{\star}))\in[a_2,a_3),$ and $\overline{x}^{\star}\in(a_3,x_1]$.
Then $\underline{x}^{\star}=(g^{\prime})^{-1}_1(g^{\prime}(\overline{x}^{\star}))$. 	Let $(\underline{x},\overline{x})$ be any solution to \eqref{3.20.w2.4}. By Lemma
			\ref{lem:slope.comparison.case.IV}, we have
			$g^{\prime}(\overline{x})\geq g^{\prime}(\overline{x}^{\star})$.
\begin{itemize}
\item If $\overline{x}\geq a_3$, then the increasing property of $g^{\prime}$ on
			$[a_3,\infty)$ yields $\overline{x}\geq\overline{x}^{\star}$. For
			$x\in(0,\overline{x}^{\star}]$,
			\begin{align*}
			V_{\underline{x}^{\star},\overline{x}^{\star}}(x)
			-V_{\underline{x},\overline{x}}(x)
			=g(x)\left(\frac{1}{g^{\prime}(\overline{x}^{\star})}
			-\frac{1}{g^{\prime}(\overline{x})}\right)\geq0.
			\end{align*}
			For $x\in(\overline{x}^{\star},\overline{x}]$,
			\begin{align*}
			V_{\underline{x}^{\star},\overline{x}^{\star}}(x)
			-V_{\underline{x},\overline{x}}(x)=V_{\underline{x}^{\star},\overline{x}^{\star}}
			(\overline{x}^{\star})
			-V_{\underline{x},\overline{x}}(\overline{x}^{\star})
			+\int_{\overline{x}^{\star}}^{x}
			\left(1-\frac{g^{\prime}(z)}{g^{\prime}(\overline{x})}\right)\mathrm{d}z
			\geq0,
			\end{align*}
For $x\in[\overline{x},\infty)$, $$V_{\underline{x}^{\star},\overline{x}^{\star}}(x)
-V_{\underline{x},\overline{x}}(x)
=
V_{\underline{x}^{\star},\overline{x}^{\star}}(\overline{x})
-V_{\underline{x},\overline{x}}(\overline{x})
\geq0,
\qquad x\in[\overline{x},\infty).$$
\item If $\overline{x}<a_3$, then necessarily $\overline{x}>a_1$. 
We have the following inequalities.
\begin{itemize}
\item For $x\in(0,\overline{x}]$,
\begin{align}
V_{\underline{x}^{\star},\overline{x}^{\star}}(x)
-V_{\underline{x},\overline{x}}(x)\nn=g(x)\left(\frac{1}{g^{\prime}(\overline{x}^{\star})}
-\frac{1}{g^{\prime}(\overline{x})}\right)\geq0.
\end{align}
\item For $x\in(\overline{x},\underline{x}^{\star}]$, we have
\begin{align}
\overline{x}\leq\underline{x}^{\star},\qquad
g^{\prime}(z)\geq g^{\prime}(\overline{x}^{\star}),
				\quad z\in[\overline{x},\underline{x}^{\star}].
				\label{eq:third.candidate.low.competitor.shape.1}
			\end{align}
Indeed, if $\overline{x}\in(a_1,a_2)$, the increasing property of $g^{\prime}$ on $[a_1,a_2]$ and the decreasing property of $g^{\prime}$ on $[a_2,a_3]$ imply \eqref{eq:third.candidate.low.competitor.shape.1}. If $\overline{x}\in[a_2,a_3)$, the decreasing property of $g^{\prime}$ on $[a_2,a_3]$ implies \eqref{eq:third.candidate.low.competitor.shape.1}. Thus, by \eqref{eq:third.candidate.low.competitor.shape.1},
\begin{align}
V_{\underline{x}^{\star},\overline{x}^{\star}}(x)
-V_{\underline{x},\overline{x}}(x)=V_{\underline{x}^{\star},\overline{x}^{\star}}(\overline{x})
-V_{\underline{x},\overline{x}}(\overline{x})
+\int_{\overline{x}}^{x}
\left(\frac{g^{\prime}(z)}{g^{\prime}(\overline{x}^{\star})}-1\right)\mathrm{d}z\geq0.
\label{eq:third.candidate.low.competitor.2}
\end{align}
\item For $x\in(\underline{x}^{\star},\overline{x}^{\star}]$, it follows from the construction of $g$ that
\begin{align}
g^{\prime}(z)\leq g^{\prime}(\overline{x}^{\star}),
\qquad z\in[\underline{x}^{\star},\overline{x}^{\star}].
\label{eq:third.candidate.low.competitor.shape.2}
\end{align}
Hence, the function $x\mapsto V_{\underline{x}^{\star},\overline{x}^{\star}}(x)-V_{\underline{x},\overline{x}}(x)$
is decreasing. 
We claim that
			\begin{align}
				\int_{\underline{x}}^{\underline{x}^{\star}}
				\left(1-\frac{g^{\prime}(z)}{g^{\prime}(\overline{x}^{\star})}\right)\mathrm{d}z
				\leq0.
				\label{eq:third.candidate.low.competitor.integral}
			\end{align}
Indeed, if $g^{\prime}(\overline{x}^{\star})<g^{\prime}(a_1)$, then $g^{\prime}(z)>g^{\prime}(\overline{x}^{\star})$ for $z\in[0,\underline{x}^{\star})$, and \eqref{eq:third.candidate.low.competitor.integral} follows immediately; if $g^{\prime}(\overline{x}^{\star})\geq g^{\prime}(a_1)$, then the continuity and the definition of $x_1$, together with
$\overline{x}^{\star}\leq x_1$, yield
\begin{align}
\int_{(g^{\prime})_1^{-1}(g^{\prime}(\overline{x}^{\star}))}^{\underline{x}^{\star}}
\left(1-\frac{g^{\prime}(z)}{g^{\prime}(\overline{x}^{\star})}\right)\mathrm{d}z\leq0.
\label{eq:third.candidate.x1.integral}
\end{align}
For $\underline{x}\leq(g^{\prime})_1^{-1}(g^{\prime}(\overline{x}^{\star}))$,
			\eqref{eq:third.candidate.low.competitor.integral} follows from
			\eqref{eq:third.candidate.x1.integral} and
			$g^{\prime}(z)\geq g^{\prime}(\overline{x}^{\star})$ on
			$[\underline{x},(g^{\prime})_1^{-1}(g^{\prime}(\overline{x}^{\star}))]$.
For $(g^{\prime})_1^{-1}(g^{\prime}(\overline{x}^{\star}))<\underline{x}<(g^{\prime})_2^{-1}(g^{\prime}(\overline{x}^{\star})),$ 	we have
			\begin{align*}
				&\int_{\underline{x}}^{\underline{x}^{\star}}
				\left(1-\frac{g^{\prime}(z)}{g^{\prime}(\overline{x}^{\star})}\right)\mathrm{d}z\nn\\
				&=\int_{(g^{\prime})_1^{-1}(g^{\prime}(\overline{x}^{\star}))}^{\underline{x}^{\star}}
				\left(1-\frac{g^{\prime}(z)}{g^{\prime}(\overline{x}^{\star})}\right)\mathrm{d}z
				-\int_{(g^{\prime})_1^{-1}(g^{\prime}(\overline{x}^{\star}))}^{\underline{x}}
				\left(1-\frac{g^{\prime}(z)}{g^{\prime}(\overline{x}^{\star})}\right)\mathrm{d}z
				\leq0.
			\end{align*}
Finally, if
			$\underline{x}\geq(g^{\prime})_2^{-1}(g^{\prime}(\overline{x}^{\star}))$, then
			$\underline{x}<\overline{x}\leq\underline{x}^{\star}$ and
			$g^{\prime}(z)\geq g^{\prime}(\overline{x}^{\star})$ on
			$[\underline{x},\underline{x}^{\star}]$, which again proves
			\eqref{eq:third.candidate.low.competitor.integral}.
At $x=\overline{x}^{\star}$, using
$\psi(\underline{x}^{\star},\overline{x}^{\star})=\beta$ and \eqref{eq:third.candidate.low.competitor.integral}, we obtain
\begin{align}
&V_{\underline{x}^{\star},\overline{x}^{\star}}(\overline{x}^{\star})
-V_{\underline{x},\overline{x}}(\overline{x}^{\star})\nn\\
&=\underline{x}-\underline{x}^{\star}
					+\frac{g(\underline{x}^{\star})}{g^{\prime}(\overline{x}^{\star})}
					-\frac{g(\underline{x})}{g^{\prime}(\overline{x})}\nn\\
					&=-\int_{\underline{x}}^{\underline{x}^{\star}}
					\left(1-\frac{g^{\prime}(z)}{g^{\prime}(\overline{x}^{\star})}\right)\mathrm{d}z
					+g(\underline{x})
					\left(\frac{1}{g^{\prime}(\overline{x}^{\star})}
					-\frac{1}{g^{\prime}(\overline{x})}\right)\geq0.
					\label{eq:third.candidate.low.competitor.3}
				\end{align}
\item For $x\in[\overline{x}^{\star},\infty)$, \begin{align}
					V_{\underline{x}^{\star},\overline{x}^{\star}}(x)
					-V_{\underline{x},\overline{x}}(x)
					&=V_{\underline{x}^{\star},\overline{x}^{\star}}(\overline{x}^{\star})
					-V_{\underline{x},\overline{x}}(\overline{x}^{\star})
					\geq0.
					\label{eq:third.candidate.low.competitor.4}
				\end{align}
\end{itemize}
\end{itemize}
The proof is complete.
\end{proof}

\begin{remark} 
From the proof of Theorem \ref{Theorem3.6}, it is clear that, in cases $(\text{\emph{IV}}_{6})$-$(\text{\emph{IV}}_{7})$ with $\overline{x}^{\star}< a$ and $\psi(0,a_3)\leq\psi(0,\overline{x}^{\star})$, \textbf{Assumption B} is necessary and sufficient for the existence of an optimal strategy for the auxiliary stochastic control problem \eqref{eq:a008.aux.} 
\end{remark}

\section{Characterization of the optimal impulsive strategy}
\label{sec.4}

In this section, we accomplish Step (II) specified at the beginning of Section \ref{sec.3}. Namely, we verify that the optimal reinsurance-dividend strategy obtained for the auxiliary control problem \eqref{eq:a008.aux.} in Section \ref{sec.3} is also optimal for the original control problem \eqref{eq:a008} by showing that the value function of \eqref{eq:a008.aux.} satisfies the HJB equation \eqref{7.17.HJB} associated with \eqref{eq:a008}.

Prior to presenting these results in Theorem \ref{thm4.1}, we establish the following technical lemma, which serves as a key component in its proof.

\begin{lemma}
\label{lem.4.1}
Let $(\underline{x}^{\star},\overline{x}^{\star})$ be the candidate optimal dividend barriers given by Theorems \ref{thm3.2.ww}, \ref{Theorem3.4}, \ref{Theorem3.5}, and \ref{Theorem3.6}. We have
\begin{eqnarray}\label{eq:4.2}
\int_{y}^{x}\left(1-\frac{g^{\prime}(z)}{g^{\prime}(\overline{x}^{\star})}\right)\mathrm{d}z\leq \beta,\quad \beta\leq y+\beta<x<\infty.
\end{eqnarray}
\end{lemma}

\begin{proof}
We verify \eqref{eq:4.2} only in the case where $0<a_{1}<a_{2}<a_{3}$; other cases can be verified similarly. To this end, we consider the following cases.
\begin{itemize}

\item Case $0\leq \underline{x}^{\star}<\overline{x}^{\star}<a_{2}$ and $y\in[0,\underline{x}^{\star}]$.

If $g^{\prime}(\overline{x}^{\star})>g^{\prime}(a_{3})$, then the function $x\mapsto \int_{y}^{x}\left(1-\frac{g^{\prime}(z)}{g^{\prime}(\overline{x}^{\star})}\right)\mathrm{d}z$ is decreasing on $[y,\underline{x}^{\star}]$, increasing on $[\underline{x}^{\star},\overline{x}^{\star}]$, decreasing on $[\overline{x}^{\star},(g^{\prime})_{3}^{-1}(g^{\prime}(\overline{x}^{\star}))]$, increasing on $[(g^{\prime})_{3}^{-1}(g^{\prime}(\overline{x}^{\star})),(g^{\prime})_{4}^{-1}(g^{\prime}(\overline{x}^{\star}))]$, and decreasing on $[(g^{\prime})_{4}^{-1}(g^{\prime}(\overline{x}^{\star})),\infty)$. In addition, since $g^{\prime}(\overline{x}^{\star})<g^{\prime}(x_{2})$, we have
\begin{align}
\int_{y}^{y}\left(1-\frac{g^{\prime}(z)}{g^{\prime}(\overline{x}^{\star})}\right)\mathrm{d}z&=0\leq \beta, \label{7.20.4.3}\\
\int_{y}^{\overline{x}^{\star}}\left(1-\frac{g^{\prime}(z)}{g^{\prime}(\overline{x}^{\star})}\right)\mathrm{d}z &\leq \int_{\underline{x}^{\star}}^{\overline{x}^{\star}}\left(1-\frac{g^{\prime}(z)}{g^{\prime}(\overline{x}^{\star})}\right)\mathrm{d}z=\beta, \label{7.20.4.4}\\
\int_{y}^{(g^{\prime})_{4}^{-1}(g^{\prime}(\overline{x}^{\star}))}\left(1-\frac{g^{\prime}(z)}{g^{\prime}(\overline{x}^{\star})}\right)\mathrm{d}z &\leq \int_{\underline{x}^{\star}}^{(g^{\prime})_{4}^{-1}(g^{\prime}(\overline{x}^{\star}))}\left(1-\frac{g^{\prime}(z)}{g^{\prime}(\overline{x}^{\star})}\right)\mathrm{d}z
\nn\\
&\leq \psi(\underline{x}^{\star},\overline{x}^{\star})+
\int_{\overline{x}^{\star}}^{(g^{\prime})_{4}^{-1}(g^{\prime}(\overline{x}^{\star}))}\left(1-\frac{g^{\prime}(z)}{g^{\prime}(x_{2})}\right)\mathrm{d}z
\nn\\
&\leq\beta.\label{7.20.4.5}
\end{align}
Therefore, \eqref{eq:4.2} holds true in this sub-case. 

If $g^{\prime}(\overline{x}^{\star})\leq g^{\prime}(a_{3})$, then the function $x\mapsto \int_{y}^{x}\left(1-\frac{g^{\prime}(z)}{g^{\prime}(\overline{x}^{\star})}\right)\mathrm{d}z$ is decreasing on $[y,\underline{x}^{\star}]$, increasing on $[\underline{x}^{\star},\overline{x}^{\star}]$, decreasing on $[\overline{x}^{\star},\infty)$. In addition, \eqref{7.20.4.4} holds. Therefore, \eqref{eq:4.2} holds true in this sub-case. 

\item Case $0\leq \underline{x}^{\star}<\overline{x}^{\star}<a_{2}$ and $y\in(\underline{x}^{\star},\overline{x}^{\star}]$.

If $g^{\prime}(\overline{x}^{\star})>g^{\prime}(a_{3})$, then the function $x\mapsto \int_{y}^{x}\left(1-\frac{g^{\prime}(z)}{g^{\prime}(\overline{x}^{\star})}\right)\mathrm{d}z$ is increasing on $[y,\overline{x}^{\star}]$, decreasing on $[\overline{x}^{\star},(g^{\prime})_{3}^{-1}(g^{\prime}(\overline{x}^{\star}))]$, increasing on $[(g^{\prime})_{3}^{-1}(g^{\prime}(\overline{x}^{\star})),(g^{\prime})_{4}^{-1}(g^{\prime}(\overline{x}^{\star}))]$, and decreasing on $[(g^{\prime})_{4}^{-1}(g^{\prime}(\overline{x}^{\star})),\infty)$. In addition, we have \eqref{7.20.4.4} and \eqref{7.20.4.5}, which
implies that \eqref{eq:4.2} holds true in this sub-case. 

If $g^{\prime}(\overline{x}^{\star})\leq g^{\prime}(a_{3})$, then the function $x\mapsto \int_{y}^{x}\left(1-\frac{g^{\prime}(z)}{g^{\prime}(\overline{x}^{\star})}\right)\mathrm{d}z$ is increasing on $[y,\overline{x}^{\star}]$, decreasing on $[\overline{x}^{\star},\infty)$. In addition, \eqref{7.20.4.4} holds. Therefore, \eqref{eq:4.2} holds true in this sub-case. 

\item Case $0\leq \underline{x}^{\star}<\overline{x}^{\star}<a_{2}$ and $y\in(\overline{x}^{\star},\infty)$.

If $g^{\prime}(\overline{x}^{\star})>g^{\prime}(a_{3})$, then the function $x\mapsto \int_{y}^{x}\left(1-\frac{g^{\prime}(z)}{g^{\prime}(\overline{x}^{\star})}\right)\mathrm{d}z$ is decreasing on $[y, y\vee (g^{\prime})_{3}^{-1}(g^{\prime}(\overline{x}^{\star}))]$, increasing on $[y\vee (g^{\prime})_{3}^{-1}(g^{\prime}(\overline{x}^{\star})),y\vee (g^{\prime})_{4}^{-1}(g^{\prime}(\overline{x}^{\star}))]$, and decreasing on $[y\vee (g^{\prime})_{4}^{-1}(g^{\prime}(\overline{x}^{\star})),\infty)$. In addition, we have \eqref{7.20.4.5}. Consequently, \eqref{eq:4.2} holds true in this sub-case. 

If $g^{\prime}(\overline{x}^{\star})\leq g^{\prime}(a_{3})$, then the function $x\mapsto \int_{y}^{x}\left(1-\frac{g^{\prime}(z)}{g^{\prime}(\overline{x}^{\star})}\right)\mathrm{d}z$ is decreasing on $[y,\infty)$. Therefore, \eqref{eq:4.2} holds true in this sub-case.

\item Case $a_{2}< \underline{x}^{\star}<a_{3}<\overline{x}^{\star}$ and $x\in[\overline{x}^{\star},\infty)$.

If $g^{\prime}(\overline{x}^{\star})> g^{\prime}(a_{1})$, then the function $y\mapsto \int_{y}^{x}\left(1-\frac{g^{\prime}(z)}{g^{\prime}(\overline{x}^{\star})}\right)\mathrm{d}z$ is increasing on $[0,(g^{\prime})_{1}^{-1}(g^{\prime}(\overline{x}^{\star}))]$, decreasing on $[(g^{\prime})_{1}^{-1}(g^{\prime}(\overline{x}^{\star})),(g^{\prime})_{2}^{-1}(g^{\prime}(\overline{x}^{\star}))]$, increasing on $[(g^{\prime})_{2}^{-1}(g^{\prime}(\overline{x}^{\star})),\underline{x}^{\star}]$, decreasing on $[\underline{x}^{\star},\overline{x}^{\star})$, and increasing on $[\overline{x}^{\star},x)$. In addition, since $g^{\prime}(\overline{x}^{\star})<g^{\prime}(x_{1})$, we have
\begin{align}
\int_{(g^{\prime})_{1}^{-1}(g^{\prime}(\overline{x}^{\star}))}^{x}\left(1-\frac{g^{\prime}(z)}{g^{\prime}(\overline{x}^{\star})}\right)\mathrm{d}z&\leq 
\int_{(g^{\prime})_{1}^{-1}(g^{\prime}(\overline{x}^{\star}))}^{\overline{x}^{\star}}\left(1-\frac{g^{\prime}(z)}{g^{\prime}(\overline{x}^{\star})}\right)\mathrm{d}z
\nn\\
&= 
\int_{(g^{\prime})_{1}^{-1}(g^{\prime}(\overline{x}^{\star}))}^{\underline{x}^{\star}}\left(1-\frac{g^{\prime}(z)}{g^{\prime}(\overline{x}^{\star})}\right)\mathrm{d}z+\psi(\underline{x}^{\star},\overline{x}^{\star})
\nn\\
&\leq 
\int_{(g^{\prime})_{1}^{-1}(g^{\prime}(\overline{x}^{\star}))}^{\underline{x}^{\star}}\left(1-\frac{g^{\prime}(z)}{g^{\prime}(x_{1})}\right)\mathrm{d}z+\beta
\nn\\
&\leq \beta,\label{7.20.4.6}\\
\int_{\underline{x}^{\star}}^{x}\left(1-\frac{g^{\prime}(z)}{g^{\prime}(\overline{x}^{\star})}\right)\mathrm{d}z&\leq 
\int_{\underline{x}^{\star}}^{\overline{x}^{\star}}\left(1-\frac{g^{\prime}(z)}{g^{\prime}(\overline{x}^{\star})}\right)\mathrm{d}z\nn\\
&
=\beta.\label{7.20.4.7.1}
\end{align}
Therefore, \eqref{eq:4.2} holds true in this sub-case.

If $g^{\prime}(\overline{x}^{\star})\leq g^{\prime}(a_{1})$, then the function $y\mapsto \int_{y}^{x}\left(1-\frac{g^{\prime}(z)}{g^{\prime}(\overline{x}^{\star})}\right)\mathrm{d}z$ is increasing on $[0,\underline{x}^{\star}]$, decreasing on $[\underline{x}^{\star},x)$. In addition, we have \eqref{7.20.4.7.1}. Consequently, \eqref{eq:4.2} holds true in this sub-case.

\item Case $a_{2}< \underline{x}^{\star}<a_{3}<\overline{x}^{\star}$ and $x\in[\underline{x}^{\star},\overline{x}^{\star})$.

If $g^{\prime}(\overline{x}^{\star})> g^{\prime}(a_{1})$, then the function $y\mapsto \int_{y}^{x}\left(1-\frac{g^{\prime}(z)}{g^{\prime}(\overline{x}^{\star})}\right)\mathrm{d}z$ is increasing on $[0,(g^{\prime})_{1}^{-1}(g^{\prime}(\overline{x}^{\star}))]$, decreasing on $[(g^{\prime})_{1}^{-1}(g^{\prime}(\overline{x}^{\star})),(g^{\prime})_{2}^{-1}(g^{\prime}(\overline{x}^{\star}))]$, increasing on $[(g^{\prime})_{2}^{-1}(g^{\prime}(\overline{x}^{\star})),\underline{x}^{\star}]$, decreasing on $[\underline{x}^{\star},x)$. In addition, we have \eqref{7.20.4.6} and \eqref{7.20.4.7.1}.
Therefore, \eqref{eq:4.2} holds true in this sub-case.

If $g^{\prime}(\overline{x}^{\star})\leq g^{\prime}(a_{1})$, then the function $y\mapsto \int_{y}^{x}\left(1-\frac{g^{\prime}(z)}{g^{\prime}(\overline{x}^{\star})}\right)\mathrm{d}z$ is increasing on $[0,\underline{x}^{\star}]$, decreasing on $[\underline{x}^{\star},x)$. In addition, we have \eqref{7.20.4.7.1}. Consequently, \eqref{eq:4.2} holds true in this sub-case.

\item Case $a_{2}< \underline{x}^{\star}<a_{3}<\overline{x}^{\star}$ and $x\in[0,\underline{x}^{\star})$.

If $g^{\prime}(\overline{x}^{\star})> g^{\prime}(a_{1})$, then
\begin{align}
\int_{y}^{x}\left(1-\frac{g^{\prime}(z)}{g^{\prime}(\overline{x}^{\star})}\right)\mathrm{d}z&
\leq 
\int_{(g^{\prime})_{1}^{-1}(g^{\prime}(\overline{x}^{\star}))}^{(g^{\prime})_{2}^{-1}(g^{\prime}(\overline{x}^{\star}))}\left(1-\frac{g^{\prime}(z)}{g^{\prime}(\overline{x}^{\star})}\right)\mathrm{d}z
\nn\\
&\leq \beta,\nn
\end{align}
where, the last inequality follows by the definition of $(\underline{x}^{\star},\overline{x}^{\star})$ and the fact of $a_{2}< \underline{x}^{\star}<a_{3}<\overline{x}^{\star}$.
Therefore, \eqref{eq:4.2} holds true in this sub-case.

If $g^{\prime}(\overline{x}^{\star})\leq g^{\prime}(a_{1})$, then $\int_{y}^{x}\left(1-\frac{g^{\prime}(z)}{g^{\prime}(\overline{x}^{\star})}\right)\mathrm{d}z\leq 0\leq \beta$. Consequently, \eqref{eq:4.2} holds true in this sub-case.

\item Case $0\leq \underline{x}^{\star}<a_{1}<a_{3}<\overline{x}^{\star}$. Then $g^{\prime}(\overline{x}^{\star})\geq \max\{g^{\prime}(x_{1}), g^{\prime}(x_{2}),g^{\prime}(a_{1}), g^{\prime}(a_{3})\}$. Suppose that $g'(a_2)>g'(\overline{x}^{\star})$.

If $y\in [0,\underline{x}^{\star}]$, then the function $x\mapsto \int_{y}^{x}\left(1-\frac{g^{\prime}(z)}{g^{\prime}(\overline{x}^{\star})}\right)\mathrm{d}z$ is decreasing on $[y,\underline{x}^{\star}]$, increasing on $[\underline{x}^{\star},(g^{\prime})_{2}^{-1}(g^{\prime}(\overline{x}^{\star}))]$, decreasing on $[(g^{\prime})_{2}^{-1}(g^{\prime}(\overline{x}^{\star})),(g^{\prime})_{3}^{-1}(g^{\prime}(\overline{x}^{\star}))]$, increasing on $[(g^{\prime})_{3}^{-1}(g^{\prime}(\overline{x}^{\star})),\overline{x}^{\star}]$, and decreasing on $[\overline{x}^{\star},\infty)$. In addition, we have
\begin{align}
\int_{y}^{(g^{\prime})_{2}^{-1}(g^{\prime}(\overline{x}^{\star}))}\left(1-\frac{g^{\prime}(z)}{g^{\prime}(\overline{x}^{\star})}\right)\mathrm{d}z&\leq 
\int_{\underline{x}^{\star}}^{(g^{\prime})_{2}^{-1}(g^{\prime}(\overline{x}^{\star}))}\left(1-\frac{g^{\prime}(z)}{g^{\prime}(\overline{x}^{\star})}\right)\mathrm{d}z
\nn\\
&=\psi(\underline{x}^{\star},\overline{x}^{\star})-
\int_{(g^{\prime})_{2}^{-1}(g^{\prime}(\overline{x}^{\star}))}^{\overline{x}^{\star}}\left(1-\frac{g^{\prime}(z)}{g^{\prime}(\overline{x}^{\star})}\right)\mathrm{d}z
\nn\\
&\leq \beta, \quad (\text{ since } \overline{x}^{\star} \geq x_{2})\label{7.20.4.7}
\\
\int_{y}^{\overline{x}^{\star}}\left(1-\frac{g^{\prime}(z)}{g^{\prime}(\overline{x}^{\star})}\right)\mathrm{d}z&\leq \int_{\underline{x}^{\star}}^{\overline{x}^{\star}}\left(1-\frac{g^{\prime}(z)}{g^{\prime}(\overline{x}^{\star})}\right)\mathrm{d}z
\nn\\
&=\beta.\label{7.20.4.8}
\end{align}
Therefore, \eqref{eq:4.2} holds true in this sub-case.

If $y\in (\underline{x}^{\star}, \overline{x}^{\star}]$, then the function $x\mapsto \int_{y}^{x}\left(1-\frac{g^{\prime}(z)}{g^{\prime}(\overline{x}^{\star})}\right)\mathrm{d}z$ is increasing on $[y,y\vee (g^{\prime})_{2}^{-1}(g^{\prime}(\overline{x}^{\star}))]$, decreasing on $[y\vee(g^{\prime})_{2}^{-1}(g^{\prime}(\overline{x}^{\star})),y\vee(g^{\prime})_{3}^{-1}(g^{\prime}(\overline{x}^{\star}))]$, increasing on $[y\vee(g^{\prime})_{3}^{-1}(g^{\prime}(\overline{x}^{\star})),\overline{x}^{\star}]$, and decreasing on $[\overline{x}^{\star},\infty)$. In addition, \eqref{7.20.4.7} and \eqref{7.20.4.8} remain true. Hence \eqref{eq:4.2} holds true in this sub-case.

If $y\in (\overline{x}^{\star},\infty)$, then the function $x\mapsto \int_{y}^{x}\left(1-\frac{g^{\prime}(z)}{g^{\prime}(\overline{x}^{\star})}\right)\mathrm{d}z$ is decreasing on $[y,\infty)$. Consequently, \eqref{eq:4.2} holds true in this sub-case.

Suppose $g'(a_2)\leq g'(\overline{x}^{\star})$.
If $y\in[0,\underline{x}^{\star}]$, then the function $x\mapsto \int_{y}^{x}\left(1-\frac{g^{\prime}(z)}{g^{\prime}(\overline{x}^{\star})}\right)\mathrm{d}z$ is decreasing on $[y,\underline{x}^{\star}]$, increasing on $[\underline{x}^{\star},\overline{x}^{\star}]$, and decreasing on $[\overline{x}^{\star},\infty)$.Moreover,
\begin{align}\label{eq:lem4.1.1}
\int_y^y\left(1-\frac{g^{\prime}(z)}{g^{\prime}(\overline{x}^{\star})}\right)\mathrm{d}z&=0\leq\beta,\nn\\
\int_y^{\overline{x}^{\star}}\left(1-\frac{g^{\prime}(z)}{g^{\prime}(\overline{x}^{\star})}\right)\mathrm{d}z
&\leq\int_{\underline{x}^{\star}}^{\overline{x}^{\star}}\left(1-\frac{g^{\prime}(z)}{g^{\prime}(\overline{x}^{\star})}\right)\mathrm{d}z=\beta.
\end{align}
Therefore, \eqref{eq:4.2} holds true in this sub-case.

If $y\in(\underline{x}^{\star},\overline{x}^{\star}]$, then the function $x\mapsto \int_{y}^{x}\left(1-\frac{g^{\prime}(z)}{g^{\prime}(\overline{x}^{\star})}\right)\mathrm{d}z$ is increasing on $[y,\overline{x}^{\star}]$ and decreasing on $[\overline{x}^{\star},\infty)$. In addition, \eqref{eq:lem4.1.1} remains true. Consequently, \eqref{eq:4.2} holds true in this sub-case.

If $y\in(\overline{x}^{\star},\infty)$, then the function $x\mapsto \int_{y}^{x}\left(1-\frac{g^{\prime}(z)}{g^{\prime}(\overline{x}^{\star})}\right)\mathrm{d}z$ is decreasing on $[y,\infty)$. Hence, \eqref{eq:4.2} holds true in this sub-case.
\end{itemize}
The proof is complete.
\end{proof}

With the aid of Lemma \ref{lem.4.1}, we are ready to characterize the optimal dividend and reinsurance strategies, along with the value function, in the following Theorem \ref{thm4.1}.

\begin{theorem}\label{thm4.1}
Let $(\underline{x}^{\star},\overline{x}^{\star})$ and $u^{\star}$ be given by Theorems \ref{thm3.2.ww}, \ref{Theorem3.4}, \ref{Theorem3.5}, and \ref{Theorem3.6}. Then, the optimal strategy for the auxiliary problem \eqref{eq:a008.aux.}, denoted as $(u^{\star},D^{\underline{x}^{\star},\overline{x}^{\star}})$, constitutes the optimal strategy for the control problem \eqref{eq:a008} if and only if at least one of the following conditions is satisfied.
\begin{itemize}
\item[\emph{(a)}] 
$\mu_{-}\geq \mu_{+}$.
\item[\emph{(b)}] 
$\overline{x}^{\star}\geq a$.
\item[\emph{(c)}] 
$\overline{x}^{\star}<a$ and $\mu_+\leq0$.
\item[\emph{(d)}] 
$\overline{x}^{\star}< a$, $\mu_{+}>0$, $\mu_{-}< \mu_{+}$, and $\mu_+-q(a-\overline{x}^{\star})-qg(\overline{x}^{\star})/g^{\prime}(\overline{x}^{\star})\leq0$. 

\end{itemize}
In particular, if
\begin{itemize}
\item[\emph{(e)}] $g^{\prime\prime}(a+)\geq0$,
\end{itemize}
holds, then either one of the conditions \emph{(a)}, \emph{(b)}, \emph{(c)}, or \emph{(d)} is necessarily satisfied.
\end{theorem}

\begin{proof}
Recall that the value function $V_{\underline{x}^{\star},\overline{x}^{\star}}(x)$ can be rewritten as 
\begin{equation}\label{eq:4.V}
V_{\underline{x}^{\star},\overline{x}^{\star}}(x)=
\left\{\begin{array}{ll}
\frac{g(x)}{g^{\prime}(\overline{x}^{\star})},&x\in[0,\overline{x}^{\star}), \\
x-\overline{x}^{\star}+\frac{g(\overline{x}^{\star})}{g^{\prime}(\overline{x}^{\star})}=x-\underline{x}^{\star}-\beta+\frac{g(\underline{x}^{\star})}{{g^{\prime}(\overline{x}^{\star})}},& x\in[\overline{x}^{\star},\infty).
\end{array}
\right.
\end{equation}
We first prove 
\begin{align}
\label{eq:thm4.1.1}
\left\{
\begin{array}{ll}
\mathcal{M}V_{\underline{x}^{\star},\overline{x}^{\star}}(x)\leq0, & \quad x\in[0,\overline{x}^{\star}), \\
\mathcal{M}V_{\underline{x}^{\star},\overline{x}^{\star}}(x)=0, & \quad x\in[\overline{x}^{\star}, \infty),
\end{array}
\right.
\end{align}
which is equivalent to
\begin{align}
\label{7.21.4.11}
\left\{
\begin{array}{ll}
V_{\underline{x}^{\star},\overline{x}^{\star}}(x)-V_{\underline{x}^{\star},\overline{x}^{\star}}(y)\geq x-y-\beta, &\quad \text{for all }\,\, x\geq y\geq0,\\
V_{\underline{x}^{\star},\overline{x}^{\star}}(x)-V_{\underline{x}^{\star},\overline{x}^{\star}}(y)= x-y-\beta, &\quad \text{for all }\,\,x\geq \overline{x}^{\star}\,\, \text{ and some }\,\,y\in[0,x].
\end{array}
\right.
\end{align} 
We distinguish the following mutually exclusive and collectively exhaustive cases.
\begin{itemize}
\item If $y+\beta>x\geq y\geq0$, it holds that
$$V_{\underline{x}^{\star},\overline{x}^{\star}}(x)-V_{\underline{x}^{\star},\overline{x}^{\star}}(y)\geq0>x-y-\beta.$$
\item If $x\geq y\geq \overline{x}^{\star}$ and $x\geq y+\beta$, it holds that
$$V_{\underline{x}^{\star},\overline{x}^{\star}}(x)-V_{\underline{x}^{\star},\overline{x}^{\star}}(y)=x-y>x-y-\beta.$$
\item If $x\geq \overline{x}^{\star}\geq y\geq 0$ and $x\geq y+\beta$, it follows from \eqref{eq:4.2} that
\begin{eqnarray}
V_{\underline{x}^{\star},\overline{x}^{\star}}(x)-V_{\underline{x}^{\star},\overline{x}^{\star}}(y)
\hspace{-0.3cm}&=&\hspace{-0.3cm}
x-\overline{x}^{\star}+\frac{g(\overline{x}^{\star})-g(y)}{g^{\prime}(\overline{x}^{\star})}
\nn\\
\hspace{-0.3cm}&=&\hspace{-0.3cm}
x-y-\beta-\left[\overline{x}^{\star}-y-\frac{g(\overline{x}^{\star})-g(y)}{g^{\prime}(\overline{x}^{\star})}-\beta\right]
\nn\\
\hspace{-0.3cm}&=&\hspace{-0.3cm}
x-y-\beta-\left[\int_{y}^{\overline{x}^{\star}}\left(1-\frac{g^{\prime}(z)}{g^{\prime}(\overline{x}^{\star})}\right)\mathrm{d}z-\beta\right]
\nn\\
\hspace{-0.3cm}& \geq &~\hspace{-0.3cm}x-y-\beta.\nn
\end{eqnarray}
In particular, if $x\geq \overline{x}^{\star}\geq y= \underline{x}^{\star}$ (Recall that $\overline{x}^{\star}-\underline{x}^{\star}-\beta\geq \int_{\underline{x}^{\star}}^{\overline{x}^{\star}}\left(1-\frac{g^{\prime}(z)}{g^{\prime}(\overline{x}^{\star})}\right)\mathrm{d}z-\beta=0$), we have
\begin{align*}
V_{\underline{x}^{\star},\overline{x}^{\star}}(x)-V_{\underline{x}^{\star},\overline{x}^{\star}}(\underline{x}^{\star})& 
=\left(x-\underline{x}^{\star}\right)-\beta,\quad x\geq \overline{x}^{\star}.
\end{align*}

\item If $\overline{x}^{\star}\geq x\geq y+\beta\geq \beta$, it follows from \eqref{eq:4.2} that
\begin{align}
V_{\underline{x}^{\star},\overline{x}^{\star}}(x)-V_{\underline{x}^{\star},\overline{x}^{\star}}(y)&=\frac{g(x)-g(y)}{g^{\prime}(\overline{x}^{\star})}
\nn\\
&=x-y-\beta-\left[\int_{y}^{x}\left(1-\frac{g^{\prime}(z)}{g^{\prime}(\overline{x}^{\star})}\right)\mathrm{d}z-\beta\right]
\nn\\
&\geq x-y-\beta.\nn
\end{align}
\end{itemize}
Combining the above arguments yields \eqref{7.21.4.11}. Consequently, \eqref{eq:thm4.1.1} is verified.

We next prove that
\begin{align}
\label{30.add.new.x}
\frac{g(a)}{g^{\prime}(a)}\leq \frac{g(\overline{x}^{\star})}{g^{\prime}(\overline{x}^{\star})}+a-\overline{x}^{\star},\quad \text{if}\quad \overline{x}^{\star}\leq a.
\end{align}
Let $h(x):=\frac{g(x)}{g^{\prime}(x)}-x$. Then we have $h'(x)=-\frac{g(x)}{(g^{\prime}(x))^2}g^{\prime\prime}(x)$. By the construction of $g$ in Section \ref{subsec.3.1}--\ref{subsec.3.4}, one knows $g^{\prime\prime}(x)\geq0$ for $x\in (\overline{x}^{\star},a)$ in case of $\overline{x}^{\star}\leq a$. Therefore, $h(x)$ is decreasing on $(\overline{x}^{\star},a)$, which implies \eqref{30.add.new.x}.

We next verify that 
\begin{align}
\label{7.22.4.13}
\left\{
\begin{array}{ll}
\max\limits_{u\in[0,1]}\mathcal{L}^uV_{\underline{x}^{\star},\overline{x}^{\star}}(x) = 0,& x\in(0,\overline{x}^{\star})\backslash\{a,\overline{x}^{\star}\}, \\
\max\limits_{u\in[0,1]}\mathcal{L}^uV_{\underline{x}^{\star},\overline{x}^{\star}}(x)\leq 0,& x\in(\overline{x}^{\star},\infty)\backslash\{a,\overline{x}^{\star}\}.
\end{array}
\right.
\end{align}
By \eqref{eq:4.V} and the construction of $V_{\underline{x}^{\star},\overline{x}^{\star}}(x)$, we obtain $V_{\underline{x}^{\star},\overline{x}^{\star}}\in C^1(\mathbb{R}_{+})\cap C^2(\mathbb{R}_{+}\backslash\{a,\overline{x}^{\star}\})$ and 
\begin{eqnarray}\label{eq:thm4.1.2}
\max\limits_{u\in[0,1]}\mathcal{L}^uV_{\underline{x}^{\star},\overline{x}^{\star}}(x)=\mathcal{L}^{u^{\star}}V_{\underline{x}^{\star},\overline{x}^{\star}}(x)=0,\quad x\in(0,\overline{x}^{\star})\backslash\{a\}.
\end{eqnarray}
In addition, it holds that
\begin{align}
\label{7.23.4.15z}
\lim_{x\rightarrow \overline{x}^{\star}+}\mathcal{L}^{u}V_{\underline{x}^{\star},\overline{x}^{\star}}(x)
&=
\lim_{x\rightarrow \overline{x}^{\star}+}(\mu_+ \mathbf{1}_{\{x>a\}} +\mu_- \mathbf{1}_{\{x\leq a\}})u-qV_{\underline{x}^{\star},\overline{x}^{\star}}(\overline{x}^{\star})
\\ \nn
&=(\mu_+ \mathbf{1}_{\{\overline{x}^{\star}\geq a\}} +\mu_- \mathbf{1}_{\{\overline{x}^{\star}<a\}})u-qV_{\underline{x}^{\star},\overline{x}^{\star}}(\overline{x}^{\star}), \quad u\in[0,1],
\\
\lim_{x\rightarrow \overline{x}^{\star}-}\mathcal{L}^{u}V_{\underline{x}^{\star},\overline{x}^{\star}}(x)
&=\frac{1}{2}(\sigma^2_+ \mathbf{1}_{\{\overline{x}^{\star}>a\}} +\sigma^2_- \mathbf{1}_{\{\overline{x}^{\star}\leq a\}})u^2V^{\prime\prime}_{\underline{x}^{\star},\overline{x}^{\star}}(\overline{x}^{\star}-)
\nn\\
&\quad +(\mu_+ \mathbf{1}_{\{\overline{x}^{\star}>a\}} +\mu_- \mathbf{1}_{\{\overline{x}^{\star}\leq a\}})u-qV_{\underline{x}^{\star},\overline{x}^{\star}}(\overline{x}^{\star}), \quad u\in[0,1].\label{7.23.4.16.zzz}
\end{align}
Combining \eqref{eq:thm4.1.2}, \eqref{7.23.4.15z}, and \eqref{7.23.4.16.zzz} yields 
\begin{align*}
0\geq\lim_{x\rightarrow \overline{x}^{\star}-}\mathcal{L}^{u}V_{\underline{x}^{\star},\overline{x}^{\star}}(x) 
&=\frac{1}{2}(\sigma^2_+ \mathbf{1}_{\{\overline{x}^{\star}>a\}} +\sigma^2_- \mathbf{1}_{\{\overline{x}^{\star}\leq a\}})u^2V^{\prime\prime}_{\underline{x}^{\star},\overline{x}^{\star}}(\overline{x}^{\star}-)\\
&\quad\;+\left(\mu_{-}-\mu_{+}\right)u\mathbf{1}_{\{\overline{x}^{\star}= a\}}
+\lim_{x\rightarrow \overline{x}^{\star}+}\mathcal{L}^{u}V_{\underline{x}^{\star},\overline{x}^{\star}}(x), \quad u\in[0,1].
\end{align*}
Using this and the fact that $V^{\prime\prime}_{\underline{x}^{\star},\overline{x}^{\star}}(\overline{x}^{\star}-)\geq 0$ (Actually, from the explicit characterizations of $(\underline{x}^{\star},\overline{x}^{\star})$, one knows that $g^{\prime\prime}(\overline{x}^{\star}-)\geq 0$, which is equivalent to $V^{\prime\prime}_{\underline{x}^{\star},\overline{x}^{\star}}(\overline{x}^{\star}-)\geq 0$), one has
\begin{align}
\label{mu-+}
\left(\mu_{-}-\mu_{+}\right)u\mathbf{1}_{\{\overline{x}^{\star}= a\}}
+
\lim_{x\rightarrow \overline{x}^{\star}+}\mathcal{L}^{u}V_{\underline{x}^{\star},\overline{x}^{\star}}(x)\leq 0, \quad u\in[0,1].
\end{align}
We next prove 
\begin{align}
\label{7.23.4.17.zz}
\mathcal{L}^{u}V_{\underline{x}^{\star},\overline{x}^{\star}}(x)\leq 0, \quad (x,u)\in(\overline{x}^{\star},\infty)\backslash\{a\}\times[0,1].
\end{align}

\begin{itemize}
\item When Condition (a) holds true, by \eqref{mu-+} we have
\begin{align}
\label{7.23.4.17z}
\lim_{x\rightarrow \overline{x}^{\star}+}\mathcal{L}^{u}V_{\underline{x}^{\star},\overline{x}^{\star}}(x)\leq 0, \quad u\in[0,1].
\end{align}
Then, if $\overline{x}^{\star}\geq a$, we have
\begin{align}
\mathcal{L}^{u}V_{\underline{x}^{\star},\overline{x}^{\star}}(x)
&=\mu_+u-qV_{\underline{x}^{\star},\overline{x}^{\star}}(x)
\nn\\
&=\mu_+u-qV_{\underline{x}^{\star},\overline{x}^{\star}}(\overline{x}^{\star})-q\left[V_{\underline{x}^{\star},\overline{x}^{\star}}(x)-V_{\underline{x}^{\star},\overline{x}^{\star}}(\overline{x}^{\star})\right]
\nn\\
&= \lim_{x\rightarrow \overline{x}^{\star}+}\mathcal{L}^{u}V_{\underline{x}^{\star},\overline{x}^{\star}}(x)-q
\left[V_{\underline{x}^{\star},\overline{x}^{\star}}(x)-V_{\underline{x}^{\star},\overline{x}^{\star}}(\overline{x}^{\star})\right]
\nn\\
&\leq 0, \quad (x,u)\in(\overline{x}^{\star},\infty)\times[0,1].\nn
\end{align}
While, if $\overline{x}^{\star}< a$, we have
\begin{align}
\mathcal{L}^{u}V_{\underline{x}^{\star},\overline{x}^{\star}}(x)
&=(\mu_+ \mathbf{1}_{\{x>a\}} +\mu_- \mathbf{1}_{\{x\leq a\}})u-qV_{\underline{x}^{\star},\overline{x}^{\star}}(x)
\nn\\
&=\left[\mu_+ \mathbf{1}_{\{x>a\}} +\mu_- \mathbf{1}_{\{x\leq a\}}-\mu_{-}\right]u-q\left[V_{\underline{x}^{\star},\overline{x}^{\star}}(x)-V_{\underline{x}^{\star},\overline{x}^{\star}}(\overline{x}^{\star})\right]
\nn\\
&\quad +\mu_{-}u-q V_{\underline{x}^{\star},\overline{x}^{\star}}(\overline{x}^{\star})
\nn\\
&=\left[\mu_+ \mathbf{1}_{\{x>a\}} +\mu_- \mathbf{1}_{\{x\leq a\}}-\mu_{-}\right]u-q\left[V_{\underline{x}^{\star},\overline{x}^{\star}}(x)-V_{\underline{x}^{\star},\overline{x}^{\star}}(\overline{x}^{\star})\right]
\nn\\
&\quad +\lim_{x\rightarrow \overline{x}^{\star}+}\mathcal{L}^{u}V_{\underline{x}^{\star},\overline{x}^{\star}}(x)
\nn\\
&\leq 0, \quad (x,u)\in(\overline{x}^{\star},\infty)\times[0,1].\nn
\end{align}
That is, \eqref{7.23.4.17.zz} holds under Condition (a).

\item When Condition (b) holds true with $\overline{x}^{\star}> a$, it follows from \eqref{mu-+} that
\begin{eqnarray}
\mathcal{L}^{u}V_{\underline{x}^{\star},\overline{x}^{\star}}(x)
\hspace{-0.3cm}&=&\hspace{-0.3cm}\mu_+ u-qV_{\underline{x}^{\star},\overline{x}^{\star}}(x)\leq \mu_+ u-qV_{\underline{x}^{\star},\overline{x}^{\star}}(\overline{x}^{\star})
\nn\\
\hspace{-0.3cm}&=&\hspace{-0.3cm}
\lim_{x\rightarrow \overline{x}^{\star}+}\mathcal{L}^{u}V_{\underline{x}^{\star},\overline{x}^{\star}}(x)\leq -\left(\mu_{-}-\mu_{+}\right)u\mathbf{1}_{\{\overline{x}^{\star}= a\}}=0,\quad (x,u)\in(\overline{x}^{\star},\infty)\times[0,1].\nn
\end{eqnarray}
That is, \eqref{7.23.4.17.zz} holds under Condition (b) with $\overline{x}^{\star}> a$.

\item Suppose Condition (b) holds with $\overline{x}^{\star}=a$. From the definition of $\overline{x}^{\star}$ and the discussions in Subsections \ref{subsec.3.1}–\ref{subsec.3.4}, we know that $\overline{x}^{\star}=a$ is possible only in one of the following cases: $\mu_{\pm}\leq 0$, $(\text{{II}}_{3})$, $(\text{{III}}_{2})$, or $(\text{{IV}}_{4})$. Furthermore, under any of these four cases, we must have $a=\overline{x}^{\star}\in (a_{3},\infty)$. Consequently, $g^{\prime\prime}(a+)=g^{\prime\prime}(\overline{x}^{\star}+)\geq 0$, which verifies Condition (e). We omit the details of this verification here, as the sufficiency of Condition (e) for \eqref{7.23.4.17.zz} will be established later.

\item When Condition (c) holds true, from \eqref{mu-+} follows \eqref{7.23.4.17z}, that is
\begin{align}
\mu_{-} u-qV_{\underline{x}^{\star},\overline{x}^{\star}}(\overline{x}^{\star})\leq 0,\nn
\end{align}
which implies
\begin{align}
\mathcal{L}^{u}V_{\underline{x}^{\star},\overline{x}^{\star}}(x)
&=\left[\mu_{-} u-qV_{\underline{x}^{\star},\overline{x}^{\star}}(\overline{x}^{\star})\right]-q\left[V_{\underline{x}^{\star},\overline{x}^{\star}}(x)-V_{\underline{x}^{\star},\overline{x}^{\star}}(\overline{x}^{\star})\right]
\nn\\
&\leq 0,\quad (x,u)\in(\overline{x}^{\star},a]\times[0,1].\nn
\end{align}
Moreover, since $\mu_+\leq 0$, we have
\begin{align}
\mathcal{L}^{u}V_{\underline{x}^{\star},\overline{x}^{\star}}(x)
&=\mu_+ u-qV_{\underline{x}^{\star},\overline{x}^{\star}}(x)
\nn\\
&\leq -qV_{\underline{x}^{\star},\overline{x}^{\star}}(x)
\nn\\
&\leq 0,\quad (x,u)\in(a,\infty)\times[0,1].\nn
\end{align}
Therefore, \eqref{7.23.4.17.zz} holds under Condition (c).

\item When Condition (d) holds true, we have
\begin{eqnarray}
\mathcal{L}^{u}V_{\underline{x}^{\star},\overline{x}^{\star}}(x)
\hspace{-0.3cm}&=&\hspace{-0.3cm}\mu_- u-qV_{\underline{x}^{\star},\overline{x}^{\star}}(x)
\nn\\
\hspace{-0.3cm}&=&\hspace{-0.3cm}
\left[\mu_-u -qV_{\underline{x}^{\star},\overline{x}^{\star}}(\overline{x}^{\star})\right]-q\left[V_{\underline{x}^{\star},\overline{x}^{\star}}(x) -V_{\underline{x}^{\star},\overline{x}^{\star}}(\overline{x}^{\star})\right]
\nn\\
\hspace{-0.3cm}&=&\hspace{-0.3cm}
\lim_{x\rightarrow \overline{x}^{\star}+}\mathcal{L}^{u}V_{\underline{x}^{\star},\overline{x}^{\star}}(x)-q\left[V_{\underline{x}^{\star},\overline{x}^{\star}}(x) -V_{\underline{x}^{\star},\overline{x}^{\star}}(\overline{x}^{\star})\right]
\nn\\
\hspace{-0.3cm}&\leq&\hspace{-0.3cm}0,\quad (x,u)\in(\overline{x}^{\star},a]\times[0,1],\nn
\end{eqnarray}
and
\begin{eqnarray}
\mathcal{L}^{u}V_{\underline{x}^{\star},\overline{x}^{\star}}(x)
\hspace{-0.3cm}&=&\hspace{-0.3cm}\mu_+ u-qV_{\underline{x}^{\star},\overline{x}^{\star}}(x)
\nn\\
\hspace{-0.3cm}&=&\hspace{-0.3cm}\mu_+ u-q(x-\overline{x}^{\star})-qV_{\underline{x}^{\star},\overline{x}^{\star}}(\overline{x}^{\star})
\nn\\
\hspace{-0.3cm}&\leq&\hspace{-0.3cm}
\mu_+u-q(a-\overline{x}^{\star})-qg(\overline{x}^{\star})/g^{\prime}(\overline{x}^{\star})
\nn\\
\hspace{-0.3cm}&\leq&\hspace{-0.3cm}
\mu_+-q(a-\overline{x}^{\star})-qg(\overline{x}^{\star})/g^{\prime}(\overline{x}^{\star})
\nn\\
\hspace{-0.3cm}&\leq&\hspace{-0.3cm}0,\quad (x,u)\in(a,\infty)\times[0,1].\nn
\end{eqnarray}
Hence, \eqref{7.23.4.17.zz} holds under Condition (d).
\end{itemize}
Combining \eqref{eq:thm4.1.2} and \eqref{7.23.4.17.zz}, we conclude that \eqref{7.22.4.13} holds under any of conditions (a), (b), (c), or (d). Combining with \eqref{eq:thm4.1.1}, \eqref{7.22.4.13} and Theorem \ref{thm2.2}, we prove the first claim of Theorem \ref{thm4.1}.

We next prove the second claim of Theorem \ref{thm4.1}. To this end, suppose that Condition (e) holds. We claim that either the conditions (a), (b), (c), or (d) of Theorem \ref{thm4.1} is satisfied. If Condition (a), Condition (b) with $\overline{x}^{\star}>a$, or Condition (c) is true, then the claim follows immediately. Now suppose that none of Condition (a), Condition (b) with $\overline{x}^{\star}>a$, and Condition (c) holds. That is, either $\overline{x}^{\star}=a$ and $\mu_{-}<\mu_{+}\leq 0$, or 
$\overline{x}^{\star}\leq a$ and $\mu_{-}<\mu_{+}>0$. 
\begin{itemize}
\item When $\overline{x}^{\star}=a$ and $\mu_{-}<\mu_{+}\leq 0$, we have
\begin{align}
\lim_{x\rightarrow \overline{x}^{\star}+}\mathcal{L}^{u}V_{\underline{x}^{\star},\overline{x}^{\star}}(x)=&~\mu_+u-qV_{\underline{x}^{\star},\overline{x}^{\star}}(\overline{x}^{\star})
\nn\\
\leq&~-qg(a+)/g^{\prime}(a+)
\nn\\
\leq&~ 0,\quad u\in[0,1],\nn
\end{align}
implying that 
\begin{align} \mathcal{L}^{u}V_{\underline{x}^{\star},\overline{x}^{\star}}(x)=&~\mu_+u-qV_{\underline{x}^{\star},\overline{x}^{\star}}(\overline{x}^{\star})
-q\left[V_{\underline{x}^{\star},\overline{x}^{\star}}(x)-V_{\underline{x}^{\star},\overline{x}^{\star}}(\overline{x}^{\star})\right]
\nn\\
=&~\lim_{x\rightarrow \overline{x}^{\star}+}\mathcal{L}^{u}V_{\underline{x}^{\star},\overline{x}^{\star}}(x)
-q\left[V_{\underline{x}^{\star},\overline{x}^{\star}}(x)-V_{\underline{x}^{\star},\overline{x}^{\star}}(\overline{x}^{\star})\right]
\nn\\
\leq&~ 0,\quad (x,u)\in(\overline{x}^{\star},\infty)\times[0,1].\nn
\end{align}
That is, \eqref{7.23.4.17.zz} holds under the case where $\overline{x}^{\star}=a$ and $\mu_{-}<\mu_{+}\leq 0$.
\item When $\overline{x}^{\star}\leq a$ and $\mu_{-}<\mu_{+}>0$, by \eqref{30.add.new.x}, we have 
\begin{align}
\label{7.23.4.16}
&\mu_+-q(a-\overline{x}^{\star})-qg(\overline{x}^{\star})/g^{\prime}(\overline{x}^{\star})
\nn\\
\leq&~\mu_{+}-qg(a+)/g^{\prime}(a+)
\nn\\
=&~\frac{1}{g^{\prime}(a+)}\left(\mu_{+}g^{\prime}(a+)-qg(a+)\right).
\end{align}
By $g^{\prime\prime}(a+)\geq 0$ and the discussions in Subsections \ref{subsec.3.2} and \ref{subsec.3.4}, there must be $\delta_{0}>0$ such that $g^{\prime\prime}(x)\geq 0$ for all $x\in(a,a+\delta_{0})$, which implies that the maximizer $u^{\star}$ of \eqref{eq:a046} is given as $u^{\star}(x)\equiv 1$ on $(a,a+\delta_{0})$.
Consequently, it holds that
\begin{align}
\label{7.23.4.17}
\frac{1}{2}\sigma_{+}^{2}g^{\prime\prime}(x)+\mu_{+}g^{\prime}(x)-qg(x)&=\frac{1}{2}\sigma_{+}^{2}\left(u^{\star}(x)\right)^{2}g^{\prime\prime}(x)+\mu_{+}u^{\star}(x)g^{\prime}(x)-qg(x)
\nn\\
&=\max_{u\in[0,1]}\left(\frac{1}{2}\sigma_{+}^{2}u^{2}g^{\prime\prime}(x)+\mu_{+}u g^{\prime}(x)-qg(x)\right)
\nn\\
&=0, \quad x\in (a,a+\delta_{0}),
\end{align}
since $g$ solves \eqref{eq:a046} by construction.
Sending $x\downarrow a$ in \eqref{7.23.4.17} yields
\begin{align}
\label{7.23.4.19}
\mu_{+}g^{\prime}(a+)-qg(a+)=-\frac{1}{2}\sigma_{+}^{2}g^{\prime\prime}(a+)\leq 0.
\end{align}
Combining \eqref{7.23.4.16} and \eqref{7.23.4.19} 
establishes the validity of Condition (d) of Theorem \ref{thm4.1}.
\end{itemize}
The proof is complete.
\end{proof}

\begin{remark}
In Theorem \ref{thm4.1}, the following additional assumption
\begin{align}
\label{add.ass}
\mu_+-q(a-\overline{x}^{\star})-qg(\overline{x}^{\star})/g^{\prime}(\overline{x}^{\star})\leq0,
\end{align}
is imposed to facilitate the characterization of an optimal dividend and reinsurance strategy in the sub-case where $\overline{x}^{\star}< a$ and $\mu_{-}< \mu_{+}>0$. Since \eqref{add.ass}
serves as a sufficient condition guaranteeing optimality among all admissible strategies, whereas \textbf{Assumption A} and \textbf{Assumption B} (in Sub-sections \ref{subs.3.2.2} and \ref{subs.3.4.2}) ensure
optimality among admissible strategies whose dividend strategy is restricted to two-barrier impulsive dividend strategies, it is natural to expect that \eqref{add.ass} is more stringent than the latter two. The following analysis is devoted to verifying this conjecture.

Consider the Cases $(\text{\emph{II}}_{1})$-$(\text{\emph{II}}_{2})$ with $\mu_{-}\leq0<\mu_{+}$, $\underline{x}^{\star}=0$, and $\overline{x}^{\star}\in [0,a)$. 
Suppose Condition \emph{(d)} holds, which, by $g(0)=0$, is equivalent to
\begin{align}
&\mu_{+}-q(a-\overline{x}^{\star})-\frac{qg(\overline{x}^{\star})}{g^{\prime}(\overline{x}^{\star})}\leq0
\nn\\
\Longleftrightarrow&\,\,
\psi(0,\overline{x}^{\star})
=\overline{x}^{\star}-\frac{g(\overline{x}^{\star})}{g^{\prime}(\overline{x}^{\star})}\leq a-\frac{\mu_{+}}{q}.
\label{eq:remark4.3.condition.d}
\end{align}
Moreover, $\underline{x}^{\star}=0$ gives
$\beta=\psi(0,\overline{x}^{\star}).$ Hence, Condition \emph{(d)} is equivalent to
\begin{align}
\beta\leq a-\frac{\mu_{+}}{q}.
\label{eq:remark4.3.case.II.condition.d}
\end{align}
By Remark \ref{remark3.7}, we know that
\begin{align}
\beta\leq\psi(0,a_{3})
\label{eq:remark4.3.case.II.assumption.A}
\end{align}
is sufficient for \textbf{Assumption A}. We claim that
\begin{align}
\label{8.15.4.28}
\psi(0,a_3)> a-\mu_+/q.
\end{align} 
Indeed, under Cases \emph{$\text{II}_{1}$} with $a_{24}<0$, we have $a_3>a$, and hence
\begin{eqnarray}
\psi(0,a_3)
\hspace{-0.3cm}&=&\hspace{-0.3cm}a_3-\frac{g(a_3)}{g^{\prime}(a_3)}
\nn\\
\hspace{-0.3cm}&=&\hspace{-0.3cm}
a_3-\frac{a_{23}\left(-\frac{a_{24}\theta^2_-}{a_{23}\theta^2_+}\right)^{\frac{\theta_+}{\theta_+-\theta_-}}+a_{24}\left(-\frac{a_{24}\theta^2_-}{a_{23}\theta^2_+}\right)^{\frac{\theta_-}{\theta_+-\theta_-}}}{\theta_+ a_{23}\left(-\frac{a_{24}\theta^2_-}{a_{23}\theta^2_+}\right)^{\frac{\theta_+}{\theta_+-\theta_-}}+\theta_-a_{24}\left(-\frac{a_{24}\theta^2_-}{a_{23}\theta^2_+}\right)^{\frac{\theta_-}{\theta_+-\theta_-}}}
\nn\\
\hspace{-0.3cm}&=&\hspace{-0.3cm}a_3-\frac{a_{23}\left(-\frac{a_{24}\theta^2_-}{a_{23}\theta^2_+}\right)+a_{24}}{\theta_+a_{23}\left(-\frac{a_{24}\theta^2_-}{a_{23}\theta^2_+}\right)+\theta_-a_{24}}
\nn\\
\hspace{-0.3cm}&=&\hspace{-0.3cm}a_3-\frac{\mu_+}{q}
\nn\\
\hspace{-0.3cm}&>&\hspace{-0.3cm}a-\frac{\mu_+}{q},\nn
\end{eqnarray}
confirming \eqref{8.15.4.28}.
Under Case \emph{$\text{II}_{1}$} with $a_{24}\geq0$, note that
\begin{align}
&a<x_{0}^{+}:=\frac{\sigma_{+}^2 \left(1-\gamma_{+}\right)}{\mu_{+}}{-\frac{\gamma_+(e^{\delta_+a}-e^{\delta_-a})}{\delta_+e^{\delta_+a}-\delta_-e^{\delta_-a}}+a}\nn\\
\Longleftrightarrow&
-\frac{\gamma_+(e^{\delta_+a}-e^{\delta_-a})}{\delta_+e^{\delta_+a}-\delta_-e^{\delta_-a}}>-\frac{\sigma^2_+(1-\gamma_+)}{\mu_+},\nn
\end{align}
which implies
\begin{eqnarray}
\psi(0,a_3)\hspace{-0.3cm}&=&\hspace{-0.3cm}
x_0^+-\frac{g(x_0^+)}{g^{\prime}(x_0^+)} 
\nn\\
\hspace{-0.3cm}&=&\hspace{-0.3cm}a-\frac{(1-\gamma_+)\mu_+}{2q}-\frac{\gamma_+(e^{\delta_+a}-e^{\delta_-a})}{\delta_+ e^{\delta_+a}-\delta_-e^{\delta_-a}}
\nn\\
\hspace{-0.3cm}&>&\hspace{-0.3cm}
a-\frac{(1-\gamma_+)\mu_+}{2q}-\frac{\sigma^2_+(1-\gamma_+)}{\mu_+}
\nn\\
\hspace{-0.3cm}&=&\hspace{-0.3cm}
a-\frac{\mu_+}{2q}\frac{\mu_+^{2}}{\mu_+^{2}+2\sigma_{+}^{2}q}-\frac{\sigma^2_+\mu_+}{\mu_+^2+2q\sigma^2_+}
\nn\\
\hspace{-0.3cm}&\geq&\hspace{-0.3cm}
a-\frac{\mu_+}{2q}-\frac{\mu_+}{2q}
\nn\\
\hspace{-0.3cm}&=&\hspace{-0.3cm}
a-\frac{\mu_+}{q},\nn
\end{eqnarray}
confirming \eqref{8.15.4.28}.
Under Case \emph{$\text{II}_{2}$}, the inequality $a_{11}\theta_{+}^{2}+a_{12}\theta_{-}^{2}<0$ implies
		$-\frac{a_{12}\theta_{-}^{2}}{a_{11}\theta_{+}^{2}}>1$. Hence,
\begin{align}
\psi(0,a_3)
=&~
a_3-\frac{g(a_3)}{g^{\prime}(a_3)}
\nn\\
=&~
a+\frac{\ln{\frac{-a_{12}\theta_{-}^2}{a_{11}\theta_{+}^2}}}{\theta_{+}-\theta_{-}}-\frac{a_{11}\left(-\frac{a_{12}\theta^2_-}{a_{11}\theta^2_+}\right)^{\frac{\theta_+}{\theta_+-\theta_-}}+a_{12}\left(-\frac{a_{12}\theta^2_-}{a_{11}\theta^2_+}\right)^{\frac{\theta_-}{\theta_+-\theta_-}}}{\theta_+ a_{11}\left(-\frac{a_{12}\theta^2_-}{a_{11}\theta^2_+}\right)^{\frac{\theta_+}{\theta_+-\theta_-}}+\theta_-a_{12}\left(-\frac{a_{12}\theta^2_-}{a_{11}\theta^2_+}\right)^{\frac{\theta_-}{\theta_+-\theta_-}}}
\nn\\
=&~a+\frac{\ln{\frac{-a_{12}\theta_{-}^2}{a_{11}\theta_{+}^2}}}{\theta_{+}-\theta_{-}}-\frac{a_{11}\left(-\frac{a_{12}\theta^2_-}{a_{11}\theta^2_+}\right)+a_{12}}{\theta_+a_{11}\left(-\frac{a_{12}\theta^2_-}{a_{11}\theta^2_+}\right)+\theta_-a_{12}}
\nn\\
=&~a+\frac{\ln{\frac{-a_{12}\theta_{-}^2}{a_{11}\theta_{+}^2}}}{\theta_{+}-\theta_{-}}-\frac{\mu_+}{q}
\nn\\
>&~a-\frac{\mu_+}{q},\nn 
\end{align}
confirming \eqref{8.15.4.28}.
Combining \eqref{eq:remark4.3.case.II.condition.d} and \eqref{8.15.4.28} implies $$\beta\leq a-\frac{\mu_{+}}{q}<\psi(0,a_{3}).$$ 
Therefore, by Remark \ref{remark3.7}, Condition \emph{(d)} implies \textbf{Assumption A}. 

We next consider the Cases $(\text{\emph{IV}}_{6})$--$(\text{\emph{IV}}_{7})$ with $\mu_{\pm}>0$, $\overline{x}^{\star}<a$, $a_{2}=a$, and $\overline{x}^{\star}\in[a_{1},a_{2})$. 
Suppose Condition \emph{(d)} holds.
For any pair $(\underline{x},\overline{x})$ appearing in \textbf{Assumption B}, we have $\overline{x}\geq a_{3}$ and $g^{\prime}(\overline{x})<g^{\prime}(a_{2})$. Moreover, for $a_{3}\leq y< (g^{\prime})_{4}^{-1}(g^{\prime}(a_{2})),$
\begin{align*}
\frac{\mathrm{d}}{\mathrm{d}y}\psi\!\left(0,(g^{\prime})_{3}^{-1}(g^{\prime}(y))\right)
&=\frac{g^{\prime\prime}(y)}{[g^{\prime}(y)]^{2}}\int_{0}^{(g^{\prime})_{3}^{-1}(g^{\prime}(y))}g^{\prime}(z)\,\mathrm{d}z\geq0.
\end{align*}
Consequently,
\begin{align}
\psi\!\left(0,(g^{\prime})_{3}^{-1}(g^{\prime}(\overline{x}))\right) \geq\psi(0,a_{3}).\label{eq:remark4.3.assumption.B.bound}
\end{align}
It follows that
\begin{align}
\psi(0,\overline{x}^{\star})\leq\psi(0,a_{3}),
\label{eq:remark4.3.case.IV.assumption.B}
\end{align}
is sufficient for \textbf{Assumption B}. 
We claim that
\begin{align}
\psi(0,a_{3})>a-\frac{\mu_{+}}{q}.
\label{eq:remark4.3.case.IV.bound}
\end{align}
Indeed, under Case $(\text{\emph{IV}}_{6})$,
we have 
\begin{align*}
\psi(0,a_{3})
&=a_{3}-\frac{g(a_{3})}{g^{\prime}(a_{3})}\\
&=a+\frac{\ln\!\left(-\dfrac{a_{36}\theta_{-}^{2}}{a_{35}\theta_{+}^{2}}\right)}{\theta_{+}-\theta_{-}}-\frac{a_{35}\left(-\dfrac{a_{36}\theta_{-}^{2}}{a_{35}\theta_{+}^{2}}\right)^{\frac{\theta_{+}}{\theta_{+}-\theta_{-}}}+a_{36}\left(-\dfrac{a_{36}\theta_{-}^{2}}{a_{35}\theta_{+}^{2}}\right)^{\frac{\theta_{-}}{\theta_{+}-\theta_{-}}}}{\theta_{+}a_{35}\left(-\dfrac{a_{36}\theta_{-}^{2}}{a_{35}\theta_{+}^{2}}\right)^{\frac{\theta_{+}}{\theta_{+}-\theta_{-}}}+\theta_{-}a_{36}\left(-\dfrac{a_{36}\theta_{-}^{2}}{a_{35}\theta_{+}^{2}}\right)^{\frac{\theta_{-}}{\theta_{+}-\theta_{-}}}}\\ 
&=a_{3}-\frac{\theta_{+}+\theta_{-}}{\theta_{+}\theta_{-}}\\
&=a_{3}-\frac{\mu_{+}}{q}\\
&>a-\frac{\mu_{+}}{q},
\end{align*}
which confirms \eqref{eq:remark4.3.case.IV.bound}.
Under Case $(\text{\emph{IV}}_{7})$ with $a_{48}<0$, we have
\begin{align*}
\psi(0,a_{3})
&=a_{3}-\frac{g(a_{3})}{g^{\prime}(a_{3})}\\
&=x_{2}^{+}+\frac{\ln\!\left(-\dfrac{a_{48}\theta_{-}^{2}}{a_{47}\theta_{+}^{2}}\right)}{\theta_{+}-\theta_{-}}-\frac{a_{47}\left(-\dfrac{a_{48}\theta_{-}^{2}}{a_{47}\theta_{+}^{2}}\right)^{\frac{\theta_{+}}{\theta_{+}-\theta_{-}}}+a_{48}\left(-\dfrac{a_{48}\theta_{-}^{2}}{a_{47}\theta_{+}^{2}}\right)^{\frac{\theta_{-}}{\theta_{+}-\theta_{-}}}}{\theta_{+}a_{47}\left(-\dfrac{a_{48}\theta_{-}^{2}}{a_{47}\theta_{+}^{2}}\right)^{\frac{\theta_{+}}{\theta_{+}-\theta_{-}}}+\theta_{-}a_{48}\left(-\dfrac{a_{48}\theta_{-}^{2}}{a_{47}\theta_{+}^{2}}\right)^{\frac{\theta_{-}}{\theta_{+}-\theta_{-}}}}\\ 
&=a_{3}-\frac{\theta_{+}+\theta_{-}}{\theta_{+}\theta_{-}}\\
&=a_{3}-\frac{\mu_{+}}{q}\\
&>a-\frac{\mu_{+}}{q},
\end{align*}
confirming \eqref{eq:remark4.3.case.IV.bound}.
Under Case $(\text{\emph{IV}}_{7})$ with $a_{48}\geq0$, we have $a_{3}=x_{2}^{+}\geq a$ and
\begin{align*}
\psi(0,a_{3})=x_{2}^{+}-\frac{g(x_{2}^{+})}{g^{\prime}(x_{2}^{+})}=x_{2}^{+}-\frac{x_{2}^{+}+\dfrac{a_{46}}{a_{45}}a}{\gamma_{+}}=-\frac{a_{46}}{a_{45}}a-\frac{(1-\gamma_{+})\mu_{+}}{2q}.
\end{align*}
Since $x_{2}^{+}\geq a$, we have
\[-\frac{a_{46}}{a_{45}}a\geq a-\frac{\sigma_{+}^{2}(1-\gamma_{+})}{\mu_{+}}.\]
Consequently, \[ \psi(0,a_{3}) \geq a-\frac{\mu_{+}}{2q} >a-\frac{\mu_{+}}{q}. \]
That is, \eqref{eq:remark4.3.case.IV.bound} holds true. Combining \eqref{eq:remark4.3.condition.d} and \eqref{eq:remark4.3.case.IV.bound} yields
\[\psi(0,\overline{x}^{\star})\leq a-\frac{\mu_{+}}{q}<\psi(0,a_{3}).\]
Consequently, Condition \emph{(d)} implies the sufficient condition stated in \eqref{eq:remark4.3.case.IV.assumption.B} for \textbf{Assumption B}; hence, \textbf{Assumption B} holds. 
\end{remark}

\begin{remark}\label{rem:special.cases}
Let $\mu_- = \mu_+>0$ and $\sigma_- = \sigma_+$. We briefly discuss how our results reduce to those of \cite{Cadenillas06}. To this end, we further let $a=0$. Then, only Case $(\mathrm{IV}_1)$, i.e., $0=a\leq x_{1}^{+}\wedge x_{0}^{-}$, can occur. 
Since $x_{1}^{+}=x_{0}^{-}$ and $\gamma_{-}=\gamma_{+}$, it follows from \eqref{3.66.g.} that $a_{41}=a$ and $a_{42}=0$. Consequently, on $[0,x_{1}^{+}]$, the function $g_{41}$ reduces to (4.5) of \cite{Cadenillas06}, where $x_{1}^{+}$ here corresponds to $x_{0}$ in \cite{Cadenillas06}. Similarly, \eqref{3.72.g.} reduces to (4.11) and (4.12) of \cite{Cadenillas06}, which, in turn, forces $g_{41}$ to reduce to (4.9) of \cite{Cadenillas06} on $(x_{1}^{+},\infty)$. Moreover, $(\underline{x}^{\star}, \overline{x}^{\star})$ here reduces to $(\tilde{x},x_{1})$ in \cite{Cadenillas06}. Therefore, the results of Theorems 5.1 and 5.2 in \cite{Cadenillas06} (for $k=1$; indeed, the case $k\in(0,1)$ presents no essential difficulty) can be recovered from Theorems \ref{thm2.2} and \ref{thm4.1} of the present paper (noting that Condition~\emph{(a)} of Theorem \ref{thm4.1} is automatically satisfied). 

The results of \cite{Wang26} can also be recovered from those of the present paper. To this end, we replace the interval $[0,1]$ where the reinsurance ratio resides, by the singleton $\{1\}$. Consequently, the solution to \eqref{eq:a046} coincides with the function $g$ in \cite{Wang26}. The concavity-convexity properties of \(g\) in Sections \ref{subsec.3.1}-\ref{subsec.3.4} reduces to the corresponding results of Section 3 in \cite{Wang26}. Consequently, Theorem 4.1 of \cite{Wang26} can be recovered by Theorem \ref{thm4.1} of this paper. 
\end{remark}

\vspace{5mm}



\end{document}